\documentclass[12pt]{amsart}
\usepackage{a4wide}
\usepackage[english]{babel}

\usepackage{amssymb,amsfonts}
\usepackage{mathrsfs}
\usepackage{graphicx}
\usepackage[all,cmtip]{xy}
\usepackage{tikz}
\usepackage{tikz-cd}
\tikzset{dynkdot/.style={circle,draw,scale=.38}}
\usepackage{cases}
\usepackage{mathtools}
\usepackage{arydshln}
\usepackage{multirow}
\usepackage{amscd}
\usepackage{pb-diagram}
\usepackage{bm}
\usepackage{bbm}
\usepackage{dsfont}
\usepackage{extarrows}
\usepackage{hyperref}
\usepackage[alphabetic]{amsrefs}

\newtheorem{theorem}{Theorem}[section]

\newtheorem{proposition}[theorem]{Proposition}
\newtheorem{lemma}[theorem]{Lemma}
\newtheorem{setting}[theorem]{Setting}
\newtheorem{corollary}[theorem]{Corollary}

\theoremstyle{definition}
\newtheorem{definition}[theorem]{Definition}
\newtheorem{example}[theorem]{Example}
\newtheorem{proposition-definition}[theorem]{Proposition-Definition}
\newtheorem{definition-theorem}[theorem]{Definition-Theorem}

\theoremstyle{remark}
\newtheorem{remark}[theorem]{Remark}
\newtheorem{conjecture}[theorem]{Conjecture}
\numberwithin{equation}{section}

\def\TT{\mathbb{T}}

\def\QQ{\mathbb{Q}}

\def\Dcal{\mathcal{D}}
\def\Fcal{\mathcal{F}}

\newcommand{\hd}{\opname{hd}}

\newcommand{\opname}[1]{\operatorname{\mathsf{#1}}}
\newcommand{\Gr}{\opname{Gr}\nolimits}

\newcommand{\End}{\opname{End}}

\newcommand{\de}{\opname{deg}}

\newcommand{\Ext}{\opname{Ext}}

\newcommand{\Hom}{\opname{Hom}}

\newcommand{\add}{\opname{add}\nolimits}
\newcommand{\ie}{\textit{i.e.,}\ }
\newcommand{\confer}{\textit{cf.}\ }

\newcommand{\im}{\opname{im}\nolimits}

\newenvironment{blue}{\relax\color{blue}}{\relax}
\newcommand{\nc}{\newcommand}
\nc{\beb}{\begin{blue}}
\nc{\eb}{\end{blue}}

\newcommand{\diag}{\operatorname{diag}}
\newcommand{\GLS}{{\sf GLS}}
\newcommand{\injc}{\mathchar`-\mathsf{inj}}
\newcommand{\projc}{\mathchar`-\mathsf{proj}}
\newcommand{\cK}{\mathcal{K}}
\newcommand{\modfp}{\mathchar`-\mathsf{mod}_{\mathrm{fp}}}
\newcommand{\modfcp}{\mathchar`-\mathsf{mod}_{\mathrm{fcp}}}
\newcommand{\modfg}{\mathchar`-\mathsf{mod}_{\mathrm{fg}}}
\newcommand{\modfcg}{\mathchar`-\mathsf{mod}_{\mathrm{fcg}}}
\newcommand{\op}{\mathrm{op}}
\newcommand{\fM}{\mathfrak{M}}
\newcommand{\Modln}{\mathchar`-\mathsf{Mod}_{\mathrm{ln}}}
\newcommand{\Modpct}{\mathchar`-\mathsf{Mod}_{\mathrm{pct}}}
\newcommand{\Mod}{\mathchar`-\mathsf{Mod}}
\newcommand{\modfd}{\mathchar`-\mathsf{mod}_{\mathrm{fd}}}

\newcommand{\DWZ}{\mathsf{DWZ}}
\newcommand{\al}{\alpha}
\newcommand{\be}{\beta}
\newcommand{\ga}{\gamma}
\tikzcdset{scale cd/.style={every label/.append style={scale=#1}, cells={nodes={scale=#1}}}}

\usepackage{appendix}
\begin{document}
\sloppy
\title[Partial $F$-invariants]{Partial $F$-invariants and cluster categorifications}


\author{Peigen Cao}
\address{School of Mathematical Sciences, University of Science and Technology of China, Hefei, 230026, People's Republic of China
}
\email{peigencao@126.com}

\author{Ryo Fujita}
\address{Research Institute for Mathematical Sciences, Kyoto University, Kitashirakawa Oiwake-cho, Sakyo, Kyoto, 606-8502, Japan}
\email{rfujita@kurims.kyoto-u.ac.jp}

\author{Kota Murakami}
\address{Kyoto University Institute for Advanced Study (KUIAS), Kyoto University,
Yoshida Ushinomiya-cho, Sakyo-ku, Kyoto 606-8501, Japan}
\email{murakami.kota.8v@kyoto-u.ac.jp}

\thanks{P. C. was supported by the National Key R\&D Program of China
(2024YFA1013801).
R.F.\ was supported by JSPS KAKENHI Grant Numbers JP23K12955 and JP26K16962. K.M. was supported by JSPS KAKENHI Grant Numbers JP21J14653, JP23KJ0337, JP22H05107, JP26K16959.}


\dedicatory{}


\date{}


\begin{abstract}

The $F$-invariant in cluster algebras is a combinatorial invariant that unifies the $E$-invariant from additive categorification and the $\mathfrak{d}$-invariant from monoidal categorification. In this paper, we study its refinement, the partial $F$-invariant, and establish its mutation formula under changes of the initial seed. As an application, we prove a conjecture of Reading, which asserts that the non-compatible cluster variables can be separated by sign-coherence of $g$-vectors upon varying the initial seed. We further show that, for cluster monomials, the partial $F$-invariants coincide with both the partial $E$-invariants for reachable decorated representations of quivers with potentials and the pole orders of normalized $R$-matrices (or partial $\mathfrak{d}$-invariants) for finite-dimensional reachable simple modules over quantum affine algebras. As consequences, we obtain a combinatorial formula for the pole orders for reachable simple modules in terms of $q$-characters; we verify the conjectural explicit formula for the pole orders between Kirillov--Reshetikhin modules. 
\end{abstract}

\maketitle

\setcounter{tocdepth}{1}
\tableofcontents

\section{Introduction}

\subsection{Cluster algebras} The notion of cluster algebras was introduced by Fomin and Zelevinsky \cite{fz_2002}  as a combinatorial approach to the dual canonical bases (or upper global bases) of quantum groups and to the theory of total positivity in algebraic groups. They often arise as the coordinate rings of various spaces of Lie-theoretic origin, and admit several interesting categorifications, some of which we are going to mention below.

A cluster algebra is a $\mathbb Z$-subalgebra of a rational function field generated by a special set of
generators called {\em cluster variables}, which are grouped into overlapping subsets,
called {\em clusters}. A {\em seed} is a pair consisting of a cluster and a rectangular
integer matrix with a skew-symmetrizable principal part. One can obtain new seeds from a given one by a procedure called {\em mutation}.
The sets of cluster variables and clusters of a cluster algebra are determined by an initial seed and the iterated mutations.
 A {\em cluster monomial} is a monomial in variables from the same cluster. 
 
 \subsection{Combinatorial invariants for cluster algebras}
A fundamental problem in cluster algebras is to characterize when two cluster variables are contained in the same cluster. Much progress has been made by introducing different types of compatibility degrees or invariants. For example,
\begin{itemize}
\item Fomin and Zelevinsky \cite{fz-2003y} introduced the compatibility degree $(-\mid\mid-)$ on the set $\Phi_{\geq -1}$ of almost positive roots associated to a Cartan matrix $C$ of finite type. It is known that the almost positive roots in  $\Phi_{\geq -1}$ are in bijection with the cluster variables of a cluster algebra of finite type, \confer \cite{fz_2003f}. 
\item Fomin, Shapiro and Thurston \cite{FST08} introduced the intersection pairing $(-\mid-)$  for tagged arcs on marked surfaces. The reachable tagged arcs are in bijection with the cluster variables of cluster algebras of surface type.
\item Cao and Li \cite{cao-li-2020} introduced the $d$-compatibility degree
on the set of cluster variables using the components of the $d$-vectors, which can be defined for any skew-symmetrizable cluster algebra. 
\item Fu and Gyoda \cite{Fu-Gyoda} introduced the $f$-compatibility degree
on the set of cluster variables using the components of the $f$-vectors, which can also be defined for any skew-symmetrizable cluster algebra. 
\end{itemize}

Now we turn to the invariants that are of interest in this paper. In \cite{Cao-2023}, the first named author of this paper introduced two new $\mathbb Z$-valued invariants (the tropical invariant and the $F$-invariant) for a pair of good elements (e.g., cluster monomials) in a cluster algebra $\mathcal{A}$ with a compatible Poisson structure. 
  The tropical invariant $\langle u,u' \rangle_{\rm trop}$ for a pair of good elements $u$ and $u'$ of $\mathcal{A}$ is defined using tropicalization. More precisely, it is proved that the good element $u'$ determines a semifield homomorphism 
 \[\beta_{u'}\colon\mathbb Q_{\rm sf}(x_{1;t_0},\ldots,x_{m;t_0})\rightarrow \mathbb Z^{\max},\]
where $\mathbb Q_{\rm sf}(x_{1;t_0},\ldots,x_{m;t_0})$ is the universal semifield generated by the initial cluster variables (frozen and unfrozen) and $\mathbb Z^{\max}=(\mathbb Z,\;+,\;\max)$ is the tropical semifield. Then, the {\em tropical invariant} $\langle u,u' \rangle_{\rm trop}$ is defined by 
\[ \langle u,u' \rangle_{\rm trop}\coloneqq \beta_{u'}(u).\]
The {\em $F$-invariant} $(u\mid\mid u')_F$ is defined to be the symmetrized sum of the tropical invariants:
\[ (u\mid\mid u')_F\coloneqq \langle u,u' \rangle_{\rm trop}+\langle u',u \rangle_{\rm trop}.\]
It turns out that $(u\mid\mid u')_F$ is always a non-negative integer. Moreover, when both $u$ and $u'$ are cluster variables, we have $(u \mid\mid u')_F = 0$ if and only if $u$ and $u'$ belong to a common cluster, \confer \cite[Theorem 4.19]{Cao-2023}.
In this sense, the $F$-invariant can be seen as a compatibility degree of cluster variables. Although not immediately apparent from the definition, the $F$-invariant extends both Fomin--Zelevinsky's compatibility degree and Fu--Gyoda's $f$-compatibility degree, \confer \cite{Cao-2023}.

We consider the tropical invariant and the $F$-invariant to be canonical combinatorial invariants of a cluster algebra.
In fact, it is shown in \cite{Cao-2023,CCG-2026} that they admit natural categorical interpretations when our cluster algebra $\mathcal{A}$ admits categorifications. 

\subsection{Partial $F$-invariant}
In this paper, we discuss a combinatorial refinement of the $F$-invariant, which we call the partial $F$-invariant, as well as its categorical interpretations.

In what follows, given a non-zero polynomial $F=\sum_{{\bf v}\in\mathbb N^n}c_{{\bf v}}{\bf y}^{{\bf v}}\in\mathbb Z[y_1,\ldots,y_n]$  and a vector ${\bf r}\in \mathbb Z^n$, we denote 
$$F[{\bf r}]\coloneqq \max\{{\bf v}^T{\bf r}\mid c_{{\bf v}}\neq 0\}\in\mathbb Z.$$
The map $F[-]\colon\mathbb Z^n\rightarrow\mathbb Z$ is called a {\em tropical polynomial}.  Clearly, if $F$ has constant term $1$, then $F[{\bf r}]\in \mathbb Z_{\geq 0}$ for any ${\bf r}\in \mathbb Z^n$.

For an ordered pair $(u,u')$ of two good elements of our cluster algebra $\mathcal{A}$, we define the {\em partial $F$-invariant} to be the non-negative integer 
\[F_u^t[S({\bf g}_{u'}^t)^\circ],\] where $F_u^t\in\mathbb{Z}_{\ge 0}[y_1,\ldots,y_n]$ and $({\bf g}_u^t)^\circ\in\mathbb Z^n$ are the $F$-polynomial and $g$-vector (supported on unfrozen part) of $u$ with respect to a seed $t$ of $\mathcal A$, and $S=\diag(s_1,\ldots,s_n)$ is a chosen skew-symmetrizer for the exchange matrices of $\mathcal A$.  
The name comes from the fact \cite[Theorem 4.10]{Cao-2023} that the ordinary $F$-invariant $(u\mid\mid u')_F$ can be written as the symmetrized sum of them:
\begin{eqnarray}\label{eqn:def-F-inv-1}
    (u\mid\mid u')_F=F_u^t[S({\bf g}_{u'}^{t})^\circ]+F_{u'}^t[S({\bf g}_{u}^t)^\circ].
\end{eqnarray}
Thus, the partial $F$-invariant can be understood as a refinement of the $F$-invariant.   

As the first result of this paper, we give the mutation formula of the partial $F$-invariant under initial seed mutations, which allows us to give a new proof of the mutation-invariance of the $F$-invariant, \ie the value of the right-hand side of \eqref{eqn:def-F-inv-1} is independent of the choice of $t\in\mathbb T_n$.

\begin{theorem}[{Theorem \ref{thm:F-inv}}]
\label{main-thm:F-inv}
Let $\mathcal A$ be a cluster algebra of full rank and let $S=\diag(s_1,\ldots,s_n)$ be a fixed skew-symmetrizer for the exchange matrices of $\mathcal A$. Let  $u$ and $u'$ be two good elements in $\mathcal A$. Then the following statements hold.
    \begin{itemize}
        \item [(i)]  For any edge \begin{xy}(0,1)*+{t}="A",(10,1)*+{t'}="B",\ar@{-}^k"A";"B" \end{xy}  in $\TT_n$, we have
       \[   F_u^{t'}[S({\bf g}_{u'}^{t'})^\circ]-F_u^t[S({\bf g}_{u'}^t)^\circ]=s_k([-g_{k;u}^{t'}]_+[-g_{k;u'}^{t}]_+-[-g_{k;u}^t]_+[-g_{k;u'}^{t'}]_+).
        \]
        \item[(ii)] For any two vertices $t,t'\in\mathbb T_n$, we have 
       \[
            F_u^t[S({\bf g}_{u'}^t)^\circ]+F_{u'}^t[S({\bf g}_u^t)^\circ]=F_u^{t'}[S({\bf g}_{u'}^{t'})^\circ]+F_{u'}^{t'}[S({\bf g}_u^{t'})^\circ].
       \]
    In particular, the $F$-invariant $(u\mid\mid u')_F=F_u^t[S({\bf g}_{u'}^t)^\circ]+F_{u'}^t[S({\bf g}_u^t)^\circ]$ only depends on $u$ and $u'$, not on the choice of vertex $t\in\mathbb T_n$.
    \end{itemize}
\end{theorem}
\begin{remark}
We can remove the full rank condition, if we only consider cluster monomials rather than good elements (see Subsection \ref{sec:trivial-coeff}).
\end{remark}

An important result in cluster algebras concerning $g$-vectors is row sign-coherence, which says that if ${\bf z}=(z_1,\ldots,z_n)$ is a cluster of a cluster algebra $\mathcal A$ (with trivial coefficients), then the $i$-th components of the $g$-vectors ${\bf g}_{z_1}^{t},\ldots, {\bf g}_{z_n}^{t}$ with respect to any vertex $t\in\mathbb T_n$ are simultaneously
non-negative or simultaneously non-positive, \confer \cite{GHKK18}. In particular, it implies that if two cluster variables $u$ and $u'$ are in the same cluster, then $g_{i;u}^tg_{i;u'}^t\geq 0$ for any $i\in[1,n]$ and $t\in\mathbb T_n$. In \cite[Conjecture 8.21]{Reading-2014}, Reading conjectured that this property can be used to characterize the compatibility between cluster variables. As an application of the mutation formula of the partial $F$-invariant, we give a positive answer to Reading's conjecture.

\begin{theorem}[Theorem \ref{thm:sign-coherent}]
\label{main-thm:sign-coherent} Let $\mathcal A$ be a cluster algebra and let $u$ and $u'$ be two cluster monomials of $\mathcal A$. Then 
    the product $uu'$ is still a cluster monomial if and only if  $u$ and $u'$ are sign-coherent (see Definition \ref{def:sign-coherent}). In particular, Reading's conjecture is true.
\end{theorem}

\subsection{Partial $F$-invariant and additive categorification}
While the original definition of cluster algebras is purely combinatorial, a central theme in the development of cluster theory has been the search for a categorical framework that lifts the cluster combinatorics and interprets the cluster-theoretic invariants (e.g., $g$-vectors, $F$-polynomials, compatibility degrees). Such a program is known as {\em categorification}. Depending on whether the multiplication in cluster algebras is categorified by direct sums or tensor products, categorification is divided into {\em additive categorification} \cite{bmrrt_2006,DWZ10, BIRSm11,gls_2011,Amiot_2009,Plamondon-2011} and {\em monoidal categorification} \cite{HL_2010,HL16,kkko-2018,kkop-2024}.

The approach of categorifications has proven highly successful: it not only provides conceptual explanations for many combinatorial phenomena but also enables us to prove some results that are difficult to obtain purely algebraically. For example, the additive categorification approach shows that the $F$-polynomials have constant term~$1$ for skew-symmetric cluster algebras \cite{DWZ10}; the monoidal categorification approach was crucial to establish that the cluster monomials are contained in the dual canonical bases of the coordinate rings of unipotent subgroups of symmetric Kac--Moody groups \cite{kkko-2018}.

In additive categorification, the cluster combinatorics are categorified by several types of ``tilting theory'', including cluster tilting in $2$-Calabi--Yau categories~\cite{bmrrt_2006,IY08,Amiot_2009}, $\tau$-tilting in module categories of finite-dimensional algebras~\cite{air_2014}, and (two-term) silting in perfect derived categories~\cite{AI12,DK-2015}.

The theory of quivers with potentials was introduced by Derksen-Weyman-Zelevinsky \cite{DWZ08,DWZ10}, and it provides a powerful tool for proving several conjectures of Fomin-Zelevinsky~\cite{fomin_zelevinsky_2007} concerning $F$-polynomials and $g$-vectors \cite[Theorem~1.7]{DWZ10}.  It has also motivated the development of 
$\tau$-tilting theory \cite{air_2014} and the theory of general presentations \cite{DK-2015}. The theory of quivers with potentials was further developed in \cite{Amiot_2009,Keller-Yang_2011,Plamondon-2011}, where the perfect derived categories of (complete) Ginzburg dg algebras and generalized cluster categories from quivers with potentials were constructed and studied. 

Our next goal is to relate the partial $F$-invariant in cluster algebras to certain invariants from additive and monoidal categorifications. For the additive categorification considered in this paper, we adopt the seminal work of Derksen--Weyman--Zelevinsky \cite{DWZ08,DWZ10}, in which decorated representations of Jacobian algebras are used to categorify skew-symmetric cluster algebras. Recall that a {\em decorated representation} $\mathcal M=(M,V)$ of a Jacobian algebra $J=J(Q,W)$ (associated to a quiver with potential $(Q,W)$) is a pair consisting of a finite-dimensional $J$-module $M$ and a finitely generated semisimple $J$-module $V$.

 Derksen, Weyman, and Zelevinsky~\cite{DWZ10} defined   the $g$-vector ${\bf g}(\mathcal M)\in\mathbb Z^n$ and the $F$-polynomial $F_{\mathcal M}\in\mathbb Z[y_1,\ldots,y_n]$ for any decorated representation $\mathcal M$ of a Jacobian algebra $J=J(Q,W)$, where $n=|Q_0|$. For a pair of decorated representations $\mathcal M=(M,V)$ and $\mathcal M'=(M',V')$ of $J$, they also defined two integers $E^{\rm inj}(\mathcal M,\mathcal M')$\footnote{
 A similar quantity was also studied independently by Auslander--Reiten ~\cite[Theorem~1.4]{AR85} for Artin algebras.} and $E^{\rm sym}(\mathcal M,\mathcal M')$:
 \begin{eqnarray*}
     E^{\rm inj}(\mathcal M,\mathcal M')&\coloneq&\langle\underline{\dim} M, \mathbf{g}(\mathcal{M}')\rangle + \dim_{\mathbb C} \Hom_J(M, M' ),\\
     E^{\rm sym}(\mathcal M,\mathcal M')&\coloneq& E^{\rm inj}(\mathcal M,\mathcal M')+ E^{\rm inj}(\mathcal M',\mathcal M),
 \end{eqnarray*}
 which are called the {\em partial $E$-invariant} and {\em $E$-invariant} of the ordered pair $(\mathcal M,\mathcal M')$, respectively.
 
Since the partial $E$-invariant can be interpreted as the dimension of a certain Hom-space in the homotopy category $\mathcal K^{b}(J\mathchar`-\mathsf{inj})$ (see Definition \ref{def:g-F-E} and Remark \ref{rmk:def-DWZ}), we have \[E^{\rm inj}(\mathcal M,\mathcal M')\geq 0\quad \text{and}\quad E^{\rm sym}(\mathcal M,\mathcal M')\geq 0.\]
Derksen, Weyman and Zelevinsky  \cite[Corollary 7.2]{DWZ10} proved that  $E^{\rm sym}(\mathcal M,\mathcal M)=0$ for any {\em negative-reachable} (or simply {\em reachable}) decorated representation,  \ie those   corresponding to cluster monomials.

Given a decorated representation $\mathcal M$ of a quiver with non-degenerate potential $(Q,W)$ at the rooted vertex $t_0$ of the $n$-regular tree $\mathbb T_n$, by applying Derksen--Weyman--Zelevinsky's mutations \cite{DWZ08}, one can obtain a family of quivers with potentials and their decorated representations:
 \[[(Q,W)]\coloneqq \{(Q_t,W_t)\mid t\in\mathbb T_n\},\quad\quad [\mathcal M]\coloneqq\{\mathcal M^t\mid t\in\mathbb T_n\},
 \]
 where $\mathcal M^t$ is a decorated representation of $(Q_t,W_t)$. 
 
In the context of cluster categorification, Derksen–Weyman–Zelevinsky~\cite{DWZ10} showed that  each cluster monomial $u$ of the cluster algebra $\mathcal A_Q$ corresponds to a family of decorated representations $\{\mathcal M_u^t\mid t\in\mathbb T_n\}$. Moreover, the $g$-vector ${\bf g}_u^t$ and $F$-polynomial $F_u^t$ of $u$ with respect to vertex $t\in\mathbb T_n$ are categorified by the $g$-vector and $F$-polynomial of the decorated representation $\mathcal M_u^t$ of $(Q_t, W_t)$, \ie we have the following equalities: 
\begin{eqnarray}\label{eqn:cat-g-f}
    {\bf g}(\mathcal M_u^t)={\bf g}_u^t\quad \text{and}\quad F_{\mathcal M_u^t}=F_u^t.
\end{eqnarray}

The following result is a direct consequence of the mutation formula for the partial $F$-invariant (Theorem~\ref{main-thm:F-inv}) and that for the partial $E$-invariant established in \cite[Theorem~7.1]{DWZ10}. It serves as one of the key steps toward an explicit formula for the pole orders of normalized $R$-matrices, as discussed in \S\ref{intro:further}.

 \begin{theorem}[{Theorem \ref{thm:F-E-inv}}] \label{main-thm:F-E-inv}
Let $(Q,W)$ be a quiver with non-degenerate potential and $\mathcal A_Q$ the corresponding cluster algebra with trivial coefficients. 
Let $u, u'$ be two cluster monomials of $\mathcal A_Q$, and let  $\{\mathcal M_u^t\mid t\in\mathbb T_n\}$ and $\{\mathcal M_{u'}^t\mid t\in\mathbb T_n\}$ be the families of decorated representations corresponding to $u$ and $u'$, respectively. Then we have
\[E^{\rm inj}(\mathcal M_u^t,\mathcal M_{u'}^t)=F_u^t[{\bf g}_{u'}^t]\;\;\;\;\text{and}\;\;\;E^{\rm sym}(\mathcal M_u^t,\mathcal M_{u'}^t)=F_u^t[{\bf g}_{u'}^t]+F_{u'}^t[{\bf g}_{u}^t]=(u\mid\mid u')_F,
\]
for any vertex $t\in\mathbb T_n$.
\end{theorem}

    We note that the first equality above was previously observed by Fei in  \cite[Theorem~3.22]{fei_2019b} in his study of \emph{tropical $F$-polynomials} of modules over finite-dimensional algebras (see Remark~\ref{rmk-end} (ii)).

\subsection{Partial $F$-invariant and monoidal categorification}
We also discuss an interpretation of the partial $F$-invariant in monoidal categorifications of cluster algebras.
The notion of monoidal categorification of cluster algebras originates in the seminal paper \cite{HL_2010} by Hernandez--Leclerc, which is completely different in nature from additive categorifications above.

Let $\mathcal{C}$ be a monoidal abelian length category whose tensor product $\otimes$ is bi-exact. The Grothendieck ring $K_0(\mathcal{C})$ carries a canonical $\mathbb{Z}$-basis formed by the classes of simple objects, with positive structure constants.
We say that the category $\mathcal{C}$ gives a monoidal categorification of a cluster algebra $\mathcal{A}$ if there is a ring isomorphism $\varphi \colon K_0(\mathcal{C}) \to \mathcal{A}$ under which the cluster monomials correspond to some simple classes. 
Given such a monoidal categorification, a simple object $M \in \mathcal{C}$ is said to be {\em reachable} if $\varphi[M] \in \mathcal{A}$ is a cluster monomial.
Note that a reachable simple object $M$ is necessarily real, meaning that $M \otimes M$ is simple.  
A cluster corresponds to a collection of simple objects, any pair of whose members \emph{strongly commute} (\ie they have simple tensor products). 

Many known examples of monoidal categorification of cluster algebras arise from representation theory of quantum affine algebras \cite{HL_2010, Nak_cl, HL_2013, KQ14, HL16, Qin_2017, BC_cl, kkop-2020, kkop-2024}, that of (symmetric) quiver Hecke algebras \cite{kkko-2018}, and also perverse coherent sheaves on affine Grassmannians \cite{CW19}. 
In these cases, our monoidal category of interest is not a braided tensor category.
In fact, there are many pairs of simple objects $M$ and $N$ such that $M \otimes N$ is not isomorphic to $N \otimes M$. 
Nevertheless, such a monoidal category is endowed with ``weak braidings'', \ie a system of non-zero morphisms $\mathbf{r}_{M,N} \colon M \otimes N \to N \otimes M$, called  (renormalized) {\em $r$-matrices}, together with an integer $\Lambda(M,N) \in \mathbb{Z}$, called \emph{the $\Lambda$-invariant}, which measure the failure of the system $\{ \mathbf{r}_{M,N}\}$ to be braidings, roughly speaking.
The symmetrized quantity
\[ \mathfrak{d}(M,N) \coloneqq \frac{\Lambda(M,N) + \Lambda(N,M)}{2}, \]
called \emph{the $\mathfrak{d}$-invariant}, is known to be a non-negative integer, and to control the commutativity of real simple objects. 
More precisely, for simple objects $M$ and $N$, assuming that at least one of them is real, we have $\mathfrak{d}(M,N)=0$ if and only if $M$ and $N$ strongly commute.

The $\Lambda$-invariant plays a crucial role in the monoidal categorification of cluster algebras.
In the framework of \cite{kkko-2018, kkop-2020}, a natural compatibility between the $\Lambda$-invariant and the cluster structure is required so that the $\Lambda$-invariant gives rise to a Poisson structure on the categorified cluster algebra. A monoidal categorification satisfying such a requirement is called a \emph{$\mathsf{\Lambda}$-monoidal categorification}.  

Given that a cluster algebra $\mathcal{A}$ admits a $\mathsf{\Lambda}$-monoidal categorification $\varphi \colon K_0(\mathcal{C}) \to \mathcal{A}$, the first named author has shown in \cite[Theorem 5.16]{Cao-2023} that, for two good elements $u, u' \in \mathcal{A}$ and simple objects $M,M' \in \mathcal{C}$ such that $\varphi[M] = u$, $\varphi[M'] = u'$, assuming that at least one of $u,u'$ is a cluster monomial, we have the equality  
\[  \langle u,u' \rangle_{\rm trop} = \Lambda(M,M') \quad \text{and hence} \quad (u \mid\mid u')_F = 2\mathfrak{d}(M,M').\]
Thus, in this case, the $F$-invariant can be interpreted as the counterpart of the $\mathfrak{d}$-invariant in monoidal categorifications. 
More generally, we further show that the inequalities
\[  \langle u,u' \rangle_{\rm trop} \le \Lambda(M,M') \quad \text{and hence} \quad (u \mid\mid u')_F \le 2\mathfrak{d}(M,M')\]
hold even when neither $u$ nor $u'$ is a cluster monomial (see Theorem \ref{thm:trop-Lambda2}).

In this paper, we discuss a similar interpretation of the partial $F$-invariant in the case of monoidal categorifications arising from quantum affine algebras.
A quantum affine algebra $U_q'(\mathfrak{g})$ is the Drinfeld--Jimbo quantized enveloping algebra associated with an affine Kac--Moody algebra $\mathfrak{g}$.
Here we assume that the quantum parameter $q$ is not a root of unity. 
For any finite-dimensional simple $U_q'(\mathfrak{g})$-modules (of type ${\bf 1}$) $M, N$, the $\Lambda$-invariant $\Lambda(M,N)$ is defined to be a certain formal degree of the universal $R$-matrix specialized at the tensor product $M\otimes N_z$, where $N_z$ is a deformation of $N$ with a formal parameter $z$ such that $N_{z}|_{z=1} = N$. Considering a normalization of the specialized universal $R$-matrix with respect to the highest weight vector, we obtain the {\em normalized $R$-matrix} $R_{M,N_z}^{\rm norm}$, which depends on $z$ only rationally. Let $\mathfrak{o}(M,N)$ denote the {\em pole order} of $R_{M,N_z}^{\rm norm}$ at $z=1$, which is a non-negative integer by definition. 
 Then, the $\Lambda$-invariant can be written as
\[ \Lambda(M,N) = 2 \mathfrak{o}(M,N) + \mathscr{N}(M,N),\] 
with $\mathscr{N}(M,N)$ being skew-symmetric (\confer \cite{kkop-2020,Fujita-Oh-2021}), and hence we have
\[ \mathfrak{d}(M,N) = \mathfrak{o}(M,N) + \mathfrak{o}(N,M).\]
Thus, the pole order $\mathfrak{o}(M,N)$ gives a refinement of $\mathfrak{d}(M,N)$, analogously to the partial $F$-invariant refining the $F$-invariant in \eqref{eqn:def-F-inv-1}.
The study of the poles of normalized $R$-matrices has a long history in relation to the highest weight cyclicity criterion of tensor product modules \cite{AK97, Kas02, KKKO-2015}. It also appears as a key ingredient in the theory of generalized quantum affine Schur--Weyl duality \cite{KKK-2018R}, where it is important to know the precise value of $\mathfrak{o}(M,N)$ for real simple modules $M,N$ of interest, not only its vanishing.
However, it is very hard to compute these values in general.

In the following result, we compare the pole orders of normalized $R$-matrices and the partial $F$-invariants, and provide a sufficient condition for their coincidence on reachable modules. 
It provides a sufficient condition for the existence of a pole at $z=1$ for general simple modules and provides a computable formula of the pole order if at least one of the two modules is reachable. 

\begin{theorem}[Theorem \ref{thm:F-o} \& Corollary \ref{Cor:o=Fg=E}]\label{main-thm:F-o} 
Let $\mathcal C$ be a monoidal subcategory of the category of finite-dimensional representations of a quantum affine algebra.
Suppose that $\mathcal{C}$ gives a $\mathsf{\Lambda}$-monoidal categorification $\varphi \colon K_0(\mathcal{C}) \to \mathcal{A}$ of a skew-symmetric cluster algebra $\mathcal{A}$ such that $\varphi[N]$ is a good element for each simple module $N$ in $\mathcal{C}$, and there exists a monoidal cluster $\{ M_{i;w}\}_{1 \le i \le m}$ (a collection of simple modules in $\mathcal{C}$ such that $\{ \varphi[M_{i;w}]\}_{1 \le i \le m}$ is a cluster of $\mathcal{A}$) satisfying $\mathfrak{o}(M_{i;w}, N) = 0$\footnote{This condition is mild for our purposes, because on the cluster algebra side we have $F_{x_{i;w}}^w[{\bf r}]=0$ for any ${\bf r}\in\mathbb Z^n$, the reason being that the $F$-polynomials of initial cluster variables are trivial, namely
$F_{x_{i;w}}^w=1$.} for all $1 \le i \le m$ and any simple module $N$ in $\mathcal{C}$. Then, the following statements hold.
\begin{itemize}
    \item [(i)] We have $
\mathscr{N}(M,N) = ({\bf g}^{w}_M)^T \Lambda_{w} {\bf g}^{w}_N$ for any simple modules $M$ and $N$ in $\mathcal{C}$.
\item[(ii)]  We have $\mathfrak{o}(M,N) \geq F_{M}^{w}[({\bf g}^{w}_N)^\circ]$ for any simple modules $M$ and $N$ in $\mathcal{C}$. Here $({\bf g}^{w}_N)^\circ\in\mathbb Z^n$ denotes the principal part of ${\bf g}^{w}_N\in\mathbb Z^m$. If at least one of $M$ and $N$ is reachable, then this is an equality.
\end{itemize} 
Here ${\bf g}^{w}_M\in \mathbb Z^m$ and $F_{M}^{w}\in\mathbb Z[y_1,\ldots,y_n]$ respectively denote the extended $g$-vector and $F$-polynomial of $\varphi[M]$ with respect to the seed $w$ given by $\{ \varphi[M_{i;w}]\}_{1 \le i \le m}$.
\end{theorem}

This result can be applied to the standard examples of monoidal categorification introduced by Hernandez--Leclerc \cite{HL16} and by Kashiwara--Kim--Oh--Park \cite{kkop-2024}.
From now on, let us assume that our quantum affine algebra $U'_q(\mathfrak{g})$ is of untwisted type\footnote{This untwisted assumption is just for simplicity and not essential (see Remark~\ref{rem:twisted-case}).}, \ie $\mathfrak{g}$ is the (untwisted) affinization of a finite-dimensional simple Lie algebra $\mathfrak{g}_0$.
By Chari--Pressley \cite[Chapter 12]{CP-1994}, simple finite-dimensional $U_q'(\mathfrak{g})$-modules (of type ${\bf 1}$) are in bijection with the set $(1+z\mathbb{C}[z])^{\mathtt I_0}$ of $\mathtt I_0$-tuples of polynomials with constant term $1$, called the Drinfeld polynomials, with $\mathtt I_0$ being a labeling set of simple roots of $\mathfrak{g}_0$.
For a finite integer interval $[a,b] \subset \mathbb{Z}$, let $\mathcal{C}_{[a,b]}$ be the Serre subcategory generated by  the simple modules whose Drinfeld polynomials factorize into products of $(1-q^pz)$ with $p \in [a,b]$ (subject to a parity condition).
The category $\mathcal{C}_{[a,b]}$ is known to give a $\mathsf{\Lambda}$-monoidal categorification of a certain skew-symmetric cluster algebra $\mathcal{A}_{[a,b]}$. We show in Propositions \ref{pro:standard-cor} and \ref{prop:good} that the condition in the above Theorem \ref{main-thm:F-o} is satisfied by a standard initial monoidal cluster $\{M_{i;t_0}\}$ formed by certain Kirillov--Reshetikhin (KR) modules.

Since the $g$-vectors and $F$-polynomials of reachable simple modules with respect to the standard initial seed can be read off from their (truncated) $q$-characters (see Lemma \ref{lem:gFab}), by Theorem \ref{main-thm:F-o} (ii), we can calculate the pole orders for such modules directly from their $q$-characters, as we explain below. Recall that for a simple module $M \in \mathcal{C}_{[a,b]}$, its {\em $q$-character} in the sense of Frenkel--Reshetikhin \cite{FR99} is a Laurent polynomial $$\chi_q(M) \in \mathbb{Z}_{\ge 0}[Y_{i,p}^{\pm 1} \mid i \in \mathtt I_0, p \in \mathbb{Z}]$$ encoding the spectral decomposition of $M$ with respect to the action of the loop Cartan part of $U'_q(\mathfrak{g})$. 
Discarding the terms containing $Y_{i,p}^{\pm 1}$ with $p > b$ from $\chi_q(M)$, we get the {\em truncated $q$-character} $\chi^{-, b}_q(M)$, which can be written in the form:
\[ \chi^{-, b}_q(M) = \left( \prod_{i \in \mathtt I_0, p \in \mathbb{Z}}Y_{i,p}^{u_{i,p}(M)} \right) F_M^{-, b}({\bf A}^{-1}),\]
where $F_M^{-, b}({\bf A}^{-1}) \in \mathbb{Z}_{\ge 0}[A_{i,p}^{-1} \mid i \in \mathtt I_0, p \le b-d_i]$ is a polynomial with constant term $1$, $A_{i,p} \coloneqq Y_{i,p+d_i}Y_{i,p-d_i}\prod_{j \in \mathtt I_0}\prod_{k=1}^{-c_{ji}}Y^{-1}_{j,p+c_{ji}+2k-1}$ is an analog of the $i$-th simple root, $(c_{ij})_{i,j \in \mathtt I_0}$ is the Cartan matrix of $\mathfrak{g}_0$, $(d_i)_{i \in \mathtt I_0}$ is the minimal left symmetrizer of $(c_{ij})_{i,j \in \mathtt I_0}$. 
Note that $u_{i,p}(M) \in \mathbb{Z}_{\ge 0}$ is nothing but the multiplicity of the linear factor $(1-q^pz)$ in the $i$-th Drinfeld polynomial of $M$.

Combining  Theorem \ref{main-thm:F-o} (ii) with Lemma \ref{lem:gFab} yields the following result.
\begin{theorem}[Theorem~\ref{thm:oF}]\label{main-thm:oF}
Let $[a,b]$ be a finite integer interval and let $M,N$ be two simple objects of $\mathcal{C}_{[a,b]}$. 
Then we have
\begin{equation}\label{maineq:oF}
    \mathfrak{o}(M,N) \geq F^{-, b}_{M}[u_{i,p+d_i}(N)-u_{i,p-d_i}(N) \mid  i \in \mathtt{I}_0, p \le b-d_i],
\end{equation}
where the right-hand side denotes the tropical polynomial $F_M^{-, b}[-]$ evaluated at $A_{i,p}^{-1} = u_{i,p+d_i}(N)- u_{i,p-d_i}(N)$ for each $(i,p)$. 
In particular, if at least one of $M$ and $N$ is reachable in $\mathcal{C}_{[a,b]}$, then \eqref{maineq:oF} is an equality.
\end{theorem}

\begin{remark}
(i) The equality in the last statement of the above theorem shows that the right-hand side is independent of the choice of $[a,b]$, as long as $M,N\in \mathcal C_{[a,b]}$ and at least one of them is reachable in $\mathcal C_{[a,b]}$. 

(ii) The above theorem also admits a variant formulated in terms of (full) $q$-characters instead of the truncated ones (see Corollary~\ref{cor:oF2}). In any case, the result tells us that we have a lower bound of the pole order $\mathfrak{o}(M,N)$ in terms of the tropical evaluation of $q$-character of $M$ for any simple objects $M, N$ in our category.
In addition, the pole order can be explicitly computed for reachable modules $M, N$ once we know their $q$-characters $\chi_q(M), \chi_q(N)$.
\end{remark}

\subsection{Further applications}\label{intro:further}
Recall that for each finite integer interval $[a,b] \subset \mathbb{Z}$, we have a skew-symmetric cluster algebra $\mathcal{A}_{[a,b]}$, which is categorified by the monoidal subcategory $\mathcal{C}_{[a,b]}$ of finite-dimensional modules over the (untwisted) quantum affine algebra $U_q'(\mathfrak{g})$ associated with the finite-dimensional simple Lie algebra $\mathfrak{g}_0$. On the other hand, the cluster algebra $\mathcal{A}_{[a,b]}$ has an additive categorification determined by an explicit quiver with potential  $(\Gamma_{[a,b]}^\circ, W_{[a,b]})$ (see \S\ref{sec:poE}).
Since the partial $F$-invariant between cluster monomials in $\mathcal A_{[a,b]}$
admits two interpretations --- both as the pole orders for reachable simple modules in $\mathcal C_{[a,b]}$ and as the partial $E$-invariants for reachable decorated representations of the Jacobian algebra $A_{[a,b]}\coloneqq J(\Gamma_{[a,b]}^\circ, W_{[a,b]})$ --- it follows that the two quantities coincide. 
In the special case when $\mathfrak{g}_0$ is of simply-laced type and $b=a+3$, the Jacobian algebra $A_{[a,b]}$ is the path algebra of a Dynkin quiver and such coincidence recovers the main result of \cite{Fuj26}.

    On the monoidal categorification side, we have a global monoidal category  $\mathcal{C}_\mathbb{Z} = \bigcup_{[a,b]}\mathcal{C}_{[a,b]}$.
    A simple module $M \in \mathcal{C}_\mathbb{Z}$  is said to be {\em reachable} if there is a finite integer interval $[a,b]$ such that $M \in \mathcal{C}_{[a,b]}$ and it is reachable in $\mathcal{C}_{[a,b]}$. 
For example, all the KR modules in $\mathcal{C}_\mathbb{Z}$ are reachable in this sense.

On the additive categorification side, for each finite integer interval $[a,b]$, we have a Jacobian algebra $A_{[a,b]}$. By varying the interval $[a,b]\subseteq \mathbb Z$, the Jacobian algebras $A_{[a,b]}$ form a suitable projective system with respect to natural projections. Thus we have a global algebra $\Lambda(\infty)$ defined by taking the projective limit of this projective system, which can be given by an infinite quiver $\Gamma$ with a potential $W_{\Gamma}$ (see Proposition \ref{prop:Lambdainfty}).

As a consequence of Theorem~\ref{main-thm:F-E-inv}, Theorem~\ref{main-thm:F-o} and the study of the above limit process, we obtain the following fundamental result:

\begin{theorem}[Proposition \ref{prop:TX}, Theorem \ref{Thm:oMN=Ext1}] \label{main-thm:oMN=Ext1} Let $T_{[a,b]}$ be the functor defined by
\[T_{[a,b]}\coloneq \Hom_{\Lambda(\infty)}(A_{[a,b]}, -) \colon \Lambda(\infty)\injc \to A_{[a,b]}\injc.\]
We have the following results.
\begin{itemize}
    \item [(i)] Let $f \colon I_0 \to I_1$ and $f' \colon I_0' \to I_1'$ be two morphisms in the category $\Lambda(\infty)\injc$.
Then, for sufficiently large $[a,b]$, the derived functor of $T_{[a,b]}$ induces an isomorphism
\[ \Hom_{\mathcal{K}^b(\Lambda(\infty)\injc)}(f,f'[1]) \cong \Hom_{\mathcal{K}^b(A_{[a,b]}\injc)}(T_{[a,b]}f, T_{[a,b]}f'[1]).\]

\item[(ii)] Let $M$ and $N$ be two reachable simple modules in $\mathcal{C}_{\mathbb{Z}}$, and let $f_M$ and $f_N$ be the corresponding\footnote{The correspondence is given in Lemma \ref{Lem:KM}, which identifies the $g$-vector of $f_M$ and that of $M$.} two-term rigid complexes in $\mathcal{K}^{[0,1]}(\Lambda(\infty)\injc)$. Then we have
\[  \mathfrak{o}(M,N) = \dim_{\mathbb{C}}\Hom_{\mathcal{K}^b(\Lambda(\infty)\injc)}(f_M, f_N[1]).\]
\end{itemize}
\end{theorem}

The limit process studied in Theorem~\ref{main-thm:oMN=Ext1} is not merely a reformulation of the previous results; it also offers a new perspective on the pole orders and $q$-characters for reachable modules in $\mathcal{C}_\mathbb{Z}$, which we will explain below. Roughly speaking, it enables us to study these invariants via the graded modules over a graded algebra $\Pi(\infty)$.

As a consequence of a classical smash product construction~\cite{CM84,Bea88}, the unital modules over the algebra $\Lambda(\infty)$ are tautologically equivalent to $\mathbb{Z}$-graded modules over an infinite-dimensional Jacobian algebra $\Pi(\infty)$ associated with $\mathfrak{g}_0$ studied by Gei\ss--Leclerc--Schr\"oer in \cite{GLS-2017}.
More precisely,  the algebra $\Pi(\infty)$ is defined as a projective limit of a projective system given by $\Pi(l)$ for $l \in \mathbb{Z}_{>0}$, where each $\Pi(l)$ is a finite-dimensional algebra which is referred to as the {\em generalized preprojective algebra} of $\mathfrak{g}_0$ in \cite{GLS-2017}.
Through this identification, we may regard the two-term complex $f_M$ in Theorem~\ref{main-thm:oMN=Ext1} as an object of the bounded homotopy category $\mathcal{K}^b(\Pi(\infty)\injc^\mathbb{Z})$ of complexes of $\mathbb{Z}$-graded finitely cogenerated injective $\Pi(\infty)$-modules. This fits into the framework of \cite{FM} by the second and third named authors of this paper.

In \cite[\S4]{HL16}, Hernandez and Leclerc give a geometric $q$-character formula for KR modules via modules over the Jacobian algebra associated to a  semi-infinite quiver with potential. Their geometric $q$-character formula can be reformulated via the graded modules over $\Pi(\infty)$ as follows:
Let $M$ be a KR module and  $f_M\in \mathcal{K}^b(\Pi(\infty)\injc^\mathbb{Z})$ the corresponding two-term rigid complex. 
Then we have
\begin{equation} \label{eq:gqchintro}
\chi_q(M)= \left(\prod_{i,p}Y_{i,p}^{u_{i,p}(M)}\right) \sum_{\mathbf d_\bullet=(d_{j,s}) \in \bigoplus_{\mathtt I_0\times\mathbb Z} \mathbb N} \chi\left(\mathop{\mathsf{Gr}^{\mathbb Z}_{\mathbf d_\bullet}}(H^0(f_M))\right) \prod_{j\in\mathtt I_0,\,s\in\mathbb Z} A_{j,s}^{-d_{j,s}},
\end{equation}
where $\chi(\mathop{\mathsf{Gr}^{\mathbb{Z}}_{\mathbf{d}_{\bullet}}}(K))$ denotes the Euler characteristic of the submodule Grassmannian of a graded $\Pi(\infty)$-module $K$.
In \cite{HL16}, for a KR module $M$, the $0$-th cohomology $H^0 (f_M)$ was studied under the name of {\em generic kernel}.

Let $M,N\in\mathcal C_{\mathbb Z}$ be two KR modules. In \cite[\S5 \& Appendix A]{FM}, the second and third named authors computed the dimension of the vector space $\Hom_{\mathcal{K}^b(\Pi(\infty)\injc^\mathbb{Z})}(f_M, f_N[1])$ explicitly in terms of the inverse quantum Cartan matrix (see Proposition \ref{Prop:HomF} below). 
A comparison with the (conjectural) denominator formulas \cite[Conjecture 6.7]{Fujita-Oh-2021} of normalized $R$-matrices between KR modules led the authors to the conjectural equality    
\begin{equation} \label{eq:o=Eintro}
\mathfrak{o}(M,N) = \dim_{\mathbb{C}}\Hom_{\mathcal{K}^b(\Pi(\infty)\injc^\mathbb{Z})}(f_M, f_N[1])  
\end{equation}
for any KR modules $M,N \in \mathcal{C}_{\mathbb{Z}}$ \cite[Conjecture 5.17]{FM}. By applying
Theorem~\ref{main-thm:oMN=Ext1} (ii)  to KR modules, we obtain the equality \eqref{eq:o=Eintro}. As a result, we obtain the following explicit formula of the pole orders for KR modules.

\begin{theorem}[{Theorem \ref{thm:oKR}, \confer \cite[Conjecture 6.7]{Fujita-Oh-2021}, \cite[Conjecture 5.17]{FM}}] \label{main-thm:oKR}
Let $M=W^{(i)}_{k,q^{p+d_i}}$ and $N=W^{(j)}_{l,q^{s+d_j}}$ be two KR modules in $\mathcal C_{\mathbb Z}$ in the notation of \eqref{eq:KR}. Then we have
\[
\mathfrak{o}(W^{(i)}_{k,q^{p+d_i}}, W^{(j)}_{l,q^{s+d_j}}) = 
\left[ F_{i,k;j,l}(q)\right]_{s-p+ld_j-kd_i},
\]
where $F_{i,k; j,l}(q)$ is an explicit polynomial in $\mathbb{Z}[q]$ given by \eqref{eq:Fikjl} and $[F_{i,k; j,l}(q)]_u$ denotes the coefficient of $q^u$ in $F_{i,k; j,l}(q)\in \mathbb Z[q]$. 
\end{theorem}

We remark that the recent preprint by Oh--Scrimshaw \cite{OS} also computes $\mathfrak{o}(M,N) $ for KR modules in many cases independently, using a different method.

  The following result, which is a consequence of the additive and monoidal categorifications of the cluster algebras $\mathcal{A}_{[a,b]}$ and the limit process, extends Hernandez--Leclerc's geometric $q$-character formula from KR modules to arbitrary reachable simple modules in $\mathcal{C}_{\mathbb{Z}}$.  The geometric $q$-character formula below is formulated in terms of the  $\mathbb Z$-graded representation theory of the generalized preprojective algebra $\Pi(\infty)$.

\begin{theorem}[Theorem \ref{thm:qchformula}] \label{main-thm:qchformula} 
The geometric $q$-character formula \eqref{eq:gqchintro} holds for any reachable simple module $M$ in $\mathcal C_{\mathbb Z}$.
\end{theorem}
We remark that the geometric $q$-character formula for $\chi_q(M)$ admits a natural specialization to a geometric formula for the classical character $\chi(M)$ (see Corollary~\ref{cor:qchformula}).

    As a future direction of \S \ref{intro:further}, we make a concluding remark.
    The infinite quiver $\Gamma$ defining $\Lambda(\infty)$ does not give the same cluster algebra corresponding to $\mathcal{C}_\mathbb{Z}$.
    Related to this, there are more general types of two-term rigid complexes in $\mathcal{K}^b(\Pi(\infty)\injc^\mathbb{Z})$ and the generating functions of Euler characteristics, which are not necessarily given in the manner of Theorem~\ref{main-thm:oMN=Ext1} (ii) and Theorem~\ref{main-thm:qchformula} but are explained from the viewpoints of generalized preprojective algebras.
    In a forthcoming project with Bernard Leclerc, we will return to such objects in $\mathcal{K}^b(\Pi(\infty)\injc^\mathbb{Z})$ in relation to the theory of shifted quantum affine algebras including infinite-dimensional real simple representations (Remark~\ref{rem:shifted}).

\subsection*{Organization}
In Section \ref{sec2}, we recall the basic notions of cluster algebras and the definitions of good elements, tropical invariant and $F$-invariant in cluster algebras.

In Section \ref{sec:3}, we give the mutation formula for the partial $F$-invariant under the initial seed mutations (see Theorem \ref{thm:F-inv}). As an application, we prove a conjecture of Reading \cite[Conjecture 8.21]{Reading-2014}, which asserts that the non-compatible cluster variables can be separated by the sign-coherence of $g$-vectors upon varying the initial seeds (see Theorem \ref{thm:sign-coherent}).

In Section \ref{sec:4},  we recall the notions of $g$-vectors, $F$-polynomials, and the (partial) $E$-invariant within the theory of quivers with potentials, along with several auxiliary results needed for our proof. We then show that the partial $E$-invariant and the partial $F$-invariant coincide on cluster monomials, by comparing their respective mutation formulas under the initial seed mutations (see Theorem \ref{thm:F-E-inv}).

In Section \ref{sec:pole-order}, we compare the pole orders of normalized $R$-matrices and the partial $F$-invariants, and provide a sufficient condition for them to coincide on cluster monomials (see Theorem \ref{thm:F-o} \& Corollary \ref{Cor:o=Fg=E}). We show that this sufficient condition is satisfied automatically for the standard examples of monoidal categorification constructed in \cite{kkop-2024}. Thanks to the coincidence of pole orders with partial $F$-invariants, and the fact that $F$-polynomials and $g$-vectors can be read off from $q$-characters, we obtain an explicit formula for the pole orders between reachable simple modules in terms of $q$-characters (see Theorem \ref{thm:oF} \& Corollary \ref{cor:oF2}).

In Section \ref{sec: application}, we introduce two algebras $\Lambda(\infty)$ and $\Pi(\infty)$ to study the pole orders and $q$-characters. The unital modules over $\Lambda(\infty)$ can be identified with graded modules over the graded algebra $\Pi(\infty)$ (see Proposition \ref{rem:compare-gpa}). We prove that the pole orders between reachable simple modules can be expressed in terms of two-term complexes of injective modules over $\Lambda(\infty)$ (see Theorem \ref{Thm:oMN=Ext1}). As a consequence, we verify the conjectural explicit formula \cite[Conjecture 5.17]{FM} (refining \cite[Conjecture 6.7]{Fujita-Oh-2021}) of the pole orders of the normalized $R$-matrices between KR modules (see Theorem \ref{thm:oKR}).
 We also give an interpretation of the $q$-character of reachable modules in terms of the graded representation theory of $\Pi(\infty)$ (see Theorem \ref{thm:qchformula}).

In Appendix~\ref{sec: appendixA}, we give an introduction to the theory of decorated representations of infinite-dimensional pseudo-compact algebras, which includes Jacobian algebras of quivers with potentials as special cases.
As a consequence, we obtain a homological interpretation of the partial $E$-invariants and the $g$-vectors of decorated representations of Jacobian algebras introduced by \cite{DWZ10} in terms of the bounded homotopy category of complexes of relative injective objects (see Theorem~\ref{thm:g-DWZ}). We remark that the representation-theoretic $g$-vectors of modules over similar pseudo-compact algebras were also studied in \cite{GHLFS26}, after the first version of this paper was posted on arXiv.

\subsection*{Conventions and notation}
Throughout this paper, we work with cluster algebras whose frozen variables are not inverted. Vectors are understood to be column vectors unless otherwise stated. All the categories in this paper are essentially small. The rings/algebras considered in this paper are not necessarily unital. The non-unital cases will be clearly indicated.
We refer to a (graded) module over a (graded) ring $A$ as a (graded) left $A$-module. When we discuss right $A$-modules, we naturally regard left $A^\op$-modules as right $A$-modules.

We say a (graded) module $M$ over a (graded) unital ring $A$ is {\em finitely generated}  if for any collection of (graded) submodules $\{M_i\}_{i\in I}$ such that $\sum_{i\in I}M_i = M$, there exists a finite subset $F\subseteq I$ such that $\sum_{i'\in F}M_{i'} = M$.
Dually, a (graded) module $M$ over a (graded) unital ring $A$ is {\em finitely cogenerated} if for any collection of (graded) submodules $\{M_i\}_{i\in I}$ such that $\bigcap_{i\in I}M_i = \{0\}$, there exists a finite subset $F\subseteq I$ such that $\bigcap_{i'\in F}M_{i'} = \{0\}$, \confer \cite[\S 10]{AF92}.

 We refer to a variety as a reduced separated scheme of finite type over an algebraically closed field, whose topology is the Zariski topology.

\section{Preliminaries}\label{sec2}

Throughout this paper, we fix a pair $(n,m)$ of integers with $m\geq n > 0$. For any positive integer $r$, we denote $[1,r] \coloneqq   \{1,\ldots,r\}$.
The matrices in this paper are always integer matrices.
\subsection{Mutation matrices and compatible pairs}

Recall that an $n\times n$ integer matrix $B$ is said to be {\em skew-symmetrizable}, if there exists a diagonal integer matrix $S=\diag(s_1,\ldots,s_n)$ with $s_i>0$ ($i\in[1,n]$) such that $SB$ is skew-symmetric. Such a diagonal matrix $S$ is called a {\em skew-symmetrizer} of $B$.

An $m\times n$ integer matrix $\widetilde B=\begin{bmatrix}
       B\\ P
   \end{bmatrix}=(b_{ij})$ is called a {\em mutation matrix}, if its top $n\times n$ submatrix $B$ is skew-symmetrizable. The submatrix $B$ is called the {\em principal part} of $\widetilde B$.

Let $\widetilde B=\begin{bmatrix}
       B\\ P
   \end{bmatrix}=(b_{ij})$ be an $m\times n$ mutation matrix. The {\em mutation} of $\widetilde B$
   in direction $k\in[1,n]$ is
defined to be the new integer matrix $\mu_k(\widetilde B)=\widetilde B'=\begin{bmatrix}
       B'\\ P'
   \end{bmatrix}=(b_{ij}')$ given by
\begin{equation}\label{eqn:b-mutation}
b_{ij}^\prime=\begin{cases}-b_{ij}, & \text{if}\;i=k\;\text{or}\;j=k;\\
 b_{ij}+[b_{ik}]_+[b_{kj}]_+-[-b_{ik}]_+[-b_{kj}]_+,&\text{otherwise},\end{cases}
\end{equation}
where $[a]_+ \coloneqq   \max\{a,0\}$ for any $a\in\mathbb R$.

The following statements are well-known,  \confer\cite[Proposition 4.5]{fz_2002}, \cite[Lemma 3.2]{bfz_2005}.
\begin{itemize}
    \item $\mu_k(\widetilde B)=\widetilde B'$ is still a mutation matrix, namely, the submatrix $B'$ of $\widetilde B'$ is skew-symmetrizable;
    \item $B$ and $B'$ share the same skew-symmetrizer;
    \item $\mu_k$ is an involution,  \ie $\mu_k^2=id$;
    \item The rank of $ \widetilde B'=\mu_k(\widetilde B)$ is equal to the rank of $\widetilde B$.
\end{itemize}

 Let $\Lambda=(\lambda_{ij})$ be an $m\times m$ skew-symmetric integer matrix, and let
$\widetilde B=\begin{bmatrix}
       B\\ P
   \end{bmatrix}=(b_{ij})$ be
 an $m\times n$ integer matrix, where $B$ is the top $n\times n$ submatrix of $\widetilde B$.  Here we do not assume that $B$ is skew-symmetrizable.

\begin{definition}[Compatible pair, \cite{bz-2005}]
   Keep $\Lambda$ and $\widetilde B$ as above. Let $S=\diag(s_1,\ldots,s_n)$ be a diagonal integer matrix with $s_j>0,\;j\in[1,n]$. The pair $(\widetilde B,\Lambda)$ is called a {\em compatible pair} of type $S$, if for any $j\in[1,n]$ and $i\in[1,m]$, we have
   \begin{eqnarray}\label{eqn:lpair}
    \sum_{k=1}^mb_{kj}\lambda_{ki}=\delta_{i,j}s_j.
   \end{eqnarray}
 In other words,  $\widetilde B^T\Lambda=(S\mid {\bf 0})$, where  ${\bf 0}$ is the $n\times (m-n)$ zero matrix.
\end{definition}

The matrix $\Lambda$ in a compatible pair $(\widetilde B,\Lambda)$ is called a {\em Poisson coefficient matrix}.

\begin{proposition}[{\cite[Proposition 3.3]{bz-2005}}]
\label{pro:fullrank}
Let $(\widetilde B,\Lambda)$ be a compatible pair of type $S$ and $B$ the top $n\times n$ submatrix of $\widetilde B$. Then the matrix $\widetilde B$ has full rank $n$ and $SB=\widetilde B^T\Lambda \widetilde B$. In particular,  $B$  is skew-symmetrizable.
\end{proposition}

Let $(\widetilde B,\Lambda)$ be a compatible pair of type $S$ and let $k\in [1,n]$. We define 
a new pair $(\widetilde B',\Lambda')$, where $\widetilde B'=\mu_k(\widetilde B)$ and $\Lambda'=(\lambda_{ij}')$ is given as follows:
\begin{eqnarray}\label{eqn:L-mutation}
    \lambda_{ij}'=\begin{cases}-\lambda_{ik}+\sum_{l=1}^m[- b_{lk}]_+\lambda_{il},&\text{if }i\neq k\text{ and }j=k;\\
-\lambda_{kj}+\sum_{l=1}^m[- b_{lk}]_+\lambda_{lj},&\text{if }i=k \text{ and }j\neq k;\\
\lambda_{ij},&\text{otherwise}.
\end{cases}
\end{eqnarray}

\begin{proposition}[{\cite[Proposition 3.4]{bz-2005}}] \label{pro:pm} The new pair $(\widetilde B',\Lambda')$ is still a compatible pair of type $S$.
\end{proposition}
Keep the above setting. We call the new pair $(\widetilde B',\Lambda')$ the {\em mutation of the compatible pair}  $(\widetilde B,\Lambda)$ in direction $k$ and denote $(\widetilde B',\Lambda')=\mu_k(\widetilde B,\Lambda)$. It is known from \cite[Proposition 3.6]{bz-2005} that  $\mu_k(\widetilde B',\Lambda')=(\widetilde B,\Lambda)$,  \ie $\mu_k$ is an involution.

\subsection{$Y$-pattern and cluster pattern}
Recall that we fixed a pair $(n,m)$ of integers with $m\geq n >0$. Let $\mathbb F$ be the field of rational functions over $\mathbb Q$ in $m$ variables.

A {\em $Y$-seed} of rank $n$ in $\mathbb F$ is a pair $({\bf y}, \widehat B)$, where
\begin{itemize}
	\item ${\bf y} = (y_1, \ldots, y_{m})$ is an ordered set of free generators of $\mathbb F$ over $\mathbb Q$;
	\item  $\widehat B=(B\mid Q)=(\hat b_{ij})$ is an $n\times m$ integer matrix such that its leftmost $n\times n$ submatrix $B$ is skew-symmetrizable.
\end{itemize}
The variables $y_1,\ldots,y_m$ are called the {\em $y$-variables} of $({\bf y}, \widehat B)$.

Let  $({\bf y}, \widehat B)$ be a $Y$-seed of rank $n$ in $\mathbb F$. The {\em mutation} of  $({\bf y}, \widehat B)$ in direction $k\in[1,n]$ is the pair  $({\bf y}', \widehat B') \coloneqq   \mu_k({\bf y}, \widehat B)$ given as follows:
\begin{align}\label{eqn:y-mutation}
 y_i^\prime&=\begin{cases}y_k^{-1}, &\text{if}\;i=k; \\
y_iy_k^{[\hat b_{ki}]_+}(1+y_k)^{-\hat b_{ki}},&\text{otherwise}.\end{cases}\\
\hat b_{ij}^\prime&=\begin{cases}-\hat b_{ij}, & \text{if}\;i=k\;\text{or}\;j=k;\\
\hat b_{ij}+[\hat b_{ik}]_+[\hat b_{kj}]_+-[-\hat b_{ik}]_+[-\hat b_{kj}]_+,&\text{otherwise}.\end{cases}\nonumber
\end{align}
One can check that the new pair  $({\bf y}', \widehat B')=\mu_k({\bf y}, \widehat B)$ is still a $Y$-seed of rank $n$ and $({\bf y}, \widehat B)=\mu_k({\bf y}', \widehat B')$.

A {\em cluster seed} or simply a {\em seed} of rank $n$ in $\mathbb F$ is a pair
$({\bf x}, \widetilde B)$, where
\begin{itemize}
	\item ${\bf x} = (x_1, \ldots, x_{m})$ is an ordered set of free generators of $\mathbb F$ over $\mathbb Q$;
	\item  $\widetilde B=\begin{bmatrix}
       B\\ P
   \end{bmatrix}=(b_{ij})$ is an $m\times n$ mutation matrix.
\end{itemize}
In this case, the tuple ${\bf x}$ is called the {\it cluster} of $({\bf x}, \widetilde B)$. The elements of ${\bf x}$ are called
{\it cluster variables}. More precisely, we call $x_1,\ldots,x_n$ {\em unfrozen cluster variables} and $x_{n+1},\ldots,x_{m}$ {\em frozen (cluster) variables}.
The matrices $B$ and $P$ are respectively called the {\it  exchange matrix} and the {\it coefficient matrix} of $({\bf x}, \widetilde B)$. We denote by $\widehat {\bf y} \coloneqq   (\hat y_1,\ldots,\hat y_n)$, where $\hat y_k \coloneqq   \prod_{j=1}^mx_j^{b_{jk}}$. The variables $\hat y_1,\ldots,\hat y_n$ are called {\em $\hat y$-variables} of $({\bf x}, \widetilde B)$.

Let  $({\bf x}, \widetilde B)$ be a seed of rank $n$ in $\mathbb F$. The {\em mutation} of $({\bf x}, \widetilde B)$ in direction $k\in[1,n]$ is the pair  $({\bf x}', \widetilde B')=\mu_k({\bf x}, \widetilde B)$ given by $\widetilde B'=\mu_k(\widetilde B)$ and
\begin{eqnarray}
\label{eqn:x-mutation}
 x_i^\prime=\begin{cases}x_i,&
 \text{if}\;i\neq k;\\
 x_k^{-1}\cdot (\prod_{j=1}^mx_j^{[b_{jk}]_+}+\prod_{j=1}^mx_j^{[-b_{jk}]_+}),&\text{if}\;i= k.\end{cases}
\end{eqnarray}
One can check that the new pair  $({\bf x}', \widetilde B')=\mu_k({\bf x}, \widetilde B)$ is still a seed of rank $n$ and $({\bf x}, \widetilde B)=\mu_k({\bf x}', \widetilde B')$.

Let $\mathbb T_n$ denote the $n$-regular tree. We
 label the edges of $\mathbb T_n$ by $1,\ldots, n$ such that the $n$ different edges adjacent to the same vertex of $\mathbb T_n$ receive different labels.

\begin{definition}
(i) A {\em $Y$-pattern} $\mathcal S_Y=\{({\bf y}_t, \widehat B_t)\mid t\in \mathbb T_n\}$ of rank $n$
	is an assignment of a $Y$-seed $({\bf y}_t, \widehat B_t)$ of rank $n$ to
 	every vertex $t$ of $\mathbb T_n$ such that $({\bf y}_{t'}, \widehat B_{t'})=\mu_k({\bf y}_t, \widehat B_t)$ whenever
	\begin{xy}(0,1)*+{t}="A",(10,1)*+{t'}="B",\ar@{-}^k"A";"B" \end{xy} in $\TT_n$.

 (ii) A {\em cluster pattern} $\mathcal S_X=\{({\bf x}_t, \widetilde B_t)\mid t\in \mathbb T_n\}$ of rank $n$
	is an assignment of a cluster seed $({\bf x}_t, \widetilde B_t)$ of rank $n$ to
 	every vertex $t$ of $\mathbb T_n$ such that $({\bf x}_{t'}, \widetilde B_{t'})=\mu_k({\bf x}_t, \widetilde B_t)$ whenever
	\begin{xy}(0,1)*+{t}="A",(10,1)*+{t'}="B",\ar@{-}^k"A";"B" \end{xy} in $\TT_n$.
\end{definition}

We usually write ${\bf y}_t=(y_{1;t},\ldots,y_{m;t})$, $\widehat B_t=(B_t\mid Q_t)=(\hat b_{ij;t})$, ${\bf x}_t=(x_{1;t},\ldots,x_{m;t})$ and \[\widetilde B_t = \begin{pmatrix}
	    B_t\\ P_t
	\end{pmatrix}=(b_{ij;t}).\]
We write $\widehat{\bf y}_t=(\hat y_{1;t},\ldots,\hat y_{n;t})$ for the ordered set of $\hat y$-variables of the seed $({\bf x}_t,\widetilde B_t)$, where $$\hat y_{k;t}=\prod_{j=1}^m x_{j;t}^{b_{jk;t}}.$$

In this paper, we usually fix a vertex $t_0$ as the {\em rooted vertex} of the $n$-regular tree $\mathbb T_n$. Clearly, both $Y$-pattern and cluster pattern are uniquely determined by the data at the rooted vertex $t_0$.

\begin{definition}\label{def:dual-pair}
Let  $\mathcal S_X=\{({\bf x}_t, \widetilde B_t)\mid t\in \mathbb T_n\}$  be a cluster pattern with $\widetilde B_{t_0}\in\mathbb Z^{m\times n}$
, $\mathcal S_Y=\{({\bf y}_t, \widehat B_t)\mid t\in \mathbb T_n\}$ a $Y$-pattern with $\widehat B_{t_0}\in\mathbb Z^{n\times l}$, and let $\mathsf{\Lambda}=\{\Lambda_t\mid t\in\mathbb T_n\}$ be a collection of $m\times m$ skew-symmetric integer matrices indexed by the vertices in $\mathbb T_n$.

\begin{itemize}
\item [(i)] The pair $(\mathcal S_X,\mathcal S_Y)$ is called a {\em cluster ensemble}, if 
$\widehat B_{t_0}$ is the principal part of $\widetilde B_{t_0}$, that is, $l=n$ and $\widehat B_{t_0}=B_{t_0}$.
    \item [(ii)] The pair $(\mathcal S_X,\mathcal S_Y)$ is called a {\em Langlands dual pair}, if $l=m$ and $\widehat B_{t_0}=-\widetilde B_{t_0}^T$.
    \item[(iii)] The pair $(\mathcal S_X, \mathsf{\Lambda})$ is called a {\em $\mathsf{\Lambda}$-cluster pattern}, if $\{(\widetilde B_t,\Lambda_t)\mid t\in\mathbb T_n\}$ forms a collection of compatible pairs and $(\widetilde B_{t'},\Lambda_{t'})=\mu_k(\widetilde B_t,\Lambda_t)$  whenever \begin{xy}(0,1)*+{t}="A",(10,1)*+{t'}="B",\ar@{-}^k"A";"B" \end{xy} in $\TT_n$. In this case, we call the triple $({\bf x}_t,\widetilde B_t,\Lambda_t)$ a {\em $\mathsf{\Lambda}$-seed}.
\end{itemize}
\end{definition}
Clearly, each cluster pattern can be embedded into a cluster ensemble and a Langlands dual pair in a unique way up to isomorphism. The following result can be easily verified.

\begin{corollary}
 Let $\mathcal S_X=\{({\bf x}_t, \widetilde B_t)\mid t\in \mathbb T_n\}$  be a cluster pattern and $\mathcal S_Y=\{({\bf y}_t, \widehat B_t)\mid t\in \mathbb T_n\}$ a $Y$-pattern. The following statements hold.
 \begin{itemize}
     \item [(i)] If  $(\mathcal S_X,\mathcal S_Y)$ is a cluster ensemble, then $\widehat B_t$ is the principal part of $\widetilde B_t$ for any vertex $t$ of $\mathbb T_n$, \ie $\widehat B_t=B_t$.
     \item[(ii)] If $(\mathcal S_X,\mathcal S_Y)$ is a Langlands dual pair, then $\widehat B_{t}=-\widetilde B_{t}^T$ for any vertex $t$ of $\mathbb T_n$.
 \end{itemize}
\end{corollary}

\begin{proposition}[{\cite[Proposition 3.9]{fomin_zelevinsky_2007}}] \label{pro:p-map}
Let $\mathcal S_X$ be a cluster pattern, and let $\widehat {\bf y}_t=(\hat y_{1;t},\ldots,\hat y_{n;t})$ be the ordered set of $\hat y$-variables of the seed $({\bf x}_t,\widetilde B_t)$. Suppose that  $\widetilde B_t$ has full rank. Then the following statements hold.
\begin{itemize}
    \item [(i)] The $\hat y$-variables $\hat y_{1;t},\ldots,\hat y_{n;t}$ are algebraically independent.
    \item [(ii)] The assignment $\mathcal S_Y=\{(\widehat {\bf y}_t, B_t)\mid t\in\mathbb T_n\}$ forms a $Y$-pattern and $(\mathcal S_X,\mathcal S_Y)$ is a cluster ensemble.
\end{itemize}

\end{proposition}

Given a $\mathsf{\Lambda}$-cluster pattern $(\mathcal S_X, \mathsf{\Lambda})$, we can define a Poisson bracket $\{-,-\}$ on the ambient field $\mathbb F$ using the $m\times m$ skew-symmetric matrix $\Lambda_{t_0}=(\lambda_{ij;t_0})$ at the rooted vertex $t_0$:
$$\{x_{i;t_0},x_{j;t_0}\} \coloneqq   \lambda_{ij;t_0}\cdot x_{i;t_0}x_{j;t_0}.$$
It turns out that this Poisson bracket is compatible with the cluster pattern $\mathcal S_X$, that is, for any cluster ${\bf x}_t$ of $\mathcal S_X$, we have $$\{x_{i;t},x_{j;t}\}=\lambda_{ij;t}\cdot x_{i;t}x_{j;t},$$
where $\lambda_{ij;t}$ is the $(i,j)$-entry of $\Lambda_t$, \confer \cite{gsv-2003}, \cite[Remark 4.6]{bz-2005}.

\begin{definition}
 Let $\mathcal S_X=\{({\bf x}_t, \widetilde B_t)\mid t\in \mathbb T_n\}$ be a cluster pattern of rank $n$ in $\mathbb F$.
\begin{itemize}
    \item [(i)]  The {\em cluster algebra} $\mathcal A$ associated to $\mathcal S_X$ is the $\mathbb Z$-subalgebra of $\mathbb F$ generated by all the cluster variables (frozen and unfrozen), \ie $\mathcal A \coloneqq   \mathbb Z[x_{1;t},\ldots,x_{m;t}\mid t\in\mathbb T_n]$. 
    
    \item[(ii)]  If $(\mathcal S_X, \mathsf{\Lambda})$ is a  $\mathsf{\Lambda}$-cluster pattern, then the corresponding  cluster algebra $\mathcal A$ 
    is endowed with extra data $\mathsf{\Lambda}=\{\Lambda_t\mid t\in\mathbb T_n\}$. In this case, we call $\mathcal A$  a {\em $\mathsf{\Lambda}$-cluster algebra}.
    \item[(iii)] If $m=n$,  \ie there are no frozen variables, the corresponding  cluster algebra is said to be {\em with trivial coefficients}.  
\end{itemize}
\end{definition}

Note that the frozen variables are not inverted in cluster algebras in this paper.

\begin{theorem}[{\cite{fz_2002}*{Laurent phenomenon}}] Let $({\bf x}_t,\widetilde B_t)$ be a seed of a cluster algebra $\mathcal A$. Then any cluster variable $z$ can be written as a Laurent polynomial in $\mathbb Z[x_{1;t}^{\pm 1}, \ldots, x_{n;t}^{\pm 1},x_{n+1;t},\ldots,x_{m;t}]$.
\end{theorem}

\begin{example}\label{ex:A2}
   Take $\widetilde B=\begin{bmatrix}
    0&1\\-1&0
\end{bmatrix}$ and ${\bf x}=(x_1,x_2)$. One can check that the  cluster algebra $\mathcal A$ defined by the initial seed $({\bf x},\widetilde B)$  has only five cluster variables:
\[x_1, \;x_2, \;x_3 \coloneqq   \frac{x_2+1}{x_1}, \;x_4 \coloneqq   \frac{x_1+x_2+1}{x_1x_2}, \;x_5 \coloneqq   \frac{x_1+1}{x_2}.\]
It is clear that all the cluster variables are contained in $\mathbb Z[x_1^{\pm 1},x_2^{\pm 1}]$.
\end{example}

\subsection{Semifield and tropical points}
Recall that $(\mathbb P, \cdot, \oplus)$ is called a {\em semifield} if $(\mathbb P,  \cdot)$ is an abelian multiplicative group endowed with a binary operation of auxiliary addition $\oplus$ which is commutative, associative and satisfies that the multiplication  distributes over the auxiliary addition. For example, $$\mathbb Z^{\rm max} \coloneqq   (\mathbb Z,+, \max\{-,-\})$$ 
is a semifield, which is called a {\em tropical semifield}.

Let $\QQ_{\rm sf}(u_1, \ldots, u_m)\coloneq \{\frac{F}{H}\mid 0\neq F,H\in\mathbb Z_{\geq 0}[u_1,\ldots,u_m]\}$ be the
set of all non-zero rational functions in $u_1, \ldots, u_m$ that have subtraction-free expressions. The set $\QQ_{\rm sf}(u_1, \ldots, u_m)$  is a semifield
with respect to the usual operations of multiplication and addition. It is called a {\em universal semifield}.

Now let $\mathbb F \coloneqq   \mathbb Q(u_1,\ldots,u_m)$ and $\mathbb F_{>0} \coloneqq   \QQ_{\rm sf}(u_1, \ldots, u_m)$. We denote by
$\Hom_{\rm sf}(\mathbb  F_{>0},\mathbb Z^{\max})$ the set of semifield homomorphisms from $\mathbb F_{>0}$ to $\mathbb Z^{\max}$. It is easy to see that the following map \[\Hom_{\rm sf}(\mathbb  F_{>0},\mathbb Z^{\max})\ni \beta\mapsto (\beta(u_1),\ldots,\beta(u_m))^T\in\mathbb Z^m\] is a bijection.

\begin{definition}[Tropical points]
    Let $\mathcal S_Y=\{({\bf y}_t, \widehat B_t)\mid t\in \mathbb T_n\}$ be a $Y$-pattern of rank $n$ in $\mathbb F$ and $\mathcal S_X=\{({\bf x}_t, \widetilde B_t)\mid t\in \mathbb T_n\}$ a cluster pattern of rank $n$ in $\mathbb F$.
\begin{itemize}
    \item [(i)] A tropical point $[{\bf g}]=\{{\bf g}^t\in\mathbb Z^m\mid t\in\mathbb T_n\}$ associated to the $Y$-pattern $\mathcal S_Y$ is an assignment of a (column) vector ${\bf g}^t=(g_{1}^t,\ldots,g_{m}^t)^T$ in $\mathbb Z^m$ to each vertex $t$ of $\mathbb T_n$ such that
    \begin{eqnarray}\label{eqn:y-trop}
g_{i}^{t'}=\begin{cases}-g_{k}^t,& \text{if}\;i=k;\\ g_{i}^t+[\hat b_{ki;t}]_+g_{k}^t+(-\hat b_{ki;t})[g_{k}^t]_+,&\text{if}\;i\neq k. \end{cases}
\end{eqnarray}
whenever
	\begin{xy}(0,1)*+{t}="A",(10,1)*+{t'}="B",\ar@{-}^k"A";"B" \end{xy} in $\TT_n$. We denote by $\mathcal S_Y(\mathbb Z^{\rm max})$ the set of tropical points associated to the $Y$-pattern $\mathcal S_Y$.

 \item[(ii)] A tropical point $[{\bf a}]=\{{\bf a}^t\in\mathbb Z^m\mid t\in\mathbb T_n\}$ associated to the cluster pattern $\mathcal S_X$ is an assignment of a (column) vector ${\bf a}^t=(a_{1}^t,\ldots,a_{m}^t)^T$ in $\mathbb Z^m$ to each vertex $t$ of $\mathbb T_n$ such that
    \begin{eqnarray}\label{eqn:x-trop}
a_{i}^{t'}=\begin{cases}-a_{k}^t+\max\{\sum_{j=1}^m[b_{jk;t}]_+a_{j}^t, \; \sum_{j=1}^m[-b_{jk;t}]_+a_{j}^t\}
,& \text{if}\;i=k;\\ a_{i}^t,&\text{if}\;i\neq k. \end{cases}
\end{eqnarray}
whenever
	\begin{xy}(0,1)*+{t}="A",(10,1)*+{t'}="B",\ar@{-}^k"A";"B" \end{xy} in $\TT_n$.
 We denote by $\mathcal S_X(\mathbb Z^{\rm max})$ the set of tropical points associated to the cluster pattern $\mathcal S_X$.
\end{itemize}
\end{definition}

\begin{remark}
The relations in \eqref{eqn:y-trop} and  \eqref{eqn:x-trop}  are obtained from the mutation relations \eqref{eqn:y-mutation} and \eqref{eqn:x-mutation} by tropicalization over the tropical semifield $\mathbb Z^{\rm max}=(\mathbb Z,+,\max\{-,-\})$. Namely, we make the following replacements.
\begin{itemize}
    \item Replace the multiplication and addition in \eqref{eqn:y-mutation} and \eqref{eqn:x-mutation}  by ``+" and ``{\rm max}\{--,--\}" over $\mathbb Z$, respectively.
     \item Replace ``1" in \eqref{eqn:y-mutation} by ``0";
\end{itemize}
\end{remark}

The following corollary can be easily checked using the mutation relations in \eqref{eqn:y-trop} and  \eqref{eqn:x-trop}.
\begin{corollary}\label{cor:bijection} 
Let  $\mathcal S_X=\{({\bf x}_t, \widetilde B_t)\mid t\in \mathbb T_n\}$  be a cluster pattern, and let $\mathcal S_Y=\{({\bf y}_t, \widehat B_t)\mid t\in \mathbb T_n\}$ be a $Y$-pattern. Denote by $\mathbb F_{>0}^X\coloneqq \mathbb Q_{\rm sf}(x_{1;t_0},\ldots, x_{m;t_0})$
and $\mathbb F_{>0}^Y\coloneqq \mathbb Q_{\rm sf}(y_{1;t_0},\ldots, y_{m;t_0})$. Then the following statements hold.
\begin{itemize}
\item[(i)] There is a bijection from 
$\Hom_{\rm sf}(\mathbb F_{>0}^X,\mathbb Z^{\max})$ to  $\mathcal S_X(\mathbb Z^{\max})$  defined by
\[ 
\beta \mapsto \{\;(\beta(x_{1;t}),\ldots,\beta(x_{m;t}))^T\in\mathbb Z^m\mid t\in\mathbb T_n \;\}. 
\]

\item[(ii)] There is a bijection from 
$\Hom_{\rm sf}(\mathbb F_{>0}^Y,\mathbb Z^{\max})$ to  $\mathcal S_Y(\mathbb Z^{\max})$ defined by
\[ 
\beta \mapsto \{\;(\beta(y_{1;t}),\ldots,\beta(y_{m;t}))^T\in\mathbb Z^m\mid t\in\mathbb T_n \;\}. 
\]
\end{itemize}
\end{corollary}

\subsection{Compatibly pointed elements and good elements}
In this subsection, we fix a {\em full rank} cluster algebra $\mathcal A$, that is, the initial mutation matrix $\widetilde B_{t_0}$ has full rank, equivalently, any mutation matrix $\widetilde B_t$ of $\mathcal A$ has full rank.

 Since $\mathcal A$ is of full rank, each seed $({\bf x}_t,\widetilde B_t)$ of $\mathcal A$ defines a partial order $\preceq_t$ on $\mathbb Z^m$. For two vectors ${\bf g},{\bf g}'\in\mathbb Z^m$, we write ${\bf g}'\preceq_t{\bf g}$ if there exists some vector ${\bf v}=(v_1,\ldots,v_n)^T\in\mathbb N^n$ such that $${\bf g}'={\bf g}+\widetilde B_t{\bf v},$$ equivalently, ${\bf x}_t^{{\bf g}'}={\bf x}_t^{\bf g}\cdot \widehat {\bf y}_t^{\bf v}$, where $\widehat {\bf y}_t^{\bf v}={\bf x}_t^{\widetilde B_t{\bf v}}$. We denote by ${\bf g}'\prec_t{\bf g}$ for the case ${\bf g}'\preceq_t{\bf g}$ and ${\bf g}'\neq {\bf g}$.

Now we recall the notion of pointed elements introduced by Qin \cite{Qin_2017}. Let 
 $$u=\sum_{{\bf h}\in\mathbb Z^m}b_{\bf h}{\bf x}_t^{\bf h}\in \mathbb Z[x_{1;t}^{\pm 1},\ldots,x_{m;t}^{\pm 1}]$$
be the Laurent expansion of an element $u\in\mathcal A$ with respect to a seed 
 $({\bf x}_t,\widetilde B_t)$. The element $u\in\mathcal A$ is said to be {\em pointed} for the seed $({\bf x}_t,\widetilde B_t)$, if the following two conditions are satisfied.
\begin{itemize}
\item The set $\{{\bf h}\in\mathbb Z^m\mid b_{\bf h}\neq 0\}$ has a unique maximal element ${\bf g}$ under the dominance order $\preceq_t$ on $\mathbb Z^m$. We denote by $\de^t(u) \coloneqq   {\bf g}$ and call it the {\em degree} of $u$ with respect to the seed $({\bf x}_t,\widetilde B_t)$.
    \item The coefficient corresponding to the degree term is $1$,  \ie $b_{\bf g}=1$ for ${\bf g}=\de^t(u)$.
\end{itemize}

Clearly, if $u\in\mathcal A$ is a pointed element for a seed $({\bf x}_t,\widetilde B_t)$, it can be written in the form:
$$u={\bf x}_t^{\bf g}+\sum_{{\bf h}\prec_t {\bf g}}b_{\bf h}{\bf x}_t^{\bf h}={\bf x}_t^{\bf g}F(\hat y_{1;t},\ldots,\hat y_{n;t}),$$
where $b_{\bf h}\in\mathbb Z$ and $F\in\mathbb Z[y_1,\ldots,y_n]$ is a polynomial with constant term $1$.

In cluster algebras, we are often interested in the pointed elements which can be controlled by tropical points associated to a $Y$-pattern. Such elements correspond to compatibly pointed elements introduced in \cite{qin_2019}.

\begin{definition}[Compatibly pointed elements and good elements]\label{def:pointed}
Let $\mathcal A$ be a cluster algebra of full rank, and let  $(\mathcal S_X,\mathcal S_Y)$ be the Langlands dual pair corresponding to $\mathcal A$.
 \begin{itemize}
 \item[(i)] An element $u\in\mathcal A$ is said to be {\em compatibly pointed}, if $u$ is pointed for any seed $({\bf x}_t,\widetilde B_t)$ of $\mathcal A$ and the collection $[{\bf g}] \coloneqq   \{\de^t(u)\in\mathbb Z^m\mid t\in\mathbb T_n\}$ forms a tropical point in $\mathcal S_Y(\mathbb Z^{\rm max})$. 

 \item [(ii)] A compatibly pointed element  $u\in\mathcal A$ is said to be a {\em good element}, if it is universally positive,  \ie $$u\in\mathbb Z_{\geq 0}[x_{1;t}^{\pm 1},\ldots,x_{m;t}^{\pm 1}]$$ for any vertex $t\in\mathbb T_n$.
\end{itemize}
\end{definition}

 Let $u$ be a compatibly pointed element in $\mathcal A$.  Then the Laurent expansion of  $u$ with respect to any seed $({\bf x}_t,\widetilde B_t)$ of $\mathcal A$ has a {\em canonical expression:}
\begin{eqnarray}\label{eqn:upointed}
   u={\bf x}_t^{{\bf g}_u^t}F_u^t(\hat y_{1;t},\ldots,\hat y_{n;t}),
 \end{eqnarray}
where ${\bf g}_u^t\coloneq \de^t(u)\in\mathbb Z^m$ and $F_u^t$ is a polynomial in $\mathbb Z[y_1,\ldots,y_n]$ with constant term $1$.
\begin{itemize}
\item The polynomial $F_u^t$ is called the {\em $F$-polynomial} of $u$ with respect to vertex $t$;
\item The vector ${\bf g}_u^t=\de^t(u)\in\mathbb Z^m$ is  called the {\em extended $g$-vector} of $u$ with respect to vertex $t$;

\item The principal part $({\bf g}_u^t)^\circ=(g_{1;u}^t,\ldots,g_{n;u}^t)^T\in\mathbb Z^n$ of ${\bf g}_u^t=(g_{1;u}^t,\ldots,g_{m;u}^t)^T\in\mathbb Z^m$ is called the {\em $g$-vector} of $u$ with respect to vertex $t$.
\end{itemize}

 Recall that a {\em cluster monomial} in  $\mathcal A$ is a monomial in cluster variables from the same cluster. That is, each cluster monomial is of the form
 ${\bf x}_w^{\bf v} \coloneqq   \prod_{j=1}^mx_{j;w}^{v_j}$
for  some vector ${\bf v}=(v_1,\ldots,v_m)^T\in \mathbb Z_{\geq 0}^m$ and cluster ${\bf x}_w$.

\begin{proposition}[{\cite{GHKK18}}]
\label{pro:ghkk}
 Let $\mathcal A$ be a cluster algebra of full rank. Then any cluster monomial $u$ of $\mathcal A$ is a good element.
\end{proposition}

\begin{example}\label{ex:A2-2}
Let us continue with Example \ref{ex:A2}. We see that $\widetilde B=\begin{bmatrix}
    0&1\\-1&0
\end{bmatrix}$ has full rank and $\hat y_1=x_2^{-1},\;\hat y_2=x_1$. The canonical expressions of the five cluster variables
$$x_1, \;x_2, \;x_3=\frac{x_2+1}{x_1}, \;x_4=\frac{x_1+x_2+1}{x_1x_2}, \;x_5=\frac{x_1+1}{x_2}$$
with respect to the initial seed $({\bf x},\widetilde B)$ are given as follows:
    \[x_1=x_1\cdot 1,\;\;\;x_2=x_2\cdot 1,\;\;\;x_3={x_1^{-1}x_2\cdot(1+\widehat y_1),}\;\;\;  x_4={x_1^{-1}\cdot (1+\widehat y_1+\widehat y_1\widehat y_2),}\;\;\;
    x_5={x_2^{-1}\cdot (1+\widehat y_2)}.\]
\end{example}

\subsection{Tropical invariant and $F$-invariant}
In this subsection, we recall the tropical invariant and the $F$-invariant introduced in \cite{Cao-2023} for $\mathsf{\Lambda}$-cluster algebras.

\begin{proposition}[{\cite[Proposition 4.4]{Cao-2023}}]
\label{pro:beta-map}
    Let  $\mathcal A$ be a $\mathsf{\Lambda}$-cluster algebra. Then any good element $u\in \mathcal A$ defines a unique semifield homomorphism 
  $$\beta_{u}:\; \mathbb F_{>0} \coloneqq   \mathbb Q_{\rm sf}(x_{1;t_0},\ldots,x_{m;t_0})\rightarrow \mathbb Z^{\rm max}$$ such that $\beta_{u}({\bf x}_t)=(\Lambda_t{\bf g}_u^t)^T$ for any vertex $t\in\mathbb T_n$, where ${\bf g}_u^t\in\mathbb Z^m$ is the extended $g$-vector of $u$ with respect to vertex $t$.
\end{proposition}

\begin{definition}[Tropical invariant and $F$-invariant] \label{def:f-invariant}
Let $\mathcal A$ be a $\mathsf{\Lambda}$-cluster algebra, and let $u,u'$ be two good elements in $\mathcal A$.

(i) The {\em tropical invariant} $\langle u,u'\rangle_{\rm trop}$ of the ordered pair $(u,u')$
 is  defined by 
 $$\langle u,u'\rangle_{\rm trop} \coloneqq   \beta_{u'}(u)\;\in \mathbb Z,$$ where
 $\beta_{u'}\colon\mathbb F_{>0} \coloneqq   \mathbb Q_{\rm sf}(x_{1;t_0},\ldots,x_{m;t_0})\rightarrow \mathbb Z^{\rm max}$ is the semifield homomorphism defined in Proposition \ref{pro:beta-map}.

(ii) The {\em $F$-invariant} $(u\mid\mid u')_F$ of the ordered pair $(u,u')$
 is defined by
 $$(u\mid\mid u')_F \coloneqq   \beta_{u'}(u)+\beta_u(u')=\langle u,u'\rangle_{\rm trop}+\langle u',u\rangle_{\rm trop}.$$
\end{definition}

 Given a non-zero polynomial $$F=\sum_{{\bf v}\in\mathbb N^n}c_{{\bf v}}{\bf y}^{{\bf v}}\in\mathbb Z[y_1,\ldots,y_n]$$ and a vector ${\bf r}\in \mathbb Z^n$, we set
$$F[{\bf r}] \coloneqq   \max\{{\bf v}^T{\bf r}\mid c_{{\bf v}}\neq 0\}\in\mathbb Z.$$
We call the map $F[-]\colon \mathbb Z^n\rightarrow\mathbb Z$ a {\em tropical polynomial}. It is easy to see that if $F$ has the constant term $1$, then $F[{\bf r}]\geq 0$ for any ${\bf r}\in\mathbb Z^n$.

\begin{example}
 Take $F=1+y_1+y_1y_2\in\mathbb Z[y_1,y_2]$ and ${\bf r}=\begin{bmatrix}
     -2\\1
 \end{bmatrix}$. Then
 \begin{eqnarray}
     F[{\bf r}]=\max\{
     \begin{bmatrix}
         0,0
     \end{bmatrix}\begin{bmatrix}
     -2\\1
 \end{bmatrix}, \begin{bmatrix}
         1,0
     \end{bmatrix}\begin{bmatrix}
     -2\\1
 \end{bmatrix}, \begin{bmatrix}
         1,1
     \end{bmatrix}\begin{bmatrix}
     -2\\1
 \end{bmatrix}
     \}=\max\{0,-2,-1\}=0.
     \nonumber
 \end{eqnarray}
\end{example}

The following theorem provides very explicit formulas to  calculate the tropical invariant and $F$-invariant.

\begin{theorem}[{\cite[Theorem 4.10]{Cao-2023}}]
\label{thm:mutation-inv}
Let $\mathcal A$ be a $\mathsf{\Lambda}$-cluster algebra and let  $u$ and $u'$ be two good elements in  $\mathcal A$. Let
\[ u={\bf x}_w^{{\bf g}_u^w}F_u^w(\hat y_{1;w},\ldots,\hat y_{n;w}),\;\;\;u'={\bf x}_w^{{\bf g}_{u'}^w}F_{u'}^w(\hat y_{1;w},\ldots,\hat y_{n;w})
\]
be the canonical expressions of $u$ and $u'$ with respect to any vertex $w\in \mathbb T_n$. Then we have
\vspace{1mm}
\begin{itemize}
    \item [(i)] $\langle u,u'\rangle_{\rm trop}= ({\bf g}_u^w)^T\Lambda_w{\bf g}_{u'}^w+F_u^w[(S\mid {\bf 0}){\bf g}_{u'}^w]$.\vspace{1.5mm}
    \item[(ii)] $(u\mid\mid u')_F= F_u^w[(S\mid {\bf 0}){\bf g}_{u'}^w]+F_{u'}^w[(S\mid {\bf 0}){\bf g}_{u}^w]$. In particular,  $(u\mid\mid u')_F\in\mathbb Z_{\geq 0}$.\vspace{1.5mm}
    \item[(iii)] $(u\mid\mid u)_F=0$ whenever $u$ is a cluster monomial.
\end{itemize}
\end{theorem}

The non-negative integer $F_u^w[(S\mid {\bf 0}){\bf g}_{u'}^w]=F_u^w[S({\bf g}_{u'}^w)^\circ]$ is called the {\em partial $F$-invariant} of the ordered pair $(u,u')$ at vertex $w\in\mathbb T_n$.  By Theorem \ref{thm:mutation-inv}, we have
\begin{eqnarray}
({\bf g}_u^t)^T\Lambda_t{\bf g}_{u'}^t+F_u^t[S({\bf g}_{u'}^t)^\circ]   =  &\langle u,u'\rangle_{\rm trop}&= ({\bf g}_u^{t'})^T\Lambda_{t'}{\bf g}_{u'}^{t'}+F_u^{t'}[S ({\bf g}_{u'}^{t'})^\circ],\nonumber\\
F_u^t[S({\bf g}_{u'}^t)^\circ]+F_{u'}^t[S({\bf g}_{u}^t)^\circ]=&(u\mid\mid u')_F&= F_u^{t'}[S({\bf g}_{u'}^{t'})^\circ]+F_{u'}^{t'}[S({\bf g}_{u}^{t'})^\circ],\nonumber
\end{eqnarray}
for any two vertices $t,t'\in\mathbb T_n$, which says that the tropical invariant and the $F$-invariant are mutation-invariant.

\section{Partial $F$-invariant and Reading's conjecture}\label{sec:3}

\subsection{Mutation formula for the partial $F$-invariant}
In this subsection, we study the mutation formula for the partial $F$-invariant. As a result, we give a new proof of the mutation-invariance for the $F$-invariant.

\begin{lemma}\label{lem:1-mutation}
Let $\mathcal A$ be a cluster algebra and let $(\mathcal S_X,\mathcal S_Y)$ be the Langlands dual pair corresponding to  $\mathcal A$.  Let $[{\bf g}]=\{{\bf g}^t=(g_1^t,\ldots,g_m^t)^T\in\mathbb Z^m\mid t\in\mathbb T_n\}$ be a tropical point in $\mathcal S_Y(\mathbb Z^{\rm max})$. Then we have
\begin{eqnarray}\label{eqn:1-mutation}
 {\bf x}_{t'}^{{\bf g}^{t'}}(1+\hat y_{k;t'})^{[-g_{k}^{t'}]_+} ={\bf x}_t^{{\bf g}^t}(1+\hat y_{k;t})^{[-g_{k}^t]_+},
\end{eqnarray}
for any edge \begin{xy}(0,1)*+{t}="A",(10,1)*+{t'}="B",\ar@{-}^k"A";"B" \end{xy} in $\TT_n$,  where $\hat y_{k;t}$ is the $k$-th $\hat y$-variable of the seed $({\bf x}_t,\widetilde B_t)$.
\end{lemma}

\begin{proof}
Since $(\mathcal S_X,\mathcal S_Y)$ is a Langlands dual pair, we have $\widehat B_t=-\widetilde B_t^T$. Thus $\hat b_{ij;t}=-b_{ji;t}$ for any $i\in[1,n]$ and $j\in[1,m]$. Since $[{\bf g}]$ is a tropical point in $\mathcal S_Y(\mathbb Z^{\rm max})$ and by  \eqref{eqn:y-trop}, we have
\begin{align}
    g_{i}^{t'}&=\begin{cases}-g_{k}^t,& \text{if}\;i=k;\\ g_{i}^t+[-b_{ik;t}]_+g_{k}^t+b_{ik;t}[g_{k}^t]_+,&\text{if}\;i\neq k. \end{cases} \nonumber\\
    &=\begin{cases}-g_{k}^t,& \text{if}\;i=k;\\ g_{i}^t+[b_{ik;t}]_+g_{k}^t,&\text{if}\;i\neq k\;\;\text{and }\;g_{k}^t\geq 0;\\
    g_{i}^t+[-b_{ik;t}]_+g_{k}^t,&\text{if}\;i\neq k\;\;\text{and }\;g_{k}^t< 0. 
    \end{cases}\nonumber
\end{align}
Since $g_k^{t'}=-g_k^t$, without loss of generality, we can assume that $g_{k}^t\geq 0$. Then the desired equality \eqref{eqn:1-mutation} becomes 
\begin{eqnarray}\label{eqn:2-mutation}
    {\bf x}_t^{{\bf g}^t}={\bf x}_{t'}^{{\bf g}^{t'}}(1+\hat y_{k;t'})^{[-g_{k}^{t'}]_+}={\bf x}_{t'}^{{\bf g}^{t'}}(1+\hat y_{k;t'})^{g_k^t}.
\end{eqnarray}
Since $({\bf x}_{t},\widetilde B_{t})=\mu_k({\bf x}_{t'},\widetilde B_{t'})$, we have $x_{i;t}=x_{i;t'}$ for $i\neq k$ and \[x_{k;t}x_{k;t'}={\bf x}_{t'}^{[-{\bf b}_{k;t'}]_+}(1+\hat y_{k;t'})={\bf x}_{t'}^{[{\bf b}_{k;t}]_+}(1+\hat y_{k;t'})\]
where $[{\bf b}_{k;t}]_+\coloneqq ([b_{1k;t}]_+,\ldots,[b_{mk;t}]_+)^T\in\mathbb Z_{\geq 0}^m$. By using these relations, we can obtain the expansion of ${\bf x}_t^{{\bf g}^t}$ in terms of ${\bf x}_{t'}$ and $\hat y_{k;t'}$, which is exactly the desired equality \eqref{eqn:2-mutation}.
\end{proof}

\begin{lemma}\label{F-mutation}
Let $\mathcal A$ be a cluster algebra of full rank, and let $\widehat {\bf y}_t=(\hat y_{1;t},\ldots, \hat y_{n;t})$ be the ordered set of $\hat y$-variables of the seed $({\bf x}_t,\widetilde B_t)$.
 Let $u$ be a compatibly pointed element in  $\mathcal A$.  Then we have 
\begin{eqnarray}\label{eqn:F-mutation}
    (1+\hat y_{k;t'})^{-[-g_{k;u}^{t'}]_+}F_u^{t'}(\hat y_{1;t'},\ldots,\hat y_{n;t'})= (1+\hat y_{k;t})^{-[-g_{k;u}^{t}]_+}F_u^{t}(\hat y_{1;t},\ldots,\hat y_{n;t}),
\end{eqnarray}
for any edge \begin{xy}(0,1)*+{t}="A",(10,1)*+{t'}="B",\ar@{-}^k"A";"B" \end{xy}  in $\TT_n$.
\end{lemma}

\begin{proof}
 For any edge \begin{xy}(0,1)*+{t}="A",(10,1)*+{t'}="B",\ar@{-}^k"A";"B" \end{xy}  in $\TT_n$, we have
\[
{\bf x}_{t'}^{{\bf g}_u^{t'}}F_u^{t'}(\widehat{\bf y}_{t'})=u={\bf x}_{t}^{{\bf g}_u^{t}}F_u^{t}(\widehat{\bf y}_{t}).
\]
Combining this with \eqref{eqn:1-mutation}, we obtain 
\[(1+\hat y_{k;t'})^{-[-g_{k;u}^{t'}]_+}F_u^{t'}(\widehat{\bf y}_{t'})=(1+\hat y_{k;t})^{-[-g_{k;u}^{t}]_+}F_u^{t}(\widehat{\bf y}_{t}).\qedhere\]
\end{proof}

\begin{lemma}\label{lem:S-map}
Let $\mathcal S_X=\{({\bf x}_t, (\widetilde B_t)_{m\times n})\mid t\in\mathbb T_n\}$ be a cluster pattern. Let $\mathcal S_Y$ and $\mathcal S_Y^p$ be two $Y$-patterns such that $(\mathcal S_X,\mathcal S_Y)$ is a Langlands dual pair and $(\mathcal S_X,\mathcal S_Y^p)$ is a cluster ensemble. Let $S=\diag(s_1,\ldots, s_n)$ be a skew-symmetrizer for $B_{t_0}$. Then there is a map from  $\mathcal S_Y(\mathbb Z^{\rm max})$ to $\mathcal S_Y^p(\mathbb Z^{\rm max})$ defined by
\[[{\bf g}]\coloneqq \{{\bf g}^t\in\mathbb Z^m\mid t\in\mathbb T_n\}\mapsto [(S\mid {\bf 0}){\bf g}]\coloneqq \{ (S\mid {\bf 0}){\bf g}^t\in\mathbb Z^n\mid t\in\mathbb T_n\},
\]
where ${\bf 0}$ is the $n\times (m-n)$ zero matrix.
\end{lemma}
\begin{proof}
We write $S_Y=\{({\bf y}_t, \widehat B_t)\mid t\in\mathbb T_n \}$ and $S_Y^p=\{ (\widehat {\bf y}_t, B_t)\mid t\in\mathbb T_n\}$, where $\widehat B_t=-\widetilde B_t^T$ and $B_t$ is the principal part of $\widetilde B_t$. Since $[{\bf g}]\coloneqq \{{\bf g}^t=(g_1^t,\ldots,g_m^t)^T\in\mathbb Z^m\mid t\in\mathbb T_n\}$ is a tropical point in $\mathcal S_Y(\mathbb Z^{\rm max})$ and by  \eqref{eqn:y-trop}, for any edge \begin{xy}(0,1)*+{t}="A",(10,1)*+{t'}="B",\ar@{-}^k"A";"B" \end{xy}  in $\TT_n$ and $i\in[1,m]$, we have
\begin{align}
    g_{i}^{t'}&=\begin{cases}-g_{k}^t,& \text{if}\;i=k,\\ g_{i}^t+[\hat b_{ki;t}]_+g_{k}^t+(-\hat b_{ki;t})[g_{k}^t]_+,&\text{if}\;i\neq k. \end{cases} \nonumber\\
    &=\begin{cases}-g_{k}^t,& \text{if}\;i=k,\\ g_{i}^t+[-b_{ik;t}]_+g_{k}^t+b_{ik;t}[g_{k}^t]_+,&\text{if}\;i\neq k. \end{cases} \nonumber
\end{align}
 Since $SB_t$ is skew-symmetric, we have that $s_ib_{ik;t}=-s_kb_{ki;t}$, where $i\in[1,n]$. We write \[{\bf q}^t=(S\mid {\bf 0}){\bf g}^t=(s_1g_1^t,\ldots,s_ng_n^t)^T=(q_1^t,\ldots,q_n^t)^T\] for $t\in\mathbb T_n$. Then for $i\in[1,n]$, we have 
 \begin{eqnarray}
  q_{i}^{t'}=s_ig_{i}^{t'}&=&  \begin{cases}-s_kg_{k}^t, &{i=k,}\\
        s_ig_{i}^t+[-s_ib_{ik;t}]_+g_{k}^t+s_ib_{ik;t}[g_{k}^t]_+,&{i\neq k}.
    \end{cases}\nonumber\\
    &=&\begin{cases}-s_kg_{k}^t, &{i=k,}\\
        s_ig_{i}^t+[b_{ki;t}]_+(s_kg_{k}^t)+(-b_{ki;t})[s_kg_{k}^t]_+,&{i\neq k}.
    \end{cases}\nonumber\\
     &=&\begin{cases}-q_k^t, &{i=k,}\\
        q_{i}^t+[b_{ki;t}]_+q_{k}^t+(-b_{ki;t})[q_{k}^t]_+,&{i\neq k}.
    \end{cases}\nonumber
\end{eqnarray}
This implies that $\{{\bf q}^t\in\mathbb Z^n\mid t\in\mathbb T_n\}=\{ (S\mid {\bf 0}){\bf g}^t\in\mathbb Z^n\mid t\in\mathbb T_n\}$ is a tropical point in  $\mathcal S_Y^p(\mathbb Z^{\rm max})$. 
\end{proof}

Recall that  given a non-zero polynomial $F=\sum_{{\bf v}\in\mathbb N^n}c_{{\bf v}}{\bf y}^{{\bf v}}\in\mathbb Z[y_1,\ldots,y_n]$  and a vector ${\bf r}\in \mathbb Z^n$, we set
$$F[{\bf r}]\coloneqq \max\{{\bf v}^T{\bf r}\mid c_{{\bf v}}\neq 0\}\in\mathbb Z.$$

\begin{theorem}\label{thm:F-inv}
Let $\mathcal A$ be a cluster algebra of full rank and let $S=\diag(s_1,\ldots,s_n)$ be a fixed skew-symmetrizer for the exchange matrices of $\mathcal A$. Let  $u$ and $u'$ be two good elements in $\mathcal A$. Then the following statements hold.
    \begin{itemize}
        \item [(i)]  For any edge \begin{xy}(0,1)*+{t}="A",(10,1)*+{t'}="B",\ar@{-}^k"A";"B" \end{xy}  in $\TT_n$, we have
        \begin{eqnarray}\label{eqn:f-inv}
            F_u^{t'}[S({\bf g}_{u'}^{t'})^\circ]-F_u^t[S({\bf g}_{u'}^t)^\circ]=s_k([-g_{k;u}^{t'}]_+[-g_{k;u'}^{t}]_+-[-g_{k;u}^t]_+[-g_{k;u'}^{t'}]_+).
        \end{eqnarray}
        \item[(ii)] For any two vertices $t,t'\in\mathbb T_n$, we have 
        \begin{eqnarray}\label{eqn:F-inv-mut}
            F_u^t[S({\bf g}_{u'}^t)^\circ]+F_{u'}^t[S({\bf g}_u^t)^\circ]=F_u^{t'}[S({\bf g}_{u'}^{t'})^\circ]+F_{u'}^{t'}[S({\bf g}_u^{t'})^\circ].
        \end{eqnarray}
    In particular, the number $(u\mid\mid u')_F\coloneqq F_u^t[S({\bf g}_{u'}^t)^\circ]+F_{u'}^t[S({\bf g}_u^t)^\circ]$ only depends on $u$ and $u'$, not on the choice of vertex $t\in\mathbb T_n$.
    \end{itemize}
\end{theorem}
\begin{proof}
(i) Let $\mathcal S_X$ be the cluster pattern corresponding to the cluster algebra $\mathcal A$. Let $\mathcal S_Y$ and $\mathcal S_Y^p$ be two $Y$-patterns such that $(\mathcal S_X,\mathcal S_Y)$ is a Langlands dual pair and $(\mathcal S_X,\mathcal S_Y^p)$ is a cluster ensemble. By Proposition \ref{pro:p-map}, we can identify the $Y$-pattern $\mathcal S_Y^p$ with the following assignment:
\[ \mathcal S_Y^p=\{(\widehat{\bf y}_t, B_t)\mid t\in\mathbb T_n\},\]
where $\widehat{\bf y}_t=(\hat y_{1;t},\ldots,\hat y_{n;t})$ is the ordered set of the $\hat y$-variables of the seed $({\bf x}_t,\widetilde B_t)$.

Since $u'$ is a good element, the assignment $\{{\bf g}_{u'}^t\in\mathbb Z^m\mid t\in\mathbb T_n\}$ is a tropical point in $\mathcal S_Y(\mathbb Z^{\rm max})$. Then by Lemma \ref{lem:S-map}, we know that $\{(S\mid {\bf 0}){\bf g}_{u'}^t\in\mathbb Z^n\mid t\in\mathbb T_n\}$ is a tropical point in $\mathcal S_Y^p(\mathbb Z^{\rm max})$. Since the initial $\hat y$-variables  $\hat y_{1;t_0},\ldots,\hat y_{n;t_0}$ are algebraically independent, they generate a universal semifield \[\hat{\mathbb F}_{>0}\coloneqq \mathbb Q_{\rm sf}(\hat y_{1;t_0},\ldots,\hat y_{n;t_0}).\] By Corollary \ref{cor:bijection}, the tropical point $\{(S\mid {\bf 0}){\bf g}_{u'}^t\in\mathbb Z^n\mid t\in\mathbb T_n\}\in \mathcal S_Y^p(\mathbb Z^{\rm max})$ corresponds to  a semifield homomorphism $\beta_{u'}:\hat{\mathbb F}_{>0}\rightarrow \mathbb Z^{\max}$ such that \[\beta_{u'}(\widehat {\bf y}_w)=((S\mid{\bf 0}){\bf g}_{u'}^w)^T\] for any vertex $w\in\mathbb T_n$.
For a non-zero polynomial $F(\widehat{ \bf y}_{w})=\sum_{{\bf v}\in\mathbb N^n}c_{\bf v}\widehat {\bf y}_w^{\bf v}\in \mathbb N[\hat y_{1;w},\ldots,\hat y_{n;w}]
$, it can be viewed as an element of $\hat{\mathbb F}_{>0}$ and
we have 
\[\beta_{u'}(F(\widehat {\bf y}_w))=\max\{{\bf v}^T (S\mid {\bf 0}){\bf g}_{u'}^w\mid c_{\bf v}\neq 0\}=F[(S\mid {\bf 0}){\bf g}_{u'}^w].\]
Since $u$ is a good element in $\mathcal A$, \eqref{eqn:F-mutation} can be viewed as an equality in $\hat{\mathbb F}_{>0}$. Thus we can apply $\beta_{u'}$ to \eqref{eqn:F-mutation} and get
\[-[-g_{k;u}^{t'}]_+\cdot \max\{0,\;\beta_{u'}(\hat y_{k;t'})\}+F_u^{t'}[(S\mid {\bf 0}){\bf g}_{u'}^{t'}]=-[-g_{k;u}^t]_+\cdot \max\{0,\;\beta_{u'}(\hat y_{k;t})\}+F_u^t[(S\mid {\bf 0}){\bf g}_{u'}^t].
\]
Since $\beta_{u'}(\hat y_{k;t'})=s_kg_{k;u'}^{t'}$ and $\beta_{u'}(\hat y_{k;t})=s_kg_{k;u'}^t$, we obtain
\[-[-g_{k;u}^{t'}]_+[s_kg_{k;u'}^{t'}]_++F_u^{t'}[(S\mid {\bf 0}){\bf g}_{u'}^{t'}]=-[-g_{k;u}^t]_+[s_kg_{k;u'}^t]_++F_u^t[(S\mid {\bf 0}){\bf g}_{u'}^t].
\]
Thus 
\begin{eqnarray}\label{eqn:f0-inv}
     F_u^{t'}[(S\mid {\bf 0}){\bf g}_{u'}^{t'}]-F_u^t[(S\mid {\bf 0}){\bf g}_{u'}^t]=s_k([-g_{k;u}^{t'}]_+[g_{k;u'}^{t'}]_+-[-g_{k;u}^t]_+[g_{k;u'}^t]_+).
\end{eqnarray}
Since $g_{k;u}^{t'}=-g_{k;u}^t$ and $g_{k;u'}^{t'}=-g_{k;u'}^t$, we can rewrite \eqref{eqn:f0-inv} in the form given in \eqref{eqn:f-inv}.

(ii) It suffices to prove that \eqref{eqn:F-inv-mut} holds for any edge \begin{xy}(0,1)*+{t}="A",(10,1)*+{t'}="B",\ar@{-}^k"A";"B" \end{xy} in $\mathbb T_n$.  So let us assume
$({\bf x}_{t'}, \widetilde B_{t'})=\mu_k({\bf x}_t, \widetilde B_t)$. By applying (i) to the ordered pair $(u,u')$ of good elements, we have 
\[  F_u^{t'}[S({\bf g}_{u'}^{t'})^\circ]-F_u^t[S({\bf g}_{u'}^t)^\circ]=s_k([-g_{k;u}^{t'}]_+[-g_{k;u'}^{t}]_+-[-g_{k;u}^t]_+[-g_{k;u'}^{t'}]_+).\]
By applying (i) to the ordered pair $(u',u)$ of good elements, we have
\[
  F_{u'}^{t'}[S({\bf g}_{u}^{t'})^\circ]-F_{u'}^t[S({\bf g}_{u}^t)^\circ]=s_k([-g_{k;u'}^{t'}]_+[-g_{k;u}^{t}]_+-[-g_{k;u'}^t]_+[-g_{k;u}^{t'}]_+).
       \]
       Then by taking the sum of the two equalities, we obtain
       \[F_u^t[S({\bf g}_{u'}^t)^\circ]+F_{u'}^t[S({\bf g}_u^t)^\circ]=F_u^{t'}[S({\bf g}_{u'}^{t'})^\circ]+F_{u'}^{t'}[S({\bf g}_u^{t'})^\circ].\]
       Then by induction, we can see that $ F_u^t[S({\bf g}_{u'}^t)^\circ]+F_{u'}^t[S({\bf g}_u^t)^\circ]$ is independent of $t\in\mathbb T_n$.
        In particular, the number $(u\mid\mid u')_F=F_u^t[S({\bf g}_{u'}^t)^\circ]+F_{u'}^t[S({\bf g}_u^t)^\circ]$ depends only on $u$ and $u'$, not on the choice of vertex $t\in\mathbb T_n$.
\end{proof}
As will be seen in the next subsection, the above theorem also holds for cluster algebras with trivial coefficients, if we only consider cluster monomials. 

\subsection{$F$-invariant for cluster algebras with trivial coefficients}\label{sec:trivial-coeff}
Originally, the $F$-invariant was defined for any pair of good elements in  $\mathsf{\Lambda}$-cluster algebras (particularly, cluster algebras of full rank). Thanks to the formula in Theorem \ref{thm:mutation-inv} (ii), the $F$-invariant for a pair of cluster monomials can be defined for any cluster algebra with trivial coefficients. In this subsection, we give more details about this fact.

\begin{setting}\label{setting}
\begin{itemize}
\item[(i)] Let $B$ be an $n\times n$ skew-symmetrizable matrix with a fixed skew-symmetrizer $S=\diag(s_1,\ldots,s_n)$. Let ${\bf z}=(z_1,\ldots,z_n)$ and ${\bf x}=(x_1,\ldots,x_m)$ be two ordered sets of indeterminates with $m=2n$.
\item[(ii)]  Set $\widetilde B=\begin{bmatrix}
    B\\ I_n
\end{bmatrix}$. Let $\mathcal A^{\rm pr}$ be the (principal coefficient) cluster algebra with initial seed $({\bf x},\widetilde B)$ and let $\mathcal A$ be the cluster algebra with trivial coefficients whose initial seed is $({\bf z},B)$. Clearly, $\mathcal A^{\rm pr}$ is a cluster algebra of full rank.
\item[(iii)] Denote by $({\bf x}_t,\widetilde B_t)$ the seed of $\mathcal A^{\rm pr}$ at vertex $t\in\mathbb T_n$ and by  $({\bf z}_t, B_t)$ the seed of $\mathcal A$ at vertex $t\in\mathbb T_n$, where ${\bf x}_t=(x_{1;t},\ldots,x_{m;t})$ and ${\bf z}_t=(z_{1;t},\ldots,z_{n;t})$.
\end{itemize}
\end{setting}

\begin{theorem}[Separation formula, \cite{fomin_zelevinsky_2007}]
 Keep the above setting. Let $u=\prod_{i=1}^nz_{i;t}^{a_i}$ be a cluster monomial in $\mathcal A$ and $u^{\rm pr}=\prod_{i=1}^nx_{i;t}^{a_i}$ the corresponding cluster monomial in $\mathcal A^{\rm pr}$. Denote by ${\bf g}_{u^{\rm pr}}^w\in\mathbb Z^m$ and $F_{u^{\rm pr}}^w\in\mathbb Z[y_1,\ldots,y_n]$ the extended $g$-vector and $F$-polynomial of $u^{\rm pr}$ with respect to vertex $w\in\mathbb T_n$. Then the Laurent expansion of $u$ with respect to the seed $({\bf z}_w,B_w)$ has a canonical expression in terms of ${\bf g}_{u^{\rm pr}}^w$ and $F_{u^{\rm pr}}^w$: 
 \[ u={\bf z}_w^{({\bf g}_{u^{\rm pr}}^w)^\circ}F_{u^{\rm pr}}^w({\bf z}_w^{B_w{\bf e}_1},\ldots,{\bf z}_w^{B_w{\bf e}_n}),
 \]
 where $({\bf g}_{u^{\rm pr}}^w)^\circ\in\mathbb Z^n$ is the principal part of ${\bf g}_{u^{\rm pr}}^w\in\mathbb Z^m$ and ${\bf e}_k$ is the $k$-th column of $I_n$. 
\end{theorem}
\begin{definition}
    Keep the setting and notations as above. We call 
   \[ {\bf g}_{u}^w \coloneqq   ({\bf g}_{u^{\rm pr}}^w)^\circ\in\mathbb Z^n\quad \text{and}\quad F_{u}^w \coloneqq   F_{u^{\rm pr}}^w\in\mathbb Z[y_1,\ldots,y_n]\] 
    the {\em $g$-vector} and {\em $F$-polynomial} of the cluster monomial $u\in\mathcal A$ with respect to vertex $w\in\mathbb T_n$.
\end{definition}
By replacing the notations, we have the following direct consequence.
\begin{proposition} Keep the setting and notations as above. Let $u,v$ be two cluster monomials in $\mathcal A$ and let $u^{\rm pr},v^{\rm pr}$ be the corresponding cluster monomials in $\mathcal A^{\rm pr}$. Then for any vertex $w\in\mathbb T_n$, we have 
\[ F_{u}^w[S{\bf g}_{v}^w]+ F_{v}^w[S{\bf g}_{u}^w]=F_{u^{\rm pr}}^w[S({\bf g}_{v^{\rm pr}}^w)^\circ]+ F_{v^{\rm pr}}^w[S({\bf g}_{u^{\rm pr}}^w)^\circ]=(u^{\rm pr}\mid\mid v^{\rm pr})_F.
\]
 In particular, the number $(u\mid\mid v)_F\coloneqq F_{u}^w[S{\bf g}_{v}^w]+ F_{v}^w[S{\bf g}_{u}^w]$ only depends on $u$ and $v$, not on the choice of $w\in\mathbb T_n$. 
\end{proposition}
The number $(u\mid\mid v)_F=F_{u}^w[S{\bf g}_{v}^w]+ F_{v}^w[S{\bf g}_{u}^w]$ above is called the {\em $F$-invariant} of the ordered pair $(u, v)$ of cluster monomials in $\mathcal A$. The number $F_{u}^w[S{\bf g}_{v}^w]$ is called the {\em partial $F$-invariant} of the ordered pair $(u, v)$ at vertex $w\in\mathbb T_n$.

\subsection{Reading's conjecture on the separation property for cluster variables} In this subsection, we confirm a conjecture of Reading, which says that the non-compatible cluster variables in cluster algebras can be separated by the sign-coherence of $g$-vectors. The mutation formula for the partial $F$-invariant will play an important role in our proof.

Since we only care about cluster monomials and cluster variables in this section, we will work with cluster algebras with trivial coefficients.

\begin{theorem}[{\cite{GHKK18}, sign-coherence of $g$-vectors}] \label{thm:GHKK}
Let $\mathcal A$ be a cluster algebra and ${\bf z}=(z_1,\ldots,z_n)$ a cluster of $\mathcal A$. Then the $i$-th components of the $g$-vectors ${\bf g}_{z_1}^{t},\ldots,{\bf g}_{z_n}^{t}$ with respect to a vertex $t\in\mathbb T_n$ are simultaneously
non-negative or simultaneously non-positive.
\end{theorem}

\begin{definition}\label{def:sign-coherent}
Let $u$ and $u'$ be two cluster monomials of a cluster algebra $\mathcal A$. 
    Let $\{{\bf g}_{u}^t\mid t\in\mathbb T_n\}$ and   $\{{\bf g}_{u'}^t\mid t\in\mathbb T_n\}$ be the families of $g$-vectors of $u$ and $u'$ with respect to vertices of $\mathbb T_n$. We say that $u$ and $u'$ are {\em sign-coherent}, if for any vertex $t\in\mathbb T_n$ and any $k\in[1,n]$, we have 
    $g_{k;u}^tg_{k;u'}^t\geq 0$.
\end{definition}
In the study of universal geometric cluster algebras, Reading \cite{Reading-2014} conjectured a separation
property for cluster variables of $\mathcal A$, which can be reformulated as follows.
\begin{conjecture}[{\cite[Conjecture 8.21]{Reading-2014}}] \label{conj:reading}
Let $u=x_{i;t}$ and $u'=x_{j;t'}$ be two cluster variables of $\mathcal A$. Then $u$ and $u'$ are contained in the same cluster if and only if $u$ and $u'$ are sign-coherent.   
\end{conjecture}
This conjecture implies that if two cluster variables are not contained in any common cluster, then they can be separated by the sign-coherence of $g$-vectors. One side of this conjecture is nothing but the sign-coherence of $g$-vectors in Theorem \ref{thm:GHKK}. To the authors' knowledge, the other side of this conjecture is still open. 

\begin{theorem}\label{thm:sign-coherent}
    Let $\mathcal A$ be a cluster algebra and let $u$ and $u'$ be two cluster monomials of $\mathcal A$. Then 
    the product $uu'$ is still a cluster monomial if and only if  $u$ and $u'$ are sign-coherent. In particular, Reading's conjecture is true.
\end{theorem}
    
\begin{proof}
``$\Rightarrow$": Suppose that the product $uu'$ is still a cluster monomial of $\mathcal A$. Then there exists a vertex $w\in\mathbb T_n$ such that $u, u'$ and $uu'$ are cluster monomials in ${\bf x}_w$. We can assume that $u=\prod_{i=1}^n x_{i;w}^{v_i}$ and $u'=\prod_{i=1}^n x_{i;w}^{v_i'}$, where $v_i,v_i'\in\mathbb Z_{\geq 0}$. Let $t\in\mathbb T_n$. By Theorem \ref{thm:GHKK}, we know that the $k$-th components of the $g$-vectors ${\bf g}_{x_{1;w}}^t,\ldots, {\bf g}_{x_{n;w}}^t$ are  simultaneously non-negative or simultaneously non-positive for each $k\in[1,n]$. This implies $$g_{k;u}^tg_{k;u'}^t=(\sum_{i=1}^nv_ig_{k;x_{i;w}}^t)(\sum_{i=1}^nv_i'g_{k;x_{i;w}}^t)\geq 0,$$ where $t\in\mathbb T_n$ and $k\in[1,n]$. So $u$ and $u'$ are sign-coherent.

``$\Leftarrow$": Suppose $u$ and $u'$ are sign-coherent. Let $S$ be a skew-symmetrizer for the exchange matrices of $\mathcal A$. We have the following claims.

Claim (a): The partial $F$-invariants $F_u^t[S{\bf g}_{u'}^t]$ and $F_{u'}^t[S{\bf g}_{u}^t]$ are independent of $t\in\mathbb T_n$ under the assumption that $u$ and $u'$ are sign-coherent.

Claim (b): $F_u^t[S{\bf g}_{u'}^t]=0=F_{u'}^t[S{\bf g}_{u}^t]$ for any vertex $t\in\mathbb T_n$. In particular,  $(u\mid\mid u')_F=0$.

Proof of claim (a): We first show that $F_u^t[S{\bf g}_{u'}^t]$ is independent of $t\in\mathbb T_n$. It suffices to show that $F_u^t[S{\bf g}_{u'}^t]=F_u^{t'}[S{\bf g}_{u'}^{t'}]$ for any edge  \begin{xy}(0,1)*+{t}="A",(10,1)*+{t'}="B",\ar@{-}^k"A";"B" \end{xy}  in $\TT_n$. In this case,
 we have $g_{k;u}^{t'}=-g_{k;u}^t$ and $g_{k;u'}^{t'}=-g_{k;u'}^t$. Since $u$ and $u'$ are sign-coherent, we have $g_{k;u}^tg_{k;u'}^t\geq 0$.
    Thus $$(-g_{k;u}^{t'})(-g_{k;u'}^t)=-g_{k;u}^tg_{k;u'}^t\leq 0\;\;\;\text{and}\;\;\;(-g_{k;u}^t)(-g_{k;u'}^{t'})=-g_{k;u}^tg_{k;u'}^t\leq 0.$$ 
    So we have 
    $$[-g_{k;u}^{t'}]_+[-g_{k;u'}^t]_+=0=[-g_{k;u}^t]_+[-g_{k;u'}^{t'}]_+.$$
Then by Theorem \ref{thm:F-inv} (i), we obtain $F_u^t[S{\bf g}_{u'}^t]=F_u^{t'}[S{\bf g}_{u'}^{t'}]$. By induction, we see that  $F_u^t[S{\bf g}_{u'}^t]$ is independent of $t\in\mathbb T_n$. By the same arguments, we can show that $F_{u'}^t[S{\bf g}_{u}^t]$ is invariant for $t\in\mathbb T_n$.

Proof of claim (b):  Since $u$ is a cluster monomial of $\mathcal A$, we can assume that $u$ is a cluster monomial in ${\bf x}_w$ for some $w\in\mathbb T_n$. In this case, $F_u^w=1$. Thus $F_u^w[{\bf r}]=0$ for any ${\bf r}\in \mathbb Z^n$. In particular, we have $F_u^w[S{\bf g}_{u'}^w]=0$.
For any vertex $t\in\mathbb T_n$, by claim (a), we have  $F_u^t[S{\bf g}_{u'}^t]=F_u^w[S{\bf g}_{u'}^w]=0$. By the same arguments,  we have $F_{u'}^t[S{\bf g}_{u}^t]=0$ for any vertex $t\in\mathbb T_n$. Thus $$(u\mid\mid u')_F=F_u^t[S{\bf g}_{u'}^t]+F_{u'}^t[S{\bf g}_{u}^t]=0.$$
Since  $(u\mid\mid u')_F=0$, we know that  the product $uu'$ is still a cluster monomial by \cite[Theorem 4.19]{Cao-2023}.

Therefore, the product $uu'$ is still a cluster monomial if and only if  $u$ and $u'$ are sign-coherent. In the case that $u$ and $u'$ are cluster variables, this corresponds to the statement in Reading's conjecture.
\end{proof}

\section{Coincidence of the
partial $E$-invariant and the partial $F$-invariant}\label{sec:4}
\subsection{The partial $E$-invariant in the theory of quivers with potentials}\label{subsec:pE-invQP}
Throughout this subsection, let $Q=(Q_0,Q_1)$ be a quiver with a vertex set $Q_0=\{1,\ldots,n\}$ and a finite arrow set $Q_1$. We will assume that $Q$ has no cycles of length $\leq 2$. Such a quiver $Q$ corresponds to a skew-symmetric matrix $B_Q=(b_{ij})$, where
\begin{eqnarray}\label{eqn:Q-B}
   b_{ij}=\#\{j\rightarrow i \text{ in }Q_1\}-\#\{i\rightarrow j \text{ in }Q_1\}. 
\end{eqnarray}
For example, if  $Q=1\leftarrow 2$, then $B_Q=\begin{bmatrix}
    0&1\\
    -1&0
\end{bmatrix}$.

 For $m\in\mathbb Z_{\geq 0}$, let $\mathbb CQ_m$ be a $\mathbb C$-vector space with a $\mathbb C$-basis labeled by the set $Q_m$ of paths
of length $m$ in $Q$. Clearly, $\mathbb CQ_m$  is finite-dimensional. The \emph{path algebra} of $Q$ is defined as a tensor algebra $\mathbb{C}Q \coloneqq \bigoplus_{m \geq 0}\mathbb{C}Q_m$ generated by the $(\mathbb{C}Q_0, \mathbb{C}Q_0)$-bimodule $\mathbb{C}Q_1$, whose multiplication is defined by compositions of paths. In this paper we read paths from right to left. It is easy to see that the space $\mathbb{C}Q_{\geq l}\coloneqq \bigoplus_{m \geq l}\mathbb{C}Q_m$ is a two-sided ideal of $\mathbb{C}Q$.

The {\em completed path algebra} of $Q$ is similarly defined as a topological algebra $\widehat{\mathbb CQ} \coloneqq   \prod_{m\geq 0}\mathbb CQ_m$ equipped with the $\mathfrak{M}_Q$-adic topology, where $\mathfrak{M}_Q$ denotes the Jacobson radical of $\widehat{\mathbb CQ}$. A {\em potential} $W$ of $Q$ is a formal sum of the form $W=\sum_{m\geq 1}w_m$ in $\widehat{\mathbb CQ}$, where $w_m$ is a $\mathbb C$-linear combination of cycles of length $m$ in $Q$. For each arrow $a \in Q_1$, the \emph{cyclic derivation} $\partial_a W \in \widehat{\mathbb CQ}$ is defined by
\begin{eqnarray}\label{eqn:cyc-der}
    \partial_a (a_1a_2\cdots a_m) \coloneqq \sum_{a_i=a}a_{i+1}a_{i+2}\cdots a_{i-1}
\end{eqnarray}
for each cycle $a_1a_2\cdots a_m$ defining $W$ and extended linearly and continuously. 

Let $I'_W\coloneqq \langle \partial_a W\mid a\in Q_1\rangle$ be the  two-sided ideal of $\widehat{\mathbb CQ}$ generated by $\partial_a W$ with $a\in Q_1$ and $I_W\coloneqq \bigcap_{l\geq 1}(\mathfrak{M}_Q^l+I'_W)$ the $\mathfrak{M}_Q$-adic closure  of  $I'_W$.
The ideal $I_W$ is  called the {\em Jacobian ideal}  (see \cite[\S3, Definition~3.1]{DWZ08}) and the quotient algebra $J=J(Q,W)=\widehat{\mathbb CQ}/I_W$ is called the  {\em Jacobian algebra}  associated with $(Q,W)$. A quiver with potential $(Q,W)$ is said to be {\em Jacobi-finite} if the corresponding Jacobian algebra $J(Q,W)=\widehat{\mathbb CQ}/I_W$ is finite-dimensional.

\begin{example}
    Consider the quiver $Q$ as follows:
    $$
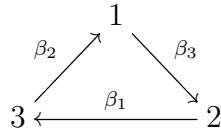
\xymatrix{&1\ar[rd]^{\beta_3}&\\
3\ar[ru]^{\beta_2}&&2\ar[ll]_{\beta_1}}$$
Then  $Q_2=\{\beta_2\beta_1,\;\beta_3\beta_2,\;\beta_1\beta_3\}$. If we take $W=\beta_3\beta_2\beta_1$, the corresponding Jacobian algebra $J=J(Q,W)$ is given by the quiver $Q$ with relations $\beta_2\beta_1=0,\;\beta_3\beta_2=0,\;\beta_1\beta_3=0$. In this case, $(Q,W)$ is Jacobi-finite.
\end{example}

    \begin{example}\label{ex:QP2}
       Now we assume the potential $W$ is given as a finite sum of cycles in the quiver $Q$. By abuse of notation, we also write \[I'_W \coloneqq \langle \partial_a W\mid a\in Q_1\rangle \subset \mathbb{C}Q\] for the ideal of the path algebra $\mathbb{C}Q$.

            (i) In general, the algebras $J(Q,W)=\widehat{\mathbb CQ}/I_W$ and $\mathbb{C}Q/I'_W$ are not necessarily isomorphic to each other, even when $(Q, W)$ is Jacobi-finite. The following example, taken from \cite{Lab09,Lab09-b}, illustrates this phenomenon. 
       Consider the following quiver $Q$ 
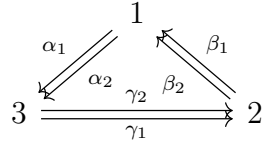
    \[\begin{tikzcd}
& 1 \arrow[dl,shift left=0.3ex,"\alpha_2"] \arrow[dl,shift right=0.3ex,"\alpha_1"'] &
\\
3 \arrow[rr,shift left=0.3ex,"\gamma_2"] \arrow[rr,shift right=0.3ex,"\gamma_1"']& & 2 \arrow[ul,shift left=0.3ex,"\beta_2"] \arrow[ul,shift right=0.3ex,"\beta_1"']
\end{tikzcd}\]
with potential $W=\al_1\be_1\ga_1 + \al_2\be_2\ga_2 - \al_1\be_2\ga_1\al_2\be_1\ga_2$. Then we have $\partial_{\al_1}(W) = \be_1\ga_1 - \be_2\ga_1\al_2\be_1\ga_2$, etc. 
We can directly check that every path of length greater than $7$ can be rewritten, using the relations, as a strictly longer path. Iterating this procedure, we see that its image in $J(Q,W)$
vanishes, since $J(Q,W)$ is separated with respect to the induced
$\mathfrak M_Q$-adic topology. For example, we have
\[\al_1\be_2\ga_1\al_2\be_1\ga_1 = \al_1\be_2\ga_1\al_2\be_2\ga_1\al_2\be_1\ga_2 = \cdots = 0\]
in $J(Q,W)$. In particular, the algebra $J(Q,W)$ is finite-dimensional. But if we work with $\mathbb{C}Q/I'_W$, we cannot kill such paths by relations.

(ii) 
If the condition $\mathbb{C}Q_{\geq l}\subseteq I_W'\subseteq \mathbb{C}Q_{\geq 2}$ holds for some $l\geq 2$, then it follows from the definition of the complete algebra $J(Q,W)$ that $J(Q,W) \cong \mathbb{C}Q/I'_W$ (see \cite[Proposition 3.3]{Pasquali-2020}). 
    \end{example}

A (finite-dimensional) {\em representation} of $Q=(Q_0,Q_1)$ is a tuple $M=(M_i, M_\alpha)_{i\in Q_0,\alpha\in Q_1}$, where each $M_i$ is a finite-dimensional $\mathbb C$-vector space and $M_\alpha\colon M_i\rightarrow M_j$ is a linear map for each arrow $\alpha\colon i\rightarrow j$. We can identify the representations of $Q$ with the finite-dimensional left modules over the path algebra $\mathbb CQ$. For a representation $M=(M_i, M_\alpha)_{i\in Q_0,\alpha\in Q_1}$ of $Q$, we denote by $\underline \dim M=(\dim M_1,\ldots,\dim M_n)^T\in\mathbb N^n$ the {\em dimension vector} of $M$.

A representation $M$ of $Q$ is {\em nilpotent} if there exists $m\in\mathbb Z_{>0}$ such that $p\cdot M=0$ for any path $p\in\mathbb CQ$ of length greater than $m$. It is known that the category of finite-dimensional nilpotent representations of $Q$ is equivalent to the category of finite-dimensional $\widehat{\mathbb CQ}$-modules.

 A {\em representation} $M$ of $(Q,W)$ is a (finite-dimensional) nilpotent representation of $Q$ whose induced $\widehat{\mathbb CQ}$-action is annihilated by the Jacobian ideal $I_W$. We can identify the representations of $(Q,W)$ with the finite-dimensional  left modules over $J=J(Q,W)$.

A {\em decorated representation} of $(Q,W)$ is a pair $\mathcal M=(M,V)$, where $M$ is a finite-dimensional left
$J$-module and $V$ is a finitely generated semisimple left $J$-module. Sometimes, we also view $V$ as a 
tuple of finite-dimensional $\mathbb C$-vector spaces
$V=(V_i)_{i\in Q_0}=(V_1,\ldots,V_n)$. A decorated representation $\mathcal M=(M,V)$ of $(Q,W)$ is said to be {\em negative}, if $M=0$.

    Following \S \ref{subsecA:inj}, we introduce numerical characteristics of decorated representations in terms of (relative) injective objects in the category of locally nilpotent $J$-modules.
Let $\mathcal{K}^b(J\mathchar`-\mathsf{inj})$ be the homotopy category of bounded complexes of finitely cogenerated injective objects in the category $J\Modln$ of locally nilpotent $J$-modules (see Definition~\ref{def:l-nil}). Let $\mathcal{K}^{[0,1]}(J\mathchar`-\mathsf{inj})$ be the subcategory of $\mathcal{K}^b(J\mathchar`-\mathsf{inj})$ consisting of two-term complexes concentrated in cohomological degree $0$ and $1$, and let $\mathcal{K}_{\rm fd}^{[0,1]}(J\mathchar`-\mathsf{inj})$ be the subcategory of $\mathcal{K}^{[0,1]}(J\mathchar`-\mathsf{inj})$ whose $0$-th cohomology is finite-dimensional. Note that if $(Q,W)$ is Jacobi-finite, then $\mathcal{K}_{\rm fd}^{[0,1]}(J\mathchar`-\mathsf{inj})=\mathcal{K}^{[0,1]}(J\mathchar`-\mathsf{inj})$.

There is a bijection between the isomorphism classes of the decorated representations of $(Q,W)$ and those of the two-term complexes in $\mathcal{K}_{\rm fd}^{[0,1]}(J\mathchar`-\mathsf{inj})$ given by
\begin{eqnarray}\label{eqn:s-t-bijection}
    \Phi\colon \mathcal M=(M,V)\mapsto  \Phi(\mathcal M)\coloneqq (I_M^{-}{\to}I_M^+)\oplus (0\to I_V),
\end{eqnarray}
where $0\to M\to I_M^-{\to} I_M^+$ is the minimal injective copresentation of $M$ as an object in $J\Modln$ and $I_V=\oplus_{i=1}^nI_i^{a_i}$ if $V=(\mathbb C^{a_1},\ldots,\mathbb C^{a_n})$. Here $I_i$ denotes the injective envelope of the simple $J$-module $S_i$ at vertex $i\in Q_0$ in the category $J\Modln$.

    The notions of $g$-vectors, $F$-polynomials and $E$-invariants of decorated representations were introduced by Derksen, Weyman and Zelevinsky \cite{DWZ10}. For convenience, we reformulate them in terms of the category $\mathcal{K}^b(J\mathchar`-\mathsf{inj})$ following Theorem~\ref{thm:g-DWZ}.
    Similar reformulations in another categorical setting can also be found in \cite{Plamondon-2011} (see also a related further study \cite{CLS-2015}). 
    \begin{definition}[{$g$-vector, $F$-polynomial and $E$-invariant}] \label{def:g-F-E}
    Let $\mathcal M=(M,V)$ be a decorated representation of $J=J(Q,W)$ and $\Phi(\mathcal M)=(I_M^{-}{\to}I_M^+)\oplus (0\to I_V)$ the corresponding two-term complex in $\mathcal{K}_{\rm fd}^{[0,1]}(J\mathchar`-\mathsf{inj})$. 
    \begin{itemize}
        \item [(i)] The  {\em $g$-vector} ${\bf g}({\mathcal M})\in\mathbb Z^n$ of $\mathcal M$ is defined as follows:
       \[{\bf g}(\mathcal{M}) \coloneqq -[\Phi(\mathcal M)]=[I_V]+[I_M^+] -[I_M^-]\in K_0(\mathcal{K}^b(J\mathchar`-\mathsf{inj}))\cong \mathbb{Z}^n,\]
        where we make the identification  $[I_i] \mapsto {\bf e}_i$ for $i\in Q_0=\{1,\ldots,n\}$.

\item [(ii)] The {\em $F$-polynomial} $F_{\mathcal M}$ of $\mathcal M=(M,V)$ is defined as follows:
\[ F_{\mathcal M}(y_1,\ldots,y_n) \coloneqq   \sum_{{\bf v}\in\mathbb N^n}\chi(\Gr_{\bf v}(M)){\bf y}^{\bf v}\;\;\in\mathbb Z[y_1,\ldots,y_n],
\]
where $\Gr_{\bf v}(M)$ is the submodule Grassmannian of $M$ with dimension vector ${\bf v}$ and $\chi$ is the Euler-Poincar\'{e} characteristic.

\item[(iii)]  The {\em partial $E$-invariant} $ E^{\rm inj}(\mathcal M,\mathcal M')$ and {\em $E$-invariant} $E^{\rm sym}(\mathcal M,\mathcal M')$ of a pair $(\mathcal M,\mathcal M')$ of decorated representations are defined as follows:
\begin{eqnarray}\label{eqn:E-inj}
    E^{\rm inj}(\mathcal M,\mathcal M') &\coloneqq&   \dim_{\mathbb C} \Hom_{\mathcal{K}^{b}(J\mathchar`-\mathsf{inj})}(\Phi(\mathcal M),\Phi(\mathcal M')[1]),\\
    E^{\rm sym}(\mathcal M,\mathcal M') &\coloneqq&    E^{\rm inj}(\mathcal M,\mathcal M')+E^{\rm inj}(\mathcal M',\mathcal M).\label{eqn:E-sym}
\end{eqnarray}
    \end{itemize}
\end{definition}

\begin{remark}\label{rmk:def-DWZ}
(i)  In Theorem~\ref{thm:g-DWZ}, we prove that the definition of $g$-vectors  here is consistent with the original definition used by Derksen, Weyman and Zelevinsky in 
\cite{DWZ10}.

(ii) Denote by $\langle-,-\rangle$ the standard inner product on $\mathbb R^n$. By  Proposition~\ref{pro:E-hom}, we see that the partial $E$-invariant $E^{\rm inj}(\mathcal M,\mathcal M')$ of two decorated representations $\mathcal M=(M,V)$ and $\mathcal M'=(M', V')$ can be written as
\[ E^{\rm inj}(\mathcal M,\mathcal M')=\langle\underline{\dim} M, \mathbf{g}(\mathcal{M}')\rangle + \dim_{\mathbb C} \Hom_J(M, M' ),
\]
which is the formula used in \cite[\S7]{DWZ10} to define the partial $E$-invariant. From the right-hand side of this equality, we see that $E^{\rm inj}(\mathcal M,\mathcal M')$ is always a finite integer.
\end{remark}

\begin{corollary}\label{cor:E-inv}
 Let $\mathcal M=(M,V)$ and  $\mathcal M'=(M',V')$ be two decorated representations of $J=J(Q,W)$. If 
 $\mathcal M=(M,V)$ is negative,  \ie $M=0$ and  $V=(\mathbb C^{a_1},\ldots,\mathbb C^{a_n})$ for some ${\bf a}=(a_1,\ldots,a_n)^T\in\mathbb Z_{\geq 0}^n$, then we have
 \[
E^{\rm inj}(\mathcal M,\mathcal M')=0,\;\;
 E^{\rm inj}(\mathcal M',\mathcal M)=\langle \underline{\dim} M', {\bf a}\rangle=E^{\rm sym}(\mathcal M,\mathcal M').
 \]
\end{corollary}
\begin{proof}
By definition, we have ${\bf g}(\mathcal M)={\bf a}$. Then the results follow from the formulas in Remark \ref{rmk:def-DWZ} (ii) and \eqref{eqn:E-sym}.
\end{proof}

\subsection{The partial $E$-invariant coincides with the partial $F$-invariant}\label{subsec:EF} We refer to \cite{DWZ08,DWZ10} for the definitions of  mutations of quivers with potentials and  mutations of decorated representations, which will not be used directly in this paper. Recall that $\mathbb T_n$ is the $n$-regular tree with a rooted vertex $t_0$. 
\begin{setting}
\begin{itemize}
\item[(i)] Let $(Q,W)=(Q_{t_0},W_{t_0})$ be a quiver with  non-degenerate potential \cite[\S7]{DWZ08} and $\mathcal M=\mathcal M^{t_0}=(M^{t_0},V^{t_0})$ a decorated representation of $(Q_{t_0},W_{t_0})$. By applying mutations, we have a family of quivers with potentials and a family of decorated representations
\begin{eqnarray}
   [(Q,W)] \coloneqq   \{(Q_t,W_t)\mid t\in\mathbb T_n\},\;\;\;\;[\mathcal M] \coloneqq   \{\mathcal M^t\mid t\in\mathbb T_n\},\nonumber
\end{eqnarray}
where $\mathcal M^t=(M^t, V^t)$ is a decorated representation of $(Q_t,W_t)$.
\item[(ii)] Let $B_Q$ be the skew-symmetric matrix corresponding to $Q$ as in \eqref{eqn:Q-B}, and $\mathcal A_Q$ the cluster algebra with trivial coefficients such that its exchange matrix at vertex $t_0$ is $B_Q$. 
\item[(iii)] For a cluster monomial $u={\bf x}_w^{\bf a}=\prod_{i=1}^nx_{i;w}^{a_i}$ of $\mathcal A_Q$, let $\{\mathcal M_u^t\mid t\in\mathbb T_n\}$ be the family of decorated representations determined by the following conditions:
\begin{itemize}
    \item[$\bullet$] For the vertex $w\in\mathbb T_n$, set $\mathcal M_u^w=(M_u^w,V_u^w)$ to be the negative decorated representation of 
    $(Q_w,W_w)$ given by
$M_u^w=0$ and $V_u^w=(\mathbb C^{a_1},\ldots,\mathbb C^{a_n})$.
\item[$\bullet$] For any edge \begin{xy}(0,1)*+{t}="A",(10,1)*+{t'}="B",\ar@{-}^k"A";"B" \end{xy}  in ${\mathbb T}_n$, $\mathcal M_u^t$ and $\mathcal M_u^{t'}$ are related by mutation in direction $k$.
\end{itemize}
\end{itemize}
\end{setting}

A decorated representation $\mathcal M=(M,V)$ of $(Q,W)$ is said to be {\em negative-reachable} (or simply {\em reachable}), if in the mutation class $[\mathcal M]=\{\mathcal M^t\mid t\in \mathbb T_n\}$ of $\mathcal M$, there exists a vertex $w\in \mathbb T_n$ such that $\mathcal M^w$ is a negative decorated representation of $(Q_w, W_w)$. 

The following theorem summarizes some results on the categorification of cluster algebras using decorated representations of quivers with potentials
by Derksen, Weyman and Zelevinsky \cite{DWZ10}.

\begin{theorem}[{\cite{DWZ10}}]
\label{thm:DWZ-g-F}
Keep the above setting. The following statements hold.
\begin{itemize}
    \item [(i)]\label{eq:DWZ-g-F1} The family $\{\mathcal M_u^t=(M_u^t,V_u^t)\mid t\in\mathbb T_n\}$  of decorated representations only depends on $u$, not on the choice of $w\in\mathbb T_n$ and ${\bf a}\in\mathbb Z_{\geq 0}^n$ such that $u={\bf x}_w^{\bf a}$.
    \item [(ii)]\label{eq:DWZ-g-F2} The correspondence $u\mapsto \{\mathcal M_u^t\mid t\in\mathbb T_n\}\mapsto \mathcal M_u\coloneqq \mathcal M_u^{t_0}$ induces a bijection from  the set of cluster monomials of $\mathcal A_Q$ to the set of isomorphism classes of negative-reachable decorated representations of $(Q,W)=(Q_{t_0},W_{t_0})$. 
    \item[(iii)]\label{eq:DWZ-g-F3} For any vertex $t\in\mathbb T_n$, we have 
    \[{\bf g}_u^t={\bf g}(\mathcal M_u^t)\;\;\;\text{and}\;\;\;F_u^t(y_1,\ldots,y_n)=F_{\mathcal M_u^t}(y_1,\ldots,y_n).
    \]
\end{itemize}
\end{theorem}

\begin{lemma}[{\cite[Theorem 7.1]{DWZ10}}]
\label{lem:E-inv}
  Let $(Q,W)$ be a quiver with non-degenerate potential, and let $\mathcal A_Q$ be the corresponding cluster algebra with trivial coefficients. 
Let $u, u'$ be two cluster monomials of $\mathcal A_Q$, and let  $\{\mathcal M_u^t\mid t\in\mathbb T_n\}$ and $\{\mathcal M_{u'}^t\mid t\in\mathbb T_n\}$ be the families of decorated representations corresponding to $u$ and $u'$. Then the following statements hold.
\begin{itemize}
\item[(i)] For any edge \begin{xy}(0,1)*+{t}="A",(10,1)*+{t'}="B",\ar@{-}^k"A";"B" \end{xy}  in $\TT_n$, 
  \begin{eqnarray*}\label{eqn:e-inv}
      E^{\rm inj}(\mathcal M_u^{t'},\mathcal M_{u'}^{t'})-E^{\rm inj}(\mathcal M_u^{t},\mathcal M_{u'}^{t})=[-g_{k;u}^{t'}]_+[-g_{k;u'}^{t}]_+-[-g_{k;u}^t]_+[-g_{k;u'}^{t'}]_+.
  \end{eqnarray*}
\item[(ii)] For any two vertices $t,t'\in\mathbb T_n$,
  \begin{eqnarray*}\label{eqn:E-inv}
      E^{\rm sym}(\mathcal M_u^t,\mathcal M_{u'}^t)=E^{\rm sym}(\mathcal M_u^{t'},\mathcal M_{u'}^{t'}).
  \end{eqnarray*}
In particular, the $E$-invariant is mutation-invariant under initial seed mutations.
\end{itemize}
\end{lemma}
\begin{proof}
(i) By \cite[Theorem 7.1]{DWZ10}, we have 
\begin{eqnarray*}\label{eqn:e0-inv}
    E^{\rm inj}(\mathcal M_u^{t'},\mathcal M_{u'}^{t'})-E^{\rm inj}(\mathcal M_u^{t}, \mathcal M_{u'}^{t})=h_k(\mathcal M_u^{t'})h_k(\mathcal M_{u'}^t)-h_k(\mathcal M_u^{t})h_k(\mathcal M_{u'}^{t'}),
\end{eqnarray*}
where $h_k(\mathcal M_u^{w})=-[-g_{k;u}^{w}]_+$ and $h_k(\mathcal M_{u'}^{w})=-[-g_{k;u'}^{w}]_+$ for any vertex $w\in\mathbb T_n$, by \cite[equality (9.1)]{DWZ10}. Then the result follows.

(ii) This follows from (i) and the definition of $E$-invariant. 
\end{proof}

\begin{corollary}\label{cor:neg-rigid}
    Let $\mathcal M$ be a negative-reachable decorated representation of $(Q,W)$. Then $E^{\rm inj}(\mathcal M,\mathcal M)=0$.
\end{corollary}
\begin{proof}
    This follows from Lemma \ref{lem:E-inv} (ii) and Corollary \ref{cor:E-inv}.
\end{proof}

We take $S=I_n$ as the fixed skew-symmetrizer for the exchange matrices of $\mathcal A_Q$, when we consider the (partial) $F$-invariant for cluster monomials in $\mathcal A_Q$.
    The following theorem is a direct consequence of Lemma~\ref{lem:E-inv} and  Theorem~\ref{thm:F-inv}. It will serve as one of the key tools for our discussion in \S \ref{sec: application}.
 \begin{theorem}
 \label{thm:F-E-inv}
Let $(Q,W)$ be a quiver with non-degenerate potential and $\mathcal A_Q$ the corresponding cluster algebra with trivial coefficients. 
Let $u, u'$ be two cluster monomials of $\mathcal A_Q$, and let  $\{\mathcal M_u^t\mid t\in\mathbb T_n\}$ and $\{\mathcal M_{u'}^t\mid t\in\mathbb T_n\}$ be the families of decorated representations corresponding to $u$ and $u'$. Then we have
\[E^{\rm inj}(\mathcal M_u^t,\mathcal M_{u'}^t)=F_u^t[{\bf g}_{u'}^t]\;\;\;\;\text{and}\;\;\;E^{\rm sym}(\mathcal M_u^t,\mathcal M_{u'}^t)=F_u^t[{\bf g}_{u'}^t]+F_{u'}^t[{\bf g}_{u}^t]=(u\mid\mid u')_F,
\]
for any vertex $t\in\mathbb T_n$.
\end{theorem}
\begin{proof}

Since $u$ is a cluster monomial of $\mathcal A_Q$, we can take a vertex $w\in\mathbb T_n$ such that $u={\bf x}_w^{\bf a}$ for some ${\bf a}=(a_1,\ldots,a_n)^T\in\mathbb Z_{\geq 0}^n$.
 In this case, the decorated representation $\mathcal M_u^w=(M_u^w,V_u^w)$ is a negative decorated representation. More precisely, we have $M_u^w=0$ and $V_u^w=(\mathbb C^{a_1},\ldots, \mathbb C^{a_n})$. Then by Corollary \ref{cor:E-inv}, we have $E^{\rm inj}(\mathcal M_u^w,\mathcal M_{u'}^w)=0$.

Since $u$ is a cluster monomial in ${\bf x}_w$, we have $F_u^w(y_1,\ldots,y_n)=1$. Thus $F_u^w[{\bf r}]=0$ for any ${\bf r}\in\mathbb Z^n$. In particular, we have
\[
F_u^w[{\bf g}_{u'}^w]=0=E^{\rm inj}(\mathcal M_u^w,\mathcal M_{u'}^w).
\]
By Theorem \ref{thm:F-inv} (i)  and Lemma \ref{lem:E-inv} (i), we know that the partial $E$-invariant $E^{\rm inj}(\mathcal M_u^t,\mathcal M_{u'}^t)$ and the partial $F$-invariant $F_u^t[{\bf g}_{u'}^t]$ satisfy the same recurrence relations for $t\in\mathbb T_n$.
Since $E^{\rm inj}(\mathcal M_u^w,\mathcal M_{u'}^w)=F_u^w[{\bf g}_{u'}^w]$ holds at the vertex $w$, we obtain that 
\begin{eqnarray}\label{eqn:e-inj=f-inv}
    E^{\rm inj}(\mathcal M_u^t,\mathcal M_{u'}^t)=F_u^t[{\bf g}_{u'}^t]
\end{eqnarray}
holds for any vertex $t\in\mathbb T_n$.
Similarly, we have $E^{\rm inj}(\mathcal M_{u'}^t,\mathcal M_{u}^t)=F_{u'}^t[{\bf g}_{u}^t]$ for any vertex $t\in\mathbb T_n$. Thus
 \[ 
E^{\rm sym}(\mathcal M_u^t,\mathcal M_{u'}^t)= E^{\rm inj}(\mathcal M_u^t,\mathcal M_{u'}^t)+ E^{\rm inj}(\mathcal M_{u'}^t,\mathcal M_{u}^t) =F_u^t[{\bf g}_{u'}^t]+F_{u'}^t[{\bf g}_{u}^t]=(u\mid\mid u')_F
 \]
holds for any vertex $t\in\mathbb T_n$.
\end{proof}

\begin{remark}\label{rmk-end}
(i) For certain classes of skew-symmetrizable cluster algebras, there are several known generalizations of the theory of quivers with potentials. In particular, generalizations of Lemma~\ref{lem:E-inv}~(i) are studied  in \cite[Theorem~10.1]{Dem10} and \cite[Theorem~5.5]{Fei26}. In these cases, by comparing Theorem~\ref{thm:F-inv}~(i) with these results, we see that the $E$-invariant and the $F$-invariant again coincide on cluster monomials, as in Theorem~\ref{thm:F-E-inv}. To the best of the authors' knowledge, the $E$-invariant has not been generally defined for arbitrary skew-symmetrizable cluster algebras, in contrast to the $F$-invariant.

        (ii) The equality \eqref{eqn:e-inj=f-inv} was first observed in \cite[Theorem~3.22]{fei_2019b}. In Fei's works \cite{fei_2019a,fei_2019b}, he defined and studied the \emph{tropical $F$-polynomials} of representations via piecewise linear functions on Harder--Narasimhan polytopes, which coincide with the Newton polytopes of the corresponding $F$-polynomials (see \cite[Remark 2.4 (1)] {fei_2019a} or \cite[Corollary 3.3]{BKT-2014}). Although his method does not impose subtraction-free conditions on $F$-polynomials, it gives, for cluster monomials, the same evaluations as our partial $F$-invariant.
\end{remark}

\section{Monoidal categorification and pole orders of normalized $R$-matrices}\label{sec:pole-order}

In this section, we discuss an interpretation of the partial $F$-invariant in the monoidal categorification of cluster algebras arising from finite-dimensional modules over quantum affine algebras. 
Our main result (Theorem \ref{thm:F-o}) presents a sufficient condition under which the partial $F$-invariants coincide with the pole orders of the normalized $R$-matrices between simple modules.
As an application, we deduce a formula expressing the pole orders for reachable simple modules in terms of their (truncated) $q$-characters (Theorem \ref{thm:oF}).

\subsection{Monoidal categorification}

Let $\mathbf{k}$ be a base field and let $(\mathcal C,\otimes)$ be a $\mathbf{k}$-linear monoidal abelian category  such that any object of $\mathcal C$ has finite length. 
We always assume that the tensor product $\otimes$ is bi-exact. 
The Grothendieck ring $K_0(\mathcal{C})$ carries a canonical $\mathbb{Z}$-basis formed by the classes of simple objects, with positive structure constants.
We denote by $[M]$ the class of an object $M\in\mathcal C$ in $K_0(\mathcal{C})$.

\begin{definition} Let $M$ and $N$ be two simple objects in $\mathcal C$.
\begin{itemize}
    \item[(i)] We say that $M$ and $N$ {\em commute}, if $M\otimes N\cong N\otimes M$.  
    \item[(ii)] We say that $M$ and $N$ {\em strongly commute}, if $M$ and $N$ commute and  $M\otimes N$ is simple. 
    \item[(iii)] We say that $M$ is {\em real} if $M\otimes M$ is simple, and that $M$ is {\em imaginary} otherwise.
\end{itemize}
\end{definition}

\begin{definition} Let $\mathcal M=(M_1,\ldots,M_r)$ be an ordered set of simple objects in $\mathcal C$.
\begin{itemize}
    \item [(i)] We say that $\mathcal M$ is a  {\em commuting set}, if $M_i$ and $M_j$ commute for any $i,j\in[1,r]$.
    \item[(ii)] We say that $\mathcal M$ is a {\em strongly commuting set}, if $M_i$ and $M_j$ strongly commute for any $i,j\in[1,r]$. In particular, each $M_i$ is a real simple object.
    \end{itemize}
\end{definition}

Given a commuting  set $\mathcal M=(M_1,\ldots,M_r)$ and a vector ${\bf a}=(a_1,\ldots,a_r)^T\in\mathbb Z_{\geq 0}^r$, 
we denote 
$$\mathcal M({\bf a}) \coloneqq M_1^{\otimes a_1}\otimes M_2^{\otimes a_2}\otimes \ldots \otimes M_r^{\otimes a_r}.$$
Clearly, for any ${\bf a}, {\bf c}\in\mathbb Z_{\geq 0}^r$, we have
 $\mathcal{M}(\mathbf{a}) \otimes \mathcal{M}(\mathbf{c}) \cong \mathcal{M}(\mathbf{a }+ \mathbf{c})$.

\begin{definition}[Monoidal categorification]
  A  $\mathbf{k}$-linear, monoidal abelian length category $\mathcal C$ is a {\em monoidal categorification} of a cluster algebra $\mathcal A$ (with non-invertible frozen variables), if there exists a $\mathbb Z$-algebra isomorphism $$\varphi\colon{K}_0(\mathcal C)\rightarrow \mathcal A$$ such that the cluster monomials of $\mathcal A$ correspond to a subset of simple objects (up to isomorphism) in $\mathcal C$. We denote this monoidal categorification by $(\mathcal A, \mathcal C, \varphi)$.
\end{definition}

The notion of monoidal categorification of cluster algebras was originally introduced by Hernandez and Leclerc \cite{HL_2010}. 
Examples of monoidal categorification of cluster algebras include various monoidal subcategories of finite-dimensional modules over quantum affine algebras and quiver Hecke algebras \cite{HL16, kkko-2018, kkop-2024}, and the category of perverse coherent sheaves on the affine Grassmannian of $GL_n$ \cite{CW19}.

From now on, we fix a monoidal categorification $(\mathcal A, \mathcal C, \varphi)$ and denote \[\varphi(M) \coloneqq \varphi([M])\in\mathcal A\]
for any object $M\in\mathcal C$. Let $({\bf x}_t,\widetilde B_t)$ be a seed of $\mathcal A$ and let $M_{i;t}\in\mathcal C$ be a simple object such that $\varphi(M_{i;t})=x_{i;t}$, which is unique up to isomorphism.  We call $\mathcal M_t \coloneqq (M_{1;t},\ldots,M_{m;t})$ a {\em monoidal cluster} and $(\mathcal M_t,\widetilde B_t)$ a {\em monoidal seed}. A simple object $M\in\mathcal C$ is said to be {\em reachable} if $\varphi(M)$ is a cluster monomial.

The following facts can be checked easily.
\begin{itemize}
    \item Each monoidal cluster $\mathcal M_t=(M_{1;t},\ldots,M_{m;t})$ is a strongly commuting set.
    \item Let ${\bf a}\in\mathbb Z_{\geq 0}^m$. Then $\varphi(\mathcal M_t({\bf a}))={\bf x}_t^{\bf a}$. In particular, $\mathcal M_t({\bf a})$ is a reachable simple object.
    \item Each reachable simple object of $\mathcal C$ is real.
\end{itemize}

\begin{proposition}[{\cite[Proposition 2.2]{HL_2010}}]
\label{prop:HL_2010}
  Let $(\mathcal A, \mathcal C, \varphi)$ be a monoidal categorification and $M$ an object in $\mathcal C$. Then $\varphi(M)$ is universally positive in $\mathcal A$, that is, for any cluster ${\bf x}=(x_{1},\ldots,x_{m})$ of $\mathcal A$, we have $\varphi(M)\in\mathbb Z_{\geq 0}[x_{1}^{\pm 1},\ldots,x_{m}^{\pm 1}]$.
\end{proposition}

\subsection{$\Lambda$-invariant for quantum affine algebras}\label{Ssec:Lqaa}
In what follows, we only consider monoidal categories consisting of  finite-dimensional modules over quantum affine algebras. 
In this subsection, we prepare the necessary notation and recall the definition and some properties of the $\Lambda$-invariant introduced by Kashiwara--Kim--Oh--Park \cite{kkop-2020}.     

Let $q$ be an indeterminate and $\mathbf{k}$ be the algebraic closure of the field $\mathbb{Q}(q)$ inside $\bigcup_{s \in \mathbb{Z}_{>0}}{\mathbb{C}}(\!(q^{1/s})\!)$. 
We consider an affine Kac--Moody algebra $\mathfrak{g}$ of type $\mathrm{X}_{N}^{(r)}$ (in the notation of \cite[Chapter 4]{Kac}) and the quantum affine algebra  $U_q'(\mathfrak{g})$ (without the degree operator).  
See \cite[\S2]{kkop-2020} for its precise definition.
This is a Hopf algebra over $\mathbf{k}$ generated by the Chevalley type generators $e_i, f_i, K_i^{\pm 1}$ for $i \in \mathtt{I}$, where $\mathtt{I}$ denotes the node set of the Dynkin diagram of $\mathfrak{g}$.   
In this paper, we use the coproduct $\Delta$ given by
\[ \Delta(e_i) =  e_i \otimes K_i^{-1} + 1 \otimes e_i, \quad 
\Delta(f_i) = f_i \otimes 1 + K_i \otimes f_i, \quad \Delta(K_i) = K_i \otimes K_i\]
for $i \in \mathtt{I}$.
We choose the affine node $0 \in \mathtt{I}$ as in \cite[Chapter 4]{Kac} except for type $\mathrm{A}_{2l}^{(2)}$, for which we choose $0\in \mathtt{I}$ so that the corresponding simple root $\alpha_0$ is the longest simple root. 
We set $\mathtt{I}_0 \coloneqq \mathtt{I} \setminus \{ 0\}$.
We normalize the standard symmetric bilinear form $(-,-)$ among weights so that the shortest simple roots have  square length $2$.  
Note that this normalization is different from the one in \cite{kkop-2020, kkop-2024} when $r=1$ and $\mathrm{X} \in \{\mathrm{B, C, F, G}\}$, but compatible with \cite{HL16}.
Let $h^\vee$ be the dual Coxeter number for $\mathfrak{g}$ as in \cite[Chapter 6]{Kac}, and $r^\vee \in \{1,2,3 \}$ the number such that the Langlands dual of $\mathfrak{g}$ is of type $\mathrm{Y}_{N^\vee}^{(r^\vee)}$.

Let $\mathscr{C}$ be the category of finite-dimensional (left) $U'_q(\mathfrak{g})$-modules of type $\mathbf{1}$ (\ie the action of $K_i$ is semisimple with eigenvalues in $q^\mathbb{Z}$ for each $i \in \mathtt{I}$).
It carries the natural structure of a $\mathbf{k}$-linear rigid monoidal abelian category, whose tensor product is $\otimes = \otimes_{\mathbf{k}}$. 
Any simple object $V$ of the category $\mathscr{C}$ has a non-zero extremal weight vector $v$  satisfying $e_i v= f_0 v = 0$ for all $i \in \mathtt{I}_0$. 
Such a vector $v$ is unique up to multiplication by a scalar in $\mathbf{k}^\times$.
We refer to this vector as a {\em dominant extremal weight vector} of $V$.

Let $z$ be another indeterminate. 
For any $U'_q(\mathfrak{g})$-module $M$, we define its {\em affinization} $M_z$ to be the $\mathbf{k}[z^{\pm 1}]$-module $M_z \coloneqq M \otimes \mathbf{k}[z^{\pm 1}]$ endowed with the $U_q'(\mathfrak{g})$-action by 
\[ e_i (m \otimes a) \coloneqq e_i m \otimes z^{\delta_{i,0}}a, \quad 
f_i (m\otimes a) \coloneqq f_i m \otimes z^{-\delta_{i,0}} a, \quad 
K_i^{\pm 1} (m \otimes a) \coloneqq K_i^{\pm1} m \otimes a\]
for any $m \in M$, $a \in \mathbf{k}[z^{\pm 1}]$ and $i\in \mathtt{I}$.

For an ordered pair $(M,N)$ of objects of $\mathscr{C}$, the universal $R$-matrix of $U_q'(\mathfrak{g})$ induces a canonical isomorphism 
\[ R_{M,N_z}^{\mathrm{univ}} \colon (M \otimes N_z) \otimes_{\mathbf{k}[z^{\pm 1}]} \mathbf{k}(\!(z)\!) \to (N_z \otimes M) \otimes_{\mathbf{k}[z^{\pm 1}]} \mathbf{k}(\!(z)\!)  \]
of $U'_q(\mathfrak{g}) \otimes_{\mathbf{k}}\mathbf{k}(\!(z)\!)$-modules, satisfying the equality
\begin{equation} \label{eq:q-tri} (R^{\mathrm{univ}}_{L, N_z} \otimes \mathsf{id}_M) \circ (\mathsf{id}_L \otimes R^{\mathrm{univ}}_{M, N_z}) = R^{\mathrm{univ}}_{L \otimes M, N_z}   \end{equation}
for any $L, M,N \in \mathscr{C}$.

For non-zero modules $M,N \in \mathscr{C}$, we say that $R_{M,N_z}^{\mathrm{univ}}$ is {\em rationally renormalizable} if there is a formal Laurent series $c(z) \in \mathbf{k}(\!(z)\!)^\times$
such that \[c(z)R^{\mathrm{univ}}_{M,N_z}(M \otimes N_z) \subset N_z\otimes M.\]
In this case, we can choose $c_{M,N}(z) \in \mathbf{k}(\!(z)\!)$, uniquely up to multiplication by an element of $\bigsqcup_{n \in \mathbb{Z}}\mathbf{k}^\times z^n$, so that 
\[c_{M,N}(z) R^{\mathrm{univ}}_{M,N_z}(M \otimes N_z) \subset N_z\otimes M \quad \text{and} \quad c_{M,N}(z) R^{\mathrm{univ}}_{M,N_z}(M \otimes N_z)  \not \subset (z-x) N_z \otimes M\]
for any $x \in \mathbf{k}^\times$.
We call $c_{M,N}(z)$ the {\em renormalizing coefficient}. Set \[R_{M,N_z}^{\rm ren} \coloneqq   c_{M,N}(z) R^{\mathrm{univ}}_{M,N_z}\colon M\otimes N_z\rightarrow N_z\otimes M.\] Then for any $a\in \mathbf{k}^\times$, the specialization of $R_{M,N_z}^{\rm ren}$ at $z=a$ does not vanish. 
The non-zero homomorphism
\[{\bf r}_{M,N} \coloneqq   R_{M,N_z}^{\rm ren}|_{z=1} \colon M\otimes N\rightarrow N\otimes M\] 
is called the {\em renormalized $r$-matrix} of the pair $(M,N)$, which is well-defined up to multiplication by a scalar in $\mathbf{k}^\times$.

Let $\Omega_{\mathscr{C}}$ be the set consisting of non-zero objects of $\mathscr{C}$ isomorphic to subquotients of tensor products of simple objects.
It is known from \cite[Propositions 2.11 \& 2.12]{kkop-2020} that $R_{M,N_z}^{\mathrm{univ}}$ is rationally renormalizable if both $M$ and $N$ belong to $\Omega_{\mathscr{C}}$.

When $M$ and $N$ are simple objects of $\mathscr{C}$, there is a formal series $a_{M,N}(z) \in \mathbf{k}[\![z]\!]^\times$, called the {\em universal coefficient}, satisfying
\[ R^{\mathrm{univ}}_{M,N_z}(u\otimes v_z) = a_{M,N}(z) v_z \otimes u\]
where $u \in M$ and $v \in N$ are the dominant extremal weight vectors of $M$ and $N$ respectively, and $v_z \coloneqq v \otimes 1 \in N_z = N \otimes \mathbf{k}[z^{\pm 1}]$. 
The isomorphism
\[ R^{\mathrm{norm}}_{M,N_z} \coloneqq a_{M,N}(z)^{-1} R^{\mathrm{univ}}_{M,N_z} \colon (M \otimes N_z) \otimes_{\mathbf{k}[z^{\pm 1}]} \mathbf{k}(z) \to (N_z \otimes M) \otimes_{\mathbf{k}[z^{\pm 1}]} \mathbf{k}(z)  \]
is called the {\em normalized $R$-matrix}. 
The {\em denominator} of $R^{\mathrm{norm}}_{M,N_z}$ is defined to be the monic polynomial $d_{M,N}(z) \in \mathbf{k}[z]$ of the smallest degree such that we have
\[ d_{M,N}(z) R^{\mathrm{norm}}_{M,N_z} (M \otimes N_z) \subset N_z \otimes M.\]  
Thus, when $M, N\in \mathscr{C}$ are simple objects, $R_{M,N_z}^{\mathrm{univ}}$ is  rationally renormalizable and we have 
\begin{equation} \label{eq:c=ad}
c_{M,N}(z) = a_{M,N}(z)^{-1} d_{M,N}(z)
\end{equation}
up to multiplication by an element of $\bigsqcup_{n \in \mathbb{Z}}\mathbf{k}^\times z^n$.

Let $\mathcal{G} \subset \mathbf{k}(\!(z)\!)^\times$ be the multiplicative subgroup generated by $\bigsqcup_{n \in \mathbb{Z}}\mathbf{k}^\times z^n$ and the elements $\varphi(cz)$ for $c \in \mathbf{k}^\times$, where $\varphi(z) \coloneqq \prod_{s \in \mathbb{Z}_{\ge 0}}(1-q^{2sr^\vee h^\vee}z)$.
Note that we have $\mathbf{k}(z)^\times \subset \mathcal{G}$. 
For an element $f(z) = az^n \prod_{b \in \mathbf{k}^\times}\varphi(bz)^{m_b} \in \mathcal{G}$, we set
\[ \mathop{\mathrm{Deg}} f(z) \coloneqq \sum_{s \in \mathbb{Z}_{\le 0}} m_{q^{2sr^\vee h^\vee}} - \sum_{s \in \mathbb{Z}_{>0}}m_{q^{2sr^\vee h^\vee}}. \]
The map $\mathrm{Deg} \colon \mathcal{G} \to \mathbb{Z}$ is a group homomorphism. 
By \cite[Lemma 3.4(i)]{kkop-2020}, we have
\[ \mathop{\mathrm{Deg}} f(z) = 2 \mathop{\mathrm{ord}}_{z=1} f(z) \quad \text{if $f(z) \in \mathbf{k}(z)^\times$}, \]
where $\mathop{\mathrm{ord}}_{z=1} f(z) \in \mathbb{Z}$ denotes the order at $z=1$ of $f(z)$.

It is shown in \cite[Proposition 3.2]{kkop-2020} that the renormalizing coefficient $c_{M,N}(z)$ belongs to $\mathcal{G}$ for any $M,N \in \mathscr{C}$ such that $R^{\mathrm{univ}}_{M,N_z}$ is rationally renormalizable.  
Following \cite{kkop-2020}, we define the map $\Lambda \colon \Omega_{\mathscr{C}} \times \Omega_{\mathscr{C}} \to \mathbb{Z}$, called the \emph{$\Lambda$-invariant},  by
\[ \Lambda(M,N) \coloneqq \mathop{\mathrm{Deg}}(c_{M,N}(z)). \]

For simple objects $M,N$ of $\mathscr{C}$, we set
\begin{eqnarray}\label{eqn:def-N}
    \mathscr{N}(M,N) \coloneqq -\mathop{\mathrm{Deg}}(a_{M,N}(z)), \quad 
\mathfrak{o}(M,N)\coloneqq \mathop{\mathrm{ord}}_{z=1}(d_{M,N}(z)).
\end{eqnarray}
Note that $\mathfrak{o}(M,N)$ is the pole order of the normalized $R$-matrix $R^{\mathrm{norm}}_{M,N_z}$ at $z=1$ and it is a non-negative integer by definition. By \eqref{eq:c=ad}, we have
\begin{equation} \label{eq:L=N+2o}
\Lambda(M,N) = \mathscr{N}(M,N) + 2 \mathfrak{o}(M,N)
\end{equation}
for any simple objects $M,N$ of $\mathscr{C}$. 
By \cite[Proposition 3.16, Corollary 3.19]{kkop-2020}, it follows that
\begin{equation} \label{eq:delta-def}
\mathfrak{d}(M,N) \coloneqq \frac{\Lambda(M,N) + \Lambda(N,M)}{2} = \mathfrak{o}(M,N)+\mathfrak{o}(N,M) \in \mathbb{Z}_{\ge 0}
\end{equation}
and hence $\mathscr{N}(M,N) = - \mathscr{N}(N,M)$
for any simple objects $M,N$ of $\mathscr{C}$.

\begin{remark}\label{Rem:spec}
For an object $M$ of $\mathscr{C}$ and a non-zero scalar $c \in \mathbf{k}^\times$, we define 
\[ M_c \coloneqq M_z/(z-c)M_z.\]
This is again an object of $\mathscr{C}$ and is called the spectral parameter shift of $M$ by $c$.
The assignment $M \mapsto M_c$ gives rise to a monoidal autoequivalence of $\mathscr{C}$.
The quantities $\Lambda(M,N)$, $\mathscr{N}(M,N)$ and $\mathfrak{o}(M,N)$ are invariant under the spectral parameter shifts, \ie we have
\[ \Lambda(M_c,N_c) = \Lambda(M,N), \quad 
\mathscr{N}(M_c,N_c) = \mathscr{N}(M,N), \quad \mathfrak{o}(M_c,N_c) = \mathfrak{o}(M,N)\]
for any $c \in \mathbf{k}^\times$. 
\end{remark}

 For an object $M\in\mathscr{C}$, we denote by $\hd(M)$ the head (the largest semisimple quotient) of  $M$. For any two objects $M,N$ in $\mathscr{C}$, we denote by
\[M\nabla N \coloneqq  \hd(M\otimes N)
\]
 the head of the tensor product $M\otimes N$.
\begin{theorem}[{\cite{KKKO-2015}}]
\label{thm:kkko-15}
Let $M$ be a real simple object in $\mathscr{C}$. Then, for any simple object $N \in \mathscr{C}$, both $M \nabla N$ and $N \nabla M$ are simple.
\end{theorem}

\begin{proposition}[{\cite[Corollary 3.17, Lemma 3.10, Lemma 4.3, Proposition 3.9]{kkop-2020}}]\label{pro:kkop-20} The following statements hold.
\begin{itemize}
\item[(i)] Let $M$ be a real simple object. Then for any simple object $N$, the tensor product $M\otimes N$ is simple if and only if $\mathfrak{d}(M,N)=0$. In particular, $\mathfrak{d}(M,M)=0$.
    \item [(ii)]  Let $M,N, L\in\Omega_{\mathscr{C}}$ such that $L$ is a simple object. Then, we have
\[\Lambda(M\otimes N,L)=\Lambda(M,L)+\Lambda(N,L)\;\;\;\text{and}\;\;\; \Lambda(L,M\otimes N)=\Lambda(L,M)+\Lambda(L,N).\]
\item[(iii)] Let $M,N, L\in\Omega_{\mathscr{C}}$ such that $L$ is a real simple object and $M,N$ are simple objects. If $L$ strongly commutes with $N$, then
\[
\Lambda(M\nabla N,L)=\Lambda(M,L)+\Lambda(N,L).
\]
Dually,
if $L$ strongly commutes with $M$, then
\[ \Lambda(L,M\nabla N)=\Lambda(L,M)+\Lambda(L,N).
\]
\item[(iv)] Let $M,N\in\Omega_{\mathscr{C}}$. Then for any subquotient $0\neq M'$ of $M$ and subquotient $0\neq N'$ of $N$, we have 
$\Lambda(M',N')\leq \Lambda(M,N)$.
\end{itemize}
\end{proposition}

\subsection{$\mathsf{\Lambda}$-monoidal categorification and tropical invariant}

\begin{definition}[$\mathsf{\Lambda}$-monoidal categorification] \label{def:L-monoidal}
Let $(\mathcal A,\mathcal C,\varphi)$ be a monoidal categorification of a cluster algebra $\mathcal{A}$ of rank $n$ with $\mathcal C$ being a monoidal Serre subcategory of $\mathscr{C}$. Such a monoidal categorification is called a  {\em ${\mathsf{\Lambda}}$-monoidal categorification} if there exists a monoidal cluster $\mathcal M_w=(M_{1;w},\ldots, M_{m;w})$ with $w \in \mathbb{T}_n$ such that 
\begin{equation} \label{eq:Lmc}
(\widetilde B_w)^T\Lambda_w=(2I_n\mid{\bf 0}),
\end{equation}
where the $(i,j)$-entry of the matrix $\Lambda_w$ is given by 
$(\Lambda_w)_{ij}\coloneq \Lambda(M_{i;w},M_{j;w})$ for $i,j\in[1,m]$.
\end{definition}

Thanks to Proposition \ref{pro:kkop-20} (i), we know that the matrix $\Lambda_w$ above is skew-symmetric and thus $(\widetilde B_w, \Lambda_w)$ forms a compatible pair of type $2I_n$. For a ${\mathsf{\Lambda}}$-monoidal categorification $(\mathcal A,\mathcal C,\varphi)$, it is known by \cite[Lemma 6.6]{kkop-2020}, \cite[Proposition~5.13]{Cao-2023} that the equality \eqref{eq:Lmc}
holds for \emph{any} vertex $w\in\mathbb T_n$. Moreover, for any edge \begin{xy}(0,1)*+{t}="A",(10,1)*+{t'}="B",\ar@{-}^k"A";"B" \end{xy} in $\mathbb T_n$, we have  $(\widetilde B_{t'},   \Lambda_{t'})=\mu_k(\widetilde B_{t}, \Lambda_{t})$ as a mutation of compatible pairs. In particular, $\mathcal{A}$ is a $\mathsf{\Lambda}$-cluster algebra, and thus the tropical invariant and the $F$-invariant can be defined for any pair of good elements in $\mathcal A$.
\begin{theorem}[{\cite[Theorem 5.16]{Cao-2023}}]
\label{thm:trop-Lambda}
Let  $(\mathcal A,\mathcal C,\varphi)$ be a $\mathsf{\Lambda}$-monoidal categorification. Let $M,N$ be two simple objects in $\mathcal C$ such that $\varphi(M)$ and $\varphi(N)$ are good elements of $\mathcal A$ in the sense of Definition \ref{def:pointed}. If at least one of $M, N$ is reachable,
then we have
    \[\Lambda(M,N)=\langle \varphi(M), \varphi(N)\rangle_{\rm trop}\;\;\;\text{and}\;\;\;2\mathfrak{d}(M,N) = (\varphi(M)\mid\mid \varphi(N))_F.\] 
 \end{theorem}

 For general simple modules $M,N$ such that $\varphi(M)$ and $\varphi(N)$ are good elements, we show in Theorem \ref{thm:trop-Lambda2} that the $\Lambda$-invariant and twice the $\mathfrak{d}$-invariant are bounded from below by the tropical invariant and the $F$-invariant, respectively.

The following result is a consequence of the Laurent phenomenon and Theorem \ref{thm:kkko-15}. 

\begin{proposition}[{\cite[Lemma 5.14]{Cao-2023}}]
\label{pro:head}
Let  $(\mathcal A,\mathcal C,\varphi)$ be a $\mathsf{\Lambda}$-monoidal categorification and  $M$ a simple object in $\mathcal C$. Then for any monoidal cluster $\mathcal M_w$,  there exist ${\bf d,h,h'} \in \mathbb Z_{\ge 0}^m$ satisfying
\[
\mathcal M_w({\bf d})\nabla M\cong\mathcal M_w({\bf h}) \quad \text{and} \quad 
M \nabla \mathcal M_w({\bf d})\cong\mathcal M_w({\bf h'}). 
 \]
\end{proposition}

\begin{lemma}\label{Lem:ng}
Let $(\mathcal{A}, \mathcal{C}, \varphi)$ be a $\mathsf{\Lambda}$-monoidal categorification. Let $L$ be a simple object of $\mathcal{C}$ such that $\varphi(L) \in \mathcal{A}$ is a good element in the sense of Definition \ref{def:pointed}.
 For any vertex $w\in\mathbb T_n$, choose ${\bf d},{\bf h} \in \mathbb{Z}_{\ge 0}^m$ so that $\mathcal{M}_w({\bf d}) \nabla L \cong \mathcal{M}_{w}({\bf h})$ (see Proposition \ref{pro:head}).
 Then 
 \[{\bf g}_L^w = {\bf h} - {\bf d},\] where ${\bf g}_L^w\in\mathbb Z^m$ is the extended $g$-vector of $\varphi(L)$ with respect to vertex $w$. 
\end{lemma}
\begin{proof}
Take any ${\bf a} \in \mathbb{Z}_{\ge 0}^m$. As $\mathcal{M}_w({\bf a})$ strongly commutes with $\mathcal{M}_w({\bf d})$ and by Proposition \ref{pro:kkop-20} (iii), we have
\[ \Lambda(\mathcal{M}_w({\bf a}), \mathcal{M}_w({\bf d}) \nabla L)=\Lambda(\mathcal{M}_w({\bf a}), \mathcal{M}_w({\bf d}))+\Lambda(\mathcal{M}_w({\bf a}),  L)={\bf a}^T\Lambda_w{\bf d}+\Lambda(\mathcal{M}_w({\bf a}),  L).\]
On the other hand, we know that
\[ \Lambda(\mathcal{M}_w({\bf a}), \mathcal{M}_w({\bf d}) \nabla L)=\Lambda(\mathcal{M}_w({\bf a}), \mathcal{M}_w({\bf h}))={\bf a}^T\Lambda_w{\bf h}.\]
From the two equalities above, we obtain $\Lambda(\mathcal{M}_w({\bf a}),  L)={\bf a}^T\Lambda_w({\bf h}-{\bf d})$ and hence
\[{\bf a}^T\Lambda_w({\bf h}-{\bf d})=\Lambda(\mathcal{M}_w({\bf a}),  L) = \langle {\bf x}_w^{\bf a}, \varphi(L) \rangle_{\rm trop}={\bf a}^T\Lambda_w{\bf g}_L^w,
\] 
where the second equality is due to Theorem \ref{thm:trop-Lambda}, and the last one is due to $F_{{\bf x}_w^{\bf a}}^w=1$. Since ${\bf a}$ is an arbitrary element of $\mathbb Z_{\geq 0}^m$, we obtain $\Lambda_w({\bf h}-{\bf d}-{\bf g}_L^w)=0$ from the above equalities.

Let $K$ denote the kernel of the quotient map $\mathcal M_w({\bf d})\otimes L \twoheadrightarrow \mathcal M_w({\bf h})$.
Then we have ${\bf x}_w^{\bf d}\varphi(L) ={\bf x}_w^{\bf h}+\varphi(K)$. Then
by Proposition~\ref{prop:HL_2010}, we see that ${\bf x}_w^{\bf h}$ occurs in the Laurent expansion of ${\bf x}_w^{\bf d}\varphi(L)$ with respect to the cluster ${\bf x}_w$. In other words, ${\bf x}_w^{{\bf h}-{\bf d}}$ occurs in the Laurent expansion of $\varphi(L)$ with respect to the cluster ${\bf x}_w$. Since $\varphi(L)$ is pointed with degree ${\bf g}_L^w$, there exists ${\bf v}\in\mathbb Z_{\geq 0}^n$ such that ${\bf h}-{\bf d}-{\bf g}_L^w=\widetilde B_w{\bf v}$.
By construction, we know that $\Lambda_w\widetilde B_w =-\begin{pmatrix}2I_n\\{\bf 0}\end{pmatrix}$.
Thus we obtain ${\bf v}=0$  from the following equalities
\[ 0=\Lambda_w({\bf h}-{\bf d}-{\bf g}_L^w)
=\Lambda_w\widetilde B_w {\bf v}=-\begin{pmatrix}2{\bf v}\\{\bf 0}\end{pmatrix}.\]
Then the equality ${\bf h}-{\bf d}-{\bf g}_L^w=\widetilde B_w{\bf v}=0$ yields ${\bf g}_L^w={\bf h}-{\bf d}$.
\end{proof}

\begin{theorem} \label{thm:trop-Lambda2}
Let  $(\mathcal A,\mathcal C,\varphi)$ be a $\mathsf{\Lambda}$-monoidal categorification. Let $M,N$ be two simple objects in $\mathcal C$ such that $\varphi(M)$ and $\varphi(N)$ are good elements of $\mathcal A$ in the sense of Definition \ref{def:pointed}.
Then we have
    \[\Lambda(M,N) \ge \langle \varphi(M), \varphi(N)\rangle_{\rm trop}\;\;\;\text{and}\;\;\;2\mathfrak{d}(M,N) \ge (\varphi(M)\mid\mid \varphi(N))_F.\] 
If at least one of $M, N$ is reachable, the equalities hold.
\end{theorem}
\begin{proof}
By Proposition~\ref{prop:HL_2010}, there exists ${\bf c}\in\mathbb Z_{\geq0}^m$ such that
${\bf x}_w^{\bf c}\varphi(M)
 =\sum_{{\bf a}\in A}n_{\bf a}{\bf x}_w^{\bf a}$,
where $A\subset \mathbb Z_{\geq0}^m$ is a finite set and $n_{\bf a}\in\mathbb Z_{>0}$.
Recall that $\varphi$ is an isomorphism, and that each cluster monomial corresponds to the class of a simple object under $\varphi$. Thus each $\mathcal M_w({\bf a})$ with ${\bf a}\in A$ occurs as a composition factor of $\mathcal M_w({\bf c})\otimes M$. Proposition~\ref{pro:kkop-20}~(ii) and (iv) imply
\[
\Lambda(\mathcal M_w({\bf a}),N) 
 \leq\Lambda(\mathcal M_w({\bf c}) \otimes M,N) = \Lambda(\mathcal M_w({\bf c}),N) + \Lambda(M,N)
\]
and hence
\[
({\bf a}-{\bf c})^T\Lambda_w{\bf g}_N^w
 =\Lambda(\mathcal M_w({\bf a}),N) -\Lambda(\mathcal M_w({\bf c}),N)
 \leq\Lambda(M,N)
\]
for any ${\bf a}\in A$. Since  $\varphi(M)
 =\sum_{{\bf a}\in A}n_{\bf a}{\bf x}_w^{\bf a - \bf c}$, Proposition~\ref{pro:beta-map} and Definition~\ref{def:f-invariant} give
\[
\langle\varphi(M),\varphi(N)\rangle_{\rm trop}
 =\max_{{\bf a}\in A}\{({\bf a}-{\bf c})^T\Lambda_w{\bf g}_N^w\}
 \leq\Lambda(M,N).
\]
This proves the first inequality.
The second inequality is an immediate consequence of the first one and the definitions (Definition \ref{def:f-invariant} and \eqref{eq:delta-def}). 
The last assertion is nothing but Theorem \ref{thm:trop-Lambda}.
\end{proof}

\begin{remark} \label{rem:trop-Lambda2}
The above proof of Lemma \ref{Lem:ng} and Theorem \ref{thm:trop-Lambda2} only requires the properties of the $\Lambda$-invariant listed in Proposition \ref{pro:kkop-20}.
In particular, these results can be proved for a general $\mathsf{\Lambda}$-monoidal categorification axiomatized in \cite{Cao-2023}.   
\end{remark}

\begin{remark}
 In the setting of symmetric quiver Hecke algebras, Kashiwara--Kim--Oh--Park \cite[Theorem 2.13, Remark 2.14]{kkop-2026} recently proved that $\Lambda(M,N) = \langle \varphi(M), \varphi(N) \rangle_{\rm trop}$ for any simple object $M$ and any real simple object $N$ (not necessarily reachable).
 In our setting of Theorem~\ref{thm:trop-Lambda2} for quantum affine algebras, their argument gives the corresponding equality whenever $N$ is real.
 The proof requires an additional property of $\Lambda$ \cite[Theorem 4.11]{kkop-2020}, not only the ones listed in Proposition \ref{pro:kkop-20}. 
 
One may expect that the equality $\Lambda(M,N) = \langle \varphi(M), \varphi(N) \rangle_{\rm trop}$ remains true even when both $M$ and $N$ are imaginary modules (\confer Conjecture \ref{conj:o=F} below).
 At this moment, the authors do not know any counter-examples.
However, it seems to be difficult to extend the known arguments in the proofs of the above equality for reachable or real modules because they rely heavily on the properties of real modules which are not always available for imaginary modules.
 \end{remark}
\subsection{Pole orders and the partial $F$-invariants}
In this subsection, we give a sufficient condition under which the partial $F$-invariant coincides with the pole order of the normalized $R$-matrix (Theorem \ref{thm:F-o}) after preparing a few lemmas. 

\begin{lemma} \label{Lem:alinear}
Let $L,M,N \in \mathscr{C}$ be simple objects. 
If $\mathfrak{o}(M,N) = 0$ and at least one of $M$, $N$ is real, then we have
\[ a_{M\nabla N, L}(z) = a_{M,L}(z) a_{N,L}(z), \quad  a_{L,M\nabla N}(z) = a_{L,M}(z) a_{L,N}(z),\]
and hence
\[
 \mathscr{N}(M \nabla N, L) = \mathscr{N}(M,L)+\mathscr{N}(N,L), \quad \mathscr{N}(L, M \nabla N) = \mathscr{N}(L,M)+\mathscr{N}(L,N).
\]
\end{lemma}
\begin{proof}
 Let $v_L \in L, v_M \in M, v_N \in N$ denote the dominant extremal weight vectors. Since $\mathfrak{o}(M,N)=0$, we can specialize the normalized $R$-matrix $R^{\mathrm{norm}}_{M,N_z}$ at $z=1$ to get the non-zero homomorphism $\mathbf{r}_{M,N} \colon M \otimes N \to N \otimes M$, whose image is isomorphic to the simple head $M \nabla N$ by \cite[Theorem 3.12]{KKKO-2015}.   
As $\mathbf{r}_{M,N}(v_M \otimes v_N) = v_N \otimes v_M$, $M \nabla N$ is isomorphic to the submodule of $N \otimes M$ generated by $v_N \otimes v_M$. 
Moreover, $v_N \otimes v_M$ is the dominant extremal weight vector of $M \nabla N$.
By the property \eqref{eq:q-tri}, we obtain
\begin{align*}
a_{M \nabla N, L}(z) (v_{L_z} \otimes v_N \otimes v_M) &= R^{\mathrm{univ}}_{N \otimes M, L_z} (v_N \otimes v_M \otimes v_{L_z}) \\
&= (R^{\mathrm{univ}}_{N,L_z} \otimes \mathsf{id}_M) \circ (\mathsf{id}_N \otimes R^{\mathrm{univ}}_{M, L_z}) (v_N \otimes v_M \otimes v_{L_z}) \\
&= a_{N,L}(z) a_{M,L}(z)  (v_{L_z} \otimes v_N \otimes v_M)
\end{align*}
and hence the desired equality $a_{M\nabla N, L}(z) = a_{M,L}(z) a_{N,L}(z)$.
The other equality $a_{L,M\nabla N}(z) = a_{L,M}(z) a_{L,N}(z)$ is obtained in a similar way.
\end{proof}

\begin{lemma} \label{Lem:oconv}
Let $L,M,N$ be simple objects of $\mathscr{C}$.
Suppose that either $M$ or $N$ is real. Then, the following statements hold.
\begin{itemize}
    \item [(i)] We have $ \Lambda(M \nabla N, L) \leq  \Lambda(M,L) + \Lambda(N,L)$ and $ \Lambda(L,M\nabla N)\leq   
    \Lambda(L,M) + \Lambda(L,N)$. The equalities hold if $M$ and $N$ strongly commute.
    \item[(ii)] If $\mathfrak{o}(M,N)=0$, we have
\[ \mathfrak{o}(M \nabla N, L) \le \mathfrak{o}(M,L) + \mathfrak{o}(N,L)\quad\text{and}\quad
\mathfrak{o}(L,M\nabla N) \le \mathfrak{o}(L,M) + \mathfrak{o}(L,N).\]
The equalities hold if $M$ and $N$ strongly commute. 
\end{itemize}
\end{lemma}
\begin{proof}
Since $L$ is simple, Proposition \ref{pro:kkop-20} (ii) (iv) implies that 
\begin{align}
    \Lambda(M \nabla N, L) &\leq \Lambda(M \otimes N, L)=  \Lambda(M,L) + \Lambda(N,L),\nonumber\\
    \Lambda(L,M\nabla N)&\leq   \Lambda(L,M\otimes N) =
    \Lambda(L,M) + \Lambda(L,N).\nonumber
\end{align}
If $M$ and $N$ strongly commute, the equalities hold as $M\nabla N\cong M\otimes N$. This proves (i). 
   The statement  (ii) follows from (i), Lemma \ref{Lem:alinear} and the equation \eqref{eq:L=N+2o}.
\end{proof}

\begin{theorem} 
\label{thm:F-o} Let $(\mathcal{A}, \mathcal{C}, \varphi)$ be a $\mathsf{\Lambda}$-monoidal categorification.
Let $M, N\in\mathcal C$ be simple objects such that $\varphi(M)$ and $\varphi(N)$ are good elements in the sense of Definition \ref{def:pointed}. 
Suppose that there exists a monoidal cluster $\mathcal M_w$ satisfying 
\begin{equation}\label{eq:assump-o=0}
\mathfrak{o}(M_{i;w}, M) =\mathfrak{o}(M_{i;w}, N)=0 \quad \text{for all $i\in[1,m]$}. 
\end{equation}
Then, the following statements hold.
\begin{itemize}
    \item [(i)] We have $
\mathscr{N}(M,N) = ({\bf g}^{w}_M)^T \Lambda_{w} {\bf g}^{w}_N$.
\item[(ii)] We have an inequality 
    \begin{equation}\label{eq:F-o2}
        \mathfrak{o}(M,N)\geq F_M^w[({\bf g}_N^w)^\circ],
    \end{equation}
     where $F_M^w\coloneqq F_{\varphi(M)}^w$, ${\bf g}^{w}_N\coloneqq {\bf g}_{\varphi(N)}^w$ and  $({\bf g}^{w}_N)^\circ\in\mathbb Z^n$ denotes the principal part of ${\bf g}^{w}_N\in\mathbb Z^m$.
In particular, \eqref{eq:F-o2} is an equality if at least one of $M$ and $N$ is reachable.
\end{itemize} 
\end{theorem}
\begin{proof}
    (i) The assumption \eqref{eq:assump-o=0} and Lemma~\ref{Lem:oconv} yield that
\[
\mathfrak{o}(\mathcal M_w({\bf a}),M)
=0=\mathfrak{o}(\mathcal M_w({\bf a}),N)
\]
for each ${\bf a}\in\mathbb Z_{\geq0}^m$.
Since $\mathcal M_w({\bf a})$ is reachable and
$F_{{\bf x}_w^{\bf a}}^w=1$, Theorem~\ref{thm:trop-Lambda} and \eqref{eq:L=N+2o} imply
\begin{equation}\label{eq:F-o1}
    \mathscr{N}(\mathcal M_w({\bf a}),N)
=\Lambda(\mathcal M_w({\bf a}),N)
={\bf a}^T\Lambda_w{\bf g}_N^w.
\end{equation}
By Proposition~\ref{pro:head}, there exist
${\bf d},{\bf h}\in\mathbb Z_{\geq0}^m$ such that
$\mathcal M_w({\bf d})\nabla M\cong\mathcal M_w({\bf h})$.
Then Lemma~\ref{Lem:alinear} and Lemma~\ref{Lem:ng} yield
\[
    \mathscr{N}(M,N)
=\mathscr{N}(\mathcal M_w({\bf h}),N) -\mathscr{N}(\mathcal M_w({\bf d}),N)
=({\bf h}-{\bf d})^T\Lambda_w{\bf g}_N^w
=({\bf g}_M^w)^T\Lambda_w{\bf g}_N^w.
\]

(ii) 
Theorem~\ref{thm:F-o}~(i), Theorem~\ref{thm:mutation-inv}~(i), and \eqref{eq:L=N+2o} imply
\[
2( \mathfrak{o}(M,N)-F_M^w[({\bf g}_N^w)^\circ] )
 =\Lambda(M,N) -\langle\varphi(M),\varphi(N)\rangle_{\rm trop}.
\]
Here we used $S=2I_n$. 
The right-hand side is greater than or equal to zero by Theorem \ref{thm:trop-Lambda2}, which proves the desired inequality \eqref{eq:F-o2}.
If at least one of $M$ and $N$ is reachable, then the equality follows from Theorem~\ref{thm:trop-Lambda}.
\end{proof}
\begin{remark}
    Theorem~\ref{thm:F-o}~(i) can be also obtained from the viewpoint of universal $R$-matrices directly:
    Retain the setting and the notation in Theorem~\ref{thm:F-o}. For any simple object $L$ in $\mathcal{C}$, Proposition~\ref{pro:head}, Lemmas~\ref{Lem:alinear} and \ref{Lem:ng} yield the following:
    \[
    a_{M,L}(z)= \prod_{i =1}^m a_{M_{i; w}, L}(z)^{h_i - d_i}= \prod_{i =1}^m a_{M_{i; w}, L}(z)^{({\bf g}_M^w)_i}.
    \]
    Applying the same argument to $N$, we obtain
    \[
    a_{M,N}(z)= \prod_{i,j =1}^m a_{M_{i; w}, M_{j; w}}(z)^{({\bf g}_M^w)_i ({\bf g}_N^w)_j}.
    \]
    By taking $-\operatorname{Deg}(-)$ of the above equality, we obtain Theorem~\ref{thm:F-o}~(i).
\end{remark}

\begin{corollary} \label{Cor:o=Fg=E}
Let $(\mathcal{A}, \mathcal{C}, \varphi)$ be a $\mathsf{\Lambda}$-monoidal categorification. Suppose that there exists a vertex $t_0 \in \mathbb{T}_n$ such that the corresponding monoidal cluster $\mathcal{M}_{t_0}$ satisfies
\begin{equation}\label{assump:t_0}
\mathfrak{o}(M_{i;t_0}, L) =0 
\quad \text{for all $1 \le i \le m$ and any simple object $L$ of $\mathcal{C}$.}
\end{equation}
Then, for any simple objects $M, N \in \mathcal{C}$ such that $\varphi(M)$ and $\varphi(N)$ are good elements of $\mathcal{A}$, we have
\begin{equation}\label{eq:o=gF}
\mathfrak{o}(M,N) \geq F_{M}^{t_0}[({\bf g}^{t_0}_N)^\circ].
\end{equation}
In particular, if at least one of $M, N$ is a reachable module, \eqref{eq:o=gF} is an equality.

\end{corollary}
In the next subsection, we will see some examples satisfying the assumption in \eqref{assump:t_0}, \confer Proposition \ref{pro:standard-cor}. 
\subsection{Pole orders and $q$-characters}\label{subsec:staex}
In this subsection, we apply Corollary \ref{Cor:o=Fg=E} to the standard examples of monoidal categorification introduced by Hernandez--Leclerc \cite{HL_2010, HL16} to reveal a connection between the pole orders of normalized $R$-matrices and $q$-characters of simple modules.

Henceforth, we assume that $\mathfrak{g}$ is of untwisted type $\mathrm{X}_N^{(1)}$,  \ie $r=1$. (The twisted case can be treated similarly, \confer Remark \ref{rem:twisted-case}.)
Note that the restriction of the Dynkin diagram of $\mathfrak{g}$ to $\mathtt{I}_0$ is identical to the Dynkin diagram of the finite-dimensional simple Lie algebra $\mathfrak{g}_0$ of type $\mathrm{X}_N$. 
For $i \in \mathtt{I}_0$, let $\alpha_i$ denote the $i$-th simple root and set $d_i \coloneqq (\alpha_i, \alpha_i)/2 \in \{1,r^\vee\}$. 

Recall that the well-known classification result due to Chari--Pressley \cite[Chapter 12]{CP-1994} gives a bijection between the set of isomorphism classes of simple modules in the category $\mathscr{C}$ and the set $(1+z\mathbf{k}[z])^{\mathtt{I}_0}$ of $\mathtt{I}_0$-tuples of polynomials in a variable $z$ with constant terms $1$, called {\em Drinfeld polynomials}.
For each $\boldsymbol{\pi} = (\pi_{i}(z))_{i \in \mathtt{I}_0} \in(1+z\mathbf{k}[z])^{\mathtt{I}_0}$, we fix a simple module $L( \boldsymbol{\pi}) \in \mathscr{C}$ which gives a representative of the simple isomorphism class corresponding to $ \boldsymbol{\pi}$ under the above bijection\footnote{Following the convention of \cite{FR99}, our polynomial $\pi_i(z)$ corresponds to the polynomial $P_i(q^{-d_i}u)$ in the notation of \cite{CP-1994} (\ie we have $\pi_i(z)=P_i(q^{-d_i}z)$).}.

\begin{remark}
The set $(1+z\mathbf{k}[z])^{\mathtt{I}_0}$ forms a commutative monoid with respect to the multiplication.
Since it is freely generated by 
\[ \boldsymbol{\pi}_{i,c} \coloneqq \left((1-cz)^{\delta_{ij}}\right)_{j \in \mathtt{I}_0}\] for $i \in \mathtt{I}_0$ and $c \in \mathbf{k}^\times$, any $ \boldsymbol{\pi}\in (1+z\mathbf{k}[z])^{\mathtt{I}_0}$ has a unique factorization (up to permutation) into a finite product of $ \boldsymbol{\pi}_{i,c}$ with $i \in \mathtt{I}_0$ and $c \in \mathbf{k}^\times$. 
\end{remark}

  A simple module isomorphic to 
\begin{equation} \label{eq:KR} 
W^{(i)}_{k,c} \coloneqq L( \boldsymbol{\pi}_{i,c} \boldsymbol{\pi}_{i,cq^{2d_i}} \cdots  \boldsymbol{\pi}_{i,cq^{2(k-1)d_i}})
\end{equation}
for some $i \in \mathtt{I}_0$, $c \in \mathbf{k}^\times$ and $k \in \mathbb{Z}_{> 0}$ is called a {\em Kirillov--Reshetikhin  (KR) module}.
If $k=1$,
it is called a  {\em fundamental module}.
For later use, we recall the following basic fact.

\begin{lemma} \label{Lem:pi}
Let $ \boldsymbol{\pi},  \boldsymbol{\pi}' \in (1+z\mathbf{k}[z])^{\mathtt{I}_0}$.
Assume that $c'/c \not\in q^{\mathbb Z_{> 0}}$ holds whenever $ \boldsymbol{\pi}_{i,c}$ is a factor of $ \boldsymbol{\pi}$ and $ \boldsymbol{\pi}_{i',c'}$ is a factor of $ \boldsymbol{\pi}'$.
Then, we have  
\[L( \boldsymbol{\pi}  \boldsymbol{\pi}') \cong L( \boldsymbol{\pi}) \nabla L( \boldsymbol{\pi}') \quad \text{and} \quad \mathfrak{o}(L( \boldsymbol{\pi}), L( \boldsymbol{\pi}'))=0. \]
\end{lemma}
\begin{proof}
It is well-known from \cite{chari, kas} that a simple module $L( \boldsymbol{\pi})$ is isomorphic to the head of a suitably ordered tensor product of fundamental modules. 
More precisely, choosing a factorization $ \boldsymbol{\pi} =  \boldsymbol{\pi}_{i_1,c_1} \boldsymbol{\pi}_{i_2,c_2} \cdots  \boldsymbol{\pi}_{i_d,c_d}$ so that $c_k / c_l \not\in q^{\mathbb Z_{> 0}}$ holds for any $1 \le l < k \le d$, we have
\[
L( \boldsymbol{\pi}) \cong \mathsf{hd}\left(L( \boldsymbol{\pi}_{i_1,c_1}) \otimes L( \boldsymbol{\pi}_{i_2,c_2}) \otimes \cdots \otimes L( \boldsymbol{\pi}_{i_d,c_d})\right).
\]
The desired isomorphism $L( \boldsymbol{\pi}  \boldsymbol{\pi}') \cong L( \boldsymbol{\pi}) \nabla L( \boldsymbol{\pi}')$ follows from this fact.
It is also known from \cite{kas} that $\mathfrak{o}(L( \boldsymbol{\pi}_{i,c}), L( \boldsymbol{\pi}_{i',c'})) = 0$ holds if $c'/c \not\in q^{\mathbb Z_{> 0}}$.
Applying Lemma \ref{Lem:oconv}, we get the desired equality $\mathfrak{o}(L( \boldsymbol{\pi}), L( \boldsymbol{\pi}'))=0$ by induction on the total degree of $ \boldsymbol{\pi} \boldsymbol{\pi}'$. 
\end{proof}

In what follows, we restrict our attention to the simple modules whose roots of Drinfeld polynomials are some integer powers of $q$.
With such a simple module $L$, we associate a unique tuple of non-negative integers $(u_{i,r}(L))_{i \in \mathtt{I}_0, r \in \mathbb{Z}}$ satisfying
\begin{equation} \label{eq:notation_u}
L \cong L(\boldsymbol{\pi}) \quad \text{with}\quad \boldsymbol{\pi} = \prod_{i \in \mathtt{I}_0, r \in \mathbb{Z}} \boldsymbol{\pi}_{i,q^r}^{u_{i,r}(L)}.
\end{equation}

Let $\epsilon \colon \mathtt{I}_0 \to \{ 0,1\}$ be a function satisfying 
\[\epsilon(i) \equiv \epsilon(j) + \min(d_i,d_j) \pmod 2 \quad \text{if $(\alpha_i,\alpha_j)<0$}. \]
There are exactly two choices of such an $\epsilon$. We fix such a parity function $\epsilon$ once and for all.
For a finite integer interval $[a,b] \subset \mathbb{Z}$, we make the following definitions:
\begin{itemize}
    \item $\mathtt{K}\coloneqq \{(i,p)\in \mathtt I_0\times \mathbb Z\mid p+d_i\in \epsilon(i)+2\mathbb{Z}\}$.
    \item $\mathtt{K}_{[a,b]}\coloneqq \{(i,p)\in \mathtt{K}    \mid p+d_i\in [a,b]\}$.
    \item $\mathtt{K}_{[a,b]}^{\rm uf}\coloneqq \{(i,p)\in \mathtt{K}_{[a,b]} \mid a\leq p-d_i\}=\{(i,p)\in \mathtt{K} \mid a\leq p-d_i< p+d_i\leq b\}$.
    \item $\mathcal{C}_\mathbb{Z}$: the Serre subcategory of $\mathscr{C}$ 
 whose simple modules are labelled by the elements of the submonoid of $(1+z\mathbf{k}[z])^{\mathtt{I}_0}$ generated by $$\{  \boldsymbol{\pi}_{i,q^{p+d_i}} \mid (i,p)\in \mathtt K\}=\{  \boldsymbol{\pi}_{i,q^{s}} \mid i\in\mathtt I_0,\ s\in \epsilon(i)+2\mathbb{Z}\}.$$ 
 \item $\mathcal{C}_{[a,b]}$: the Serre subcategory of $\mathscr{C}$ 
 whose simple modules are labelled by the elements of the submonoid of $(1+z\mathbf{k}[z])^{\mathtt{I}_0}$ generated by $$\{  \boldsymbol{\pi}_{i,q^{p+d_i}} \mid (i,p)\in \mathtt K_{[a,b]}\}=\{  \boldsymbol{\pi}_{i,q^{s}} \mid i\in\mathtt I_0,\ s\in [a,b]\cap(\epsilon(i)+2\mathbb{Z})\}.$$ 
\end{itemize}
Note that both $\mathcal C_{\mathbb Z}$ and $\mathcal C_{[a,b]}$ are monoidal subcategories of $\mathscr{C}$, \confer \cites{HL_2010, kkop-2024}. 
\begin{remark}
    The categories $\mathcal{C}_\mathbb{Z}$ and $\mathcal C_{\ell}\coloneqq \mathcal C_{[0,2\ell +1]}$ were originally introduced by Hernandez--Leclerc \cite[\S 3]{HL_2010}. For a general integer interval $[a,b]$, the category $\mathcal{C}_{[a,b]}$ is a special case of the category $\mathscr{C}_{\mathfrak{g}}^{[a',b'],\mathfrak{s}}$ for  an admissible sequence $\mathfrak{s}$ and a finite interval $[a',b'] \subset \mathbb{Z}$ (different from $[a,b]$ in general) in the sense of \cite[\S6.3]{kkop-2024} (up to spectral parameter shift). 
\end{remark}

\begin{remark}\label{Rem:CZ}
Recall the spectral parameter shift from Remark \ref{Rem:spec}.
For any $\boldsymbol{\pi} = (\pi_i(z))_{i\in \mathtt I_0} \in (1+z\mathbf{k}[z])^{\mathtt{I}_0}$ and $c \in \mathbf{k}^\times$, we have $L(\boldsymbol{\pi})_c \cong L(\boldsymbol{\pi}_c)$, where $\boldsymbol{\pi}_c \coloneqq (\pi_i(cz))_{i \in \mathtt{I}_0}$. 
Therefore, the category $\mathcal{C}_\mathbb{Z}$ is stable under the spectral parameter shifts by elements of $q^{2\mathbb{Z}}$.
It is known that any prime simple module (\ie a simple module without non-trivial tensor factorizations) in $\mathscr{C}$ is isomorphic to a spectral parameter shift of a prime simple module in $\mathcal{C}_\mathbb{Z}$. 
Moreover, for any simple modules $M,N \in \mathcal{C}_\mathbb{Z}$, we have
$\mathfrak{o}(M,N_c) = 0$ if $c \not \in q^{2\mathbb{Z}}$.   
Thus, the study of the multiplicative nature of simple modules in $\mathscr{C}$ reduces to that in $\mathcal{C}_\mathbb{Z}$.
\end{remark}

\begin{remark}
For any simple modules $M,N$ in $\mathcal{C}_\mathbb{Z}$, there is an integer interval $[a,b]$ large enough so that both $M$ and $N$ belong to $\mathcal{C}_{[a,b]}$.
Combined with Remarks \ref{Rem:spec} and \ref{Rem:CZ}, it reduces the study of the multiplicative nature of simple modules in the category $\mathscr{C}$ to that of simple modules in the subcategories $\mathcal{C}_{[a,b]}$ for finite integer intervals $[a,b]$.   
\end{remark}

The category $\mathcal{C}_{[a,b]}$ gives a monoidal categorification of a cluster algebra $\mathcal{A}_{[a,b]}$. 
Let us describe its standard initial seed.
We first define an infinite quiver $\Gamma$ whose vertex set is 
\begin{equation}\label{eq:Gamma0}
\Gamma_0 \coloneqq \mathtt{K}=\{(i,p)\in \mathtt I_0\times \mathbb Z\mid p+d_i\in \epsilon(i)+2\mathbb{Z}\}
\end{equation}
and whose arrow set $\Gamma_1$ satisfies
\begin{equation}\label{eq:Gamma1} 
\# \{ \text{$(i,p) \to (j,s)$ in $\Gamma_1$}\} = \begin{cases} 
1, & \text{if $(\alpha_i,\alpha_j) \neq 0$ and $s = p + (\alpha_i,\alpha_j)$},\\
0, & \text{otherwise.}
\end{cases}
\end{equation}
Then, we consider the full subquiver $\Gamma_{[a,b]}$ of $\Gamma$ whose vertex set is 
\[ (\Gamma_{[a,b]})_0 \coloneqq \mathtt{K}_{[a,b]}=\{(i,p)\in \mathtt{K}    \mid p+d_i\in [a,b]\}.\]
We make $\Gamma_{[a,b]}$ into an ice quiver by declaring that a vertex $(i,p)\in (\Gamma_{[a,b]})_0=\mathtt{K}_{[a,b]}$ is frozen if $(i,p)\in \mathtt{K}_{[a,b]}\setminus \mathtt{K}_{[a,b]}^{\rm uf}$. 
We denote by $\Gamma_{[a,b]}^\circ$ the principal part of $\Gamma_{[a,b]}$, \ie the full subquiver of $\Gamma_{[a,b]}$ supported on the set $\mathtt{K}_{[a,b]}^{\rm uf}$ of unfrozen vertices.
\begin{example}\label{ex:Gammaab}
When $[a,b] = [-5,5]$, we give two examples of the quiver $\Gamma_{[-5,5]}$ for type $\mathrm{X}_N = \mathrm{A}_5$ and $\mathrm{B}_3$. The vertices $\text{\tiny $\square$}$ are frozen. 
\begin{itemize}
\item[(i)] When $\mathrm{X}_N = \mathrm{A}_5$, we identify the set $\mathtt{I}_0$ with $\{1,2,3,4,5 \}$ so that we have $(\alpha_i, \alpha_j) < 0$ if and only if $|i-j|=1$. 
If we choose the parity function $\epsilon \colon \mathtt{I}_0 \to \{ 0,1\}$ with $\epsilon(1) =0$,
then we have $\epsilon(1)=\epsilon(3)=\epsilon(5)=0$ and $\epsilon(2)=\epsilon(4)=1$. The ice quiver $\Gamma_{[-5,5]}$ is depicted as:
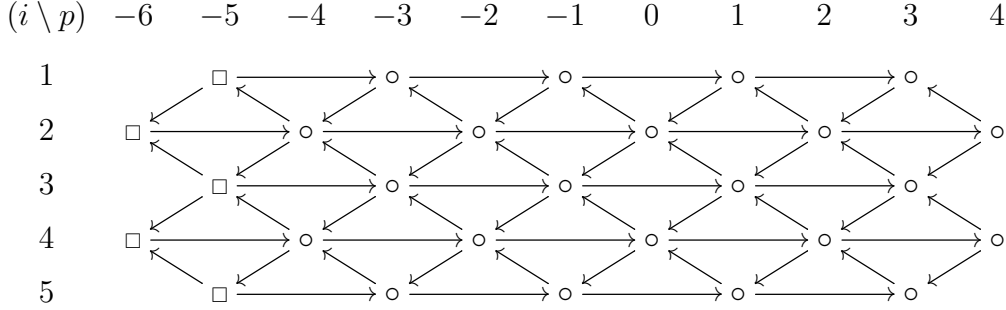
\[
\xymatrix@!C=3mm@R=2mm{
(i\setminus p) &-6&-5&-4& -3 & -2 & -1 &0&1 &2& 3 &4 \\
1&&\text{\tiny $\square$} \ar@{<-}[dr]&&\circ \ar@{<-}[ll]\ar@{<-}[dr]&& \circ \ar@{<-}[ll]\ar@{<-}[dr] &&\circ \ar@{<-}[ll]\ar@{<-}[dr]
&& \ar@{<-}[ll]\circ \ar@{<-}[dr]&\\ 
2&\text{\tiny $\square$}\ar@{<-}[dr]\ar@{<-}[ur]&&\circ \ar@{<-}[ll]\ar@{<-}[dr]\ar@{<-}[ur]&& \circ \ar@{<-}[ll]\ar@{<-}[dr] \ar@{<-}[ur]&&\circ \ar@{<-}[ll]\ar@{<-}[dr]\ar@{<-}[ur]
&&\circ \ar@{<-}[ll]\ar@{<-}[dr]\ar@{<-}[ur] && \circ \ar@{<-}[ll]\\
3&&\text{\tiny $\square$}\ar@{<-}[dr]\ar@{<-}[ur]&&\circ \ar@{<-}[ll]\ar@{<-}[dr]\ar@{<-}[ur]&& \circ \ar@{<-}[ll]\ar@{<-}[dr]\ar@{<-}[ur] &&\circ \ar@{<-}[ll]\ar@{<-}[dr]\ar@{<-}[ur]
&& \ar@{<-}[ll]\circ \ar@{<-}[dr]\ar@{<-}[ur]&\\ 
4&\text{\tiny $\square$}\ar@{<-}[ur]\ar@{<-}[dr]&&\circ \ar@{<-}[ll]\ar@{<-}[ur]\ar@{<-}[dr]&& \circ \ar@{<-}[ll]\ar@{<-}[ur] \ar@{<-}[dr] &&\circ\ar@{<-}[ll]\ar@{<-}[ur] \ar@{<-}[dr]&&\circ \ar@{<-}[ll]\ar@{<-}[ur] \ar@{<-}[dr]&& \circ \ar@{<-}[ll] \\
5&&\text{\tiny $\square$}\ar@{<-}[ur]&&\circ \ar@{<-}[ll]\ar@{<-}[ur]&& \circ \ar@{<-}[ll]\ar@{<-}[ur] &&\circ \ar@{<-}[ll]\ar@{<-}[ur]
&& \ar@{<-}[ll]\circ \ar@{<-}[ur]&}
\]
\item[(ii)] When $\mathrm{X}_N = \mathrm{B}_3$, we identify the set $\mathtt{I}_0$ with $\{1,2,3 \}$ so that we have $(d_1,d_2,d_3)=(2,2,1)$ and $(\alpha_i, \alpha_j) < 0$ if and only if $|i-j|=1$. 
If we choose the parity function $\epsilon \colon \mathtt{I}_0 \to \{ 0,1\}$ with $\epsilon(1) =1$, then we have $\epsilon(1)=\epsilon(2)=1$ and $\epsilon(3)=0$. The ice quiver $\Gamma_{[-5,5]}$ is depicted as:
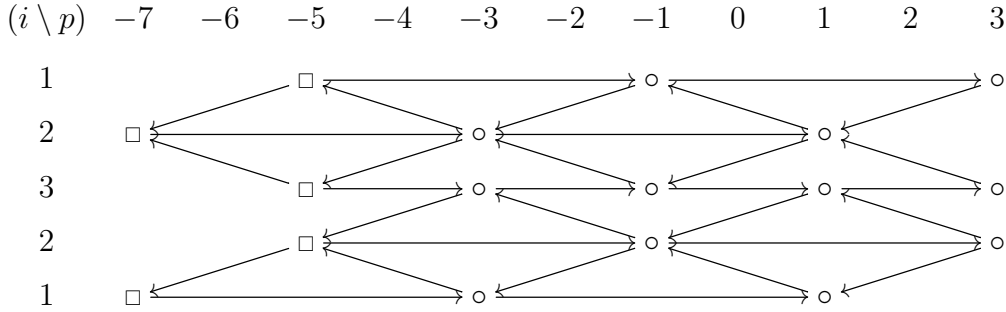
\[
\xymatrix@!C=3mm@R=2mm{
(i\setminus p) &-7&-6&-5& -4 & -3 & -2 &-1&0&1& 2 &3 \\
1&&&\text{\tiny $\square$}  \ar@{<-}[drr]&&&&\circ\ar@{<-}[llll]\ar@{<-}[drr]&&&& \circ \ar@{<-}[llll] \\
2&\text{\tiny $\square$} \ar@{<-}[drr]\ar@{<-}[urr]&&&& \circ \ar@{<-}[llll]\ar@{<-}[drr] \ar@{<-}[urr]&&&&\circ \ar@{<-}[llll]\ar@{<-}[drr] \ar@{<-}[urr]&&\\
3&&&\text{\tiny $\square$} \ar@{<-}[urr]&&\circ \ar@{<-}[ll]\ar@{<-}[drr]&& \circ \ar@{<-}[ll]\ar@{<-}[urr] &&\circ \ar@{<-}[ll]\ar@{<-}[drr]
&& \ar@{<-}[ll]\circ \\ 
2&&&\text{\tiny $\square$}  \ar@{<-}[urr]\ar@{<-}[drr]&&&&\circ\ar@{<-}[llll]\ar@{<-}[urr] \ar@{<-}[drr]&&&& \circ \ar@{<-}[llll] \\
1&\text{\tiny $\square$} \ar@{<-}[urr]&&&& \circ \ar@{<-}[llll] \ar@{<-}[urr]&&&&\circ \ar@{<-}[llll] \ar@{<-}[urr]&&
}
\]
\end{itemize}
\end{example}

\vspace{3mm}
To each vertex $(i,p) \in \mathtt K_{[a,b]}$ of the quiver $\Gamma_{[a,b]}$, we assign a KR module 
\begin{eqnarray}\label{eqn:M-ip}
    M_{(i,p)}^{[a,b]} \coloneqq W^{(i)}_{l, q^{p+d_i}} 
    \quad \text{with}\quad l &\coloneqq& 1+\max\{ k\in \mathbb N\mid (i,p+2kd_i)\in \mathtt K_{[a,b]}\}\\
    &=&1+\max\{k\in \mathbb N\mid p+(2k+1)d_i\leq b\}.\nonumber
\end{eqnarray}
Note that $M_{(i,p)}^{[a,b]}$ belongs to the subcategory $\mathcal{C}_{[a,b]}$ and
the definition of $M_{(i,p)}^{[a,b]}$ depends on the integer $b$, but does not depend on $a\in \mathbb Z$ such that $(i,p)\in \mathtt K_{[a,b]}$. In particular,
 we have $M_{(i,p)}^{[a,b]}=M_{(i,p)}^{[a',b]}$ for any integer $a'$ with $a'\leq a$. 

We define a mutation matrix  $\widetilde B_{[a,b]}=(b_{{(i,p),(j,s)}})$ indexed by the set  $\mathtt{K}_{[a,b]} \times \mathtt{K}_{[a,b]}^{\rm uf}$ by
 \[ b_{(i,p),(j,s)} \coloneqq \# \{\text{$(j,s) \to (i,p)$ in $(\Gamma_{[a,b]})_1$}\} - \# \{\text{$(i,p) \to (j,s)$ in $(\Gamma_{[a,b]})_1$}\}.\]
 Let $\mathcal A_{[a,b]}$ be the cluster algebra whose initial exchange matrix is $\widetilde B_{t_0}\coloneqq\widetilde B_{[a,b]}$. The cluster of $\mathcal A_{[a,b]}$ at a vertex $t$ is denoted by ${\bf x}_t\coloneqq (x_{(i,p);t})_{(i,p)\in \mathtt{K}_{[a,b]}}$. When we discuss several different integer intervals, we also denote the initial vertex by $t_0^{[a,b]}$ to clarify it. (If there is no confusion, we abbreviate it simply as $t_0$).

By an analog of \cite[Theorem 5.1]{HL16}, the assignment $[M_{(i,p)}^{[a,b]}]\mapsto x_{(i,p);t_0}$, $(i,p)\in \mathtt{K}_{[a,b]}=(\Gamma_{[a,b]})_0$, extends to a ring isomorphism \[\varphi_{[a,b]} \colon {K}_0(\mathcal{C}_{[a,b]}) \to \mathcal{A}_{[a,b]}.\] 
By \cite{kkop-2024}, $(\mathcal A_{[a,b]}, \mathcal C_{[a,b]}, \varphi_{[a,b]})$ is a $\mathsf{\Lambda}$-monoidal categorification.  In particular, the initial monoidal cluster $\mathcal M_{t_0}$ of $\mathcal C_{[a,b]}$ is given by the KR modules $\{ M^{[a,b]}_{(i,p)} \mid (i,p) \in \mathtt{K}_{[a,b]}\}$.

\begin{proposition}\label{pro:standard-cor}
The initial monoidal cluster $\mathcal{M}_{t_0}$ defined above satisfies the condition \eqref{assump:t_0}.
Namely, for any $(i,p) \in \mathtt{K}_{[a,b]}=(\Gamma_{[a,b]})_0$ and any simple module $L \in \mathcal{C}_{[a,b]}$, we have 
\begin{equation} \label{eq:oMN=0}
\mathfrak{o}(M_{(i,p)}^{[a,b]}, L) = 0.
\end{equation}
\end{proposition}
\begin{proof}
By Lemmas \ref{Lem:pi} and \ref{Lem:oconv}, it suffices to consider the case when $L$ is a fundamental module, say $L \cong L( \boldsymbol{\pi}_{j,q^{s+d_j}})$ with $(j,s) \in (\Gamma_{[a,b]})_0$. Recall the definition of $M^{[a,b]}_{(i,p)}$ from \eqref{eqn:M-ip}. 
If $s+d_j \le p+d_i$, we apply Lemma \ref{Lem:pi} to obtain \eqref{eq:oMN=0}.
If $s+d_j > p+d_i$, it is known from \cite[Proposition 4.19]{kkop-2024} that $M^{[a,b]}_{(i,p)}$ and $L$ strongly commute, which implies \eqref{eq:oMN=0}.
\end{proof}

\begin{proposition}\label{prop:good}
For any simple module $M \in \mathcal{C}_{[a,b]}$, the element $\varphi_{[a,b]}(M)\in\mathcal A_{[a,b]}$ is a good element in the sense of Definition \ref{def:pointed}.
\end{proposition}
\begin{proof}
The proof can be parallel to the discussion in \cite[\S\S3.2--3.5]{kk_2019} for quiver Hecke algebras, where the
key ingredients are the $\Lambda$-invariant and a notion of weights. 
The same discussion as in \emph{loc.~cit.} works for our category $\mathcal{C}_{[a,b]}$ as follows:

For quantum affine algebras, we use the weights $\mathsf{E}(-)$ of \cite{kkop-2022} in addition to the $\Lambda$-invariant. 
The weights of the initial monoidal cluster span a lattice of rank $m-n$ by \cite[Lemma~7.16]{kkop-2024}, and this rank is preserved under mutation.
Thanks to \eqref{eq:Lmc} and \cite[Definition~3.2, Lemmas~4.1--4.2]{kkop-2022},
we obtain the analogs of \cite[Lemma~3.2, Proposition~3.3]{kk_2019}. In other words, the dominance order $\preceq_t$ can be reformulated in terms of the $\Lambda$-invariant and weights.
The properties concerning head and socle required in the remaining arguments are provided by \cite[Theorem~4.11]{kkop-2020}.
In particular, the assertion follows from the analog of \cite[Lemma 3.6(ii)]{kk_2019}, which shows that $\varphi_{[a,b]}(M)$ is pointed and positive with respect to any seed, and the analog of \cite[Theorem 3.16]{kk_2019}, which ensures the required compatibility among the degrees.
Note that our $g$-vector ${\bf g}^t_M$ (resp.\ $\preceq_t$) corresponds to ${\bf g}_{\mathscr{S}}^{L}(M)$ (resp.\  the opposite of $\preceq_{\mathscr{S}}$) in \cite{kk_2019}.
\end{proof}

Now, we recall that the extended $g$-vector ${\bf g}_M^{t_0}$ and $F$-polynomial $F_M^{t_0}$ of the good element $\varphi_{[a,b]}(M)\in \mathcal A_{[a,b]}$ given by a simple object $M\in\mathcal{C}_{[a,b]}$ can be read off from the $q$-character of $M$ according to \cite{HL16}.

For an object $M\in\mathcal C_{\mathbb Z}$, its {\em $q$-character} $\chi_q(M)$ in the sense of Frenkel--Reshetikhin \cite{FR99} is a Laurent polynomial in variables $Y_{i,p+d_i}$ for $(i,p)\in \mathtt K=\{(i,p)\in \mathtt{I}_0\times \mathbb Z\mid p+d_i\in \epsilon(i)+2\mathbb{Z}\}$
with non-negative integer coefficients encoding the $\ell$-weight space decomposition of $M$ (\ie the spectral decomposition with respect to the action of the imaginary root vectors of $U_q'(\mathfrak{g})$).

By \cite[Lemma 3.8]{HL16}, if $M \in \mathcal{C}_{\mathbb Z}$ is simple, its $q$-character $\chi_q(M)$ is of the form (with the notation in \eqref{eq:notation_u})
\[ \chi_q(M) = \left( \prod_{(i,p) \in \mathtt K} Y_{i,p+d_i}^{u_{i,p+d_i}(M)} \right) F_M({\bf A}^{-1}),\]
where $F_M({\bf A}^{-1}) \in \mathbb{Z}_{\ge 0}[A_{i,p}^{-1} \mid (i,p) \in \mathtt{K}]$ with constant term one, and
\[ A_{i,p} \coloneqq Y_{i,p+d_i}Y_{i,p-d_i} \prod_{j \in \mathtt{I}_0, c_{ji}=-1} Y_{j,p}^{-1} \prod_{j \in \mathtt{I}_0, c_{ji}=-2} Y_{j,p-1}^{-1}Y_{j,p+1}^{-1} \prod_{j \in \mathtt{I}_0, c_{ji}=-3}Y_{j,p-2}^{-1}Y_{j,p}^{-1}Y_{j,p+2}^{-1}\]
with $c_{ji} \coloneqq 2(\alpha_j,\alpha_i)/(\alpha_j,\alpha_j)$ being the Cartan integer.
Note that $A_{i,p} \in \mathbb Z[Y_{j,s+d_j}^{\pm 1}\mid (j,s)\in \mathtt{K}]$ for any $(i,p)\in \mathtt{K}$, and the set $\{ A_{i,p} \mid (i,p) \in  \mathtt{K}\}$ is algebraically independent.

Fix a finite integer interval $[a,b]$ as before. For an object $M\in \mathcal C_{[a,b]}\subseteq \mathcal {C}_{\mathbb Z}$, we define its {\em truncated $q$-character} $\chi^{-, b}_q(M)$ by deleting all the terms containing $Y_{i,p+d_i}^{\pm 1}$ with $p+d_i > b$ as a non-trivial factor from $\chi_q(M)$. 
By \cite[Proposition 3.10]{HL16},
the assignment $M \mapsto \chi^{-, b}_q(M)$ gives rise to an injective ring homomorphism 
\[ \chi_q^{-, b}(-) \colon {K}_0(\mathcal{C}_{[a,b]}) \to \mathbb{Z}[Y_{i,p+d_i}^{\pm 1} \mid (i,p) \in \mathtt{K}_{[a,b]}].\]
If $M \in \mathcal{C}_{[a,b]}$ is simple, its truncated $q$-character $\chi^{-, b}
_q(M)$ is of the form 
\begin{equation}\label{eq:abqch}
    \chi^{-, b}_q(M) = \left( \prod_{(i,p) \in \mathtt K_{[a,b]}} Y_{i,p+d_i}^{u_{i,p+d_i}(M)} \right) F_M^{-, b}({\bf A}^{-1}),
\end{equation}
where $F_M^{-, b}({\bf A}^{-1}) \in \mathbb{Z}_{\ge 0}[A_{i,p}^{-1} \mid (i,p) \in \mathtt{K}_{[a,b]}^{\rm uf}]$ is obtained from $F_M({\bf A}^{-1})$ by deleting all the terms containing $A_{j,s}^{-1}$ with $(j,s) \not \in \mathtt{K}_{[a,b]}^{\rm uf}$. (Note that $A_{j,s}$ belongs to $\mathbb{Z}[Y_{i,p+d_i}^{\pm 1} \mid (i,p) \in \mathtt{K}_{[a,b]}]$ if and only if $(j,s) \in \mathtt{K}^{\rm uf}_{[a,b]}$.)
Moreover, we have $F^{{-, b}}_{M}({\bf A}^{-1})=1$ if $M$ belongs to  the initial monoidal cluster
    $\{M_{(i,p)}^{[a,b]}\mid (i,p)\in \mathtt{K}_{[a,b]} \}$  of $\mathcal C_{[a,b]}$.

\begin{lemma}[{\confer\cite{HL16}}]\label{lem:gFab}
  Keep the above notation.   Let $M$ be a simple object in $\mathcal C_{[a,b]}$. Let ${\bf g}_M^{t_0}$ and $F_M^{t_0}({\bf y})$ be the extended $g$-vector and
    $F$-polynomial of the good element $\varphi_{[a,b]}(M)\in \mathcal A_{[a,b]}$ with respect to the standard initial seed. Then
    \begin{itemize}
\item[(i)] ${\bf g}_M^{t_0} = (g_{(i,p)})_{(i,p)\in \mathtt{K}_{[a,b]}}$ is given by $g_{(i,p)} \coloneqq u_{i,p+d_i}(M) - u_{i,p-d_i}(M)$ in the notation of \eqref{eq:notation_u}.
\item[(ii)] $F_M^{t_0}({\bf y})$ is obtained from $F^{-, b}_M(\mathbf{A}^{-1})$ by substituting $A_{i,p}^{-1}$ with ${y}_{(i,p)}$ for $(i,p)\in \mathtt{K}_{[a,b]}^{\rm uf}$.
\end{itemize}
\end{lemma}
\begin{proof}
Let $\psi \colon \mathbb{Z}[x_{(i,p);t_0}^{\pm 1}\mid (i,p)\in \mathtt{K}_{[a,b]}] \to \mathbb{Z}[Y_{i,p+d_i}^{\pm 1} \mid (i,p) \in \mathtt{K}_{[a,b]}]$ be the monomial transformation whose inverse is given by
\[
\psi^{-1}(Y_{i,p+d_i}) =\frac{x_{(i,p);t_0}}{x_{(i,p+2d_i);t_0}} \qquad ((i,p)\in\mathtt K_{[a,b]}),
\]
where $x_{(i,p+2d_i);t_0}=1$ if $(i,p+2d_i)\notin\mathtt K_{[a,b]}$.
By \cite[Theorem~5.1, Lemma~4.15]{HL16}, we have an equality $\psi\circ\varphi_{[a,b]} =\chi_q^{-,b}(-)$ of maps ${K}_0(\mathcal{C}_{[a,b]}) \to \mathbb{Z}[Y_{i,p+d_i}^{\pm 1} \mid (i,p) \in \mathtt{K}_{[a,b]}]$, and have an equality $\psi(\hat y_{(i,p);t_0}) =A_{i,p}^{-1}$ for each $(i,p)\in\mathtt K_{[a,b]}^{\rm uf}$, where
$\hat{y}_{(i,p);t_0} \coloneqq 
\prod_{(j,s)\in \mathtt{K}_{[a,b]}} x_{(j,s); t_0}^{b_{(j,s),(i,p)}^{t_0}}$.
Applying $\psi^{-1}$ to \eqref{eq:abqch}, we directly obtain
\[
\varphi_{[a,b]}(M)
 ={\bf x}_{t_0}^{{\bf g}_M^{t_0}}
 F_M^{-,b}\left(
   (\hat y_{(i,p);t_0})_{(i,p)\in\mathtt K_{[a,b]}^{\rm uf}}
 \right),
 \]
 where $g_{(i,p)}=u_{i,p+d_i}(M)-u_{i,p-d_i}(M)$.
Since $F_M^{-,b}$ has constant term $1$, the expression \eqref{eqn:upointed} yields the assertions.
\end{proof}

\begin{theorem}\label{thm:oF}
Let $[a,b]$ be a finite integer interval and $M,N$ two simple objects of $\mathcal{C}_{[a,b]}$.
Then we have
\begin{equation}\label{eq:oFab}
    \mathfrak{o}(M,N) \geq F^{-, b}_{M}[u_{i,p+d_i}(N)-u_{i,p-d_i}(N) \mid (i,p) \in \mathtt{K}_{[a,b]}^{\rm uf}],
\end{equation}
where the right-hand side denotes the tropical polynomial $F_M^{-, b}[-]$ evaluated at $A_{i,p}^{-1} = u_{i,p+d_i}(N)- u_{i,p-d_i}(N)$ for each $(i,p) \in \mathtt{K}_{[a,b]}^{\rm uf}$.
 In particular, if at least one of $M$ and $N$ is reachable in $\mathcal{C}_{[a,b]}$, then \eqref{eq:oFab} is an equality. 
\end{theorem}
\begin{proof}
By Propositions \ref{pro:standard-cor} and \ref{prop:good}, we can apply Corollary \ref{Cor:o=Fg=E} to get the inequality
\[ \mathfrak{o}(M,N) \geq F_{M}^{t_0}[({\bf g}_N^{t_0})^\circ],\]
which is identical to the desired inequality by Lemma \ref{lem:gFab}.
\end{proof}

It would be useful to reformulate Theorem \ref{thm:oF} in such a way that it does not refer to the choice of integer interval $[a,b]$.
For this purpose, we introduce the following terminology.

\begin{definition} \label{def:reachableCZ}
Let $M$ be a simple module in $\mathcal{C}_{\mathbb Z}$.
We say that $M$ is \emph{reachable} in $\mathcal{C}_{\mathbb Z}$ if there exists a finite integer interval $[a,b]$
such that $M$ belongs to  $\mathcal{C}_{[a,b]}$ and $\varphi_{[a,b]}(M)$ is a cluster monomial of $\mathcal A_{[a,b]}$.
\end{definition}

For example, the KR module $W^{(i)}_{k,q^{p+d_i}}$ is reachable in $\mathcal{C}_{\mathbb Z}$ for any $(i,p) \in \mathtt{K}$ and $k \in \mathbb{Z}_{>0}$. 

\begin{remark} \label{rem:reachableCZ}
Actually, the category $\mathcal{C}_\mathbb{Z}$ is identical to the category $\mathcal{C}^{[-\infty,+\infty], \mathfrak{s}}_{\mathfrak{g}}$ for an admissible sequence $\mathfrak{s}$ in the sense of \cite{kkop-2024}, and it is a $\mathsf{\Lambda}$-monoidal categorification of a certain cluster algebra $\mathcal{A}_{\mathbb{Z}}$ of infinite rank by \cite[Theorem 8.1]{kkop-2024}.
In particular, there is a ring isomorphism $\varphi_\mathbb{Z} \colon K_0(\mathcal{C}_{\mathbb Z}) \to \mathcal{A}_{\mathbb Z}$ under which the set of cluster monomials in $\mathcal{A}_{\mathbb Z}$ gets identified with a subset of the set of simple classes in $\mathcal{C}_{\mathbb Z}$. By the proof of \cite[Theorem 8.1]{kkop-2024},  we see that a simple module $M$ in $\mathcal{C}_\mathbb{Z}$ is reachable in the sense of Definition \ref{def:reachableCZ} if and only if $\varphi_{\mathbb Z}(M)$ is a cluster monomial of $\mathcal{A}_{\mathbb Z}$.
\end{remark}

\begin{corollary}\label{cor:oF2}
    Let $M,N$ be simple objects of $\mathcal C_{\mathbb Z}$. Then we have
\begin{equation}\label{eq:oF2}
    \mathfrak{o}(M,N)
\geq F_M[u_{i,p+d_i}(N)-u_{i,p-d_i}(N) \mid (i,p)\in\mathtt K],
\end{equation}
where the right-hand side denotes the tropical polynomial $F_M[-]$ evaluated at $A_{i,p}^{-1}=u_{i,p+d_i}(N)-u_{i,p-d_i}(N)$ for each $(i,p)\in\mathtt K$.
In particular, \eqref{eq:oF2} is an equality if at least one of $M$ and $N$ is reachable in $\mathcal C_{\mathbb Z}$ in the sense of Definition~\ref{def:reachableCZ}.
\end{corollary}

\begin{example}
Let $\mathfrak{g}_0$ be of type $\mathrm{A}_2$.
We identify $\mathtt{I}_0$ with $\{1,2\}$ and let $(\epsilon(1), \epsilon(2)) \coloneqq (1,0)$.
When $M$ is a fundamental module in $\mathcal{C}_{\mathbb Z}$, say $M= L(\boldsymbol{\pi}_{1,q})$,
it is reachable and its $q$-character is
\[ \chi_q(L(\boldsymbol{\pi}_{1,q})) = Y_{1,1} + Y_{2,2}Y^{-1}_{1,3} + Y^{-1}_{2,4} 
= Y_{1,1}(1+A^{-1}_{1,2}(1+A^{-1}_{2,3})).\]
Therefore, we have $F_{L(\boldsymbol{\pi}_{1,q})}({\bf A}^{-1}) = 1+A^{-1}_{1,2}(1+A^{-1}_{2,3})$.
By Corollary \ref{cor:oF2}, we have
\[ \mathfrak{o}(L(\boldsymbol{\pi}_{1,q}),N) = \max( 0, u_{1,3}(N)-u_{1,1}(N)+\max(0,u_{2,4}(N)-u_{2,2}(N)) )\]
for any simple module $N \in \mathcal{C}_{\mathbb Z}$ (possibly neither reachable nor real).
For example, when $N = L(\boldsymbol{\pi}_{1,q^p}\boldsymbol{\pi}_{2,q^s})$ $(p\in1+2\mathbb Z,\ s\in2\mathbb Z)$, we get
\begin{align*} \mathfrak{o}(L(\boldsymbol{\pi}_{1,q}),L(\boldsymbol{\pi}_{1,q^p}\boldsymbol{\pi}_{2,q^s})) &= \max( 0, \delta_{p,3}-\delta_{p,1}+\max(0,\delta_{s,4}-\delta_{s,2}) ) \\
&= \begin{cases}
2 & \text{if $(p,s) = (3,4)$}, \\
1 & \text{if ($p=3$ \& $s \neq 4$) or ($s = 4$ \& $p \not \in \{1,3 \}$)}, \\
0 & \text{otherwise}.
\end{cases}
\end{align*}
\end{example}

    \begin{proof}[Proof of Corollary~\ref{cor:oF2}]
Since $F_M$ is a polynomial, we can find an interval $[a,b]\subset \mathbb Z$ such that $M,N\in\mathcal C_{[a,b]}$ and each variable $A_{i,p}^{-1}$ appearing in $F_M$ satisfies $(i,p)\in\mathtt K_{[a,b]}^{\rm uf}$. Then it is clear that $F_M^{-,b}=F_M$.
The desired inequality \eqref{eq:oF2} immediately follows from Theorem~\ref{thm:oF}.

If at least one of $M$ and $N$ is reachable in $\mathcal C_{\mathbb Z}$, then we can find $[a,b]$ such that it is also reachable in $\mathcal C_{[a,b]}$ thanks to Definition~\ref{def:reachableCZ} and Lemma~\ref{lem:clust_ab}.
The desired equality follows from Theorem~\ref{thm:oF}.
\end{proof}

\begin{lemma}[{\confer \cite{HL16}}]\label{lem:clust_ab}
Let $[a,b]$ be a finite integer interval and $M$ a simple object in $\mathcal C_{[a,b]}$. 
Every monoidal cluster of $\mathcal{C}_{[a, b]}$ is contained in a monoidal cluster of $\mathcal{C}_{[a', b']}$ for any integer interval $[a',b']$ containing $[a,b]$.
In particular, if $M$ is reachable in $\mathcal C_{[a,b]}$ (\ie $\varphi_{[a,b]}(M)$ is a cluster monomial of $\mathcal{A}_{[a,b]}$), then $M$ is also reachable in  $\mathcal C_{[a',b']}$.
\end{lemma}
\begin{proof}
The assertion follows from the discussion in \cite[\S3]{HL16}. We shall give some details for the sake of clarity.  
By induction on the length of the interval $[a',b']$, it suffices to consider the case when $(a',b') = (a-1,b)$ or $(a,b+1)$. 

First, we consider the case when $(a',b') = (a-1,b)$. 
We have $M^{[a,b]}_{(i,p)} = M^{[a-1,b]}_{(i,p)}$ for any $(i,p) \in \mathtt{K}_{[a,b]} \subset \mathtt{K}_{[a-1,b]}$ by definition.
Moreover, the quiver $\Gamma_{[a,b]}$ is the full subquiver of $\Gamma_{[a-1,b]}$ whose vertex set is $\mathtt{K}_{[a,b]}$, and $\mathtt{K}^{\rm uf}_{[a,b]} \subset \mathtt{K}^{\rm uf}_{[a-1,b]}$ holds.
We see that no vertex in
$\mathtt{K}_{[a-1,b]}\setminus\mathtt{K}_{[a,b]}$
is adjacent to a vertex in $\mathtt{K}_{[a,b]}^{\rm uf}$.
Thus any mutation sequence for $\mathcal A_{[a,b]}$
can be performed in $\mathcal A_{[a-1,b]}$
at the corresponding vertices with the same exchange relations.
The assertion follows from these facts.

Next, we consider the case when $(a',b') = (a,b+1)$. Let $\mathtt{I}_0(b+1) \coloneqq \{i \in \mathtt{I}_0 \mid (i, b+1-d_i) \in \mathtt{K}_{[a,b+1]}^{\rm uf}\}$. For each $i \in \mathtt{I}_0(b+1)$, we define a mutation sequence 
\[ \boldsymbol{\mu}_{i} \coloneqq \mu_{(i,a_i)} \circ \mu_{(i,a_i + 2d_i)} \circ \cdots \circ \mu_{(i,b+1-3d_i)} \circ \mu_{(i,b+1-d_i)}\]
for the cluster algebra $\mathcal{A}_{[a,b+1]}$, where $a_i \coloneqq  \min\{p\in b+1-d_i + 2d_i\mathbb Z\mid (i,p)\in \mathtt K_{[a,b+1]}^{\rm uf}\}$. 
Choose a total ordering $\mathtt{I}_0(b+1) = \{ i_1, i_2, \ldots, i_l \}$ and set $\boldsymbol{\mu} \coloneqq  \boldsymbol{\mu}_{i_1}  \circ  \boldsymbol{\mu}_{i_2} \circ \cdots \circ \boldsymbol{\mu}_{i_l}$. Applying the mutation sequence $\boldsymbol{\mu}$ to the initial seed $t_0 = ({\bf x}_{t_0}, \widetilde B_{[a,b+1]})= ({\bf x}_{t_0}, \Gamma_{[a,b+1]})$ of the cluster algebra $\mathcal{A}_{[a,b+1]}$, we get a new seed $t = \boldsymbol{\mu}(t_0) = ({\bf x}_t, \boldsymbol{\mu}(\Gamma_{[a,b+1]}))$.
The seed $t$ does not depend on the choice of total ordering of $\mathtt{I}_0(b+1)$.
By \cite[\S3.2.3]{HL16}, we have $\varphi_{[a,b+1]}(M^{[a,b]}_{(i,p)}) = x_{\iota(i,p); t}$ for any $(i,p) \in \mathtt{K}_{[a,b]}$,
where $\iota \colon \mathtt{K}_{[a,b]} \to \mathtt{K}_{[a,b+1]}$ is the injective map defined by 
\[ \iota(i,p) \coloneqq \begin{cases}
(i,p+2d_i) & \text{if $i \in \mathtt{I}_0(b+1)$ and $p \in b+1-d_i + 2d_i\mathbb Z$}, \\
(i,p) & \text{otherwise}.
\end{cases} \]
Moreover, the map $\iota$ identifies the initial seed for $\mathcal{C}_{[a,b]}$ with the restriction of the seed $t$ to $\iota(\mathtt{K}_{[a,b]})$ with $\{(i,a_i) \mid i \in \mathtt{I}_0(b+1)\}$ frozen. 
In the quiver $\boldsymbol{\mu}(\Gamma_{[a,b+1]})$,
no vertex outside $\iota(\mathtt{K}_{[a,b]})$ is adjacent to
a vertex in $\iota(\mathtt{K}_{[a,b]}^{\rm uf})$.
Thus any mutation sequence for $\mathcal A_{[a,b]}$
can be performed from the seed $t$ at the corresponding vertices with the same exchange relations.
The assertion follows from these facts.
\end{proof}

We expect the following conjecture. Motivated from additive categorifications, one might think of this as an analog of equalities in \cite[Theorem~5.4]{Sch92} in the study of generic dimensions of first extension groups and the Euler form from quiver representations: 
\begin{conjecture} \label{conj:o=F}
    Keep the notation in Corollary~\ref{cor:oF2}.
       Let $M,N$ be simple objects of $\mathcal C_{\mathbb Z}$. Then we have an equality
\[
    \mathfrak{o}(M,N) = F_M[u_{i,p+d_i}(N)-u_{i,p-d_i}(N) \mid (i,p)\in\mathtt K].
\]
\end{conjecture}

\section{Infinite quivers with potentials and their applications to reachable modules} \label{sec: application}
\subsection{Pole orders and the partial $E$-invariants}\label{sec:poE}

Keep the quivers $\Gamma, \Gamma_{[a,b]}$ and $\Gamma_{[a,b]}^\circ$ used in the previous section. We fix a choice of sign: $$\omega\colon\{(i,j)\in \mathtt I_0\times \mathtt I_0\mid c_{ij}<0\}\to \{\pm 1\},\quad \text{such that}\;\;\omega_{ij}=-\omega_{ji}.$$  Let $\alpha_{ij}(q)$ denote the arrow from $(j,q)$ to $(i,q+d_ic_{ij})$ and let $\varepsilon_i(p)$ denote the arrow from $(i,p)$ to $(i,p+2d_i)$ in the infinite quiver $\Gamma$. In the quiver $\Gamma$, there are oriented cycles
\[w_{ij}(p)\coloneqq\varepsilon_i(p-2d_i(c_{ij}+1))\cdots \varepsilon_i(p)\alpha_{ij}(p-d_i c_{ij})\alpha_{ji}(p-2d_i c_{ij})\] 
of the form
\begin{equation} \label{eq:cycles}
\vcenter{\xymatrix@C=2em{
(i,p) \ar[r] & (i,p+2d_i) \ar[r] & \cdots \ar[r] & (i,p-2d_ic_{ij}) \ar[dl]\\
& & (j,p-d_ic_{ij}) \ar[llu]& 
}}
\end{equation}
when we have $c_{ij}\leq -1$. 
Using such cycles, we introduce the following potentials:
\begin{itemize}
    \item $W_\Gamma$: the formal sum of all the signed cycles $\omega_{ij}w_{ij}(p)$ in the infinite quiver $\Gamma$;
    \item $W_{[a,b]}$: the sum of all the signed cycles $\omega_{ij}w_{ij}(p)$ contained in the full subquiver $\Gamma_{[a,b]}^\circ$ of $\Gamma$.
\end{itemize}
Now we define a non-negative grading on the infinite quiver $\Gamma$ via the degree map $\deg_{\Gamma}\colon \Gamma_1\to \mathbb Z$ given  by
    \begin{align}\label{eq:deggam}
\deg_{\Gamma}(\varepsilon_i(p))\coloneqq 0, \quad \deg_{\Gamma}(\alpha_{ij}(q))\coloneqq 1.
\end{align}
We find each cycle in $\Gamma$ has a strictly positive degree by inspection. 
 All the cycles of the form \eqref{eq:cycles} in $\Gamma$ defining the potentials $W_\Gamma$ and $W_{[a,b]}$ have degree $2$, and so each cyclic derivation of them is a homogeneous element with respect to these gradings.

We consider the following algebra: 
\begin{eqnarray}\label{eqn:A_ab}
A_{[a,b]}\coloneqq \mathbb{C}\Gamma_{[a,b]}^{\circ}/I'_{W_{[a,b]}}\cong  \mathbb C\Gamma/(\langle e_{k,r} \mid (k,r)\in \left(\Gamma_0\setminus (\Gamma_{[a,b]}^{\circ})_0\right)\rangle+ I'_{W_{\Gamma}}),
\end{eqnarray}
where $I'_{W_{[a,b]}}$ (resp.\ $I'_{W_\Gamma}$) is the two-sided ideal of $\mathbb{C}\Gamma_{[a,b]}^{\circ}$  (resp.\ $\mathbb C\Gamma$) generated by cyclic derivations of the potential $W_{[a,b]}$ (resp.\ $W_{\Gamma}$), \confer \eqref{eqn:cyc-der}. Notice that each cyclic derivation of the potential $W_\Gamma$ is a linear combination of finitely many paths in the infinite quiver $\Gamma$, since each arrow in $\Gamma$ is contained in finitely many cycles appearing in $W_\Gamma$.
\begin{proposition}[{\confer \cite[Proposition 4.17]{HL16}}]\label{pro:HL16-4.17}
    Keep the above notation. The following statements hold.
\begin{itemize}
    \item [(i)] The  algebra $A_{[a,b]}$ is finite-dimensional.
    \item[(ii)] The quiver with potential $(\Gamma_{[a,b]}^\circ, W_{[a,b]})$ is rigid in the sense of \cite[(8.1)]{DWZ08}. In particular, it is non-degenerate.
    \item[(iii)]  For each integer interval $[a,b]$, the ideal $I'_{W_{[a,b]}}$ of $\mathbb{C}Q\coloneq \mathbb C\Gamma^\circ_{[a,b]}$ satisfies the following condition for some $l \geq 2$:
    \[\mathbb{C}Q_{\geq l} \subseteq I'_{W_{[a,b]}} \subseteq \mathbb{C}Q_{\geq 2}.\]
    In particular, we have an algebra isomorphism $A_{[a,b]} \cong J(\Gamma_{[a,b]}^\circ, W_{[a,b]})$.
\end{itemize}
\end{proposition}

\begin{proof} 
Assertions (i) and (ii) follow easily from \cite[Proposition~4.17]{HL16} and \cite[Corollary~8.2, Proposition~8.9]{DWZ08}. The proof there essentially implies that $\mathbb{C}Q_{\geq l} \subseteq I'_{W_{[a,b]}} \subseteq \mathbb{C}Q_{\geq 2}$ for some $l\geq 2$. Assertion (iii) then follows from the discussion in Example~\ref{ex:QP2}~(ii).
\end{proof}

\begin{remark}\label{rem:HL}
(i) The potential $W'_{[a,b]}$ used
in  \cite{HL16} is the sum of cycles $w_{ij}(p)$ without signs, but the same proof in \cite[Proposition 4.17]{HL16} works for our signed potential $W_{[a,b]}$ and the algebra $A_{[a,b]}$ as well.

(ii) The same calculation as in \cite[Remark 4.3]{FM} shows that two quivers with potentials $(\Gamma_{[a,b]}^\circ, W'_{[a,b]})$ and $(\Gamma_{[a,b]}^\circ, W_{[a,b]})$ are right equivalent in the sense of \cite[Definition 4.2]{DWZ08}. In particular, there is an 
isomorphism $J(\Gamma_{[a,b]}^\circ, W_{[a,b]}) \cong J(\Gamma_{[a,b]}^\circ, W'_{[a,b]})$.
\end{remark}

Now we view the algebras $\mathbb C\Gamma\footnote{Notice that this is a non-unital algebra. }$ and $A_{[a,b]}$ as $\mathbb Z$-graded algebras via the degree map  $\deg_{\Gamma}\colon \Gamma_1\to \mathbb Z$. Let $\mathbb{C}\mathchar`-\mathsf{cat}$  be the category of $\mathbb{Z}$-graded $\mathbb{C}$-linear categories (a.k.a.\ enriched categories on the category of $\mathbb{Z}$-graded $\mathbb{C}$-vector spaces). Following Remark~\ref{rmk:R-C-cat} below, we know that the algebras $\mathbb C\Gamma$ and $A_{[a,b]}$ can be viewed as objects in $\mathbb{C}\mathchar`-\mathsf{cat}$. In particular, $A_{[a,b]}$ can be viewed as an ideal quotient of $\mathbb C\Gamma$ in $\mathbb{C}\mathchar`-\mathsf{cat}$.

\begin{remark}\label{rmk:R-C-cat}
    If  $R$ is a non-unital $\mathbb Z$-graded algebra given by a graded quiver $Q=(Q_0,Q_1)$ with homogeneous relations, then $R$ can be identified with a small graded $\mathbb C$-linear category whose object set is $Q_0$ and the morphism space from $j\in Q_0$ to $i\in Q_0$ is given by the $\mathbb{Z}$-graded $\mathbb{C}$-vector space $e_{i}R e_{j}$, where $e_{i}$ is the idempotent associated with the vertex $i\in Q_0$. The composition of morphisms is naturally given by the multiplication of $R$.
\end{remark}

When $A_{[a,b]}$ is regarded as a $\mathbb C$-linear category
indexed by $\Gamma_0$, every vertex
$v\notin(\Gamma_{[a,b]}^\circ)_0$ is regarded as a zero object.
Now we introduce a ``limit'' of the family of algebras $A_{[a,b]}$ indexed by the finite integer intervals $[a,b]$ in the following sense:
\begin{definition}\label{def:Lambdainf}
    Let $\Psi$ be the directed partially ordered set of all intervals $[a,b]\subset \mathbb{Z}$ with natural inclusion. We view each algebra $A_{[a,b]}=\mathbb{C}\Gamma_{[a,b]}^{\circ}/I'_{W_{[a,b]}}$ as an object in $\mathbb{C}\mathchar`-\mathsf{cat}$.
    For each inclusion $[a,b]\subset [a',b']$, the natural projection $A_{[a',b']}\to A_{[a,b]}$ can be identified with a morphism in  $\mathbb{C}\mathchar`-\mathsf{cat}$. We define $\Lambda(\infty)$ to be the projective limit \[\Lambda(\infty)\coloneqq\varprojlim_{[a,b]\in \Psi}A_{[a,b]}\] of the projective system $\{A_{[a,b]}\mid [a,b]\in \Psi\}$ in the category $\mathbb{C}\mathchar`-\mathsf{cat}$ of $\mathbb{Z}$-graded $\mathbb{C}$-linear categories. 
\end{definition}

We have the following explicit description of $\Lambda(\infty)$:
\begin{proposition}\label{prop:Lambdainfty}
    Let $I'_{W_{\Gamma}}$ be the two-sided ideal of the path algebra $\mathbb{C}\Gamma$ generated by all the cyclic derivations of $W_{\Gamma}$ and let $\mathfrak{M}_{\Gamma}$ be the two-sided ideal generated by all the arrows of $\Gamma$. Then we have the algebra isomorphism
    \[
    \Lambda(\infty) \cong \mathbb{C}\Gamma/ \bigcap_{l\geq 1}(\mathfrak{M}_{\Gamma}^l+I'_{W_{\Gamma}}).
    \]
\end{proposition}
\begin{proof}
    We work with the non-negative grading induced by \eqref{eq:deggam}. For fixed source and target vertices and fixed degree, the lengths of paths are bounded. That is, the number of arrows $\alpha$ is fixed by the degree, and then the number of degree $0$ arrows $\varepsilon$ is bounded by the second coordinates. In particular, every homogeneous morphism space of $\mathbb C\Gamma$ is finite-dimensional. We have the following isomorphisms of projective limits in the category $\mathbb{C}\mathchar`-\mathsf{cat}$:
    \begin{align*}
     \varprojlim_{[a,b]\in \Psi}\mathbb C\Gamma/ (\langle e_{k,r} \mid (k,r)\in \left(\Gamma_0\setminus (\Gamma_{[a,b]}^{\circ})_0\right)\rangle)\cong \mathbb{C}\Gamma \cong \varprojlim_{l\geq 1} \mathbb{C}\Gamma/\mathfrak{M}_{\Gamma}^l.
    \end{align*}
We regard $A_{[a,b]}$ as an object in $\mathbb{C}\mathchar`-\mathsf{cat}$ by \eqref{eqn:A_ab}, and thus we have
\[ \Lambda(\infty)= \varprojlim_{[a,b]\in \Psi}A_{[a,b]}\cong \mathbb C\Gamma/\bigcap_{[a,b] \in \Psi} (\langle e_{k,r} \mid (k,r)\in \left(\Gamma_0\setminus (\Gamma_{[a,b]}^{\circ})_0\right)\rangle+ I'_{W_{\Gamma}}).
\]
    It is enough to show the following equality between ideals of $\mathbb{Z}$-graded $\mathbb{C}$-linear category $\mathbb{C}\Gamma$ under the identification in Remark~\ref{rmk:R-C-cat}: \[\mathtt{LHS}\coloneqq \bigcap_{[a,b] \in \Psi} (\langle e_{k,r} \mid (k,r)\in \left(\Gamma_0\setminus (\Gamma_{[a,b]}^{\circ})_0\right)\rangle+ I'_{W_{\Gamma}})=\bigcap_{l\geq1}(\mathfrak{M}_{\Gamma}^l+ I'_{W_{\Gamma}}) \eqqcolon \mathtt{RHS}.\]
This is equivalent to showing an equality $e_{i,p}(\mathtt{LHS})e_{j,q}=e_{i,p}(\mathtt{RHS})e_{j,q}$ for any $(i,p), (j,q)\in\Gamma_0$.

Now we fix $(i,p), (j,q)\in\Gamma_0$. We take any interval $[a,b]$ such that $(i,p), (j,q) \in (\Gamma_{[a,b]}^{\circ})_0$. Every subspace of the form $$e_{i,p}(\langle e_{k,r} \mid (k,r)\in \left(\Gamma_0\setminus (\Gamma_{[a,b]}^{\circ})_0\right)\rangle+ I'_{W_{\Gamma}})e_{j,q}$$ contains $e_{i,p}(\mathfrak{M}_{\Gamma}^l+ I'_{W_{\Gamma}})e_{j,q}$ for all sufficiently large $l$ since the ideal of $A_{[a,b]}$ generated by all the arrows in $\Gamma_{[a,b]}^\circ$ is nilpotent thanks to Proposition~\ref{pro:HL16-4.17}~(iii). Conversely, for each $l\geq1$, it is clear that  $e_{i,p}(\mathfrak{M}_{\Gamma}^l+ I'_{W_{\Gamma}})e_{j,q}$ contains the subspace $e_{i,p}(\langle e_{k,r} \mid (k,r)\in \left(\Gamma_0\setminus (\Gamma_{[a,b]}^{\circ})_0\right)\rangle+ I'_{W_{\Gamma}})e_{j,q}$ for
sufficiently large $[a,b]$ such that $(i,p), (j,q)\in (\Gamma^\circ_{[a,b]})_0$. Hence, we have 
$e_{i,p}(\mathtt{LHS})e_{j,q}=e_{i,p}(\mathtt{RHS})e_{j,q}$ for any $(i,p), (j,q)\in\Gamma_0$. Since our ideals are determined by homogeneous subspaces with respect to the $\Gamma_0\times \Gamma_0$-grading, this yields the desired equality.  
\end{proof}

We give a connection between the algebra $\Lambda(\infty)$ and the $\mathbb{Z}$-graded algebras $\Pi(l)\,(l\in \mathbb Z_{\geq 1})$ and $\Pi(\infty)$ studied by Gei\ss--Leclerc--Schr\"oer in \cite{GLS-2017} and by the second and third named authors in \cite[\S\S4--5]{FM}.

Let $C = (c_{ij})_{i,j \in \mathtt{I}_0}$ be the Cartan matrix 
associated to the finite-dimensional simple Lie algebra $\mathfrak{g}_0$.
For $i,j\in \mathtt{I}_0$, we write $i\sim j$ if $c_{ij}<0$.

\begin{definition}[{\cite{GLS-2017,FM}}]\label{def:GLS1}
Keep the above notation. Associated with the finite-dimensional simple Lie algebra $\mathfrak{g}_0$, we define a quiver $\widetilde{Q}$ and a degree map $\deg \colon \widetilde{Q}_1 \to \mathbb{Z}$
as follows:
\begin{itemize}
    \item [(i)] The quiver $\widetilde{Q} \coloneqq \widetilde{Q}(\mathfrak{g}_0)$ is defined by
    \begin{gather*}
    \widetilde{Q}_0=\mathtt{I}_0, \quad \widetilde{Q}_1= \{ \alpha_{ij}\mid (i, j)\in \mathtt{I}_0 \times \mathtt{I}_0, c_{ij}<0\}\cup \{\varepsilon_i \mid i\in \mathtt{I}_0\}\\
s(\alpha_{ij})=j,\quad t(\alpha_{ij})=i,\quad s(\varepsilon_i)=t(\varepsilon_i)=i.
\end{gather*}
\item[(ii)] The degree map $\deg\colon \widetilde{Q}_1 \to \mathbb{Z}$  
on $\widetilde{Q}$ is defined by \begin{align*}
    \deg(\alpha_{ij})=-\mathop{\mathrm{max}}(d_i, d_j),\quad \deg(\varepsilon_i)=2d_i,
\end{align*}
where $D=\diag(d_i\mid i\in \mathtt{I}_0)$ is the minimal left symmetrizer of $C=(c_{ij})$.
\end{itemize}
\end{definition}

The path algebra $\mathbb{C}\widetilde{Q}$ is regarded as a $\mathbb{Z}$-graded algebra with respect to the degree map  $\deg\colon \widetilde{Q}_1 \to \mathbb{Z}$. It is easy to see that
\begin{eqnarray}\label{eqn:silon}
   \varepsilon\coloneqq \sum_{i\in \mathtt{I}_0}\varepsilon_i^{r^\vee /d_i}  
\end{eqnarray}
is a homogeneous element in $\mathbb{C}\widetilde{Q}$ of degree $2r^\vee$.

\begin{definition}\label{def:GLS2}
    Keep the above notation.  Let $\omega_{ij}\in \{1, -1\}$ be a choice of sign such that $\omega_{ij}= -\omega_{ji}$ if $i\sim j$. We define the following $\mathbb{Z}$-graded algebras $\Pi(l)\,(l\in \mathbb{Z}_{\geq 1})$ and $\Pi(\infty)$:
    \begin{enumerate}
        \item[(i)]\label{eq:GLS2_1} For each $l \in \mathbb{Z}_{\geq 1}$, the $\mathbb{Z}$-graded algebra $\Pi(l)$ is defined by
    \[\Pi(l)\coloneqq \mathbb{C}\widetilde{Q}/( I'_{\GLS} +\langle\varepsilon^l\rangle),\]
    where $I'_{\GLS}\coloneqq I'_{\GLS}(\mathfrak{g}_0)$ is the ideal of $\mathbb{C}\widetilde{Q}$ generated by the following homogeneous relations:
    \begin{itemize}
        \item[(R1)] $\varepsilon_i^{-c_{ij}}\alpha_{ij}-\alpha_{ij}\varepsilon_j^{-c_{ji}}$ for each $i,j \in \mathtt{I}_0$ with $i\sim j$;
        \item[(R2)] $\sum_{j\in \mathtt{I}_0; j\sim i}\sum_{k=0}^{-c_{ij}-1}\omega_{ij}\varepsilon_i^k\alpha_{ij}\alpha_{ji}\varepsilon_i^{-c_{ij}-1-k}$ for each $i\in \mathtt{I}_0$.
    \end{itemize}
    \item[(ii)]\label{eq:GLS2_2} The $\mathbb{Z}$-graded algebra $\Pi(\infty)$ is defined by
    \[\Pi(\infty)\coloneqq \varprojlim_{l\geq 1} \Pi(l)\]
    as the projective limit with respect to the projective system given by natural projections in the category of $\mathbb{Z}$-graded algebras.
    \end{enumerate}
\end{definition}
\begin{remark} \label{rmk:gls-potential}
    \begin{enumerate}
        \item[(i)]  
        We caution that the algebra $\Pi(l)$ and the associated simple Lie algebra $\mathfrak{g}_0$ here are of Langlands dual type to the ones studied in \cite{GLS-2017} (see \cite[Remark~2.1]{FM}). The algebras $\Pi(l)$ and $\Pi(\infty)$ are referred to as the generalized preprojective algebras in \cite{GLS-2017,FM}.
The algebra $\Pi(1)$ for $\mathfrak{g}_0$ of simply-laced type is identical to the classical preprojective algebra of a Dynkin quiver.
        \item[(ii)]  It is known that each algebra $\Pi(l)$ associated with a simple Lie algebra in Definition~\ref{def:GLS2}~(i) is finite-dimensional, \confer \cite[Corollary~11.13]{GLS-2017}.
        Note that the generalized preprojective algebras in \cite{GLS-2017} are defined for arbitrary symmetrizable Kac-Moody algebras and they are infinite-dimensional in general.
        \item[(iii)] One can easily observe that the homogeneous generators of the ideal $I'_{\GLS}$ in Definition~\ref{def:GLS2}~(i) are given by derivations of the ``potential'' $W_\GLS \coloneqq W_\GLS(\mathfrak{g}_0)$\footnote{Potentials similar to those defining $\Pi(\infty)$ and $\Pi(l)$ have previously appeared in the study of McKay correspondences for simply-laced types (see, e.g., \cite{BGK02,Moz11}), as well as in the context of theoretical physics and generalizations of quiver varieties for non-simply-laced types \cite{CD12,Yam10}.} defined by
\[
W_\GLS = \sum_{i, j\in \mathtt{I}_0; i\sim j} \omega_{ij}\varepsilon_i^{-c_{ij}} \alpha_{ij}\alpha_{ji}.
\]
Since $d_ic_{ij}+\max(d_i,d_j)=0$  for any $i\sim j$, we see that the potential $W_\GLS$ is homogeneous and $\deg(W_\GLS)=0$ (see \cite[\S 4.1, 4.3]{FM} and \cite[\S 1.7.3]{GLS-2017} for more details).
\item[(iv)] By the construction of $\Pi(\infty)$ (see the discussion in Proposition~\ref{Prop:KS} below), we can easily see that the canonical homomorphism $\mathbb C \widetilde{Q}/I'_\GLS \to \Pi(\infty)$ is surjective and there is an isomorphism
    \[\Pi(\infty)\cong \mathbb{C}\widetilde{Q}/\bigcap_{l\geq 1}(I'_\GLS + \langle \varepsilon^l\rangle).\]
    \end{enumerate}
\end{remark}

Following Cohen--Montgomery~\cite{CM84}, it is of classical interest in Morita theory to find some (non-unital) algebra whose category of modules is equivalent to the category of graded modules over a given graded algebra. We relate $\Lambda(\infty)$ and $\Pi(\infty)$ in this context. 

Recall that the infinite quiver $\Gamma$ has exactly two possible choices depending on the parity functions $\epsilon\colon \mathtt{I}_0 \to \{0, 1\}$. Let $\widetilde{\Gamma}$ be the infinite quiver which has two connected components given by these two possible choices of $\Gamma$. The vertex set $\widetilde{\Gamma}_0$ of $\widetilde{\Gamma}$ is given by
$$\widetilde{\Gamma}_0=\mathtt{I}_0\times \mathbb Z.$$
Let $W_{\widetilde{\Gamma}}$ be the formal sum of all the signed oriented cycles $\omega_{ij}w_{ij}(p)$ of the form \eqref{eq:cycles} with signs $\omega_{ij}$ in the infinite quiver $\widetilde{\Gamma}$. We define the algebra $\widetilde{\Lambda}(\infty)$ as the projective limit similarly to the definition of $\Lambda(\infty)$ in Definition~\ref{def:Lambdainf}.
 \begin{example}\label{ex:quiver}
    The quivers $\widetilde{Q}$ and $\Gamma$ of type $B_3$ are illustrated in Figure \ref{fig:1} and Figure \ref{fig:2}. The quiver $\widetilde{\Gamma}$ has exactly two connected components isomorphic to $\Gamma$ (up to shift). Here we take the same parity function as Example~\ref{ex:Gammaab}~(ii):
    \begin{figure}[htbp]
        \centering
        \[\begin{tikzcd}
	1 && 2 && 3
	\arrow["{\varepsilon_1}", from=1-1, to=1-1, loop, in=55, out=125, distance=10mm]
	\arrow["{\alpha_{21}}", from=1-1, to=1-3]
	\arrow["{\alpha_{12}}", shift left=3, from=1-3, to=1-1]
	\arrow["{\varepsilon_2}", from=1-3, to=1-3, loop, in=55, out=125, distance=10mm]
	\arrow["{\alpha_{32}}", from=1-3, to=1-5]
	\arrow["{\alpha_{23}}", shift left=3, from=1-5, to=1-3]
	\arrow["{\varepsilon_3}", from=1-5, to=1-5, loop, in=55, out=125, distance=10mm]
\end{tikzcd}\]
\caption{the quiver $\widetilde{Q}$ of type $B_3$}
\label{fig:1}
    \end{figure}
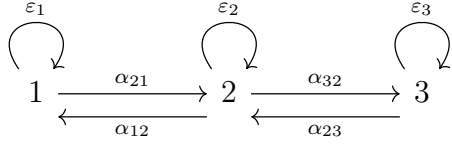
    \begin{figure}[htbp]
        \centering
        \[\begin{tikzcd}[column sep={11mm, between origins}, row sep=2mm]
(i\setminus p) & \cdots & -7 & -6 & -5 & -4 & -3 & -2 & -1 & 0 & 1 & 2 & 3 & \cdots \\
1 &\cdots& & & \circ\arrow[lll, <-] \arrow[drr, <-] & & & & \circ \arrow[llll, <-] \arrow[drr, <-] & & & & \circ \arrow[llll, <-] &\cdots \arrow[l,-]\\
2 &\cdots& \circ\arrow[l,<-] \arrow[drr, <-] \arrow[urr, <-] & & & & \circ \arrow[llll, <-] \arrow[drr, <-] \arrow[urr, <-] & & & & \circ \arrow[llll, <-] \arrow[drr, <-] \arrow[urr, <-] & & &\cdots\arrow[lll,-]\\
3 &\cdots& \circ\arrow[l,<-]\arrow[drr, <-]& & \circ\arrow[ll,<-] \arrow[urr, <-] & & \circ \arrow[ll, <-] \arrow[drr, <-] & & \circ \arrow[ll, <-] \arrow[urr, <-] & & \circ \arrow[ll, <-] \arrow[drr, <-] & & \circ \arrow[ll, <-] &\cdots \arrow[l,-]\\ 
2 &\cdots& & & \circ\arrow[lll, <-] \arrow[urr, <-] \arrow[drr, <-] & & & & \circ \arrow[llll, <-] \arrow[urr, <-] \arrow[drr, <-] & & & & \circ \arrow[llll, <-] &\cdots \arrow[l,-]\\
1 &\cdots& \circ\arrow[l,<-] \arrow[urr, <-] & & & & \circ \arrow[llll, <-] \arrow[urr, <-] & & & & \circ \arrow[llll, <-] \arrow[urr, <-] & & &\cdots\arrow[lll,-]
\end{tikzcd}\]
\caption{the quiver $\Gamma$ of type $B_3$}
\label{fig:2}
    \end{figure}
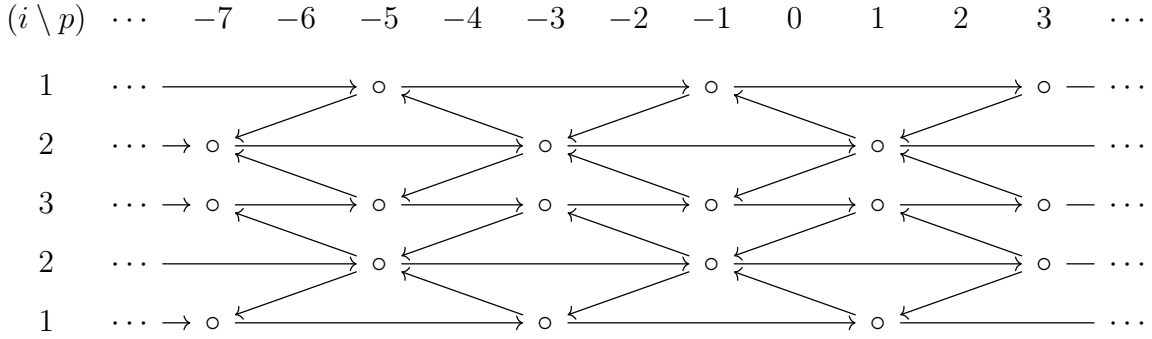
 \end{example}
If we have a $\mathbb Z$-graded vector space $V=\bigoplus_{d\in \mathbb Z}V_d$, we define its $p$-th grading shift $V\langle p \rangle$ to be the graded vector space whose $d$-th graded component is given by $(V \langle p \rangle)_d = V_{d-p}$ for each $p\in \mathbb Z$.

A non-unital algebra $A$ is said to have {\em local units} if each finite subset of $A$ is contained in some subring $wAw$, where $w\in A$ is an idempotent. By definition, the algebra $\Lambda(\infty)$ (resp.\ $\widetilde{\Lambda}(\infty)$) has local units, which are given by a finite sum of idempotents in $\{e_{(i,p)}\}_{(i,p) \in \Gamma_0}$ (resp.\ in $\{e_{(j,q)}\}_{(j,q)\in \widetilde{\Gamma}_0}$). 
A left $\Lambda(\infty)$-module $M$ is called a \emph{unital left $\Lambda(\infty)$-module} 
if $\Lambda(\infty)M=M$, equivalently,
 $M = \bigoplus_{(i,p) \in \Gamma_0} e_{i,p}M$ as a $\mathbb{C}$-vector space. Let $\Lambda(\infty)\mathchar`-\mathsf{Mod}$ denote the category of unital left $\Lambda(\infty)$-modules. We similarly define unital left $\widetilde{\Lambda}(\infty)$-modules, and let $\widetilde{\Lambda}(\infty)\mathchar`-\mathsf{Mod}$ denote the category of unital left $\widetilde{\Lambda}(\infty)$-modules. Let $\Pi(\infty)\mathchar`-\mathsf{Mod}^\mathbb{Z}$ denote the category of $\mathbb{Z}$-graded left $\Pi(\infty)$-modules. We define similar notions for right modules dually.
\begin{proposition}[{\confer \cite[Remark 4.3 \& Proposition 4.4]{FM}}] \label{rem:compare-gpa}
    There is a categorical isomorphism $\widetilde{\Lambda}(\infty)\mathchar`-\mathsf{Mod} \cong \Pi(\infty)\mathchar`-\mathsf{Mod}^\mathbb{Z}$. There is a similar isomorphism for right modules as well.
\end{proposition}
\begin{proof}
    We follow a general construction \cite{CM84,Bea88} of the smash product of a group graded algebra. In particular, we construct an algebra isomorphism $\widetilde{\Lambda}(\infty)\cong \Pi(\infty)\#\mathbb{Z}$, and this yields the desired categorical isomorphism thanks to \cite[\S 2, Theorem~2.6]{Bea88}. Here the (non-unital) ring $R\# \mathbb{Z}$ for a unital $\mathbb{Z}$-graded ring $R$ is defined as a free left $R$-module $\bigoplus_{m \in \mathbb{Z}} R p_m$ with a basis $\{p_m \mid m \in \mathbb{Z}\}$. The multiplication is given by the formula $(\lambda p_m)(\mu p_n) = \lambda \mu_{m-n}p_n$ for $\lambda, \mu \in R$ and $n,m \in \mathbb{Z}$, where $\mu_{m-n}$ is the homogeneous component of $\mu$ of degree $m-n$.

 We define a ring homomorphism $\psi\colon \mathbb{C}\widetilde{\Gamma} \to \mathbb{C}\widetilde{Q}\#\mathbb{Z}$ by generators as follows:
\begin{align*}
e_{i,m} \mapsto e_i p_m, \quad
\alpha_{ij}(m) \mapsto \alpha_{ij} p_m, \quad
\varepsilon_i(m) \mapsto \varepsilon_i p_m.
\end{align*}
It is easy to check that $\psi\colon \mathbb{C}\widetilde{\Gamma} \to \mathbb{C}\widetilde{Q}\#\mathbb{Z}$ is a $\mathbb C$-algebra isomorphism. 

We can show the following isomorphisms:
\[
   \widetilde \Lambda(\infty)  \cong \mathbb{C} \widetilde \Gamma/ \bigcap_{l\geq 1}(\mathfrak{M}_{\widetilde \Gamma}^l+I'_{W_{\widetilde \Gamma}})\quad\text{and}\quad
   \Pi(\infty)\coloneqq \varprojlim_{l\geq 1} \Pi(l)\cong \mathbb{C}\widetilde{Q}/\bigcap_{l\geq 1}(\mathfrak{M}_{\widetilde{Q}}^l+I'_{\mathsf{GLS}}).
\]
The first isomorphism follows from a similar argument to that of Proposition~\ref{prop:Lambdainfty}.
The second isomorphism also follows from a similar argument thanks to Remark~\ref{rmk:gls-potential}~(iv) and the fact that $\Pi(l)$ is a finite-dimensional algebra \cite[Corollary~11.13]{GLS-2017}. In particular, we can easily see from the relations of $\Pi(l)$ that $\fM_{\widetilde{Q}}$ is nilpotent in $\Pi(l)$  (\confer \cite[Theorem~2.2~(2)]{FM}).

Since the isomorphism $\psi$ identifies the ideal $\mathfrak{M}_{\widetilde \Gamma}^l+I'_{W_{\widetilde \Gamma}}$ of
$\mathbb{C} \widetilde \Gamma$ with the ideal $\bigoplus_{m\in\mathbb Z}(\mathfrak{M}_{\widetilde{Q}}^l+I'_{\mathsf{GLS}})p_m$ of $\mathbb{C}\widetilde{Q}\#\mathbb{Z}$ for any $l\geq 1$, we know that  $\psi$ induces the $\mathbb C$-algebra isomorphism $\widetilde{\Lambda}(\infty)\cong \Pi(\infty)\#\mathbb{Z}$. Thus the assertion holds.
\end{proof}

Since ${\Lambda}(\infty)\mathchar`-\mathsf{Mod}$ can be viewed as a full subcategory of $\widetilde{\Lambda}(\infty)\mathchar`-\mathsf{Mod}$, each object in ${\Lambda}(\infty)\mathchar`-\mathsf{Mod}$ can be identified with a $\mathbb{Z}$-graded module over $\Pi(\infty)$ via the categorical isomorphism in Proposition~\ref{rem:compare-gpa}.
\begin{remark}
    \begin{enumerate}
        \item[(i)] The projective objects $\Lambda(\infty)e_{i,p}$ and the $p$-th grading shift of $\Pi(\infty)e_i$ are identified under the categorical isomorphism in Proposition~\ref{rem:compare-gpa}. Under the categorical isomorphism for right modules, the projective object $e_{i,p}\Lambda(\infty)$ is identified with $(-p)$-th grading shift of $e_i \Pi(\infty)$ in our convention.
        \item[(ii)] By construction, there is a canonical isomorphism of $\mathbb{C}$-vector spaces,
\[ e_{i,p} \Lambda(\infty) e_{j,s} \cong (e_i \Pi(\infty) e_j)_{p-s}, \qquad \forall (i,p), (j,s) \in \Gamma_0,\]
 where $V_{u}$ denotes the $u$-th homogeneous component of a $\mathbb{Z}$-graded vector space $V$. 
        \item[(iii)] The smash product used in the proof of Proposition~\ref{rem:compare-gpa} is sometimes referred to as the Galois covering in the sense of Gabriel~\cite{Gab81} with an interest in Auslander--Reiten theory. The relationship between the smash product and the Galois covering was studied by Cibils--Marcos~\cite{CM06} in a more general context (see also \cite{Asas22} for a textbook of modern perspectives).
    \end{enumerate}
\end{remark}

Let $A$ be either $\Lambda(\infty)$ or $\widetilde{\Lambda}(\infty)$.
Analogously to \cite[\S 10]{AF92}, we give the following definitions. 
A unital $A$-module $M$ is \emph{finitely generated} if it satisfies the following condition: for any epimorphism $\bigoplus_{m\in H} U_m \twoheadrightarrow M$ in $A\Mod$ with an arbitrary index set $H$, there exists a finite subset $F\subseteq H$ such that the canonical injection induces an epimorphism $\bigoplus_{m\in F} U_{m} \twoheadrightarrow M$.
A unital $A$-module $M$ is \emph{finitely cogenerated} if it satisfies the following condition: for any monomorphism $M \hookrightarrow \prod_{n\in H} V_n$ in $A\Mod$ with an arbitrary index set $H$,
there exists a finite subset $F\subseteq H$ such that the canonical projection induces a monomorphism $M \hookrightarrow \prod_{n\in F} V_n$.

     \vspace{1.5mm}
Let $\widetilde \Lambda(\infty)\mathchar`-\mathsf{mod}_\mathsf{fcg}$ (resp.\ $\widetilde \Lambda(\infty)\mathchar`-\mathsf{mod}_\mathsf{fg}$) denote the full subcategory of $\widetilde \Lambda(\infty)\mathchar`-\mathsf{Mod}$ consisting of unital finitely cogenerated (resp.\ finitely generated) modules. The categorical isomorphism $\widetilde{\Lambda}(\infty)\Mod \cong \Pi(\infty)\Mod^\mathbb{Z}$ induces, by restriction, the following isomorphisms:
\[
\widetilde{\Lambda}(\infty)\modfg\cong \Pi(\infty)\modfg^{\mathbb{Z}}\quad \text{and}\quad \widetilde{\Lambda}(\infty)\modfcg\cong \Pi(\infty)\modfcg^{\mathbb{Z}}.
\]

Let $\Lambda(\infty)\injc$ (resp.\ $\Lambda(\infty)\projc$) denote the full subcategory of injective objects in $\Lambda(\infty)\mathchar`-\mathsf{mod}_\mathsf{fcg}$ (resp.\ projective objects in $\Lambda(\infty)\mathchar`-\mathsf{mod}_\mathsf{fg}$).
\begin{proposition}\label{Prop:KS}
    \begin{enumerate}
        \item[(i)]\label{eq:KS1} The category $\Lambda(\infty)\mathchar`-\mathsf{mod}_\mathsf{fcg}$ is a $\Hom$-finite abelian category and each object in $\Lambda(\infty)\mathchar`-\mathsf{mod}_\mathsf{fcg}$ admits an injective resolution in this category.
        \item[(ii)]\label{eq:KS2} The triangulated category $\mathcal{K}^b(\Lambda(\infty)\injc)$ is $\Hom$-finite and Krull--Schmidt (\ie any object is isomorphic to a finite direct sum of objects whose endomorphism rings are local rings).
    \end{enumerate}
\end{proposition}
\begin{proof}
By Proposition~\ref{rem:compare-gpa}, the category $\Lambda(\infty)\mathchar`-\mathsf{mod}_\mathsf{fcg}\subseteq \widetilde\Lambda(\infty)\mathchar`-\mathsf{mod}_\mathsf{fcg}$ is equivalent to an abelian subcategory of the category $\Pi(\infty)\modfcg^{\mathbb{Z}}$ of finitely cogenerated $\mathbb{Z}$-graded $\Pi(\infty)$-modules. We know that $\Pi(\infty)$ is a graded free module of finite rank over a central subalgebra isomorphic to a $\mathbb{Z}$-graded polynomial ring $\mathbb{C}[\varepsilon]$ with $\deg \varepsilon = 2r^\vee> 0$ (and hence $\Pi(\infty)$ is a left and right noetherian algebra). This is because we have $e_i\Pi(\infty)e_j\cong \varprojlim_{l\geq 1} e_i\Pi(l)e_j$ and each $e_i\Pi(l)e_j$ is a $\mathbb{Z}$-graded free $\mathbb{C}[\varepsilon_i]/\langle \varepsilon_i^{r^\vee l/d_i} \rangle$-module of finite rank by left multiplication and is a $\mathbb{Z}$-graded free $\mathbb{C}[\varepsilon_j]/\langle \varepsilon_j^{r^\vee l/d_j} \rangle$-module of finite rank by right multiplication thanks to \cite[Theorem~11.12]{GLS-2017} and the symmetry of each $\Pi(l)$ given by the anti-involution $\phi$ in \cite[\S 2.3]{FM}. 
The corresponding homogeneous bases can be chosen compatibly
with the natural projections $\Pi(l+1)\twoheadrightarrow\Pi(l)\,(l\in \mathbb{Z}_{>0})$.
Hence the inverse systems stabilize degreewise, and passing to the
projective limit shows that each $e_i\Pi(\infty)e_j$ is $\mathbb{Z}$-graded free of
finite rank over $\mathbb C[\varepsilon]$, since $\mathbb{C}[\varepsilon_i](\cong \varprojlim_{l\geq 1} \mathbb{C}[\varepsilon_i]/\langle \varepsilon_i^{r^\vee l/d_i} \rangle)$ and $\mathbb{C}[\varepsilon_j](\cong \varprojlim_{l\geq 1} \mathbb{C}[\varepsilon_j]/\langle \varepsilon_j^{r^\vee l/d_j} \rangle)$ are $\mathbb{Z}$-graded free over $\mathbb{C}[\varepsilon]$. This also yields the fact that the $\mathbb Z$-gradation of $\Pi(\infty)$ has a lower bound and each degree component is finite-dimensional. We have the (restricted) $\mathbb{C}$-duality $\mathbb{D}\colon \Lambda(\infty)\mathchar`-\mathsf{mod}_\mathsf{fcg} \xrightarrow{\sim} \Lambda(\infty)^{\text{op}}\mathchar`-\mathsf{mod}_\mathsf{fg}$ thanks to Proposition~\ref{rem:compare-gpa}. All the assertions in (i) follow from these facts.
The assertion (ii) is a consequence of the assertion (i) and the well-known fact that the bounded homotopy category of a $\Hom$-finite Krull--Schmidt linear category is $\Hom$-finite and Krull--Schmidt as well.
\end{proof}
\begin{remark}
    The discussions in Proposition~\ref{Prop:KS} show that the Grothendieck groups $K_0(\mathcal{K}^b(\Pi(\infty)\injc^\mathbb{Z}))\cong K_0(\mathcal{K}^b(\widetilde{\Lambda}(\infty)\injc))$ and $K_0(\mathcal{K}^b(\Lambda(\infty)\injc))$ are free abelian groups.
\end{remark}
For each right unital $\Lambda(\infty)$-module $M=\bigoplus_{(i,p)\in \Gamma_0} Me_{i,p}$, we naturally associate an object $\mathbb{D}(M)\coloneqq \bigoplus_{(i,p)\in \Gamma_0}\Hom_{\mathbb{C}}(Me_{i,p},\mathbb{C})$ in $\Lambda(\infty)\mathchar`-\mathsf{Mod}$, and this yields a $\mathbb{C}$-duality $\Lambda(\infty)^{\text{op}}\mathchar`-\mathsf{mod}_\mathsf{fg}\to \Lambda(\infty)\mathchar`-\mathsf{mod}_\mathsf{fcg}$  as seen in the proof of Proposition~\ref{Prop:KS}.
 An indecomposable injective left $\Lambda(\infty)$-module is given by
\begin{equation}\label{eq:injdef}
    I_{i,p} \coloneqq \mathbb{D}(e_{i,p}\Lambda(\infty))
\end{equation}
for each vertex $(i,p) \in \Gamma_{0}$, and the set $\{ I_{i,p} \mid (i,p) \in \Gamma_0\}$ gives a complete system of isomorphism classes of indecomposable objects in $\Lambda(\infty)\injc$.
In what follows, we regard a homomorphism $f \colon I_0 \to I_1$ in $\Lambda(\infty)\injc$ as a two-term complex in $\mathcal{K}^{[0,1]}(\Lambda(\infty)\injc)$ concentrated in cohomological degree $0$ and $1$. One can regard it as an object in the bounded homotopy category $\mathcal{K}^b(\Lambda(\infty)\injc)$.

The algebra $A_{[a,b]}$ is finite-dimensional by Proposition~\ref{pro:HL16-4.17}~(i). The algebra $\Lambda(\infty)$ is an algebra with local units. Any object $P$ in $\Lambda(\infty)\projc$ satisfies $P=\bigoplus_{(i,p)\in \Gamma_0} e_{i,p}P$ as a $\mathbb{C}$-vector space and $\dim_{\mathbb{C}}(e_{i,p}P) <\infty$ for each $(i,p)\in \Gamma_0$ thanks to Proposition~\ref{Prop:KS}. In particular, we have the \emph{Nakayama functors}, which give linear equivalences:
   \begin{eqnarray*}
      \nu_{[a,b]}(-)&\coloneqq& \mathbb{D}\circ\Hom_{A_{[a,b]}}(-, A_{[a,b]})\colon A_{[a,b]}\projc \xrightarrow{\sim} A_{[a,b]}\injc,\\ 
      \nu(-)&\coloneqq& \mathbb{D}\circ\Hom_{\Lambda(\infty)}(-, \Lambda(\infty))\colon \Lambda(\infty)\projc \xrightarrow{\sim} \Lambda(\infty)\injc.
   \end{eqnarray*}
To relate the two categories $\Lambda(\infty)\injc$ and $A_{[a,b]}\injc$, we consider the following functor
\[T_{[a,b]}\coloneq \Hom_{\Lambda(\infty)}(A_{[a,b]}, -) \colon \Lambda(\infty)\injc \to A_{[a,b]}\injc.\]
It is easy to see that 
\begin{eqnarray}\label{eqn:nu}
T_{[a,b]}\circ \nu \cong \nu_{[a,b]}\circ (A_{[a,b]}\otimes_{\Lambda(\infty)}(-))
\end{eqnarray}
as a functor  $\Lambda(\infty)\projc\to A_{[a,b]}\injc$, and we have \[T_{[a,b]}(I_{i,p}) \cong \begin{cases}
\mathbb{D}(e_{i,p}A_{[a,b]}) & \text{if $(i,p) \in (\Gamma_{[a,b]}^\circ)_0$}, \\
\{ 0 \} & \text{otherwise.}
\end{cases}\]
Note that there is a natural isomorphism
\begin{equation}\label{eq:HomII}
\Hom_{\Lambda(\infty)}(I_{i,p}, I_{j,s}) \cong e_{i,p} \Lambda(\infty) e_{j,s}.
\end{equation}
\begin{definition}
    Let $A$ denote the algebra $\Lambda(\infty)$ or $A_{[a,b]}$.
    A two-term complex $f\in \mathcal{K}^{[0,1]}(A\injc)$ is referred to as {\em rigid}, if  $\Hom_{\mathcal{K}^{b}(A\injc)}(f,f[1])=0$.
\end{definition}
\begin{remark}\label{rem:T_X2}
    Thanks to \eqref{eqn:nu}, the functor $T_{[a,b]}$ and its derived functor are identical to the truncation functors associated to factor algebras. We can easily see that if $f\in \mathcal{K}^{[0,1]}(\Lambda(\infty)\injc)$ is rigid, then $T_{[a,b]}f \in \mathcal{K}^{[0,1]}(A_{[a,b]}\injc)$ is also rigid for any integer interval $[a,b]$ (\confer \cite[Lemma~2.13]{KI22} for a study of unital rings).
\end{remark}
\begin{proposition} \label{prop:TX}
Let $f \colon I_0 \to I_1$ and $f' \colon I_0' \to I_1'$ be morphisms in the category $\Lambda(\infty)\injc$.
Then, for sufficiently large $[a,b]$, the derived functor of $T_{[a,b]}$ induces an isomorphism
\[ \Hom_{\mathcal{K}^b(\Lambda(\infty)\injc)}(f,f'[1]) \cong \Hom_{\mathcal{K}^b(A_{[a,b]}\injc)}(T_{[a,b]}f, T_{[a,b]}f'[1]).\]
(For simplicity, we also denote the derived functor by $T_{[a,b]}$).
\end{proposition}
\begin{proof}
By definition of homotopy relations, Definition~\ref{def:Lambdainf}, and the identification~\eqref{eq:HomII}, 
the functor $T_{[a,b]}$ induces the following commutative diagram with exact rows:
\[\begin{tikzcd}[scale cd=0.83]
	{\displaystyle{\prod_{m\in\{0,1\}}\Hom_{\Lambda(\infty)}(I_m,I'_m)}} & {\Hom_{\Lambda(\infty)}(I_0, I_1')} & {\Hom_{\mathcal{K}^b(\Lambda(\infty)\injc)}(f, f'[1])} & 0 && {} \\
	\\
	{\displaystyle{\prod_{m\in\{0,1\}}\Hom_{A_{[a,b]}}(T_{[a,b]}I_m,T_{[a,b]}I'_m)}} & {\Hom_{A_{[a,b]}}(T_{[a,b]}I_0, T_{[a,b]}I_1')} & {\Hom_{\mathcal{K}^b(A_{[a,b]}\injc)}(T_{[a,b]}f,T_{[a,b]} f'[1])} & 0
	\arrow[from=1-1, to=1-2]
	\arrow["{\alpha_1}", from=1-1, to=3-1]
	\arrow[from=1-2, to=1-3]
	\arrow["{\alpha_2}", from=1-2, to=3-2]
	\arrow[from=1-3, to=1-4]
	\arrow["{\alpha_3}", from=1-3, to=3-3]
	\arrow[from=3-1, to=3-2]
	\arrow[from=3-2, to=3-3]
	\arrow[from=3-3, to=3-4]
\end{tikzcd}
\]
Here $\alpha_1$ and $\alpha_2$ are natural maps obtained from the projective system in Definition~\ref{def:Lambdainf} and the identification~\eqref{eq:HomII}. The map $\alpha_3$ is an induced map.
Since the projective system in Definition~\ref{def:Lambdainf} is given by natural projections and satisfies the Mittag-Leffler condition, we see that the morphisms $\alpha_1$ and $\alpha_2$ are surjective. This implies that $\alpha_3$ is also surjective.

Since the spaces $\prod_{m\in\{0,1\}}\Hom_{\Lambda(\infty)}(I_m,I'_m)$ and $\Hom_{\Lambda(\infty)}(I_0, I'_1)$
are finite-dimensional, the morphisms $\alpha_1$ and $\alpha_2$ must be isomorphisms for any sufficiently large $[a,b]$ by construction of $\Lambda(\infty)$. Thus we conclude that $\alpha_3$ is an isomorphism by the five lemma.
\end{proof}

\begin{remark}\label{rem:field}
    The arguments in this subsection, up to this point, hold over an arbitrary base field. Those in the remainder of this subsection and in \S\ref{subsec:PKR} are valid over any algebraically closed field. However, we take the ground field to be $\mathbb C$ for simplicity and to maintain consistency with \S \ref{sec:4}.
\end{remark}

\begin{definition}\label{def:gvec}
(i)  Let $A$ denote the algebra $\Lambda(\infty)$ or $A_{[a,b]}$. For a two-term complex $f\colon I_0 \to I_1$ in $\mathcal{K}^{[0,1]}(A\injc)$, its  {\em $g$-vector} $\mathbf{g}(f)$ 
is defined to be
 \[\mathbf{g}(f)\coloneqq [I_1]-[I_0],\]
viewed as an element of the Grothendieck group $K_0(\mathcal{K}^{b}(A\injc))$.

(ii)  For a simple module $M \in \mathcal{C}_{\mathbb{Z}}$, we define the vector ${\bf g}_M^\infty = (g_{(i,p)})_{(i,p) \in \Gamma_0}$ by the formula
\[ g_{(i,p)} \coloneqq u_{i,p+d_i}(M) - u_{i,p-d_i}(M)\]
in the notation of \eqref{eq:notation_u}, for all $(i,p) \in \Gamma_0$. 
\end{definition}
\begin{remark}\label{rem:reach-g-vector}
If a simple module $M\in\mathcal C_{\mathbb Z}$ is reachable in $\mathcal C_{[a,b]}$, then by Lemma \ref{lem:gFab}, we know that the components of ${\bf g}_M^\infty= (g_{(i,p)})_{(i,p) \in \Gamma_0}$ supported on the vertex set $\mathtt K_{[a,b]}=(\Gamma_{[a,b]})_0$ give the extended $g$-vector ${\bf g}_M^{[a,b]}=(g_{(i,p)})_{(i,p) \in \mathtt K_{[a,b]}}$ of the cluster monomial $\varphi_{[a,b]}(M)$ of $\mathcal A_{[a,b]}$ with respect to its standard initial seed. 
\end{remark}

\begin{lemma}\label{lem:gvec-rigid}
 Let $A$ denote the algebra $\Lambda(\infty)$ or $A_{[a,b]}$.
Let $f$ and $f'$ be two-term complexes in $\mathcal{K}^{[0,1]}(A\injc)$. If both $f$ and $f'$ are rigid and $\mathbf{g}(f) = \mathbf{g}(f')$, then $f \cong f'$ as objects in $\mathcal{K}^{[0,1]}(A\injc)$.
\end{lemma}
\begin{proof}
    This type of statement about (split) Grothendieck groups is well-known in various kinds of general settings (see e.g., \cite[Theorem~10]{EJR18}, \cite[Theorem~5.5]{air_2014} in the setting of finite-dimensional algebras, \cite[Theorem~2.3]{DK08} in that of suitable $2$-Calabi-Yau categories). Following the discussions in Proposition~\ref{Prop:KS} (ii) and Remark~\ref{rem:T_X} below, the result stated in the lemma holds as well.
\end{proof}

\begin{remark}\label{rem:T_X}
Let $f\in \Hom_{\Lambda(\infty)}(I_0, I_1)$ with  $I_0, I_1\in \Lambda(\infty)\injc$. It is  known that the orbit $\mathcal{O}_f$ of $f$ under the natural action of the connected algebraic group $\mathsf{Aut}_{\Lambda(\infty)}(I_0) \times \mathsf{Aut}_{\Lambda(\infty)}(I_1)$ is a smooth locally closed subvariety in the affine space $\Hom_{\Lambda(\infty)}(I_0, I_1)$ with respect to the Zariski topology, and there is a $\mathbb{C}$-linear isomorphism
$$\Hom_{\mathcal{K}^b(\Lambda(\infty)\injc)}(f, f[1])\cong T_f(\Hom_{\Lambda(\infty)}(I_0,I_1))/T_f(\mathcal{O}_f)$$ given by the Zariski tangent spaces. In particular, if $f$ is rigid, then the orbit $\mathcal{O}_f$ is Zariski open dense in the affine space $\Hom_{\Lambda(\infty)}(I_0, I_1)$.
\end{remark}

Let $\mathcal A_{[a,b]}^\circ$ be the cluster algebra with trivial
coefficients whose initial quiver is $\Gamma_{[a,b]}^\circ$, and let
$\operatorname{sp}_{[a,b]}\colon
\mathcal A_{[a,b]}\longrightarrow \mathcal A_{[a,b]}^\circ$
be the coefficient specialization sending any frozen variable to $1$.
We remark that $\overline u\coloneqq\operatorname{sp}_{[a,b]}(u)$
is a cluster monomial of $\mathcal A_{[a,b]}^\circ$ for every cluster monomial $u$ of $\mathcal A_{[a,b]}$. With
respect to the corresponding seeds, the separation formula gives
$F_{\overline u}^{t}=F_u^{t}$ and
${\bf g}_{\overline u}^{t}=({\bf g}_u^{t})^\circ$.
\begin{lemma}\label{lem:f-Mab}
 Let $M$ be a reachable simple object in $\mathcal{C}_{[a,b]}$ and let ${\bf g}_M^{[a,b]}\in\mathbb Z^{\mathtt K_{[a,b]}}$ be the extended $g$-vector of $\varphi_{[a,b]}(M)$ with respect to the standard initial seed of $\mathcal A_{[a,b]}$. Then there exists a unique rigid two-term complex $f_M^{[a,b]}\in\mathcal{K}^{[0,1]}(A_{[a,b]}\injc)$  (up to isomorphism) such that  $\mathbf{g}(f_M^{[a,b]}) = (\mathbf{g}_M^{{[a,b]}})^\circ$ in $\mathbb Z^{\mathtt K_{[a,b]}^{\rm uf}}$.
\end{lemma}
\begin{proof}
Since $M$ is reachable in $\mathcal C_{[a,b]}$, we know that $u\coloneq \varphi_{[a,b]}(M)$ is a cluster monomial of $\mathcal A_{[a,b]}$. Set $\overline u\coloneqq\operatorname{sp}_{[a,b]}(u)\in \mathcal A_{[a,b]}^\circ$. By Theorem \ref{thm:DWZ-g-F}, there is a negative-reachable decorated representation $\mathcal M_{\overline u}^{[a,b]}$ of $(\Gamma_{[a,b]}^\circ, W_{[a,b]})$, unique up to isomorphism, such that
\[
{\bf g}(\mathcal M_{\overline u}^{[a,b]})
={\bf g}_{\overline u}^{t_0}
=({\bf g}_u^{t_0})^\circ
=({\bf g}_M^{[a,b]})^\circ.
\]
By the bijection in \eqref{eqn:s-t-bijection}, the decorated representation $\mathcal M_{\overline u}^{[a,b]}$ corresponds to a two-term complex $f_M^{[a,b]}\in \mathcal{K}^{[0,1]}(A_{[a,b]}\injc)$ with $\mathbf{g}(f_M^{[a,b]}) = {\bf g}(\mathcal M_{\overline u}^{[a,b]})=(\mathbf{g}_M^{{[a,b]}})^\circ$. Since $\mathcal M_{\overline u}^{[a,b]}$ is negative-reachable, Corollary \ref{cor:neg-rigid} implies that the two-term complex $f_M^{[a,b]}$ is rigid. This proves the existence. The uniqueness of $f_M^{[a,b]}$ follows from Lemma \ref{lem:gvec-rigid}.
\end{proof}

\begin{lemma}\label{Lem:KM}
Keep the previous notation.
Let $M$ be a reachable simple module in $\mathcal{C}_\mathbb{Z}$. Then there is a unique two-term rigid complex $f_M\in \mathcal{K}^{[0,1]}(\Lambda(\infty)\injc)$ (up to isomorphism) such that ${\bf g}(f_M)={\bf g}_M^\infty$ in $\bigoplus_{\Gamma_0}\mathbb Z$. 
\end{lemma}
\begin{proof}
We write ${\bf g}_M^\infty=(g_v)_{v\in\Gamma_0}$ and set
\[
 I_M^0\coloneqq\bigoplus_{v\in\Gamma_0}I_v^{[-g_v]_+},
 \qquad
 I_M^1\coloneqq\bigoplus_{v\in\Gamma_0}I_v^{[g_v]_+},
\]
where we put $I_v\coloneqq I_{i,p}$ for $v=(i,p)\in \Gamma_0$.
Since ${\bf g}_M^\infty$ has finite support, both $I_M^0$ and $I_M^1$ are objects in $\Lambda(\infty)\injc$.

We choose a sufficiently large interval $L=[a,b]$ such that $M$ is reachable in $\mathcal C_L$;
the non-zero components of ${\bf g}_M^\infty$ are supported on $(\Gamma_L^\circ)_0$;
and for all $(r,s)\in \{(0,0),(1,1),(0,1)\}$, the natural maps
\[
\rho_L^{rs}\colon \Hom_{\Lambda(\infty)}(I_M^r,I_M^s)
\xrightarrow{\sim} \Hom_{A_L}(T_LI_M^r,T_LI_M^s)
\]
are isomorphisms as in Proposition~\ref{prop:TX}. Such an interval exists by the argument in the proof of Proposition~\ref{prop:TX}.

It is well-known that each rigid two-term object in $\cK^{[0,1]}(A_L\injc)$ has a minimal representative whose two terms have no common indecomposable direct summand (\confer \cite[Proposition~2.5]{air_2014} in the dual setting).
Thus we see the two terms of such a complex are determined by its $g$-vector. Recall that the rigid object $f_M^L \in \cK^{[0,1]}(A_L\injc)$ from Lemma~\ref{lem:f-Mab} satisfies ${\bf g}(f_M^L)=({\bf g}_M^L)^\circ={\bf g}_M^\infty|_{(\Gamma_{L}^\circ)_0}$ (see Remark~\ref{rem:reach-g-vector}).

After choosing a minimal representative of $f_M^L$,  it is represented by
\[
\left(T_LI_M^0\xrightarrow{d_{M;L}}T_LI_M^1\right).
\]
Let $d_M\colon I_M^0\to I_M^1$ be the unique lift of $d_{M;L}$
under the isomorphism $\rho_L^{01}$, and define
\[
 f_M\coloneqq\left(I_M^0\xrightarrow{d_M}I_M^1\right).
\]
The above isomorphisms identify the corresponding homotopy quotients as in Proposition~\ref{prop:TX},
and hence
\[
 \Hom_{\mathcal K^b(\Lambda(\infty)\injc)}(f_M,f_M[1])
 \cong
 \Hom_{\mathcal K^b(A_L\injc)}(f_M^L,f_M^L[1])
 =0.
\]
Thus $f_M$ is rigid and we have $\mathbf g(f_M)={\bf g}_M^\infty$.
This proves the existence of $f_M$. By Lemma~\ref{lem:gvec-rigid},
the isomorphism class of $f_M$ is unique and independent of the choice of $L$ (see also the discussion concerning $T_{L'}f_M\cong f_M^{L'}$ for general intervals $L'$ in the proof of Theorem~\ref{Thm:oMN=Ext1}).
\end{proof}

\begin{theorem}\label{Thm:oMN=Ext1}
Let $M$ and $N$ be two reachable simple modules in $\mathcal{C}_{\mathbb{Z}}$, and let $f_M$ and $f_N$ be the corresponding two-term rigid complexes in $\mathcal{K}^{[0,1]}(\Lambda(\infty)\injc)$. Then we have
\begin{eqnarray}\label{eqn:o-E-infty}
    \mathfrak{o}(M,N) = \dim_{\mathbb{C}}\Hom_{\mathcal{K}^b(\Lambda(\infty)\injc)}(f_M, f_N[1]).
\end{eqnarray}
\end{theorem}
\begin{proof}
Since $M$ and $N$ are reachable in $\mathcal C_{\mathbb Z}$, there exists an integer interval $[a,b]$ such that 
$M$ and $N$ are reachable in $\mathcal{C}_{[a,b]}$. 
Set
\[
u_M\coloneqq\varphi_{[a,b]}(M),\qquad
u_N\coloneqq\varphi_{[a,b]}(N),\qquad
\overline{u}_M\coloneqq\operatorname{sp}_{[a,b]}(u_M),\qquad
\overline{u}_N\coloneqq\operatorname{sp}_{[a,b]}(u_N).
\]

For every interval $[a,b]$ such that the objects $M, N$ are reachable in $\mathcal C_{[a,b]}$, the complexes $T_{[a,b]}f_M$ and $T_{[a,b]}f_N$ are rigid (see Remark~\ref{rem:T_X2}) and have $g$-vectors ${\bf g}(T_{[a,b]}f_M)={\bf g}_M^\infty|_{(\Gamma_{[a,b]}^\circ)_0}=({\bf g}_M^{[a,b]})^\circ$ and ${\bf g}(T_{[a,b]}f_N)={\bf g}_N^\infty|_{(\Gamma_{[a,b]}^\circ)_0}=({\bf g}_N^{[a,b]})^\circ$, respectively (see Remark~\ref{rem:reach-g-vector}). Hence $T_{[a,b]}f_M\cong f_M^{[a,b]}$ and $T_{[a,b]}f_N\cong f_N^{[a,b]}$ by Lemma~\ref{lem:f-Mab} and Lemma~\ref{lem:gvec-rigid}.

Let $\mathcal M_{\overline{u}_M}^{[a,b]}$ and $\mathcal M_{\overline{u}_N}^{[a,b]}$ be the negative-reachable decorated representations of $(\Gamma_{[a,b]}^\circ, W_{[a,b]})$ corresponding to $\overline{u}_M$ and $\overline{u}_N$, respectively.
Under the bijection in \eqref{eqn:s-t-bijection}, they are identified with the two-term complexes $T_{[a,b]}f_M\cong f_M^{[a,b]}$ and $T_{[a,b]}f_N\cong f_N^{[a,b]}$ in $\mathcal K^b(A_{[a,b]}\injc)$, respectively.
By the definition of the partial $E$-invariant, we have
\[
\dim_{\mathbb{C}}\Hom_{\mathcal K^b(A_{[a,b]}\injc)}(T_{[a,b]}f_M,T_{[a,b]}f_N[1])
=E^{\mathrm{inj}}(\mathcal M_{\overline{u}_M}^{[a,b]},\mathcal M_{\overline{u}_N}^{[a,b]}).
\]
Since \begin{eqnarray*}
   E^{\mathrm{inj}}(\mathcal M_{\overline{u}_M}^{[a,b]},\mathcal M_{\overline{u}_N}^{[a,b]})
\overset{\text{Thm.~\ref{thm:F-E-inv}}}{=}F_{\overline{u}_M}^{t_0}[{\bf g}_{\overline{u}_N}^{t_0}]
=F_{u_M}^{t_0}[({\bf g}_{u_N}^{t_0})^\circ]
=F_M^{t_0}[({\bf g}_N^{t_0})^\circ]
\overset{\text{Cor.~\ref{Cor:o=Fg=E}}}{=}
\mathfrak{o}(M,N),
\end{eqnarray*}
we obtain
\[
\dim_{\mathbb{C}}\Hom_{\mathcal K^b(A_{[a,b]}\injc)}(T_{[a,b]}f_M,T_{[a,b]}f_N[1])=\mathfrak{o}(M,N),
\]
where the right-hand side is independent of the choice of $[a,b]$.
By enlarging $[a,b]$ if necessary, we may apply Proposition~\ref{prop:TX} to the pair $(f_M,f_N)$. Thus we obtain
\[\dim_{\mathbb{C}} \Hom_{\mathcal{K}^b(\Lambda(\infty)\injc)}(f_M, f_N[1]) =\mathfrak{o}(M,N).\qedhere\]
\end{proof}
\begin{remark}\label{rem:oKR2}
Let $f,f'\in \mathcal{K}^{[0,1]}(\Lambda(\infty)\injc)$. In Proposition~\ref{prop:TX}, 
   the interval $[a,b]$ is assumed to be
   sufficiently large to get the 
 following isomorphism 
 \[ \Hom_{\mathcal{K}^b(\Lambda(\infty)\injc)}(f,f'[1]) \cong \Hom_{\mathcal{K}^b(A_{[a,b]}\injc)}(T_{[a,b]}f, T_{[a,b]}f'[1]).\]
Now let $f_M$ and $f_N$ be the two-term complexes in $\mathcal{K}^{[0,1]}(\Lambda(\infty)\injc)$ corresponding to two reachable simple modules $M$ and $N$ in $\mathcal C_{\mathbb Z}$, respectively. From the proof of Theorem \ref{Thm:oMN=Ext1},
 we have the equality:  \[\dim_{\mathbb{C}} \Hom_{\mathcal{K}^b(\Lambda(\infty)\injc)}(f_M, f_N[1]) = \dim_{\mathbb{C}} \Hom_{\mathcal{K}^b(A_{[a,b]}\injc)}(T_{[a,b]}f_M,T_{[a,b]}f_N[1])\]
 for any interval $[a,b]$ such that $M,N\in\mathcal C_{[a,b]}$ and both $M$ and $N$ are reachable in $\mathcal{C}_{[a,b]}$. Thus the epimorphism $\alpha_3$ in the proof of Proposition~\ref{prop:TX} induces the following isomorphism
\[ \Hom_{\mathcal{K}^b(\Lambda(\infty)\injc)}(f_M,f_N[1]) \cong \Hom_{\mathcal{K}^b(A_{[a,b]}\injc)}(T_{[a,b]}f_M, T_{[a,b]}f_N[1])\]
 for any such interval $[a,b]$.
\end{remark}
\begin{remark}\label{rem:shifted}
In general, not every two-term rigid complex in $\mathcal{K}^{[0,1]}(\Lambda(\infty)\injc)$ arises in the manner described in Lemma~\ref{Lem:KM}.
    We will return to such general cases in forthcoming work with Bernard Leclerc, approaching them from the viewpoint of real simple representations of shifted quantum affine algebras (including infinite-dimensional representations).

The right-hand side of \eqref{eqn:o-E-infty} can be viewed as the partial $F$-invariant of a certain infinite-rank cluster algebra in the following sense.
To the infinite quiver $\Gamma$, we can associate an infinite-rank cluster algebra $\mathcal{A}_\Gamma$ categorified by a certain category $\mathcal{C}^{\mathrm{shift}}_{\mathbb{Z}}$ of finite-dimensional representations of shifted quantum affine algebras \cite{Her23,HL16b}.
One can show that $\mathbf{g}_M^\infty$ for any reachable simple object $M$ in $\mathcal{C}_\mathbb{Z}$ appears as the $g$-vector of a cluster monomial in the infinite-rank cluster algebra $\mathcal{A}_\Gamma$, whose initial seed is given by the standard seed associated with $\Gamma$. By developing the theory of partial $F$-invariants for infinite-rank cluster algebras, we can interpret the right-hand side of \eqref{eqn:o-E-infty} as the partial $F$-invariant between cluster monomials in $\mathcal{A}_\Gamma$.
To the best of the authors' knowledge, the theories of $\Lambda$-invariants and (normalized) $R$-matrices for such categories have not yet been fully developed (see \cite{HZ24} for a study in the context of shifted Yangians).
\end{remark}

 \subsection{Explicit formulas for pole orders for KR modules}\label{subsec:PKR}
 We retain the notation from the previous subsections.
In this subsection, we apply the above results to verify a conjecture from \cite{FM}, which gives an explicit expression of the pole orders of the normalized $R$-matrices between Kirillov--Reshetikhin modules in terms of the $q$-deformed Cartan matrix of the finite-dimensional Lie algebra $\mathfrak{g}_0$ of type $\mathrm{X}_N$. 

  The $q$-deformed Cartan matrix  $C(q) = (C_{ij}(q))_{i,j \in \mathtt{I}_0}$ of $\mathfrak{g}_0$ is defined by
\[ C_{ij}(q) \coloneqq \begin{cases}
  q^{d_i}+q^{-d_i}, &i=j;\\
  [c_{ij}]_q,&i\neq j,
\end{cases}  \quad \quad\text{where}\quad [k]_q\coloneqq\frac{q^k-q^{-k}}{q-q^{-1}}\;\;\text{for } k\in\mathbb Z.\]
We regard $C(q) \in GL(\mathbb{Q}(q))$ and take its inverse $\widetilde{C}(q) \coloneqq C(q)^{-1} \in GL(\mathbb{Q}(q))$.
Let 
\[ \widetilde{C}_{ij}(q) = \sum_{u \in \mathbb{Z}} \widetilde{c}_{ij}(u) q^u \quad \in \mathbb{Q}(\!(q)\!) \]
denote the Laurent expansion at $q=0$ of the $(i,j)$-entry of $\widetilde{C}(q)$.
It turns out that the coefficients $\widetilde{c}_{ij}(u)$ are all integers and satisfy the following properties (\confer \cite[\S4]{Fujita-Oh-2021}):
\begin{itemize}
\item $\widetilde{c}_{ij}(u) = 0$ if $u < d_i$;
\item $\widetilde{c}_{ij}(u) = \widetilde{c}_{ij^*}(r^\vee h^\vee - u)\ge 0$ if $0 \le u \le r^\vee h^\vee$;
\item $\widetilde{c}_{ij}(u+r^\vee h^\vee) = -\widetilde{c}_{ij^*}(u)$ if $ u \ge 0$.
\end{itemize} 
Here, $i \mapsto i^*$ is the involution of $\mathtt{I}_0$ satisfying $w_0 \alpha_i = -\alpha_{i^*}$, where $w_0$ is the longest Weyl group element and $\alpha_i$ is the $i$-th simple root.
Explicit values of these coefficients $\widetilde{c}_{ij}(u)$ ($i,j \in \mathtt{I}_0, 0 \le u \le r^\vee h^\vee$) have been computed in the literature. See \cite[Appendix A]{GTL}, \cite[Appendix C]{FR-W}, \cite[\S4.3]{Fujita-Oh-2021} for instance.

In what follows, for a formal Laurent series $a(q) = \sum_{u \in \mathbb{Z}} a_u q^u\in\mathbb Q(\!(q)\!)$, we write $[a(q)]_u$ for the coefficient $a_u$ of $q^u$ in $a(q)$.
With this notation, we have $\widetilde{c}_{ij}(u) = [\widetilde{C}_{ij}(q)]_u$ by definition.

 Recall that,  for each $(i,p) \in \Gamma_0=\mathtt K=\{(i,p)\in \mathtt I_0\times \mathbb Z\mid p+d_i\in \epsilon(i)+2\mathbb{Z}\}$ and $k \in \mathbb{Z}_{>0}$, we have the KR module  $W^{(i)}_{k,q^{p+d_i}} \in \mathcal{C}_{\mathbb{Z}}$ defined in \eqref{eq:KR}.
 Note that any KR module in the category $\mathcal{C}_{\mathbb{Z}}$ is of this form.
 
 One of the main results of this subsection is the following theorem, which will be proved at the end of the subsection.
 
\begin{theorem}[{cf.\ \cite[Conjecture 5.17]{FM}, a refinement of \cite[Conjecture 6.7]{Fujita-Oh-2021}}] \label{thm:oKR}
For any $(i,p), (j,s) \in \Gamma_0$ and $k,l \in \mathbb{Z}_{>0}$, we have
\begin{equation}
\mathfrak{o}(W^{(i)}_{k,q^{p+d_i}}, W^{(j)}_{l,q^{s+d_j}}) = 
\left[ F_{i,k;j,l}(q)\right]_{s-p+ld_j-kd_i}.
\end{equation}
Here, $F_{i,k; j,l}(q) \in \mathbb{Z}[q]$ is a polynomial in $q$ explicitly given by 
\begin{equation} \label{eq:Fikjl}
F_{i,k; j,l}(q) = F_{j,l; i,k}(q) = q^{kd_i}[l]_{q^{d_j}} \sum_{u=0}^{r^\vee h^\vee }\widetilde{c}_{ij}(u)q^u - \delta(\clubsuit) \Delta_{ij}(q),
\quad \text{if $kd_i \ge l d_j$},
\end{equation}
where $\delta(\clubsuit)$ is $1$ or $0$ depending on whether the following condition $(\clubsuit)$ is true or not:
\[ (\clubsuit): \text{$\mathrm{X} \in \{ \mathrm{C,F,G}\}$, $d_i=d_j=1$, and $k=l \not \in r^\vee \mathbb{Z}$,}\]
and $\Delta_{ij}(q) \in \mathbb{Z}[q]$ is a polynomial, depending only on $i$ and $j$ (not on $k$ and $l$), given case by case as follows\footnote{In the proof of \cite[Proposition~A.1]{FM}, there are minor calculation errors and the formula for type $\mathrm{C}_N$ presented in \cite{FM} is incorrect. The formula here is the corrected version.}:
\begin{itemize}
\item If $\mathrm{X}_N = \mathrm{C}_N$, we have
\[ 
\Delta_{ij}(q) = \sum_{a=1}^{i+j-N}q^{2N-i-j+2a} \qquad (1 \le i,j <N),
\]
where we identify the set $\mathtt{I}_0$ with $\{1,2,\ldots,N\}$ so that the Dynkin diagram is depicted as:
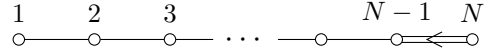
\begin{center}
\begin{tikzpicture}
\foreach \x in {1,2,3}
{\node[dynkdot,label={above:\footnotesize$\x$}] (C\x) at (\x,0) {};}
\node (Cm) at (4, 0) {$\cdots$};
\node[dynkdot] (C4) at (5,0) {};
\node[dynkdot,label={above:\footnotesize$N-1$}] (C5) at (6,0) {};
\node[dynkdot,label={above:\footnotesize$N$}] (C6) at (7,0) {};
\draw[-] (C1) -- (C2);
\draw[-] (C2) -- (C3);
\draw[-] (C3) -- (Cm);
\draw[-] (Cm) -- (C4);
\draw[-] (C4) -- (C5);
\draw[-] (C5.30) -- (C6.150);
\draw[-] (C5.330) -- (C6.210);
\draw[-] (6.6,0+.1) -- (6.4,0) -- (6.6,0-.1);
\end{tikzpicture}
\end{center}
\item If $\mathrm{X}_N = \mathrm{F}_4$, we have 
\[ \Delta_{33}(q) = q^4+q^8+q^{10}+q^{14}, \qquad \Delta_{34}(q)=\Delta_{43}(q)=q^9, \qquad \Delta_{44}(q)=0,\]
where we identify the set $\mathtt{I}_0$ with $\{1,2,3,4\}$ so that the Dynkin diagram is depicted as:
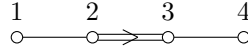
\begin{center}
\begin{tikzpicture}
\foreach \x in {1,2}
{\node[dynkdot, label={above:\footnotesize$\x$}] (F\x) at (\x,0) {};}
\foreach \x in {3,4}
{\node[dynkdot, label={above:\footnotesize$\x$}] (F\x) at (\x,0) {};}
\draw[-] (F1.east) -- (F2.west);
\draw[-] (F3) -- (F4);
\draw[-] (F2.30) -- (F3.150);
\draw[-] (F2.330) -- (F3.210);
\draw[-] (2.6,0) -- (2.4,0+.1);
\draw[-] (2.6,0) -- (2.4,0-.1);
\end{tikzpicture} 
\end{center}
\item If $\mathrm{X}_N = \mathrm{G}_2$, we have
\[ \Delta_{22}(q) = q^6, \]
where we identify the set $\mathtt{I}_0$ with $\{1,2\}$ so that the Dynkin diagram is depicted as:
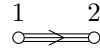
\begin{center}
\begin{tikzpicture}
\node[dynkdot, label={above:\footnotesize$1$}] (G1) at (1,0){};
\node[dynkdot,  label={above:\footnotesize$2$}] (G2) at (2,0) {};
\draw[-] (G1) -- (G2);
\draw[-] (G1.40) -- (G2.140);
\draw[-] (G1.320) -- (G2.220);
\draw[-] (1.4,0+.1) -- (1.6,0) -- (1.4,0-.1);
\end{tikzpicture}
\end{center}
\end{itemize}
\end{theorem}

Now we study the following specific two-term complexes in $\mathcal{K}^b(\Lambda(\infty)\injc)$ which explain several aspects of KR modules:
For each $(i,p) \in \Gamma_0$ and $k \in \mathbb{Z}_{>0}$, let 
\[f^{(i)}_{k,p} \colon I_{i,p+2kd_i} \to I_{i,p}
\] denote the $\Lambda(\infty)$-homomorphism corresponding under \eqref{eq:HomII} to the path
\begin{equation} \label{eq:ep^k}
 (i,p) \to (i, p+2d_i) \to \cdots \to (i,p+2(k-1)d_i)\to (i,p+2kd_i).
\end{equation}
\begin{lemma}[{\confer \cite[\S5]{FM}}] \label{lem:KKR}
For any $(i,p) \in \Gamma_0$ and $k \in \mathbb{Z}_{>0}$, the following statements hold.
\begin{enumerate}
\item[(i)] \label{eq:KKR1} The object $f^{(i)}_{k,p} \in \mathcal{K}^{[0,1]}(\Lambda(\infty)\injc)$ is rigid.
\item[(ii)] \label{eeq:KKR2}
The $0$-th cohomology $H^0(f^{(i)}_{k,p})$ is finite-dimensional, and for each $(j,s) \in \Gamma_0$, we have
\[ \dim_{\mathbb{C}} e_{j,s}H^0(f^{(i)}_{k,p}) = \left[ q^{kd_i} [k]_{q^{d_i}} \sum_{u=d_j}^{r^\vee h^\vee -d_j} \widetilde{c}_{ji}(u)q^u\right]_{s-p}. \]
\item[(iii)] \label{eq:KKR3} The first cohomology $H^1(f^{(i)}_{k,p})$ is equal to $\{0\}$.
\end{enumerate}
\end{lemma}
\begin{proof}
    All the assertions have been essentially proved in \cite[\S5]{FM}. For the reader's convenience, we give an alternative explanation of (i) in our current setting, and give a guide to \cite{FM} for (ii) 
    and (iii).

    Note that the object $f^{(i)}_{k,p}$ is rigid if and only if for any morphism $\psi\in\Hom_{\Lambda(\infty)}(I_{i,p+2kd_i}, I_{i,p})$ there exist morphisms $g\in\Hom_{\Lambda(\infty)} (I_{i, p+2kd_i}, I_{i,p+2kd_i})$ and $h\in\Hom_{\Lambda(\infty)} (I_{i,p}, I_{i,p})$ such that $\psi = f^{(i)}_{k,p}\circ g + h\circ f^{(i)}_{k,p}$. It is sufficient to prove that any morphism $\psi\in \Hom_{\Lambda(\infty)}(I_{i,p+2kd_i}, I_{i,p})$ factors through $f^{(i)}_{k,p}$. By the identification \eqref{eq:HomII} and Proposition~\ref{rem:compare-gpa}, we have  \[\Hom_{\Lambda(\infty)}(I_{i,p+2kd_i}, I_{i,p})\cong e_{i,p+2kd_i} \Lambda(\infty) e_{i,p} \cong (e_i\Pi(\infty)e_i)_{2kd_i}.\] The space $e_i\Pi(\infty)e_i$ is a graded free $\mathbb{C}[\varepsilon_i]$-module of finite rank by left multiplication as seen in the proof of Proposition~\ref{Prop:KS}. Note that we can see this free module has a free basis given by elements of non-positive degree from the defining relations of $\Pi(\infty)$, and there is an isomorphism $(e_i\Pi(\infty)e_i)_{2kd_i} \cong \varepsilon_i^k(e_i\Pi(\infty)e_i)_{0}$. Since $f^{(i)}_{k,p}\in \Hom_{\Lambda(\infty)}(I_{i,p+2kd_i}, I_{i,p})$ corresponds to $\varepsilon_i^k$ via our isomorphisms, this shows that any morphism $\psi\in \Hom_{\Lambda(\infty)}(I_{i,p+2kd_i}, I_{i,p})$ factors through $f^{(i)}_{k,p}$, and we obtain the assertion (i). The assertion (iii) also follows from the fact that $e_i\Pi(\infty)$ is a graded free $\mathbb{C}[\varepsilon_i]$-module of finite rank by left multiplication (see the proof of Proposition~\ref{Prop:KS}). The assertion (ii) was proved in \cite[Lemma~5.3, \S 3]{FM} based on analyzing a $\mathbb{Z}$-graded structure of bimodule resolution of $\Pi(l)$ (see also \cite[Proposition~12.1]{GLS-2017}).
\end{proof}
\begin{remark} \label{rem:KKR}
        The $0$-th cohomology $H^0(f^{(i)}_{k,p})$ is identical to (under the equivalence explained in Proposition~\ref{rem:compare-gpa})  what we denoted $q^{p+kd_i}K^{(i)}_k$ and referred to as \emph{generic kernel} in \cite[\S5]{FM} (see also \cite[Definition~4.5]{HL16}).
   
\end{remark}

Let us consider the space $\Hom_{\mathcal{K}^b(\Lambda(\infty)\injc)}(f^{(i)}_{k,p}, f^{(j)}_{l,s}[1])$ for $(i,p), (j,s) \in \Gamma_0$ and $k,l \in \mathbb{Z}_{> 0}$. 
In \cite[\S5 \& Appendix A]{FM}, the dimensions of these spaces have been computed explicitly in terms of the inverse of the $q$-deformed Cartan matrix:
\begin{proposition}[{\cite[Propositions 5.6 \& A.1]{FM}}]\label{Prop:HomF}
We have the equality
\begin{equation}\label{eq:HomF}
    \dim_{\mathbb{C}} \Hom_{\mathcal{K}^b(\Lambda(\infty)\injc)}(f^{(i)}_{k,p}, f^{(j)}_{l,s}[1]) = [F_{i,k;j,l}(q)]_{s-p+ld_j-kd_i},
\end{equation}
for any $(i,p), (j,s) \in \Gamma_0$ and $k,l \in \mathbb{Z}_{> 0}$, 
where $F_{i,k;j,l}(q)$ is the polynomial given in \eqref{eq:Fikjl}.
\end{proposition}  
\begin{proof}
    This equality has been proved in \cite[\S5 \& Appendix A]{FM} as explained in Remark \ref{rem:HomF}.
    For convenience, we give an alternative explanation here.
    By definition of the homotopy relation, the identification \eqref{eq:HomII}, and the construction in Proposition~\ref{rem:compare-gpa}, there is an isomorphism
    \begin{equation}\label{eq:ePie}
        \Hom_{\mathcal{K}^b(\Lambda(\infty)\injc)}(f^{(i)}_{k,p}, f^{(j)}_{l,s}[1]) \cong \left(\frac{e_i \Pi(\infty) e_j}{e_i\varepsilon_i^k\Pi(\infty)e_j + e_i\Pi(\infty)\varepsilon_j^l e_j}\right)_{p+2kd_i-s}.
    \end{equation}
The following equality
\[\dim_{\mathbb{C}}\left(\frac{e_i \Pi(\infty) e_j}{e_i\varepsilon_i^k\Pi(\infty)e_j + e_i\Pi(\infty)\varepsilon_j^l e_j}\right)_{p+2kd_i-s} = [F_{i,k;j,l}(q)]_{s-p+ld_j-kd_i}\]
is obtained by a case-by-case calculation in \cite[Propositions 5.6 \& A.1]{FM}. Note that the calculation in most cases reduces to understanding Lemma~\ref{lem:KKR} (ii) (see \textit{loc.~cit.}). Thus the desired equality holds.
\end{proof}
\begin{remark}\label{rem:HomF}
     We see that $f^{(i)}_{k,p}$ and $f^{(j)}_{l,s}$ are isomorphic to $H^0(f^{(i)}_{k,p})$ and $H^0(f^{(j)}_{l,s})$ respectively in the bounded derived category $\mathcal{D}^b(\Lambda(\infty)\mathchar`-\mathsf{mod}_\mathsf{fcg})$ thanks to Proposition~\ref{Prop:KS} (i) and Lemma \ref{lem:KKR} (iii). In particular, there is an isomorphism $$\Hom_{\mathcal{K}^b(\Lambda(\infty)\injc)}(f^{(i)}_{k,p}, f^{(j)}_{l,s}[1])\cong \Ext^1_{\Lambda(\infty)}(H^0(f^{(i)}_{k,p}), H^0(f^{(j)}_{l,s})).$$
By \cite[Propositions 5.6 \& A.1]{FM} and Remark \ref{rem:KKR}, we know that the dimension of the latter is equal to the right-hand side of \eqref{eq:HomF}. 

    This type of isomorphism holds for all reachable simple modules in $\mathcal C_{\mathbb Z}$:
Let $M\in\mathcal C_{\mathbb Z}$ be a reachable simple module and let $f_M$ be as in Lemma~\ref{Lem:KM}.
Thanks to the formula for $\mathbf g_M^\infty$, there exists a direct sum $h\coloneqq (I\xrightarrow{1_I}I) \oplus f_M$ with $I\in \Lambda(\infty)\injc$ such that this gives a copresentation with the same domain and codomain as
\[
 h'\coloneqq\bigoplus_{(i,p)\in\Gamma_0}(f^{(i)}_{1,p})^{\oplus u_{i,p+d_i}(M)}.
\]
By Lemma~\ref{lem:KKR}, $h'$ is surjective with finite-dimensional kernel.
Since $h$ is homotopy equivalent to the rigid complex $f_M$, the orbit of $h$ is open dense in the $\Hom$-space as in Remark~\ref{rem:T_X}.
Thus the surjectivity of $h'$ shows that $h$ and $f_M$ are surjective because the function $d_{j,s}(g)\coloneqq\dim_{\mathbb C}e_{j,s}(\opname{coker} g)$ is upper semicontinuous for each $(j,s)\in\Gamma_0$ and constant on orbits. 
Since $f_M$ gives an injective resolution of $H^0(f_M)$, 
Theorem~\ref{Thm:oMN=Ext1} yields
\[
 \mathfrak{o}(M,N)
 =\dim_{\mathbb C}\Ext^1_{\Lambda(\infty)}(H^0(f_M),H^0(f_N))
\]
for any two reachable simple modules $M,N\in\mathcal C_{\mathbb Z}$.
\end{remark}

\begin{theorem}[{\confer \cite[Conjecture 5.17]{FM}}] \label{thm:oKR2}
We have the equality
\begin{equation} \label{eq:oKR2} 
\mathfrak{o}(W^{(i)}_{k,q^{p+d_i}}, W^{(j)}_{l,q^{s+d_j}}) = \dim_{\mathbb{C}} \Hom_{\mathcal{K}^b(\Lambda(\infty)\injc)}(f^{(i)}_{k,p}, f^{(j)}_{l,s}[1]) 
\end{equation} 
for any $(i,p), (j,s) \in \Gamma_0$ and $k,l \in \mathbb{Z}_{>0}$.
\end{theorem}
\begin{proof}[Proof of Theorem \ref{thm:oKR} and 
Theorem \ref{thm:oKR2}]
Thanks to Proposition \ref{Prop:HomF}, we know that the statements in Theorem 
\ref{thm:oKR} and Theorem \ref{thm:oKR2} are equivalent. So we only need to prove Theorem \ref{thm:oKR2}.

By \eqref{eqn:M-ip}, we know that the KR module $W^{(i)}_{k,q^{p+d_i}}$ (resp.\ $W^{(j)}_{l,q^{s+d_j}}$)
is contained in the initial monoidal cluster of $\mathcal C_{[p+d_i,p+(2k-1)d_i]}$ (resp.\ $\mathcal C_{[s+d_j,s+(2l-1)d_j]}$). So $M\coloneq W^{(i)}_{k,q^{p+d_i}}$ and $N\coloneq W^{(j)}_{l,q^{s+d_j}}$ are reachable simple modules in $\mathcal C_{\mathbb Z}$.

By Lemma \ref{Lem:KM}, there exist two-term rigid complexes $f_M, f_N\in \mathcal{K}^{[0,1]}(\Lambda(\infty)\injc)$, unique up to isomorphism, such that ${\bf g}(f_M)={\bf g}^\infty_M$ and ${\bf g}(f_N)={\bf g}^\infty_N$. Since
$M=W^{(i)}_{k,q^{p+d_i}}$ and $N=W^{(j)}_{l,q^{s+d_j}}$, by Definition \ref{def:gvec} (ii), we have
\[{\bf g}^\infty_M={\bf e}_{i,p}-{\bf e}_{i,p+2kd_i}\quad\text{and}\quad {\bf g}^\infty_N={\bf e}_{j,s}-{\bf e}_{j,s+2ld_j},
\]
where ${\bf e}_{i,p}$ is the standard basis vector of $\mathbb Z^{\Gamma_0}$ indexed by $(i,p)\in \Gamma_0$. By Lemma \ref{lem:KKR} (i), we know that both of the two-term complexes $f_{k,p}^{(i)}\colon I_{i,p+2kd_i}\to I_{i,p}$ and $f_{l,s}^{(j)}\colon I_{j,s+2ld_j}\to I_{j,s}$ are rigid. Since ${\bf g}(f_{k,p}^{(i)})={\bf g}(f_M)$ and ${\bf g}(f_{l,s}^{(j)})={\bf g}(f_N)$, we have 
\[
f_M\cong f_{k,p}^{(i)}\quad \text{and}\quad f_N\cong f_{l,s}^{(j)}
\] in $\mathcal{K}^{[0,1]}(\Lambda(\infty)\injc)$ by the uniqueness in Lemma \ref{Lem:KM}. Then by Theorem \ref{Thm:oMN=Ext1}, we have \[\mathfrak{o}(W^{(i)}_{k,q^{p+d_i}}, W^{(j)}_{l,q^{s+d_j}}) = \dim_{\mathbb{C}} \Hom_{\mathcal{K}^b(\Lambda(\infty)\injc)}(f^{(i)}_{k,p}, f^{(j)}_{l,s}[1]).\] This completes the proof of Theorem \ref{thm:oKR} and 
Theorem \ref{thm:oKR2}.
\end{proof}

\subsection{Geometric $q$-character formulas for reachable modules}\label{subsec:qch}
We retain the notation from the previous subsections.
In this subsection, we give an interpretation of the $q$-character of reachable modules in terms of the graded representation theory of $\Pi(\infty)$.

Let $M$ be a reachable simple module in $\mathcal C_{\mathbb Z}$. The following notation will be used in Theorem \ref{thm:qchformula}.
\begin{itemize}
    \item $\chi_q(M)\in \mathbb{Z}[Y_{i,p+d_i}^{\pm 1} \mid (i,p) \in \mathtt K]$: the $q$-character of $M$;
    \item $f_M$: the two-term rigid complex in $\mathcal{K}^{[0,1]}(\Lambda(\infty)\injc)$ with ${\bf g}(f_M)={\bf g}_M^\infty$ (see Lemma \ref{Lem:KM});
    \item $K_M$: the graded module in $\Pi(\infty)\mathchar`-\mathsf{Mod}^\mathbb{Z}$ corresponding to $H^0(f_M)\in \Lambda(\infty)\mathchar`-\mathsf{Mod}$ under the identification in Proposition~\ref{rem:compare-gpa};
    \item $\mathop{\mathsf{Gr}^{\mathbb{Z}}_{\mathbf{d}_\bullet}}(K_M)$: the projective variety of (finite-dimensional) $\mathbb{Z}$-graded submodules $V$ of  $K_M$ such that each $\mathbb{Z}$-graded composition multiplicity $[V\colon S_i\langle p \rangle]$ satisfies
\[[V\colon S_i \langle p \rangle]= \dim_{\mathbb{C}}(e_i V)_p = d_{i,p}\]
for $\mathbf{d}_{\bullet}=(d_{i, p})_{(i,p)\in \mathtt{I}_0\times \mathbb{Z}}\in \bigoplus_{\mathtt{I}_0\times \mathbb{Z}}\mathbb{N}$. Here $S_i\langle p \rangle$ is the $p$-th grading shift of the simple top of the indecomposable projective module $P_i\coloneqq \Pi(\infty)e_i$;
\item $F^{\mathbb{Z}}_{K_M}(\mathbf{y})$: a polynomial in $
\mathbb{Z}[y_{i,p}\mid (i,p)\in \mathtt{I}_0\times \mathbb{Z}]$ 
  defined by
\[F^{\mathbb{Z}}_{K_M}(\mathbf{y})\coloneqq \sum_{\mathbf{d}_{\bullet}\in \bigoplus_{\mathtt I_0\times\mathbb Z}\mathbb N} \chi(\mathop{\mathsf{Gr}^{\mathbb{Z}}_{\mathbf{d}_{\bullet}}}(K_M))\mathbf{y}^{\mathbf{d}_\bullet},\]
where $\chi(-)$ denotes the Euler-Poincar\'e characteristic and $\mathbf{y}^{\mathbf{d}_\bullet}\coloneqq \prod_{(i,p)\in {\mathtt{I}_0\times \mathbb{Z}}}y_{i,p}^{d_{i,p}}$ for $\mathbf{d}_{\bullet}=(d_{i, p})_{(i,p)\in \mathtt{I}_0\times \mathbb{Z}}\in \bigoplus_{\mathtt{I}_0\times \mathbb{Z}}\mathbb{N}$;
\item $A_{j,s}$: the Laurent monomial in $\mathbb Z[Y_{i,p+d_i}^{\pm 1}\mid (i,p)\in \mathtt{K}]$ defined in 
 \S \ref{subsec:staex}.
\end{itemize}
\begin{remark}
Note that  $F^{\mathbb{Z}}_{K_M}(\mathbf{y})\in \mathbb{Z}[y_{i,p}\mid (i,p)\in \mathtt{I}_0\times \mathbb{Z}]$ actually belongs to $\mathbb{Z}[y_{i,p}\mid (i,p)\in \mathtt{K}]$, since $K_M$ is obtained from a $\Lambda(\infty)$-module.   
\end{remark}

For a polynomial $H({\bf y})\in \mathbb{Z}[y_{i,p}\mid (i,p)\in \mathtt{I}_0\times \mathbb{Z}]$, we denote by $H(A_{j,s}^{-1}\mid (j,s)\in \mathtt K)$ the evaluation of $H({\bf y})$ given by
\begin{align*}
    y_{j,s}\mapsto A_{j,s}^{-1}\;\; \text{for}\;\;(j,s)\in \mathtt{K}\quad\text{and}\quad y_{j',s'}\mapsto 0\;\;\text{for}\;\;(j',s')\in (\mathtt{I}_0\times \mathbb{Z})\setminus \mathtt{K}.
\end{align*}

\begin{theorem}\label{thm:qchformula}
Keep the above notation. 
We have the following equality:
\[\chi_q(M)= \left( \prod_{(i,p) \in \mathtt K} Y_{i,p+d_i}^{u_{i,p+d_i}(M)} \right)F_{K_M}^{\mathbb{Z}}(A_{j,s}^{-1}\mid (j,s)\in \mathtt K).\]
\end{theorem}
\begin{proof}
    We choose a finite integer interval $L_0=[a_0,b_0]$ such that $M$ is reachable in $\mathcal C_{L_0}$.
    By Lemma~\ref{lem:clust_ab}, $M$ is reachable in $\mathcal C_L$ for every interval $L\supseteq L_0$.
For $L=[a,b]$, we set
\[
X_L\coloneqq \Hom_{\Lambda(\infty)}(A_L, H^0(f_M))\subseteq H^0(f_M).
\]
Since $\Hom_{\Lambda(\infty)}(A_L, -)$ is left exact, we have an isomorphism $X_L\cong H^0(T_Lf_M)$, and we can see $X_L\subseteq X_{L'}$ when $L\subseteq L'$.

We write $f_M=(I_M^0\to I_M^1)$.
Then we have
\[
I_M^0= \bigcup_{L\supseteq L_0}T_LI_M^0,
\]
because $I_M^0$ is a finite direct sum of objects in $\Lambda(\infty)\injc$.
It follows that
\[
H^0(f_M)=\bigcup_{L\supseteq L_0}X_L.
\]
On the other hand, Lemma~\ref{Lem:KM} and the equalities ${\bf g}(T_{L}f_M)={\bf g}_M^\infty|_{(\Gamma_{L}^\circ)_0}=({\bf g}_M^{L})^\circ$ give
$T_Lf_M\cong f_M^L$ as seen in the proof of Theorem~\ref{Thm:oMN=Ext1}, and hence
$X_L\cong H^0(f_M^L)$.

For any sufficiently large $L=[a,b]$, the truncated polynomial
$F_M^{-,b}({\bf A}^{-1})$ coincides with $F_M({\bf A}^{-1})$.
Therefore, by Lemma~\ref{lem:gFab} and Theorem~\ref{thm:DWZ-g-F}~(iii), the $F$-polynomials of the modules $X_L$ are independent of $L$.
Since $\deg_{y_v}F_{X_L}=\dim_{\mathbb C}e_vX_L$ for each vertex $v$,
the dimension vectors of the modules $X_L$ are also independent of $L$ for any sufficiently large $L$.
Thus the natural inclusions $X_L\subseteq X_{L'}$ are equalities for any sufficiently large $L\subseteq L'$.
This yields $H^0(f_M)=X_L\cong H^0(f_M^L)$ for any sufficiently large $L$. In particular, $H^0(f_M)$ is finite-dimensional.

Under Proposition~\ref{rem:compare-gpa}, submodules of $H^0(f_M)$ correspond to graded submodules of $K_M$. Consequently, we obtain
\[
F_{K_M}^{\mathbb Z}({\bf y})
=F_M({\bf A}^{-1})|_{A_{i,p}^{-1}=y_{i,p}}.
\]
The assertion follows from the factorization of $\chi_q(M)$ in terms of $F_M({\bf A}^{-1})$.
\end{proof}
    \begin{remark}
Let $M$ be a reachable simple module in $\mathcal C_{\mathbb Z}$ and
let $N$ be any simple module in $\mathcal C_{\mathbb Z}$. Then we have another formula for the pole order
\[
  \mathfrak{o}(M,N)
  = \max_{U\subseteq K_M}\sum_{(i,p)\in\mathtt K}(u_{i,p+d_i}(N)-u_{i,p-d_i}(N))\dim_{\mathbb C}(e_iU)_p,
\]
where $U$ runs over all $\mathbb{Z}$-graded $\Pi(\infty)$-submodules of $K_M$.
Indeed, we know
$\mathfrak{o}(M,N)=F_{K_M}^{\mathbb Z}[\mathbf g_N^\infty]=F_{H^0(f_M^L)}[\mathbf g_N^\infty|_{(\Gamma_L^\circ)_0}]$
for sufficiently large $L$.
Thanks to Remark~\ref{rmk-end}~(ii), the Newton polytope of $F_{H^0(f_M^L)}$ is same as the convex hull of dimension vectors of submodules corresponding to these $U$ in $\mathbb{R}\otimes_{\mathbb{Z}}K_0(A_L\modfd)$. These immediately give the above formula.
\end{remark}

As a corollary, we interpret the classical limits of $q$-characters studied in \cite{FR99} for reachable modules in $\mathcal{C}_{\mathbb{Z}}$. Our construction of $\Pi(\infty)$ ensures that each submodule of a finite-dimensional graded $\Pi(\infty)$-module $K_M$ is a nilpotent module.
Indeed, since the homogeneous central element
$\varepsilon$ has degree $2r^\vee>0$, there exists $l\in \mathbb{Z}_{>0}$ such that
$\varepsilon^lK_M=0$. Hence the action of $\Pi(\infty)$ on $K_M$
factors through $\Pi(l)$. Since the arrow ideal
$\mathfrak M_{\widetilde Q}$ is nilpotent in $\Pi(l)$, the module
$K_M$ is nilpotent.
In particular, each submodule of $K_M$ has a composition series of finite length whose composition factors are simple modules $S_i\,(i\in \mathtt{I}_0)$. 

We retain the previous notation. The following are additionally used in Corollary~\ref{cor:qchformula}:
\begin{itemize}
    \item $\mathsf{Gr}_{\mathbf{d}}(K_M)$: the projective variety of (ungraded) submodules $V$ of $K_M$ such that the multiplicity $[V\colon S_i]$ of each composition factor $S_i$ satisfies
    \[[V\colon S_i]=\dim_{\mathbb{C}}(e_i V)= d_i,\]
    where $\mathbf{d}=(d_i)_{i\in \mathtt{I}_0}\in \mathbb{N}^{\mathtt{I}_0}$;
    \item $F_{K_M}(\widetilde{\mathbf{y}})$: a polynomial in $\mathbb{Z}[y_i\mid i\in \mathtt{I}_0]$ defined by
    \[F_{K_M}(\widetilde{\mathbf{y}})=\sum_{\mathbf{d}\in \mathbb{N}^{\mathtt{I}_0}}\chi(\mathsf{Gr}_{\mathbf{d}}(K_M))\widetilde{\mathbf{y}}^{\mathbf{d}}\]
    where we denote $\widetilde{\mathbf{y}}^{\mathbf{d}}\coloneqq \prod_{i\in \mathtt{I}_0}y_i^{d_i}$;
    \item $\widetilde{A}_i\coloneqq e^{\alpha_i}$: a usual formal symbol for each simple root $\alpha_i$;
    \item $\widetilde{\chi}(M)\in \mathbb{Z}[e^{\pm \varpi_i}\mid i\in \mathtt{I}_0]$: the normalized character of $M$ as a finite-dimensional $U_q(\mathfrak{g}_0)$-module, where the normalization is taken with respect to the highest weight term.
\end{itemize}

\begin{corollary}\label{cor:qchformula}
    We keep the above notation. We have the following equality:
    \[\widetilde{\chi}(M)= F_{K_M}(\widetilde{A}_i^{-1} \mid i\in \mathtt{I}_0)\]
    after evaluations $\widetilde{A}_i=e^{\alpha_i}\mapsto \prod_{j\in \mathtt{I}_0}(e^{\varpi_j})^{c_{ji}}$ for $i\in\mathtt{I}_0$.
\end{corollary}
\begin{proof}
    Let $\mathsf{Gr}^{\mathbb{Z}}_{\mathbf{d}}(K_M)$ be the projective variety of $\mathbb{Z}$-graded submodules of $K_M$ with dimension vector $\mathbf{d}=(d_i)_{i\in \mathtt{I}_0}\in \mathbb{N}^{\mathtt{I}_0}$.
    We have a $\mathbb{C}^{\times}$-action on $K_M$ defined by
    \[t. m = t^{\deg(m)}m\]
    for each homogeneous element $m \in K_M$, and it induces a $\mathbb{C}^{\times}$-action on $\mathsf{Gr}_{\mathbf{d}}(K_M)$.
    The projective variety $\mathsf{Gr}^{\mathbb{Z}}_{\mathbf{d}}(K_M)$ is the fixed point set of $\mathsf{Gr}_{\mathbf{d}}(K_M)$ with respect to this $\mathbb{C}^{\times}$-action. In particular, we obtain an equality
$\chi(\mathsf{Gr}_{\mathbf{d}}(K_M))=\chi(\mathsf{Gr}^{\mathbb{Z}}_{\mathbf{d}}(K_M))$
    thanks to Bia{\l}ynicki-Birula's theorem~\cite[Corollary~2]{BB73}.
    We have a decomposition
    \begin{equation}\label{eq:Grdecomp}
        \mathsf{Gr}^{\mathbb{Z}}_{\mathbf{d}}(K_M)= \bigsqcup_{\substack{\sum_{p\in \mathbb{Z}}d_{i,p}=d_i\\ (i\in \mathtt{I}_0)}}\mathsf{Gr}^{\mathbb{Z}}_{\mathbf{d}_{\bullet}}(K_M)
    \end{equation}
    into a disjoint union of finitely many subvarieties $\mathsf{Gr}^{\mathbb{Z}}_{\mathbf{d}_{\bullet}}(K_M)$ studied in Theorem~\ref{thm:qchformula}.  Since the Euler-Poincar\'e characteristic $\chi(\mathsf{Gr}^{\mathbb{Z}}_{\mathbf{d}}(K_M))$ is equal to the sum of all the $\chi(\mathsf{Gr}^{\mathbb{Z}}_{\mathbf{d}_{\bullet}}(K_M))$ in \eqref{eq:Grdecomp}, Theorem~\ref{thm:qchformula} and \cite[Theorem~3]{FR99} show that the construction of $F_{K_M}$ is compatible with our normalized specialization $\chi_q(M)\mapsto \widetilde{\chi}(M)$. Thus we obtain the assertion.
    
    Note that $M$ is not necessarily irreducible as a $U_q(\mathfrak{g}_0)$-module. But the ordinary character $\chi(M)$ is easily recovered as the product of a unique monomial and $\widetilde{\chi}(M)$ such that the product is invariant under the action of the Weyl group.
\end{proof}
\begin{remark}
    \begin{enumerate}
        \item[(i)] When $\mathfrak{g}_0$ is of simply-laced type and $f_M=f_{1, p}^{(i)}$, then the projective variety $\mathsf{Gr}_{\mathbf{d}}(K_M)$ (resp.\ $\mathsf{Gr}^{\mathbb{Z}}_{\mathbf{d}}(K_M)$) in Corollary~\ref{cor:qchformula} is identical to the (resp.\ $\mathbb{Z}$-graded) submodule Grassmannian of an (resp.\ $\mathbb{Z}$-graded) injective module over the classical preprojective algebra $\Pi(1)$. It is known that our (resp.\ $\mathbb{Z}$-graded) submodule Grassmannian is homeomorphic to (resp.\ the fixed point set of a $\mathbb{C}^\times$-action in) a certain Lagrangian subvariety of Nakajima quiver variety (\confer\cite{Lus98,ST11,Shi10}).
        \item[(ii)] For $\mathfrak{g}$ of arbitrary untwisted type, the realization of (truncated) $q$-characters of KR modules similar to Theorem~\ref{thm:qchformula} was studied in \cite[Theorem~4.8]{HL16}. In particular, the submodule Grassmannian which they studied is identical to ours via the identification given in Proposition~\ref{rem:compare-gpa} when their formula in \textit{loc.~cit.} gives the complete $q$-character (see \cite[Remark~4.9]{HL16}).
    \end{enumerate}
\end{remark}
\begin{remark} \label{rem:twisted-case}
It is straightforward to extend the discussions in \S\ref{sec:pole-order} and \S\ref{sec: application} 
to the category $\mathscr{C}_{\mathfrak{g}}^{[a',b'],\mathfrak{s}}$ for any admissible sequence $\mathfrak{s}$ and any finite interval $[a',b']$ in the sense of \cite[\S6]{kkop-2024}, including twisted types, with obvious modifications.
(Strictly speaking, for twisted type, we have to replace ``$q$-characters'' with ``twisted $q$-characters'' in the sense of \cite{Her10}.)
We restricted our exposition to the case of the subcategories $\mathcal{C}_{[a,b]}$ (of untwisted type) only for the sake of simplicity.

It can also be seen that the case of twisted type is parallel to the case of untwisted type, through the folding isomorphism introduced by Hernandez \cite{Her10}.
Let $\mathfrak{g}$ be an untwisted affine Lie algebra of simply-laced type $\mathrm{X}_N^{(1)}$ ($\mathrm{X} \in \{ \mathrm{A, D, E}\}$), and $\mathfrak{g}^{(r)}$ the twisted affine Lie algebra of type $\mathrm{X}_N^{(r)}$ ($r \in \{2,3\}$).
Then, for any admissible sequence $\mathfrak{s}$ and an integer interval $[a',b']$, we have the associated monoidal categories $\mathcal{C}^{[a',b'], \mathfrak{s}}_{\mathfrak{g}}$ and $\mathcal{C}^{[a',b'], \mathfrak{s}}_{\mathfrak{g}^{(r)}}$ for $\mathfrak{g}$ and $\mathfrak{g}^{(r)}$ respectively. 
By \cite{kkop-2024}, both categories are $\mathsf{\Lambda}$-monoidal categorifications of a common cluster algebra, whose initial monoidal clusters consist of certain KR modules in both cases.  
The folding isomorphism restricts to an explicit ring isomorphism $K_0(\mathcal{C}^{[a',b'], \mathfrak{s}}_{\mathfrak{g}}) \cong K_0(\mathcal{C}^{[a',b'], \mathfrak{s}}_{\mathfrak{g}^{(r)}})$, 
under which KR modules in $\mathcal{C}^{[a',b'], \mathfrak{s}}_{\mathfrak{g}}$ are in bijection with KR modules in $\mathcal{C}^{[a',b'], \mathfrak{s}}_{\mathfrak{g}^{(r)}}$.
Since the initial clusters match under the folding isomorphism, it induces a bijection of reachable modules\footnote{Very recently, this bijection was extended to all simple modules in \cite{NW26}.}.
Moreover, the folding isomorphism is compatible with the (twisted) $q$-character homomorphisms by construction.
Therefore, for our purposes, there is no essential difference between the untwisted cases and the twisted cases.  
\end{remark}

\appendix
\section{Decorated representations and numerical invariants}\label{sec: appendixA}

Fix any base field $\Bbbk $. Let $Q$ be a finite quiver and let $\widehat{ \Bbbk  Q}$ be its complete path algebra. Let $\fM \coloneqq \mathfrak{M}_Q$ denote the Jacobson radical of $\widehat{ \Bbbk  Q}$.
We refer to an element $r\in \widehat{ \Bbbk  Q}$ as \emph{basic} if $r$ is a formal linear sum of paths in $Q$ with a common source $i$ and a common target $j$. Note that any ideal of $\widehat{ \Bbbk  Q}$ is generated by basic elements.

We endow the algebra $\widehat{ \Bbbk  Q}$ with the $\mathfrak{M}$-adic topology given by a fundamental system $\{\mathfrak{M}^p\}_{p\geq 0}$ of open neighborhoods of $0$.
We choose an ideal $I$ of $\widehat{ \Bbbk  Q}$ generated by finitely many basic elements in $\mathfrak{M}^2$.  Let $\overline I= \bigcap_{p\geq 0}(I+\mathfrak{M}^p)$ be the $\mathfrak{M}$-adic closure of $I$. Throughout the appendix, we consider the following complete and separated algebra \[J \coloneqq \widehat{\Bbbk  Q}/\overline I \cong \varprojlim_{p\geq 0} \widehat{\Bbbk  Q}/ (I + \fM^p)\] with respect to the $\fM$-adic topology. Here we also use $\fM$ to denote the Jacobson radical of $J$, as it is given by the natural projection.

Since each $J_p \coloneqq \widehat{\Bbbk  Q}/ (I + \fM^p)$ is a finite-dimensional algebra, the $\fM$-adic topology of $J$ coincides with a complete and separated topology given by the fundamental system $\{U_\lambda\}_{\lambda}$ of open neighborhoods of $0$ consisting of ideals $U_\lambda$ such that $J/U_\lambda$ is finite-dimensional. Such a topological algebra is called a {\em pseudo-compact algebra}~\cite{Gab62}. As pointed out in \cite[\S 7]{Keller-Yang_2011}, the Jacobian algebra of a quiver with potential is an example of a pseudo-compact algebra.

In general, the algebra $J$ considered above is neither noetherian nor coherent, and the category of finitely generated (resp.\ finitely presented) modules is not abelian but just an exact category. However, it is well-known that $J$ is a \emph{semi-perfect} algebra (\confer \cite[Lemma~4.3]{KM22}), \ie it satisfies the following two conditions (see \cite[\S 27]{AF92}): 
\begin{itemize}
\item the ring $J/\fM$ is semisimple; \item any idempotent of $J/\fM$ is the image of an idempotent of $J$ under the projection. 
\end{itemize}
In particular, the category $J\projc$ of finitely generated projective $J$-modules is a Krull--Schmidt category (see \cite[Proposition~4.1]{Kra15}).

\subsection{Injective objects in the category of locally nilpotent modules}\label{subsecA:inj}
Keep the previous setting. We prepare several basic facts about modules over the algebra $J$.

\begin{lemma}\label{lem:fd_fp}
    Every finite-dimensional $J$-module is finitely presented.
\end{lemma}
\begin{proof}
    Since $J$ is a semi-perfect algebra, each finite-dimensional $J$-module $M$ is of finite length whose composition factors are given by the simple modules $S_i$ with
 $i\in Q_0$. We know from \cite[Proposition~3.3]{BIRSm11} that each $S_i$ is finitely presented. If we have a short exact sequence
    $0\to X\to Y\to Z \to 0$ with $X, Z$ finitely presented, then $Y$ is also finitely presented by the horseshoe lemma. Since $M$ is contained in the extension closure of $\{S_i\mid i\in Q_0\}$, we obtain the assertion.
\end{proof}
\begin{definition}
    We consider the following subcategories of $J\Mod$.
    \begin{enumerate}
    \item[(i)]  Let $J\modfd$ denote the full subcategory of finite-dimensional $J$-modules.
    \item [(ii)]  Let $J\modfg$ denote the full subcategory of finitely generated $J$-modules.
        \item[(iii)] Let $J\modfp$ denote the full subcategory of finitely presented $J$-modules.
        \item[(iv)] Let $J\projc$ denote the full subcategory of finitely generated projective $J$-modules.
    \end{enumerate}
\end{definition}
\begin{lemma}[{\cite{KM22,Keller-Yang_2011}}] \label{lem:A-Krull-Sch}
    Keep the above notation. We have the following.
    \begin{enumerate}
        \item[(i)] The category $\cK^b(J\projc)$ is Krull--Schmidt.
        \item[(ii)] The subcategory $\cK^{[-1, 0]}(J\projc)$ of two-term complexes in  $\cK^b(J\projc)$ is Krull--Schmidt.
        \item[(iii)] The exact category $J\modfp$ of finitely presented $J$-modules is Krull--Schmidt.
    \end{enumerate}
\end{lemma}
\begin{proof}
        All the assertions are special cases in \cite[Corollary~4.6~(ii), Lemma~4.10]{KM22}. We remark that the first assertion (i) also follows from \cite[Lemma~2.17]{Keller-Yang_2011}.
     \end{proof}
\begin{remark}
    By Schanuel's lemma, each finitely presented module over a semi-perfect ring has a minimal projective presentation.
\end{remark}
We will rephrase the above facts in terms of (locally nilpotent) finitely copresented modules.
   
\begin{definition} \label{def:l-nil}
Let $M$ be a $J$-module.
    \begin{enumerate}
        \item[(i)] A $J$-module $M$ is \emph{nilpotent} if  there exists some $m\in \mathbb{Z}_{>0}$ such that $\fM^m M=\{0\}$.
        \item[(ii)] A $J$-module $M$ is \emph{locally nilpotent} if  for each $x\in M$ there exists $m\in \mathbb{Z}_{> 0}$ such that $\fM^m\cdot x = 0$.
    \end{enumerate}
\end{definition}
\begin{remark}
    Every nilpotent $J$-module is locally nilpotent, and so every finite-dimensional $J$-module is locally nilpotent (\confer \cite[\S 10]{DWZ08}).
\end{remark}
\begin{remark}\label{rem:Gro}
    One can easily check that the category $J\Modln$ of locally nilpotent $J$-modules is an abelian category. It is also directly confirmed that $J\Modln$ is a Grothendieck abelian category (\ie it is closed under arbitrary coproducts, has filtered  colimits with the exactness property, and has a system of generators $\{J/\fM^m\}_{m\in \mathbb{Z}_{>0}}$). In particular, each object in $J\Modln$ has an injective envelope \cite{Gab62}. Thus we can see each finitely copresented locally nilpotent $J$-module has a minimal injective copresentation in $J\Modln$.
    The relative injective envelope $E_{\mathsf{ln}}(M)$ of an object $M\in J\Modln$ does not necessarily coincide with the injective envelope $E(M)$ of $M$ in the category of all the $J$-modules unless we assume a noetherian property on $J$. But there is an inclusion $E_{\mathsf{ln}}(M) \subseteq E(M)$.
\end{remark}
 Let $J_l\coloneqq J/\fM^l$, which is a finite-dimensional algebra  for each $l\in \mathbb{Z}_{\geq 1}$. The following characterization of locally nilpotent modules is immediate from the definition:
\begin{lemma}\label{lem:locnilp}
    Let $M$ be a $J$-module. The following are equivalent:
    \begin{enumerate}
        \item[(i)] A $J$-module $M$ is locally nilpotent.
        \item[(ii)]  There are inclusions $M_m \to M_{m+1}$ between submodules $M_m\,(m\in \mathbb{Z}_{>0})$ of $M$ such that each $M_m$ is a $J_m$-module and there is an isomorphism $M\cong \varinjlim_{m> 0} M_m$.
        \item[(iii)] A $J$-module $M$ is isomorphic to the direct limit of a direct system consisting of finite-dimensional $J$-modules.
    \end{enumerate}
\end{lemma}
\begin{proof}
    The $\fM^m$-torsion submodule $M_m\coloneqq \{x\in M\mid \fM^m\cdot x= 0\}$ of $M$ has a $J_m$-module structure and it is isomorphic to $\Hom_J(J_m, M)$. A $J$-module $M$ is locally nilpotent if and only if $M=\bigcup_{m>0}M_m$ by definition. This yields the equivalence (i)$\Leftrightarrow$(ii).

    We prove (i)$\Leftrightarrow$(iii). Let $M$ be a locally nilpotent $J$-module. For each $x\in M$, there exists $m$ such that $\fM^m\subseteq \mathop{\rm{ann}}(x)$
    by definition. Then we have $Jx \cong J/\mathop{\rm{ann}}(x)$ and there is a surjection $J/\fM^m \to J/\mathop{\rm{ann}}(x)$. Thus $Jx$ is finite-dimensional. If we take $\Fcal$ as the filtered set of finite-dimensional submodules of $M$ with respect to natural inclusions, then the previous argument shows $\bigcup_{V\in \Fcal}V = M$. Conversely, let $M \coloneqq \varinjlim_{\mu} V_{\mu}$, where $\{V_\mu\}$ forms an inductive system consisting of finite-dimensional $J$-modules. We have the canonical maps $\iota_\mu\colon V_\mu \to \varinjlim_{\mu} V_{\mu}$ and $M= \bigcup_\mu \iota_\mu(V_\mu)$. Since each $\iota_\mu(V_\mu)$ is finite-dimensional and thus nilpotent, the module $M$ is locally nilpotent.
\end{proof}
\begin{definition}\label{def:fcp} We consider the following subcategories of $J\Modln$.

    \begin{enumerate}
        \item [(i)] Let $J\modfcg$ be the full subcategory of finitely cogenerated locally nilpotent $J$-modules.
        \item[(ii)] Let $J\modfcp$ be the full subcategory of finitely copresented locally nilpotent $J$-modules.
        \item[(iii)] Let $J\injc$ be the full subcategory of finitely cogenerated injective objects in $J\Modln$.
    \end{enumerate}
\end{definition}

Note that for each $M\in J\modfg$ (resp.\ $N\in J\modfcg$), we have a finitely generated module $J_l\otimes_J M$ (resp.\ a finitely cogenerated module $\Hom_J(J_l, N)$) over a finite-dimensional algebra $J_l$ for each $l\in \mathbb{Z}_{>0}$.
In particular, we can introduce the following continuous version of $\Bbbk $-duality:

\begin{definition}\label{def:Cdual}
    \begin{enumerate}
        \item[(i)] The contravariant functor $D\colon J\modfg\to J^\op\modfcg$ is defined by
        \[D(-)\coloneqq \varinjlim_{l>0}\Hom_{\Bbbk }((J_l\otimes_J -), \Bbbk ),\] where the direct limit is taken with respect to the direct system of natural inclusions.
        \item[(ii)] The contravariant functor $D'\colon J\modfcg \to J^\op\modfg$ is defined by \[D'(-)\coloneqq \varprojlim_{m>0}\Hom_{\Bbbk }(\Hom_{J}(J_m, -), \Bbbk ),\] where the inverse limit is taken with respect to the inverse system of natural projections.
    \end{enumerate}
     
\end{definition}
\begin{remark}
A module $D(M)$ for $M\in J\modfg$ is just a submodule of $\Hom_{\Bbbk }(M, \Bbbk )$, and $D(M)$ is not isomorphic to $\Hom_{\Bbbk }(M, \Bbbk )$ in general.
\end{remark}
\begin{remark}\label{rem:dual}
    One can directly extend the functor defined in Definition~\ref{def:Cdual}~(ii) to a functor $D'(-)\colon J\Modln \to J^\op\Mod$. Since
    the inverse limit in $\varprojlim_{m>0}\Hom_{\Bbbk }(\Hom_{J}(J_m, -), \Bbbk )$ always commutes with $\Hom$,  we have the following natural isomorphisms
\[D'(-)\cong \Hom_{\Bbbk }(\varinjlim_{m>0}\Hom_{J}(J_m, -), \Bbbk )\cong \Hom_{\Bbbk }(-, \Bbbk ).\]
In particular, the functor $D'(-)\colon J\Modln \to J^\op\Mod$ is faithful and exact. However, this extended functor is not full (\confer Remark~\ref{rem:contdual}).
\end{remark}

In the proof of Theorem~\ref{thm:Dequiv1} below, we use the following conventions:
For a finitely generated $J$-module $M$ (resp.\ finitely cogenerated $J$-module $N$) and $p\in\mathbb N$, we set
\[
M_{(p)}\coloneqq J_p\otimes_J M\in J_p \mathchar`-\modfg,\quad N^{(p)}\coloneqq \Hom_J(J_p,N)\in J_p \mathchar`-\modfcg. 
\]
 For similar $J^\op$-modules $M$, $N$, we also apply the analogous notations $M_{(p)}\coloneqq J_p^\op\otimes_{J^\op} M$ and $N^{(p)}\coloneqq \Hom_{J^\op}(J_p^\op,N)$.
\begin{theorem}\label{thm:Dequiv1}
   The functor $D$ gives a contravariant additive equivalence $D\colon J\projc \cong J^\op\injc$.
\end{theorem}
\begin{proof}
We set $(-)^\vee \coloneqq \Hom_{\Bbbk }(-, \Bbbk )$. For any $P \in J\projc$ (resp.\ $I \in J^\op \injc$), one can directly confirm that $P_{(l)}$ (resp.\ $I^{(l)}$) is a finitely generated projective (resp.\ finitely cogenerated injective) module over the finite-dimensional algebra $J_l$ (resp.\ $J_l^\op$) for each $l \in \mathbb{Z}_{>0}$ by universal properties.

For a fixed $P\in J\projc$, we show that $D(P)\in J^\op\injc$. We know that $(P_{(l)})^\vee \in J^\op_l\injc$ and $D(P)=\varinjlim_{l>0} (P_{(l)})^\vee\in J^\op\modfcg$. Since the category $J^\op\Modln$ is a Grothendieck abelian category with a system of generators  $ \{J_l^\op\}_{l>0}$ (see Remark~\ref{rem:Gro}), we will use Baer's criterion for a Grothendieck abelian category (see \cite[\S 1.10,~Lemme~1]{Tohoku}, \cite[Prop. 8.4.7]{KS06}) to show $D(P)$ is an injective object in $J^\op\Modln$. For
any submodule $X$ of $J_m^{\op}\in J^\op\Modln$ with $m>0$ and any morphism $f\colon X\to D(P)$, we need to prove that $f$ factors through the inclusion $i_X\colon X\hookrightarrow J_m^{\op}$. Since $D(P)=\varinjlim_{l>0} (P_{(l)})^\vee$,  there exists $k\geq m$ such that $f\colon X\to D(P)$ factors through the inclusion $\iota_{k}\colon (P_{(k)})^\vee \hookrightarrow D(P)$.  Then we have the following commutative diagram:
\[\begin{tikzcd}
	{J_m^\op} && {(P_{(k)})^\vee} && {D(P)} \\
	\\
	X
	\arrow[dotted,"{g_k}", from=1-1, to=1-3]
	\arrow["{\iota_k}", hook, from=1-3, to=1-5]
	\arrow["{i_X}", hook, from=3-1, to=1-1]
	\arrow["{f_k}", from=3-1, to=1-3]
	\arrow["f"', from=3-1, to=1-5]
\end{tikzcd}\]
Since the three $J^\op$-modules $X, J_m^{\op}$ and $(P_{(k)})^\vee$ can be viewed as $J^\op_k$-modules and $(P_{(k)})^\vee$ is an injective  $J^\op_k$-module, there exists a homomorphism $g_k\colon J_m^{\op}\to (P_{(k)})^\vee$ such that $f_k=g_ki_X$. Hence, we have
$f=\iota_kf_k=\iota_kg_ki_X$. This proves that $f\colon X \to D(P)$ factors through the inclusion $i_X\colon X\hookrightarrow J_m^{\op}$. Then by  Baer's criterion for the Grothendieck abelian category $J^\op\Modln$, we know that $D(P)\in J^\op\modfcg$ is an injective object in $J^\op\Modln$. Hence, $D(P)\in J^\op\injc$.

For a fixed $I\in J^\op\injc$, we show that $D'(I)\in J\projc$. We claim that
the socle $\operatorname{soc}I$ of  $I$ is an essential submodule. Indeed, if we assume $\{0\}\neq X\subset I$ and $0\neq x\in X$ and choose $r$
minimal such that $x\mathfrak M^r=0$, then we have
$\{0\}\neq x\mathfrak M^{r-1}\subseteq
X\cap\operatorname{soc}I$.
Thus $\operatorname{soc}I$ is essential in $I$.
Set $E_i\coloneqq D(Je_i)$. By construction, we have an isomorphism
\[
\operatorname{soc}E_i \cong (J_1e_i)^\vee(\eqqcolon S_i^{\op})
\]
in $J^{\op}\Modln$.
Since we have proved that $D(P)\in J^\op\injc$ for each $P\in J\projc$ and its socle is essential, we know that $E_i$ is the injective envelope of $S_i^{\op}$ in $J^{\op}\Modln$.
Since $I$ is
finitely cogenerated, its socle has finite length and we write
\[
\operatorname{soc}I\cong
\bigoplus_{i\in Q_0}(S_i^{\op})^{a_i}.
\]
Since both $I$ and $\bigoplus_{i\in Q_0}E_i^{a_i}$ are injective envelopes
of $\operatorname{soc}I$ in $J^{\op}\Modln$, we have
$I\cong\bigoplus_{i\in Q_0}E_i^{a_i}$.
By construction, we obtain isomorphisms
$I^{(l)}\cong
\bigoplus_{i\in Q_0}\left((J_le_i)^\vee\right)^{a_i}$.
Thus, by using $J\cong\varprojlim_{l>0} J_l$ and the fact that inverse
limits commute with finite direct sums, we obtain
\[
D'(I)\cong
\varprojlim_{l>0} (I^{(l)})^\vee
\cong \bigoplus_{i\in Q_0}(Je_i)^{a_i}\in J\projc.
\]

One can directly confirm the following isomorphisms\footnote{The isomorphisms here  can be  regarded as a special case of Proposition~\ref{prop:DD}.}
\[D'(D(P)) \cong P\quad \text{and}\quad 
D(D'(I)) \cong  I
\]
(\confer Remark~\ref{rem:dual}). The key reason is that we have isomorphisms $P \cong \varprojlim_{l>0} P_{(l)}$, $I \cong \varinjlim_{l>0} I^{(l)}$, and a duality $(-)^\vee\colon J_l\projc \cong J_l^\op\injc$ for each finite-dimensional algebra $J_l$. 
Thus $D$ and $D'$ give contravariant additive equivalences between $J\projc$ and $J^\op\injc$. 
\end{proof}
\begin{theorem}\label{thm:Dequiv2}
    The functor $D$ gives a contravariant exact equivalence $D\colon J\modfp\cong J^\op\modfcp$. The functor $D'$ gives a quasi-inverse of $D$.
\end{theorem}
\begin{proof}
    First we show that $D$ gives a categorical equivalence $J\modfp \cong J^\op\modfcp$ whose quasi-inverse is given by $D'$. We take any $X\in J\modfp$ and any $Y\in J^\op\modfcp$. Since the functor $D$ is left exact by definition, we have $D(X)\in J^\op\modfcp$ thanks to Theorem~\ref{thm:Dequiv1}. Similarly, we have $D'(Y)\in J\modfp$ because $D'$ is exact.
     
We can obtain the following isomorphisms\footnote{The isomorphisms here  also follow from Proposition~\ref{prop:DD}.} 
\[D'(D (X))\cong X\quad\text{and}\quad  D(D'(Y))\cong Y\]
by taking (co)presentations thanks to Theorem~\ref{thm:Dequiv1} and the five lemma.  Thus we have a contravariant equivalence $J\modfp \cong J^\op\modfcp$, and $D$ and $D'$ are mutually quasi-inverses.

We know that the functor $D'$ is exact and $D$ is left exact. Now let us prove that $D$ is also right exact. Let $f: M\to N$ be an injective morphism in $J\modfp$. Applying the exact functor $D'$ in Remark~\ref{rem:dual}, we have $D'(\mathrm{Cok}(D(f))) \cong \mathrm{Ker}(D'(D(f))) \cong  \mathop{\rm{Ker}}(f) = \{ 0 \}$, since we know $D' \circ D \cong \mathop{\rm{id}}_{J\modfp}$. It implies that  $\mathop{\rm{Cok}}(D(f))=\{0\}$ as $D'$ is faithful. This shows $D$ is exact.
\end{proof} 
Since there is an $A$-duality $\Hom_{A}(-, A)\colon A\projc \xrightarrow{\sim} A^\op\projc$ for any unital ring $A$, we have an additive equivalence $\nu\coloneqq D\circ \Hom_{J}(-, J)\colon J\projc \to J\injc$ and a triangle equivalence $\nu[-1]\colon \cK^b(J\projc)\xrightarrow{\sim} \cK^b(J\injc)$. Here we also denote the derived functor by the same symbol $\nu$. Then the following result follows from Lemma~\ref{lem:A-Krull-Sch}.
\begin{lemma}\label{lem:injKS}
    Keep the above notation. We have the following:
    \begin{enumerate}
        \item[(i)] The categories $J\injc$ and $\cK^b(J\injc)$ are Krull--Schmidt.
        \item[(ii)] The subcategory $\cK^{[0, 1]}(J\injc)$ of two-term complexes in  $\cK^b(J\injc)$ is Krull--Schmidt.
        \item[(iii)] The exact category $J\modfcp$ is Krull--Schmidt.
    \end{enumerate}
\end{lemma}
 Each two-term complex in $\cK^{[0, 1]}(J\injc)$ is homotopy equivalent to a minimal complex of the form $C_{\mathrm{min}}\oplus I'\oplus I''[-1]$ (unique up to isomorphism of complexes). Here minimal means having no nonzero contractible direct summand. The complex $C_{\mathrm{min}}=(I^0\xrightarrow{f} I^1)$ is a two-term complex such that $f$ is a left and right minimal morphism in $J\injc$, and $I', I''\in J\injc$.
 We denote $M\coloneqq H^0(C_{\mathrm{min}}\oplus I')$ and the socle of $I''$ by $V$.
 Thanks to Lemma~\ref{lem:injKS}, this notation yields a bijection $\Phi$ from the set of pairs $(M, V)$, consisting of an object $M\in J\modfcp$ and a finitely generated semisimple $J$-module $V$, to the set of isoclasses of objects in $\cK^{[0, 1]}(J\injc)$.

 We consider the subcategory $\cK_{\rm fd}^{[0, 1]}(J\injc)$ of $\cK^{[0, 1]}(J\injc)$ whose $0$-th cohomology is finite-dimensional. We refer to the corresponding pairs $(M,V)$ as \emph{decorated representations} of $J$ under the above bijection.

\subsection{$g$-vector and partial $E$-invariant for decorated representations}\label{subsecA:g-E}
 We still keep the algebra $J$ as in previous subsection. 
\begin{lemma}\label{lem:fdHom}
    \begin{enumerate}
        \item[(i)] Let $M\in J\modfd$ and let $I\in J\injc$. Then we have an inequality \[\dim_{\Bbbk }\Hom_J(M, I)<\infty.\]
        \item[(ii)] Let $M^\bullet\in \Dcal^b(J\modfd)$ and let $I^\bullet\in \cK^b(J\injc)$. Then we have an inequality \[\dim_{\Bbbk }\Hom_{\Dcal^b(J\Modln)}(M^\bullet, I^\bullet)<\infty.\]
    \end{enumerate}
 \end{lemma}
 \begin{proof}
    The assertion (i) is obvious from Theorem~\ref{thm:Dequiv2}.
Since $I^\bullet\in \cK^b(J\injc)$ is $\cK$-injective in $\mathcal{K}(J\Modln)$, we have an isomorphism $\Hom_{\Dcal^b(J\Modln)}(M^\bullet, I^\bullet)\cong H^0(\Hom^\bullet_J(M^\bullet, I^\bullet))$. The assertion (ii) follows from (i).
 \end{proof}
 
 By the Krull--Schmidt properties and our discussions, we have isomorphisms $K_0(\cK^b(J\injc))\cong K_0(J\injc) \cong \mathbb{Z}^n$ and $K_0(\Dcal^b(J\modfd))\cong K_0(J\modfd)\cong \mathbb{Z}^n$ of Grothendieck groups, where $n=|Q_0|$.
 By Lemma~\ref{lem:fdHom}, for
$M^\bullet\in\mathcal D^b(J\modfd)$ and
$I^\bullet\in\mathcal K^b(J\injc)$, the vector spaces
\[
\Hom_{\mathcal D^b(J\Modln)}
(M^\bullet,I^\bullet[m])
\]
are finite-dimensional and vanish for all but finitely many
$m\in\mathbb Z$. Hence we have a bilinear pairing
\[
\langle-,-\rangle_{\rm EP}\colon
K_0(\mathcal D^b(J\modfd))
\times K_0(\mathcal K^b(J\injc))
\longrightarrow\mathbb Z
\]
defined by
\[
\langle[M^\bullet],[I^\bullet]\rangle_{\rm EP}
\coloneqq
\sum_{m\in\mathbb Z}(-1)^m
\dim_{\Bbbk}
\Hom_{\mathcal D^b(J\Modln)}
(M^\bullet,I^\bullet[m]).
\]
The pairing is well-defined since it is the Euler characteristic
of the bounded Hom-complex
$\Hom_J^\bullet(M^\bullet,I^\bullet)$, and Euler
characteristics are additive with respect to distinguished
triangles in either variable.
Note that we have an equality
\[\langle[S_i],[I_j]\rangle_{\rm EP}=\delta_{ij}\]
for each simple $J$-module $S_i$ and each indecomposable object $I_j \in J\injc$ ($i, j\in Q_0$).
 \begin{definition}[$g$-vector, $h$-vector and partial $E$-invariant]
    Let $\mathcal{M}$ be a decorated representation of $J$ and choose the minimal representative
$I_{\mathcal M}^0\to I_{\mathcal M}^1$
of $\Phi(\mathcal M)$ in $\cK^{[0,1]}(J\injc)$.
    \begin{itemize}
        \item [(i)]  The \emph{$g$-vector} $\mathbf{g}(\mathcal{M})$ of $\mathcal M$ is defined by
$$\mathbf{g}(\mathcal{M})\coloneqq -[\Phi(\mathcal M)]=[I_{\mathcal M}^1]-[I_{\mathcal M}^0]\in K_0(\cK^b(J\injc))\cong \mathbb Z^n.$$ 
\item[(ii)] The \emph{$h$-vector} $\mathbf{h}(\mathcal{M})$ of $\mathcal M$ is defined by
\[
\mathbf{h}(\mathcal{M})=-([I_{\mathcal{M}}^0\colon I_1],\ldots,[I_{\mathcal{M}}^0\colon I_n])^T\in\mathbb Z^n,
\]
where $[I_{\mathcal{M}}^0\colon I_i]$ is 
the multiplicity of $I_i$ in $I_{\mathcal{M}}^0$.
\item[(iii)] The {\em partial $E$-invariant} $E^{\rm inj}(\mathcal{M},\mathcal{M}')$ and {\em $E$-invariant} of a pair $(\mathcal{M},\mathcal{M}')$ of decorated representations of $J$ are defined by
\begin{eqnarray*}
    E^{\rm inj}(\mathcal{M},\mathcal{M}')&\coloneqq&
    \dim_{\Bbbk} \Hom_{\cK^{b}(J\mathchar`-\mathsf{inj})}(\Phi(\mathcal M),\Phi(\mathcal M')[1]),\\
     E^{\rm sym}(\mathcal{M},\mathcal{M}')&\coloneqq&  E^{\rm inj}(\mathcal{M},\mathcal{M}')+ E^{\rm inj}(\mathcal{M}',\mathcal{M}).
\end{eqnarray*}
    \end{itemize}
 \end{definition}
 
\begin{remark}
From the proof of Theorem \ref{thm:g-DWZ}, we see that the $g$-vector defined above is consistent with the definitions used in \cite[(1.13)]{DWZ10} and  \cite[\S3.1]{CLS-2015}. 
\end{remark}

  Note that  both of the categories $\cK^b(J\injc)$ and $\Dcal^b(J\modfd)$ have infinite-dimensional Hom-spaces in general. By definition, the decorated representations correspond to the two-term complexes in  $\cK_{\rm fd}^{[0,1]}(J\injc)$. In this case,  the following proposition implies that
  $$E^{\rm inj}(\mathcal M,\mathcal M')=\dim_\Bbbk \Hom_{\mathcal K^{b}(J\mathchar`-\mathsf{inj})}(\Phi(\mathcal M),\Phi(\mathcal M')[1])<\infty.$$ 
 
 \begin{proposition} \label{pro:E-hom}
 For two decorated representations $\mathcal M=(M, V)$ and $\mathcal M'=(M', V')$ of $J$, we have
\begin{eqnarray}\label{eqn:e-inv-DWZ}
  E^{\rm inj}(\mathcal M,\mathcal M')=\langle\underline{\dim} M, \mathbf{g}(\mathcal{M}')\rangle + \dim_{\Bbbk} \Hom_J(M, M' ).  
\end{eqnarray}
\end{proposition}

\begin{proof}
     We twist the Euler-Poincar\'e pairing $\langle-,-\rangle_{\mathrm{EP}} \colon K_0(\mathcal{D}^b(J\modfd))\times K_0(\mathcal{K}^b(J\injc))\to \mathbb{Z}$ as follows:
    \[\langle [M^\bullet], -[I^\bullet]\rangle_{\mathrm{EP}} =\sum_{k\in \mathbb{Z}}(-1)^{k-1}\dim_{\Bbbk }\Hom_{\mathcal{D}^b(J\Modln)}(M^\bullet, I^\bullet [k])\]
    We regard  $M= H^0(\Phi(\mathcal{M}))$ and $M'= H^0(\Phi(\mathcal{M}'))$ as objects in $\mathcal{D}^b(J\modfd)$. Then we have the following equalities:
    \begin{align*}
        \langle\underline{\dim} M, \mathbf{g}(\mathcal{M}')\rangle
        &=\langle[M], -[\Phi(\mathcal{M}')]\rangle_{\mathrm{EP}}\\
        &=-\dim_{\Bbbk}\Hom_{\mathcal{D}^b(J\Modln)}(M, \Phi(\mathcal{M}'))+\dim_{\Bbbk}\Hom_{\mathcal{D}^b(J\Modln)}(M, \Phi(\mathcal{M}')[1])\\
        &=-\dim_{\Bbbk} \Hom_J(M, M' ) + \dim_{\Bbbk }\Hom_{\mathcal{K}^b(J\mathchar`-\mathsf{inj})}(\Phi(\mathcal{M}), \Phi(\mathcal{M}')[1])\\
         &=-\dim_{\Bbbk} \Hom_J(M, M' ) + E^{\rm inj}(\mathcal M,\mathcal M').
    \end{align*}
    For the third equality, we used the standard $t$-structure in $\mathcal{D}^b(J\Modln)$ and the fact that $\Phi(\mathcal{M})$ and $\Phi(\mathcal{M}')$ are concentrated in cohomological degree $0$ and $1$. Then the result follows.
\end{proof}

\subsection{Comparison with DWZ's $g$-vectors and partial $E$-invariants}\label{subsecA:DWZgE}
The notions of $g$-vectors, $h$-vectors and (partial) $E$-invariants were first introduced by Derksen--Weyman--Zelevinsky~\cite{DWZ10}
for decorated representations of Jacobian algebras using slightly different definitions. In this subsection, we prove that the definitions given in the previous subsection are consistent with Derksen--Weyman--Zelevinsky's definitions. So in this subsection, we assume that $J$ is the Jacobian algebra of a quiver with potential $(Q,W)$.

We use $\mathbf{g}_\DWZ(\mathcal{M})$,  $\mathbf{h}_{\DWZ}(\mathcal{M})$  and $E^{\mathrm{inj}}_\DWZ(\mathcal{M}, \mathcal{M}')$ to denote $g$-vector, $h$-vector and partial $E$-invariant defined in \cite{DWZ10} by Derksen--Weyman--Zelevinsky. The reader can refer to \cite{DWZ10} or the proof of Theorem \ref{thm:g-DWZ} for the precise definitions.

    The following result is an analogue of \cite[Proposition~4.8]{Plamondon-2011}. Our argument presented here is logically independent of the theory of generalized cluster categories.
    \begin{theorem}[{\confer\cite[Proposition~4.8]{Plamondon-2011}}]\label{thm:g-DWZ}
       Let $J=J(Q,W)$ be the Jacobian algebra associated to a quiver with potential $(Q,W)$. Let $\mathcal M$ and $\mathcal M'$ be two decorated representations of $J$. Then we have \begin{enumerate}
\item[(i)]$\mathbf{g}_\DWZ(\mathcal{M})=\mathbf{g}(\mathcal{M})$;
            \item[(ii)]
            $\mathbf{h}_{\DWZ}(\mathcal{M})=\mathbf{h}(\mathcal{M})$;
            \item[(iii)]
$E^{\mathrm{inj}}_\DWZ(\mathcal{M}, \mathcal{M}')=E^{\mathrm{inj}}(\mathcal{M}, \mathcal{M}')$. In particular, $E^{\mathrm{inj}}_\DWZ(\mathcal{M}, \mathcal{M}')$ takes values in $\mathbb Z_{\geq 0}$. 
        \end{enumerate}
    \end{theorem}
    \begin{proof}
        By \cite[Proposition~4.2]{BIRSm11}, we have the following complex for each $i\in Q_0$:
    \begin{equation}\label{eq:cpxSi}
        P_i\xrightarrow{(\phi(a))_a} \bigoplus_{a\in Q_1, t(a)=i} P_{s(a)} \xrightarrow{\phi(\partial_{ba}W)_{a,b}} \bigoplus_{b\in Q_1, s(b)=i} P_{t(b)} \xrightarrow{(\phi(b))_b} P_i \to S_i \to 0.
    \end{equation}
    Here $\phi\colon \widehat{\Bbbk Q} \to J$ is the natural projection.
    The above complex is exact except at the second term from the left. Thus we know an explicit description of the first three terms of the projective resolution of the simple $J$-module $S_i$ for each $i\in Q_0$.
    We apply the functor $\Hom_J(-, M)$ to this complex for a decorated representation $\mathcal{M}=(M, V)= (\bigoplus_{i\in Q_0} M(i), \bigoplus_{j\in Q_0} V(j))$. Then we have the following complex which is consistent with \cite[(1.8)]{DWZ10}:
    \[M(i)\xrightarrow{\beta_i}M_{\mathrm{out}}(i)\xrightarrow{\gamma_i}M_{\mathrm{in}}(i) \xrightarrow{\alpha_i} M(i),\]
    where we denote $M(k)=e_k M\cong \Hom_J(P_k, M)$, and
    \begin{align*}
        M_{\mathrm{out}}(i)=\bigoplus_{b\in Q_1, s(b)=i} M(t(b)),\quad M_{\mathrm{in}}(i)=\bigoplus_{a\in Q_1, t(a)=i} M(s(a)).
    \end{align*}
    In the above, the maps $\alpha_i, \beta_i, \gamma_i$ are naturally induced maps from~\eqref{eq:cpxSi}.

    Denote $\Phi(\mathcal{M})=(I_{M}^0\xrightarrow{f} I_{M}^1)\oplus (0\to I_{V})$,
     where $I_M^0 \xrightarrow{f} I_M^1$ is the minimal injective copresentation of $M$ and $I_{V}$ is the injective envelope of $V$. We have the following equalities:
    \begin{align*}
        & \quad \dim_{\Bbbk} \ker \gamma_i - \dim_{\Bbbk} M(i) + \dim_{\Bbbk} V(i)\\ 
        &= \dim_{\Bbbk}(\ker \gamma_i/\im \beta_i) + (\dim_{\Bbbk}(\im \beta_i) - \dim_{\Bbbk} M(i)) + \dim_{\Bbbk} V(i)\\ 
        &= \dim_{\Bbbk}(\ker \gamma_i/\im \beta_i) - \dim_{\Bbbk}(\ker\beta_i) + \dim_{\Bbbk} V(i)\\
        &=\dim_{\Bbbk}\Ext^1_J(S_i, M) - \dim_{\Bbbk}\Hom_J(S_i, M) +  \dim_{\Bbbk} V(i)\\
        &=(\dim_{\Bbbk}\Ext^1_{J\Modln}(S_i, M) - \dim_{\Bbbk}\Hom_{J\Modln}(S_i, M)) +  \dim_{\Bbbk} V(i)\\
        &= [I_M^1\colon I_i]-[I_M^0\colon I_i]+ [I_V\colon I_i]\\
        &= \mathbf{g}_i(\mathcal{M}),
    \end{align*}
    where the left-hand side is the $i$-th component of $\mathbf{g}_\DWZ(\mathcal{M})$ and the right-hand side is that of $\mathbf{g}(\mathcal{M})$ by definition.
    (For the fourth equality above, we used the fact that $J\modfd\subseteq J\Modln$ and they are extension closed subcategories of $J\Mod$.)  So we have $\mathbf{g}_\DWZ(\mathcal{M})=\mathbf{g}(\mathcal{M})$.

    In the above argument, we see that $- \dim_\Bbbk(\ker\beta_i)=-[I_M^0:I_i]$,
    where the left-hand side is the $i$-th component of $\mathbf{h}_\DWZ(\mathcal{M})$ and the right-hand side is that of $\mathbf{h}(\mathcal{M})$ by definition. So we have $\mathbf{h}_\DWZ(\mathcal{M})=\mathbf{h}(\mathcal{M})$.

Since the right-hand side of \eqref{eqn:e-inv-DWZ} is $E^{\mathrm{inj}}_\DWZ(\mathcal{M}, \mathcal{M}')$ by definition in \cite[\S7]{DWZ10},   $E^{\mathrm{inj}}_\DWZ(\mathcal{M}, \mathcal{M}')=E^{\mathrm{inj}}(\mathcal{M}, \mathcal{M}')$ is exactly the equality \eqref{eqn:e-inv-DWZ}. 
    \end{proof}

\subsection{Remarks on the relationship to silting theory}
In this subsection, we assume that the base field $\Bbbk$ is algebraically closed.
We briefly remark on a relationship between Derksen-Weyman-Zelevinsky's theory and the (two-term) silting theory as an analogue of comparison in \cite[\S 7.3]{Cao-2023} in terms of partial $F$-invariants and two-term complexes.
We refer the reader to \cite{AI12,KY14,BY13} for background and terminology on silting theory not recalled here. 

Let $(Q, W)$ be a (not necessarily Jacobi-finite) quiver with non-degenerate potential. Write $Q_0=\{1, \ldots, n\}$. One can associate to $(Q,W)$ a homologically smooth, non-positive dg algebra $\Gamma\coloneqq \Gamma(Q, W)$, called the \emph{complete Ginzburg dg algebra} (see \cite{Keller-Yang_2011}). The Jacobian algebra $J= J(Q,W)$ is the $0$-th cohomology of this dg algebra $\Gamma$.
In the derived category $\mathcal{D}(\Gamma)$ of dg $\Gamma$-modules, the perfect derived category $\mathsf{per}(\Gamma)$ is the smallest thick subcategory of $\mathcal{D}(\Gamma)$ containing $\Gamma$. The triangulated category $\mathsf{per}(\Gamma)$ is known to be Krull--Schmidt \cite[Lemma~2.17]{Keller-Yang_2011} and has a silting object $\Gamma$.

Let $\mathsf{silt}(\Gamma)$ denote the set of isomorphism classes of basic silting objects in $\mathsf{per}(\Gamma)$, which is endowed with a partial order $\geq$ defined by $T \geq T'$ iff $\Hom_{\mathsf{per}(\Gamma)}(T, T'[>0])=0$. We say that a silting object $T$ is \emph{two-term} if $\Gamma \geq T \geq \Gamma[1]$. Let $\mathsf{2silt}(\Gamma)$ denote the subposet of $\mathsf{silt}(\Gamma)$ consisting of two-term silting objects.

For a basic silting object $T=\bigoplus_{i=1}^n T_i$ with indecomposable direct summands $T_i$,
we write $\mu_i^-(T)$ and $\mu_i^+(T)$ for its left and
right mutations whenever the corresponding minimal
approximations exist
\footnote{Our notation for left and right mutations is
opposite to that in \cite{AI12} and agrees with
\cite{air_2014}.}.
Let $\mathsf{2silt}^\circ(\Gamma)$ denote the set of
isomorphism classes of basic two-term silting objects
obtained from $\Gamma$ by iterated mutations through
two-term silting objects.
At these \emph{reachable} objects from $\Gamma$, the required approximations
are inductively obtained by the exchange triangles in
\cite[Theorem~3.2,  Remark~3.3]{Keller-Yang_2011} (thanks to derived equivalences mentioned below).
Precisely one of the two mutations at each summand remains
two-term, with the direction determined by
Theorem~\ref{thmA:KY}(ii) below.
Therefore we can write
$$\mathsf{2silt}^\circ(\Gamma)=\{T(t)\mid t\in\mathbb T_n\},$$
where $\mathbb T_n$ denotes the $n$-regular tree with a rooted vertex $t_0$ and $T(t_0)\coloneqq \Gamma$.

Let $\Gamma=\bigoplus_{i=1}^n\Gamma e_i$ and
$T(t)=\bigoplus_{i=1}^n T_i(t)$ be the indecomposable decompositions of $\Gamma$ and $T(t)$. We identify $K_0(\mathsf{per}(\Gamma))$ with $\mathbb Z^n$ using the standard basis
$\{[\Gamma e_i]\}_{i=1}^n$. We write ${\bf g}_i(t)\in\mathbb Z^n$ for the  $i$-th $g$-vector in the $G$-matrix $G_t^{t_0}$ of the cluster algebra $\mathcal A_Q$ and write ${\bf c}_i(t)\in\mathbb Z^n$ for the $i$-th $c$-vector in the $C$-matrix $C_t^{t_0}$. Recall that we have a family of quivers with potentials $\{(Q_t,W_t)\mid t\in\mathbb T_n\}$ obtained from $(Q_{t_0},W_{t_0})\coloneq (Q,W)$ by iterated mutations.

\begin{theorem} [\cite{Keller-Yang_2011}, {\cite[\S 4.7]{BY13} and \cite{Nag13}}]\footnote{The authors learned the contents of this theorem from Osamu Iyama.}\label{thmA:KY}
    Keep the above notation. Then we have the following:
    \begin{enumerate}
        \item[(i)]  For each $t \in \mathbb{T}_n$, the opposite of the dg endomorphism algebra of the silting object $T(t)$ is quasi-isomorphic to the complete Ginzburg dg algebra $\Gamma_t\coloneqq \Gamma(Q_t, W_t)$ and $(\End_{\mathcal{D}(\Gamma)}T(t))^\op$ is isomorphic to the Jacobian algebra $J(Q_t, W_t)$.
        \item[(ii)] The left mutation $\mu_i^-(T(t))$ (resp. right mutation $\mu_i^+(T(t))$) of $T(t)$ is again two-term if and only if ${\bf c}_{i}(t)\in\mathbb Z_{\geq 0}^n$ (resp.  ${\bf c}_{i}(t)\in\mathbb Z_{\leq 0}^n$).
        \item[(iii)] The $g$-vector ${\bf g}_i(t)$ is identified with the class $[T_i(t)]$ of  $T_i(t)$ in $K_0(\mathsf{per}(\Gamma))\cong \mathbb{Z}^n$.
    \end{enumerate}
\end{theorem}
The above theorem seems to be folklore, and we do not give a proof of this here. The main ingredients of the above theorem are the result of Keller--Yang \cite{Keller-Yang_2011} on derived equivalences among the derived categories $\mathcal{D}(\Gamma_t)$ and their restrictions induced from compact generators $T(t)$ of perfect derived categories (\confer \cite{Kel94,Kel06}), and the categorical interpretation of sign-coherence of $c$-vectors in terms of reachable two-term simple-minded collections ($=$smc's) $\{S_1(t),\ldots,S_n(t)\}$ in the derived categories $\mathcal{D}_{\rm fd}(\Gamma)$ of dg $\Gamma$-modules with finite-dimensional total cohomology (\ie smc's obtained by iterated smc mutations \cite[\S 7.2]{KY14} through two-term smc's from the initial smc $\{S_1(t_0),\ldots, S_n(t_0)\}$ given by simple objects in the (length) heart of the standard $t$-structure in $\mathcal{D}_{\rm fd}(\Gamma)$). Here a two-term smc means a smc consisting of dg modules whose cohomology is concentrated in degrees $-1, 0$. The classes of objects in each reachable two-term simple-minded collection $\{[S_1(t)],\ldots, [S_n(t)]\}$ give the dual basis of $\{[T_1(t)],\ldots,[T_n(t)]\}$ under the Euler--Poincar\'e pairing on $K_0(\mathsf{per}(\Gamma)) \times K_0(\mathcal{D}_{\rm fd}(\Gamma))$. Note that $\mathcal{D}_{\rm fd}(\Gamma)$ is $\Hom$-finite and the Euler--Poincar\'e pairing is well-defined since $\Gamma$ is homologically smooth.

We have the canonical induction functor $p_*\colon \mathsf{per}(\Gamma) \to \cK^b(J\projc)$ induced from the canonical projection $p\colon \Gamma \to J$.
By the results of Br\"ustle--Yang \cite[\S A.2.1,\S A.2.2]{BY13}\footnote{Note that \cite[\S A.2.2]{BY13} assumes $\Hom$-finiteness of $\mathsf{per}(\Lambda)$ for a non-positive dg algebra $\Lambda$,
but the same argument still works for $\mathsf{per}(\Gamma)$ because the required
approximations exist as mentioned. In particular, we can define $\mathsf{2silt}^\circ(J)$ in
$\cK^b(J\projc)$ similarly to $\mathsf{2silt}^\circ(\Gamma)$.} together with Theorem~\ref{thmA:KY},
we can immediately see the following:
\begin{corollary}
    Keep the above notation. The $g$-vector ${\bf g}_i(t)$ of the cluster variable $x_{i;t}$ of $\mathcal{A}_Q$ is identified with $[p_* T_i(t)] \in K_0(\cK^b(J\projc))\cong \mathbb{Z}^n$. In particular, the  clusters of $\mathcal{A}_Q$ are in bijection with the silting objects in $\mathsf{2silt}^\circ(J)$. 
\end{corollary}
We refer to an object $M \in \cK^{[-1,0]}(J\projc)$ as a \emph{reachable two-term presilting object} if $M$ belongs to $\add(p_* T(t))$ for some $t \in \mathbb{T}_n$, and let $\mathsf{2psilt}^\circ(J)$ denote the set of isoclasses of them.
By the same argument as in the proof of \cite[Theorem 3.8]{KM22}, each two-term presilting object is uniquely determined by its class in the Grothendieck group $K_0(\cK^b(J\projc))$.
Thanks to this uniqueness, Theorem~\ref{thm:DWZ-g-F}(iii), Theorem~\ref{thm:g-DWZ}(i), and Theorem~\ref{thmA:KY}(iii), through the functor $\nu[-1]\colon \cK^{[-1,0]}(J\projc) \to \cK^{[0,1]}(J\injc)$ and the correspondence $\Phi$ in \S\ref{subsecA:g-E}, there is a bijection between $\mathsf{2psilt}^\circ(J)$ and the set of isomorphism classes of negative-reachable decorated representations of $(Q, W)$ (see \S \ref{subsec:EF}), which is consistent with mutations. Thus we have the following:
\begin{corollary}
    Keep the above notation. Let $u=\prod_{i=1}^n x_{i;t}^{a_i}$ be a cluster monomial in $\mathcal{A}_Q$ and let $\mathcal{M}_u=(M_u, V_u)$ be the negative-reachable decorated representation of $(Q, W)$. Then we have a $J$-module isomorphism $M_u \cong \bigoplus_{i=1}^n H^0(\nu(p_*  T_i(t))[-1])^{\oplus a_i}$.
\end{corollary}

Some correspondences between cluster algebras and silting theory are given in Table~\ref{table}. Note that the partial order $\geq_{t_0}$ in the last row of Table~\ref{table} is the opposite of the one studied in \cite[\S 7.3]{Cao-2023}.
\begin{table}[htbp]
\centering
\small
\renewcommand{\arraystretch}{1.4}
\begin{tabular}{|p{\dimexpr.40\textwidth-2\tabcolsep-1.5\arrayrulewidth\relax}|p{\dimexpr.60\textwidth-2\tabcolsep-1.5\arrayrulewidth\relax}|}
\hline
\textbf{cluster algebras} & \textbf{silting theory} \\
\hline
clusters ${\bf x}_t$ & reachable two-term silting objects $T(t)$ \\
\hline
exchange matrices $B_t=B_{Q_t}$ & quivers of the opposite endomorphism algebras\newline
$(\End_{\mathcal D(\Gamma)} T(t))^\op\cong J(Q_t,W_t)$ \\
\hline
mutations of seeds & two-term silting mutations\newline
(green: left mutation; red: right mutation) \\
\hline
initial cluster variables $x_{i;t_0}$ & indecomposable direct summands $\Gamma e_i$ of $\Gamma$ \\
\hline
cluster variables $x_{i;t}$ & indecomposable direct summands $T_i(t)$ of $T(t)$ \\
\hline
cluster monomials $\prod_{i=1}^n x_{i;t}^{a_i}$ & reachable two-term presilting objects
$\bigoplus_{i=1}^n (T_i(t))^{\oplus a_i}$ \\
\hline
$g$-vectors ${\bf g}_i(t)$ & classes $[T_i(t)]\in K_0(\mathsf{per}(\Gamma))$ \\
\hline
$c$-vectors ${\bf c}_i(t)$ & classes $[S_i(t)]\in K_0(\mathcal D_{\rm fd}(\Gamma))$ \\
\hline
partial $F$-invariants\newline on cluster monomials & partial $E$-invariants $=\dim_{\Bbbk}\Hom_{\mathsf{per}(\Gamma)}(-,-[1])$ \newline on $\mathsf{2psilt}^\circ(\Gamma)$\\
\hline
partial order ${\bf x}_t  \geq_{t_0} {\bf x}_{t'}$ from  partial $F$-invariants:
$F_{u_t}^{t_0}[{\bf g}_{u_{t'}}^{t_0}]=0$, where
$u_t\coloneq \prod_{i=1}^nx_{i;t}$ & partial order $T(t)\geq T(t')$ in $\mathsf{2silt}^\circ(\Gamma)$:\newline
$\Hom_{\mathsf{per}(\Gamma)}(T(t),T(t')[>0])=0$, equivalently,\newline
$\Hom_{\mathsf{per}(\Gamma)}(T(t),T(t')[1])=0$\\
\hline
\end{tabular}
\vspace{2mm}
\caption{cluster algebras vs. silting theory}\label{table}
\end{table}
\subsection{Remarks on topologies on modules}
Let $J$ be a pseudo-compact algebra as introduced at the beginning of this section again.
As a supplement, we mention several topologies on modules over the algebra $J$, and the duality between pseudo-compact modules and locally nilpotent modules related to the functors $D$ and $D'$.
\begin{proposition}\label{lem:limisom}
    \begin{enumerate}
        \item[(i)] For each $M\in J\modfp$, there is an isomorphism
        \[M\cong \varprojlim_m (J_m\otimes_J M)\cong \varprojlim_m (M/\fM^m M).\]
        In other words, the $\fM$-adic topology on $M\in J\modfp$ is complete and separated.
        \item[(ii)] For each $N\in J\modfcp$, there is an isomorphism
        \[N\cong \varinjlim_l \Hom_J(J_l, N).\]  
    \end{enumerate}
\end{proposition}
\begin{proof}
        (i). Let $P_1 \to P_0 \to M \to 0$ be a projective presentation in $J\modfp$ of $M$. Then we have the exact sequence
        \begin{equation}\label{eq:pres1}
            J_m\otimes_J P_1 \to J_m\otimes_J P_0 \to J_m\otimes_J M \to 0.
        \end{equation}
        Since each object appearing in \eqref{eq:pres1} is finite-dimensional, the induced inverse systems satisfy the Mittag-Leffler conditions and we obtain an exact sequence
        \[\varprojlim_m (J_m\otimes_J P_1) \to \varprojlim_m (J_m\otimes_J P_0) \to \varprojlim_m (J_m\otimes_J M) \to 0.\]
        Since the algebra $J$ is complete and separated with respect to the $\fM$-adic topology, we have isomorphisms $\varprojlim_m (J_m \otimes_J P_i) \cong P_i\,(i=0, 1)$. By the five lemma, we obtain the assertion~(i).

        The assertion~(ii) is obvious since $N$ is locally nilpotent.
    \end{proof}
\begin{proposition}\label{prop:DD}
    \begin{enumerate}
        \item[(i)] For each $\fM$-adic complete object $M\in J\modfg$, we have $D'(D(M))\cong M$.
        \item[(ii)] For any object $N\in J^\op\modfcg$, we have $D(D'(N))\cong N$.
    \end{enumerate}
\end{proposition}
\begin{proof}

For  $M \in J\modfg$ and $N \in J^\op \modfcg$, we know that $J_l \otimes_J M$ and $\Hom_{J^{\op}}(J_l^\op,N)$ are finite-dimensional  $J_l$-modules and $J_l^{\op}$-modules, respectively for each $l \in \mathbb{Z}_{>0}$. Set $(-)^\vee \coloneqq \Hom_{\Bbbk }(-, \Bbbk )$.

(i) Let $M \in J\modfg$ be an $\fM$-adic complete object. We have  the following isomorphisms:
    \[
    D'(D(M)) \overset{\text{Rem~\ref{rem:dual}}}{\cong} (\varinjlim_l (J_l \otimes_J M)^\vee)^\vee 
     \cong \varprojlim_l ((J_l \otimes_J M )^\vee)^\vee 
     \cong \varprojlim_l (J_l \otimes_J M) 
    \cong M.
    \]

   (ii) By Lemma~\ref{lem:fd_fp}, we know that the finite-dimensional $J^\op$-module $J_l^\op$ is finitely presented. Thus for  $N\in J^\op\modfcg$, we have the isomorphism $J_l \otimes_J (N^\vee) \cong (\Hom_{J^{\op}}(J_l^\op,N))^\vee$ of finite-dimensional $J_l$-modules. Then we have
    \[
    D(D'(N)) \overset{\text{Rem~\ref{rem:dual}}}{\cong} \varinjlim_l (J_l \otimes_J (N^\vee) )^\vee  
    \cong \varinjlim_l ((\Hom_{J^{\op}}(J_l^\op,N))^\vee)^\vee \cong \varinjlim_l \Hom_{J^{\op}}(J_l^\op,N)
    \overset{\text{Lem \ref{lem:locnilp}}}{\cong} N.\qedhere\]
    \end{proof}
\begin{remark}\label{rem:contdual}
     A $J$-module $M$ is \emph{pseudo-compact} if and only if $M$ has a complete and separated linear topology given by a fundamental system $\{U_\mu\}_{\mu\in I}$ of (open) neighborhoods of $0$ consisting of submodules $U_\mu$ for which $M/U_\mu$ is of finite length. By the result in \cite[Proposition~2.3]{Bru66} and Lemma~\ref{lem:locnilp}, there is a contravariant equivalence between the category $J\Modpct$ of pseudo-compact topological $J$-modules and the category $J^\op\Modln$, where the morphisms in $J\Modpct$ are given by continuous $J$-module homomorphisms with respect to the pseudo-compact topology. For each $\fM$-adic complete object $M\in J\modfg$, one can see $M$ is pseudo-compact, and the pseudo-compact topology on $M$ coincides with the $\fM$-adic one since each $M/\fM^m M$ is of finite length. Note that the $J$-homomorphisms among such objects are always continuous.
    One can see our equivalences in Theorem~\ref{thm:Dequiv1} and Theorem~\ref{thm:Dequiv2} are obtained by restricting the equivalence in \textit{loc.~cit.}
\end{remark}
\begin{remark}
    Since finite-dimensional $J$-modules are compact objects in the module category thanks to Lemma~\ref{lem:fd_fp},
    one can show that the category $J\Modln$ is equivalent to the ind-completion of the category $J\modfd$ of finite-dimensional $J$-modules thanks to Lemma~\ref{lem:locnilp}. In particular, the category of pseudo-compact topological $J^\op$-modules is equivalent to the pro-completion of the category $J^\op\modfd$ via the contravariant equivalence in Remark~\ref{rem:contdual}.
    
    Dually to the compact property of finite-dimensional modules, we have an isomorphism $\varinjlim_m \Hom_J(X/\fM^m X, Y)\cong \Hom_{J}(\varprojlim_m X/\fM^m X, Y)$ for $X\in J\modfg$ and $Y\in J\modfd$ since finite-dimensional $J$-modules are nilpotent.
    As a generalization, if we take a pseudo-compact $J$-module $X\cong \varprojlim_\lambda X_\lambda$ with $X_\lambda\in J\modfd$ and an object $Y\in J\modfd$, then we have an isomorphism $\varinjlim_\lambda\Hom_J(X_\lambda, Y)\cong  \Hom_{J\Modpct}(\varprojlim_\lambda X_\lambda, Y)$.
\end{remark}
\section*{Acknowledgements}
The authors thank Hideto Asashiba, Osamu Iyama, and Bernard Leclerc for helpful discussions.
In particular, aspects of \S \ref{sec: application} were motivated by discussions during joint work with Bernard Leclerc. The authors are grateful for his collaboration. 
The second named author thanks the organizers of the conference ``Quantum Groups, Monoidal Categorification and Related Topics'' held on Jeju Island, Korea, December 2025, where he presented the results of this paper prior to the submission to arXiv.

\section*{Disclosure of AI assistance}
The authors first developed their ideas and fully wrote the manuscript by themselves.
OpenAI's GPT was used for proofreading and improving presentations including \LaTeX~formatting.
In the final stage, GPT was also used for checking some examples and logics of some proofs. A minor calculation error in \cite[Proposition~A.1]{FM} was found in this process, and it is corrected in Theorem~\ref{thm:oKR}.

\bibliography{myref}
\end{document}